\documentclass[11pt]{aart}

\usepackage[letterpaper, hmargin=1in, top=1in, bottom=1in, footskip=0.6in]{geometry}

\usepackage{titlesec}
\titleformat{\section}[block]{\filcenter\normalfont\bfseries\large}{\thesection.}{.5em}{}\titlespacing*{\section}{0pt}{2\baselineskip}{1\baselineskip}
\titleformat{\subsection}[runin]{\normalfont\bfseries}{\thesubsection.}{.4em}{}[.]\titlespacing{\subsection}{0pt}{2ex plus .1ex minus .2ex}{.8em}
\titleformat{\subsubsection}[runin]{\normalfont\itshape}{\thesubsubsection.}{.3em}{}[.]\titlespacing{\subsubsection}{0pt}{1ex plus .1ex minus .2ex}{.5em}
\titleformat{\paragraph}[runin]{\normalfont\itshape}{\theparagraph.}{.3em}{}[.]\titlespacing{\paragraph}{0pt}{1ex plus .1ex minus .2ex}{.5em}

\usepackage[T1]{fontenc}
\usepackage[utf8]{inputenc}
\usepackage{lmodern}

\usepackage[labelfont=bf,font=small,labelsep=period]{caption}
\usepackage{microtype}

\let\originalleft\left
\let\originalright\right
\renewcommand{\left}{\mathopen{}\mathclose\bgroup\originalleft}
\renewcommand{\right}{\aftergroup\egroup\originalright}

\usepackage{amsmath}
\usepackage{amssymb}
\usepackage{amsfonts}
\usepackage{latexsym}
\usepackage{amsthm}
\usepackage{amsxtra}
\usepackage{amscd}
\usepackage{bbm}
\usepackage{mathrsfs}
\usepackage{bm}
\usepackage{mathtools}

\usepackage{graphicx, color}

\definecolor{darkred}{rgb}{0.9,0,0.3}
\definecolor{darkblue}{rgb}{0,0.3,0.9}
\definecolor{purple}{rgb}{0.6,0,0.7}

\definecolor{vdarkred}{rgb}{0.6,0,0.2}
\definecolor{vdarkblue}{rgb}{0,0.2,0.6}
\usepackage[pdftex, colorlinks, linkcolor=vdarkblue,citecolor=vdarkred]{hyperref}

\usepackage{booktabs}
\usepackage[nottoc,notlof,notlot]{tocbibind}
\usepackage{cite} 

\numberwithin{equation}{section}
\numberwithin{figure}{section}

\usepackage{enumitem}

\theoremstyle{plain} 
\newtheorem{theorem}{Theorem}[section]
\newtheorem*{theorem*}{Theorem}
\newtheorem{lemma}[theorem]{Lemma}
\newtheorem*{lemma*}{Lemma}
\newtheorem{corollary}[theorem]{Corollary}
\newtheorem*{corollary*}{Corollary}
\newtheorem{proposition}[theorem]{Proposition}
\newtheorem*{proposition*}{Proposition}

\newtheorem*{conjecture*}{Conjecture}

\theoremstyle{definition} 
\newtheorem{definition}[theorem]{Definition}
\newtheorem*{definition*}{Definition}

\newtheorem*{example*}{Example}
\newtheorem{remark}[theorem]{Remark}
\newtheorem*{remark*}{Remark}

\newtheorem*{assumption*}{Assumption}

\newtheorem*{convention*}{Convention}

\newcommand{\f}{\mathbf} 
\renewcommand{\r}{\mathrm}  
\newcommand{\bb}{\mathbb} 
\renewcommand{\cal}{\mathcal} 
 
\newcommand{\fra}{\mathfrak} 
\newcommand{\ul}[1]{\underline{#1} \!\,} 
\newcommand{\ol}[1]{\overline{#1} \!\,} 
\newcommand{\wh}{\widehat}
\newcommand{\wt}{\widetilde}
\newcommand{\op}{\operatorname}

\renewcommand{\P}{\mathbb{P}}
\newcommand{\E}{\mathbb{E}}
\newcommand{\R}{\mathbb{R}}

\newcommand{\N}{\mathbb{N}}
\newcommand{\Z}{\mathbb{Z}}

\newcommand{\ee}{\mathrm{e}}

\newcommand{\dd}{\mathrm{d}}
\newcommand{\col}{\mathrel{\vcenter{\baselineskip0.75ex \lineskiplimit0pt \hbox{.}\hbox{.}}}}
\newcommand*{\deq}{\mathrel{\vcenter{\baselineskip0.5ex \lineskiplimit0pt\hbox{\scriptsize.}\hbox{\scriptsize.}}}=}
\newcommand*{\eqd}{=\mathrel{\vcenter{\baselineskip0.5ex \lineskiplimit0pt\hbox{\scriptsize.}\hbox{\scriptsize.}}}}
\newcommand{\eqdist}{\overset{\r d}{=}}

\renewcommand{\leq}{\leqslant}
\renewcommand{\geq}{\geqslant}
\renewcommand{\epsilon}{\varepsilon}

\newcommand{\floor}[1] {\lfloor #1 \rfloor}
\newcommand{\ceil}[1]  {\lceil  #1 \rceil}

\newcommand{\qq}[1]{[\![#1]\!]}

\newcommand{\ind}[1]{\mathbbm 1_{#1}}

\newcommand{\p}[1]{(#1)}
\newcommand{\pb}[1]{\bigl(#1\bigr)}
\newcommand{\pB}[1]{\Bigl(#1\Bigr)}
\newcommand{\pbb}[1]{\biggl(#1\biggr)}
\newcommand{\pBB}[1]{\Biggl(#1\Biggr)}

\newcommand{\qb}[1]{\bigl[#1\bigr]}
\newcommand{\qB}[1]{\Bigl[#1\Bigr]}
\newcommand{\qbb}[1]{\biggl[#1\biggr]}
\newcommand{\qBB}[1]{\Biggl[#1\Biggr]}
\newcommand{\qa}[1]{\left[#1\right]}

\newcommand{\h}[1]{\{#1\}}
\newcommand{\hb}[1]{\bigl\{#1\bigr\}}
\newcommand{\hB}[1]{\Bigl\{#1\Bigr\}}

\newcommand{\abs}[1]{\lvert #1 \rvert}
\newcommand{\absb}[1]{\bigl\lvert #1 \bigr\rvert}
\newcommand{\absB}[1]{\Bigl\lvert #1 \Bigr\rvert}
\newcommand{\absbb}[1]{\biggl\lvert #1 \biggr\rvert}
\newcommand{\absBB}[1]{\Biggl\lvert #1 \Biggr\rvert}

\newcommand{\norm}[1]{\lVert #1 \rVert}
\newcommand{\normb}[1]{\bigl\lVert #1 \bigr\rVert}

\newcommand{\normbb}[1]{\biggl\lVert #1 \biggr\rVert}

\newcommand{\ang}[1]{\langle #1 \rangle}

\newcommand{\scalar}[2]{\langle#1 \mspace{2mu}, #2\rangle}

\newcommand{\scalarbb}[2]{\biggl\langle#1 \,\mspace{2mu},\, #2\biggr\rangle}

\DeclareMathOperator{\diag}{diag}

\DeclareMathOperator{\supp}{supp}

\DeclareMathOperator{\dist}{dist}

\DeclareMathOperator{\spec}{spec}

\newcommand{\eps}{\varepsilon}

\newcommand*{\defeq}{\mathrel{\vcenter{\baselineskip0.5ex \lineskiplimit0pt\hbox{\scriptsize.}\hbox{\scriptsize.}}}=}

\newcommand{\id}{I} 					

\newcommand{\am}{\ensuremath{\mathfrak{a}}}

\title{Poisson statistics and localization at the spectral edge of sparse Erd\H{o}s--R\'enyi graphs}
\author{Johannes Alt \and Raphael Ducatez \and Antti Knowles}

\begin{document}

\maketitle

\begin{abstract}
We consider the adjacency matrix $A$ of the Erd{\H o}s--R\'enyi graph on $N$ vertices with edge probability $d/N$. 
For $(\log \log N)^4 \ll d \lesssim \log N$, we prove that the eigenvalues near the spectral edge form asymptotically a Poisson point process and the associated eigenvectors are exponentially localized. As a corollary, at the critical scale $d \asymp \log N$, the limiting distribution of the largest nontrivial eigenvalue does not match with any previously known distribution. Together with \cite{ADK20}, our result establishes the coexistence of a fully delocalized phase and a fully localized phase in the spectrum of $A$. The proof relies on a three-scale rigidity argument, which characterizes the fluctuations of the eigenvalues in terms of the fluctuations of sizes of spheres of radius 1 and 2 around vertices of large degree.
\end{abstract}

\section{Introduction}

\subsection{Overview}

This paper is about the eigenvalue fluctuations of the Erd\H{o}s--R\'enyi graph near the spectral edge.
In spectral graph theory, obtaining precise information on the spectral edge is of fundamental importance and has attracted much attention in the past thirty years. See for instance \cite{Chu, HLW06, Alo98} for reviews. The Erd\H{o}s--R\'enyi graph $ \mathbb{G} \equiv \mathbb{G}(N,d/N)$ is the simplest model of a random graph, where each edge of the complete simple\footnote{By \emph{simple} we mean that the graph is undirected and has no loops or multiple edges.} graph on $N$ vertices is kept independently with probability $d/N$, with $0 < d < N$. Here, $d \equiv d_N$ is a parameter whose interpretation is the expected degree of a vertex. The adjacency matrix of $\bb G$, denoted by $A = (A_{xy})_{x,y \in [N]}$, is the archetypal sparse random matrix, and its eigenvalue fluctuations have been extensively studied in the random matrix theory literature.

If $d \asymp N$ then the graph $\bb G$ is dense, and the matrix $H \deq d^{-1/2} A$ is (up to a centring) a Wigner matrix with Bernoulli entries. In that regime, the fluctuations of the extreme eigenvalues of $H$ have been analysed in great detail \cite{EYY3, Sosh1, lee2014necessary, TV2}, and are known to be governed by the universal Tracy-Widom distribution of random matrix theory.  In \cite{EKYY1, EKYY2} it is proved that the Tracy-Widom distribution persists down to $d \gg N^{2/3}$, and this result is further extended to $d \gg N^{1/3}$ in \cite{lee2018local}. A crossover appears at the scale $d \asymp N^{1/3}$, where the Tracy-Widom fluctuations for $d \gg N^{1/3}$ give way to Gaussian fluctuations for $d \ll N^{1/3}$. This phenomenon is identified in \cite{huang2020transition}, where Gaussian fluctuations are proved for $N^{2/9} \ll d \ll N^{1/3}$. In \cite{he2020fluctuations}, Gaussian fluctuations are established in the full polynomial regime $N^{o(1)} \leq d \ll N^{1/3}$.

The goal of this paper is to analyse the regime $d \lesssim \log N$. The scale $d \asymp \log N$ is well known to be critical for the Erd\H{o}s--R\'enyi graph $ \mathbb{G}$: if $d \gg \log N$ the graph is with high probability homogeneous, in the sense that its vertices all have comparable degrees, and if $d \ll \log N$ the graph is with high probability inhomogeneous, in the sense that its vertices have wildly differing degrees, which leads to the proliferation of isolated vertices, leaves, and hubs. The most famous manifestation of this transition is the well-known connectivity transition for $\bb G$ around $d = \log N$ \cite{ER59}. On the spectral side, it is proved in \cite{BBK1} that  if $d \gg \log N$ then the extreme eigenvalues of $H$ converge to the boundary of the asymptotic support $[-2,2]$ of the spectrum, while in \cite{BBK2} it is proved that if $d \ll \log N$ then they do not. The critical scale $d \asymp \log N$ is analysed in \cite{ADK19,tikhomirov2021outliers}. There, it is proved that the behaviour established in \cite{BBK1} persists down to $d > b_* \log N$, where $b_* \deq \frac{1}{\log 4 - 1} \approx 2.59$. If $1 < d < b_* \log N$, then the extreme eigenvalues of $H$ are determined by the largest degrees of $\bb G$: with high probability, each vertex $x$ with normalized degree $\alpha_x \deq \frac{1}{d} \sum_{y} A_{xy}$ greater than $2$ gives rise to two eigenvalues near $\pm \frac{\alpha_x}{\sqrt{\alpha_x - 1}}$. In other words, with high probability, there is an approximate bijection between vertices of normalized degree greater than $2$ and eigenvalues larger than $2$.

The works \cite{BBK2, ADK19, tikhomirov2021outliers}, as well as the somewhat improved bounds from \cite{ADK20}, only give weak estimates on the locations of the eigenvalues. In particular, they are far from describing the microscopic fluctuations of the eigenvalues. A particularly striking manifestation of this observation is the well-known fact \cite{Bol01} that if $d \ll \log N$ then with high probability the largest normalized degree $\max_x \alpha_x$ of $\bb G$ is deterministic, so that the approximation $\max_x \frac{\alpha_x}{\sqrt{\alpha_x - 1}}$ derived in \cite{ADK19,tikhomirov2021outliers} is also deterministic.

In this paper we derive the full joint fluctuations of the eigenvalues near the spectral edges in the regime $(\log \log N)^4 \ll d < b_* \log N$. We prove that they form asymptotically a Poisson point process with an explicit intensity. This intensity does in general not have a limit as $N \to \infty$, and its form depends strongly on the scale $d$. We refer to Figure \ref{fig:intensity} below for an illustration. In particular, for $d \ll \log N$, we show that the largest eigenvalue of $H$ has asymptotically Gumbel fluctuations, provided that $d$ stays away from a set of resonant densities. At the critical scale $d \asymp \log N$, we identify the asymptotic distribution of the largest eigenvalue of $H$, which we find to be a new law that does not satisfy 
the conclusion of the Fisher--Tippett--Gnedenko theorem.

An important observation of our proof is that the fluctuations of the extreme eigenvalues are determined not just by the degrees of the large degree vertices, but also by the sizes of the spheres of radius two around them. Following this observation, we establish very precise rigidity bounds on the extreme eigenvalues, comparing the location of each eigenvalue with an explicit function of the spheres of radii 1 and 2 around some vertex, with error bounds that are much smaller than the magnitude of the fluctuations. To that end, we develop a three-level rigidity argument, which balances the precision of the estimates with the number of vertices to which they apply.  The finest rigidity estimate is the most involved analytical part of our proof. It relies on an approximate tridiagonalization argument. Its starting point is the construction of a trial basis that strikes a delicate balance between, on the one hand, being explicit enough to yield very precise estimates on the tridiagonal, and, on the other hand, ensuring that the off-tridiagonal part is small enough. This allows us to compare the neighbourhoods of large degree vertices with rooted regular trees, where the degree of a vertex depends only on its distance to the root, and whose spectra can be analysed explicitly. It turns out, however, that the errors made in this comparison are larger than the scale of the fluctuations, and a further important ingredient of our proof is to account for deviations arising from the irregularities in the neighbourhoods of vertices of large degree. We remark that all of these steps, required to reach a sufficient degree of precision, are qualitatively novel and represent a major departure from \cite{ADK19, ADK20}. Many difficulties of the proof stem from justifying the heuristic that the independent random variables associated with the edges of the complete graph determine both the neighbourhoods of the large degree vertices as well as the random geometry connecting them; this collective contribution of the independent random variables precludes any simple structuring or splitting of the randomness. We refer to Section \ref{sec:proof_overview} for an overview of the proof.
The methods developed in this paper also apply to sparse Wigner matrices, as defined e.g.\ in \cite{ADK19, ADK20, tikhomirov2021outliers}; the details will appear elsewhere.

Another important motivation for our work is the general universality conjecture for disordered quantum systems. A \emph{disordered quantum system} is described by its Hamiltonian $H$, a self-adjoint matrix acting on a typically high-dimensional space. The eigenvalues of $H$ represent the system's energy levels, and the corresponding eigenvectors its stationary states. Disorder, arising for instance from impurities and irregularities in a medium, is mathematically modelled by randomness in $H$. The general universality conjecture for disordered quantum systems states that the spectrum of $H$ can consist of two distinct phases:
(i) the \emph{localized (or insulating) phase}, where the local eigenvalue statistics are \emph{Poisson} and the associated eigenvectors are \emph{localized};
and (ii) the \emph{delocalized (or metallic) phase}, where the local eigenvalue statistics are governed by \emph{random matrix theory} and the associated eigenvectors are \emph{delocalized}.
The archetypal model expected to exhibit both phases is the Anderson model \cite{anderson1958absence}, for which Anderson famously conjectured that in dimensions $d > 2$ and for small enough disorder, the spectrum splits into a delocalized phase in the bulk of the spectrum and a localized phase near the edges of the spectrum (for large disorder, the delocalized phase disappears). Much progress has been achieved in the localized phase \cite{FroSpe, AizMol, goldsheid1977random, Min}, but the delocalized phase remains wide open.

Transitions in the localization behaviour of eigenvectors have also been analysed in several models of Wigner matrices. In \cite{LS1, LS2} the authors consider the sum of a Wigner matrix and a diagonal matrix with independent random entries with a large enough variance. They show that the eigenvectors in the bulk are delocalized while near the edge they are partially localized at a single site. Heavy-tailed Wigner matrices, or Lévy matrices, whose entries have $\alpha$-stable laws for $0 < \alpha < 2$, were proposed in \cite{CB1} as a simple model that exhibits a transition in the localization of its eigenvectors. They have been extensively studied in the physics and mathematics literature; we refer to \cite{ALY1} for a summary of the predictions from \cite{CB1, TBT1}. In \cite{BG1, BG2} it is proved that for energies in a compact interval around the origin, eigenvectors are weakly delocalized, and for $0 < \alpha < 2/3$ for energies far enough from the origin, eigenvectors are weakly localized. In \cite{ALY1}, full delocalization is proved in a compact interval around the origin, and the authors even establish GOE local eigenvalue statistics in the same spectral region. In \cite{ALM1}, the law of the eigenvector components of Lévy matrices is computed. Moreover, the fluctuations of the extreme eigenvalues are determined in \cite{ABP, SoshPoi}, where they are shown to form asymptotically a Poisson process with power law intensity measure.

In \cite{ADK20}, we proposed the Erd\H{o}s-R\'enyi graph at and below criticality, $d < b_* \log N$, as a natural and attractive new model on which to analyse the phase coexistence stipulated by the above universality conjecture. To the best of our knowledge, the coexistence of phases at different energies in this model had not been previously analysed even in the physics literature.  It has two features that make it particularly appealing: its graph structure provides an intrinsic and nontrivial notion of distance, and it is amenable to rigorous analysis, including a proof of phase coexistence. It is proved in \cite{ADK20} that for $\sqrt{\log N} \ll d < b_* \log N$ the spectrum of $H$ splits into two phases: a \emph{delocalized phase} in the bulk of the spectrum and a \emph{semilocalized phase} in its complement. The delocalized phase is characterized by completely delocalized eigenvectors $\f w$, in the sense that $\norm{\f w}^2_\infty / \norm{\f w}^2_2 \leq N^{-1 + o(1)}$. The semilocalized phase is characterized by eigenvectors satisfying $\norm{\f w}^2_\infty / \norm{\f w}^2_2 \geq N^{-\gamma(\f w) + o(1)}$ for some (explicit) $\gamma(\f w) < 1$.

In this paper we prove the existence of a fully \emph{localized phase} phase near the spectral edge, by establishing both hallmarks given above -- Poisson statistics and eigenvector localization. The localization holds in a strong sense: exponential decay around a unique vertex. We also show that each localized eigenvector is approximately radial.
Together with \cite{ADK20}, we have therefore rigorously established the coexistence of a fully delocalized phase and a fully localized phase in the spectrum of $H$.

\paragraph{Conventions}
Every quantity that is not explicitly \emph{constant} depends on $N$. We omit this dependence in our notation. We use $C,c$ to denote generic positive constants, which may change from step to step. We write $X \lesssim Y$ or $X = O(Y)$ to mean 
$X \leq C Y$. We write $X \asymp Y$ to mean 
$X \lesssim Y$ and $Y \lesssim X$. Moreover, we write $X \ll Y$ or $X = o(Y)$ to mean $X/Y \to 0$ as $N \to \infty$. We say that an event $\Omega$ holds with \emph{high probability} if $\P(\Omega) \to 1$ as $N \to \infty$.

\subsection{Results}
In order to describe the appropriate rescaling of the eigenvalue process, we need a few definitions. Define the function $f \equiv f_d$ on $[1,\infty)$ through
\begin{equation} \label{def_f}
f(u) \deq u \log u - (u  - 1) + \frac{1}{d} \log \sqrt{2 \pi d u}\,.
\end{equation}
Clearly, $f$ is increasing and $f(2) = \frac{1}{b_*} + O(\frac{\log d}{d})$, where we recall the definition $b_* \deq \frac{1}{\log 4 - 1}$.
 This function is well known, with the following interpretation: if $\cal P_d$ is a Poisson random variable with expectation $d$ then by Stirling's approximation we have
\begin{equation} \label{Poisson_f}
\P(\cal P_d = k) = \ee^{-df(k/d)} \pbb{1 + O\pbb{\frac{1}{k}}}
\end{equation}
for $k \in \N$. We denote by $\fra u \equiv \fra u(d,N) > 2$ the unique solution of the equation
\begin{equation} \label{def_u}
f(\fra u) = \frac{\log N}{d}\,.
\end{equation}
The number $d \fra u$ has the interpretation of the typical maximal degree of the graph, since the distribution of any degree of the graph is approximately $\cal P_d$.

We introduce the parameters 
\begin{equation} \label{def_tau_etc}
\tau(\fra u) \deq \frac{2(\fra u - 1)^{5/2}}{\fra u^{1/2} (\fra u - 2)}\,, \qquad
\theta(\fra u) \deq \frac{\fra u - 1}{\fra u^{1/2}}\,, \qquad 
\sigma(\fra u) \deq \sqrt{\fra d} \Lambda\p{\fra u / \fra d, 1/\fra d}\,, \qquad \fra d \deq 1 + \frac{1}{d}\,,
\end{equation}
where, for any $\alpha \geq 2$ and $\beta \geq 2(\sqrt 2 - 1)$, we set 
\begin{equation} \label{eq:def_Lambda} 
\Lambda (\alpha,\beta) \deq \alpha\pbb{\alpha-\frac{\beta}{2}(\alpha+\beta)+\frac{\beta}{2}\sqrt{(\alpha+\beta)^{2}-4\alpha}}^{-1/2}\,. 
\end{equation}
The function $\Lambda$ can be naturally interpreted in terms of the largest eigenvalue of the infinite $(p,q,s)$-regular tree (see Appendix \ref{app:spectral_analysis}), or, alternatively, of the infinite tridiagonal matrix $Z_1(\alpha,\beta)$ defined in \eqref{eq:def_Zd} below. We refer to Appendix \ref{app:spectral_analysis} for details and further properties of $\Lambda(\alpha, \beta)$. The parameter $\sigma(\fra u)$ represents the typical location of the largest nontrivial eigenvalue of $H$, while $\frac{1}{d \tau(\fra u)}$ represents the typical eigenvalue spacing of $H$ near the spectral edge. Thus, we shall rescale the eigenvalue process according to the following definition.

\begin{definition}[Eigenvalue process near right edge] \label{def:def_H_and_Phi_process_near_edge}
We define the rescaled eigenvalue process of $H \deq A / \sqrt{d}$ near the right edge as
\begin{equation*}
\Phi \deq \sum_{\lambda \in \spec(H)} \delta_{d\tau(\fra u) (\lambda - \sigma(\fra u))} \,.
\end{equation*}
\end{definition}

Denote by $g(s) \deq \frac{1}{\sqrt{2 \pi}} \ee^{-\frac{1}{2} s^2}$ the density of the standard Gaussian.

\begin{definition}[Poisson reference process]
Define the intensity measure $\rho$ on $\R$ through
\begin{equation} \label{def_rho}
\rho(\dd s) \deq \sum_{\ell \in \Z} \fra u^{d \fra u + \ell} \, g \pb{s + \theta(\fra u) (d \fra u + \ell)} \, \dd s\,,
\end{equation}
and denote by $\Psi$ the Poisson point process on $\R$ with intensity measure $\rho$.
\end{definition}

Our first main result states that $\Phi$ is close to $\Psi$ as $N \to \infty$. The convergence holds in the region $[-\kappa, \infty)$ containing an expected number $\cal K$ of rescaled eigenvalues. 
Thus, for given $\cal K$ we define
\begin{equation}  \label{def_kappa}
\kappa \deq - \inf \{ s \in \R \col \rho([s,\infty)) \leq \cal K\}\,.
\end{equation}
An elementary argument shows that for $\cal K$ large enough we always have $\kappa > 0$, uniformly in $N$ and $1 \leq d \leq 3 \log N$.

In general neither process $\Phi$ or $\Psi$ has a limit as $N \to \infty$, and we establish asymptotic closeness with respect to the standard metric $\cal D_\kappa$ on the space of point processes on $[-\kappa, \infty)$ defined through
\begin{equation*}
\cal D_\kappa(\Phi, \Psi) \deq \sum_{n \in \N^*} 2^{-n} \,  \sup_{s_1, \dots, s_n \geq -\kappa} \; \sup_{k_1, \dots, k_n \in \N} \, \absBB{\P\pBB{\bigcap_{i \in [n]} \{\Phi([s_i,\infty)) \geq k_i\} } - \P\pBB{\bigcap_{i \in [n]} \{\Psi([s_i,\infty)) \geq k_i\} }} \,.
\end{equation*}

\begin{theorem}[Poisson statistics] \label{thm:point_process}
For any constant $\zeta > 4$ and any small enough constant $\xi > 0$ the following holds. Suppose that
\begin{equation} \label{d_assumptions}
(\log \log N)^{\zeta} \leq d \leq (b_* - (\log N)^{-\xi}) \log N\,.
\end{equation}
Define
\begin{equation} \label{calK_choice}
\cal K \deq d^{1/2 - 2 / \zeta - 16 \xi}
\end{equation}
and let $\kappa$ be as in \eqref{def_kappa}. Then there exists an eigenvalue $\nu \in \spec(H)$ such that
\begin{equation*}
\cal D_\kappa\pb{\Phi - \delta_{d\tau(\fra u) (\nu - \sigma(\fra u))} , \Psi} \longrightarrow 0
\end{equation*}
as $N \to \infty$.
\end{theorem}

\begin{remark}[Stray eigenvalue] \label{rem:stray} 
We call the eigenvalue $\nu$ from Theorem \ref{thm:point_process} the \emph{stray eigenvalue}. It is approximately equal to $d^{1/2}$ with high probability; see Corollary~\ref{Cor:SpecMax} below for a precise statement. For $d^2 \gg \frac{\log N}{\log \log N}$, it is the largest eigenvalue, an outlier separated from the other eigenvalues, and coincides with the well-known outlier eigenvalue of the dense $\bb G(N,d/N)$. For $d^2 \lesssim \frac{\log N}{\log \log N}$, it is near or inside the main spectrum, and it is no longer the largest eigenvalue. 
All of these claims follow easily from the asymptotics $\sigma(\fra u) \asymp \sqrt{\fra u}$ for $\fra u \gg 1$ combined with $\fra u \asymp \frac{\log N}{d \log \log N}$ for $d \ll \log N$, as follows from Lemma \ref{lem:u_a} below.

In fact, in the regime where the stray eigenvalue is in the region $[-\kappa, \infty)$, it fluctuates on a much smaller scale than the other eigenvalues. In particular, its fluctuations are negligible compared to the scale $1 / d \tau(\fra u)$ of the rescaled eigenvalue process $\Phi$. This observation follows from Corollary \ref{Cor:SpecMax}, combined with the behaviour of $\Lambda$ from Lemma \ref{lem:properties_Lambda} below and the bound on $\kappa$ from Lemma \ref{lem:kappa_condition} below, which imply that if $d < (\log N)^{1/2 - c}$ for some constant $c > 0$ then the stray eigenvalue is outside of the region $[-\kappa, \infty)$.

Thus, Theorem \ref{thm:point_process} combined with Corollary \ref{Cor:SpecMax} give a detailed picture of the transition at the scale $d^2 \asymp \frac{\log N}{\log \log N}$, where the stray eigenvalue enters the main spectrum.  Our result also gives a precise description of the crossover between the two regimes for the largest eigenvalue of $H$ first discussed in \cite{krivelevich2003largest}.

Finally, a straightforward extension of our analysis in Sections \ref{sec:block_diagonal_approximation} and \ref{sec:proof_thm_localization} shows that, with high probability, the eigenvector associated with the stray eigenvalue is delocalized and close (in the Euclidean) norm to the flat vector $N^{-1/2} (1, \dots, 1)$. This is in stark contrast to all other eigenvectors near the spectral edge, which are localized by Theorem \ref{thm:localisation} below. We omit further details.
\end{remark}

\begin{figure}[!hb]
\begin{center}
{\small 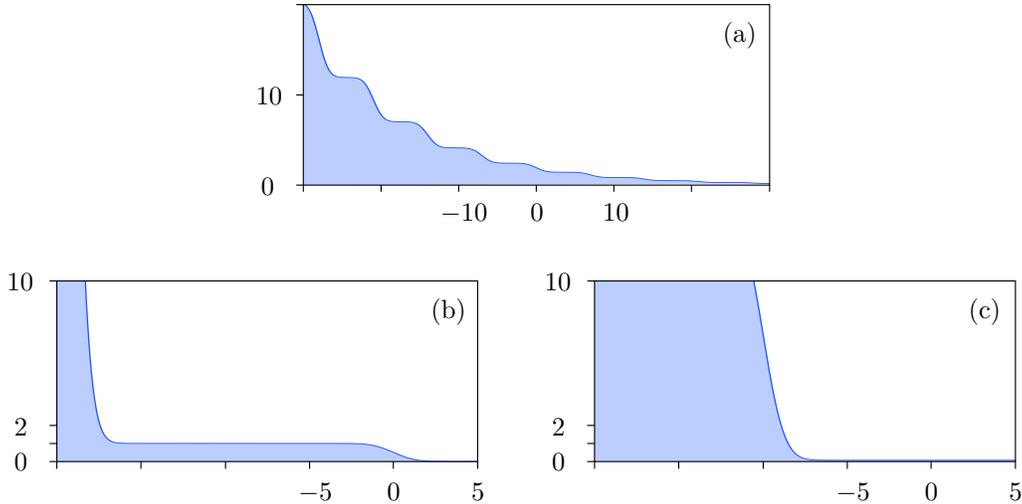}
\end{center}
\caption{
An illustration of the tail distribution function $s \mapsto \rho([s,\infty))$ of the intensity measure $\rho$ defined in \eqref{def_rho} and discussed in Remark \ref{rem:extreme_evs}. We plot the three cases: (a) the critical regime, $d \asymp \log N$; (b) the subcritical resonant regime, $d \ll \log N$ and $\ang{d \fra u}= 0$; (c) the subcritical nonresonant regime, $d \ll \log N$ and $\abs{\ang{d \fra u}} \geq c$ for some constant $c > 0$.
\label{fig:intensity}}
\end{figure}

\begin{remark}[Fluctuations of extreme eigenvalues] \label{rem:extreme_evs}
An immediate corollary of Theorem \ref{thm:point_process} is the convergence of the joint law of any bounded number of eigenvalues with index up to $\cal K / 2$. Denote by $\lambda_1 \geq \lambda_2 \geq \cdots \geq \lambda_{N-1}$ the eigenvalues in $\spec(H) \setminus \{\nu\}$, and for simplicity consider only the distribution of a single eigenvalue. We have, uniformly for $k \in \N^*$ and $s \geq -\kappa$,
\begin{align*}
\P(d \tau(\fra u)(\lambda_k - \sigma(\fra u)) \geq s) &= \P(\Phi([s, \infty)) \geq k) = \P(\Psi([s, \infty)) \geq k) + o(1)
\\
&= \ee^{-\rho([s,\infty))} \sum_{l \geq k} \frac{1}{l!} \rho([s,\infty))^l + o(1)\,.
\end{align*}
Since the right-hand side is $1 - o(1)$ for $s = -\kappa$ and $k \leq \cal K / 2$ by Lemma \ref{lem:Bennett} below (applied to a Poisson random variable with expectation $\rho([s,\infty))$), we therefore obtain the fluctuations of $\lambda_k$ for any $k \leq \cal K/2$.

We discuss the distribution of the rescaled top eigenvalue $X_1 \deq d \tau(\fra u)(\lambda_1 - \sigma(\fra u)) $ in more detail. Its distribution function is $\P(X_1 \leq s) = \ee^{-\rho([s,\infty))} + o(1)$. 
To analyse the right-hand side, we denote by $\ang{x} \in [-1/2, 1/2)$ the $1$-periodic representative of $x \in \R$, and distinguish three cases: (a) the critical regime, $d \asymp \log N$; (b) the subcritical resonant regime, $d \ll \log N$ and $\ang{d \fra u}= 0$; (c) the subcritical nonresonant regime, $d \ll \log N$ and $\abs{\ang{d \fra u}} \geq c$ for some constant $c > 0$. The tail distribution function $s \mapsto \rho([s,\infty))$ in each of the three cases is illustrated in Figure~\ref{fig:intensity}.
In the critical regime (a), where $\fra u \asymp 1$, we suppose that $\fra u \to \bar {\fra u}$ and $\ang{d \fra u} \to h$. Then we find that, for all $s \in \R$,
\begin{equation} \label{critical_distribution}
\lim_{N \to \infty} \P \pb{X_1 \leq s} = \exp\pbb{-
\sum_{\ell \in \Z} \bar {\fra u}^{h + \ell} \, G \pb{s + \theta(\bar {\fra u}) (h + \ell)}}\,,
\end{equation}
where $G (s) \deq \int_s^\infty \dd t \, g(t)$ is the Gaussian tail distribution function. The right-hand side of \eqref{critical_distribution} is a distribution function on $\R$ that seems not to have appeared previously in the literature. In particular, $X_1$ does not satisfy the conclusion of the Fisher--Tippett--Gnedenko theorem of classical extreme value theory.
In the subcritical resonant regime (b), where $\fra u \to \infty$, one easily finds that $\rho([s,\infty)) = \fra u G(s + \theta(\fra u)) + G(s) + o(1)$ provided that $s$ is chosen so that this expression is $O(1)$. We conclude that the distribution of $X_1$ does not have a limit. Instead, it is a mixture of two distributions on different scales: asymptotically, with probability $1 - 1/\ee$, the variable $X_1$ has a standard normal distribution, and with probability $1/\ee$ it has a Gumbel distribution on the scale $1 / \sqrt{\log \fra u}$ around $-\theta(\fra u) + O(\sqrt{\log \fra u})$.
Finally, in the subcritical nonresonant regime (c), it is easy to see that, after a suitable affine rescaling, $X_1$ has asymptotically a Gumbel distribution.

\end{remark}

\begin{remark}[The left edge]  \label{rem:left_edge_intro}
An analogous result to Theorem \ref{thm:point_process} holds at the left edge of the spectrum. In fact, near the spectral edges the spectrum of $H$ is with high probability approximately symmetric: in the notation of Remark \ref{rem:extreme_evs}, $\abs{\lambda_k  + \lambda_{N - k}}$ is with high probability much less than the scale of the fluctuations of $\lambda_k$ for $k \leq \cal K/2$. In particular, the correlation coefficient of $\lambda_k$ and $-\lambda_{N - k}$ is $1 - o(1)$ for any $k \leq \cal K/2$. A more precise formulation may be found in Corollary \ref{Cor:SpecMax} and Remark \ref{rem:left_edge} below. As a consequence, the point process near the left edge,
$\Phi_- \deq \sum_{\lambda \in \spec(H)} \delta_{d\tau(\fra u) (-\lambda - \sigma(\fra u))}$,
satisfies $\cal D_\kappa(\Phi_-,\Psi) \to 0$ under the assumptions of Theorem \ref{thm:point_process}.
\end{remark}

We now move on to the eigenvectors of $H$. We denote by $B_i(x)$ the ball of radius $i$ around $x$, and by $S_i(x)$ the sphere of radius $i$ around $x$. We denote by $\f w |_X$ the restriction of the vector $\f w$ to the set $X$. 

\begin{theorem}[Localization] \label{thm:localisation}
For any constant $\zeta > 4$ and any small enough constant $\xi > 0$ the following holds with high probability.
Suppose that \eqref{d_assumptions} holds.
Let $\f w$ be the $\ell^2$-normalized eigenvector associated with any of the  largest $d^{1/4 - 1 / \zeta - 8 \xi}$ eigenvalues of $H$ except $\nu$.
Then $\f w$ is exponentially localized around some vertex $x$ in the sense that for all $i \geq 0$ we have
\begin{equation*}
\norm{\f w |_{B_i(x)^c}} \lesssim \frac{q^i}{(1 - q)^2}
+ \frac{1}{(d \fra u)^{8}} \,, \qquad q = \frac{2 + O(d^{-1/3})}{\sigma(\fra u)}\,.
\end{equation*}
\end{theorem}

Note that in Theorem \ref{thm:localisation} we have $q < 1$ since $\sigma(\fra u) -2 \gtrsim d^{-2\xi}$, as follows from the definition of $\sigma(\fra u)$, \eqref{Lambda_geq_2_est} below, and \eqref{a_2_lower_bound} below.

\begin{remark}[Approximate structure of $\f w$] \label{rem:localization}
The eigenvector $\f w$ in Theorem \ref{thm:localisation} can be very precisely approximated by the eigenvector $\f w(x)$ of the largest eigenvalue of the matrix obtained by restricting $H$ to the ball $B_r(x)$ for some suitably chosen radius $r$: $\norm{\f w - \f w(x)} \leq (d \fra u)^{-8}$. Moreover, $\f w(x)$ is approximately radial and exponentially decaying, in the sense that
\begin{equation*}
\normbb{\f w(x) - \sum_{i = 0}^r u_i(x) \f s_i(x)} \lesssim  \frac{ d^{-1/2 +  3 \xi}}{\sqrt{\fra u}} 
 +  \frac{1}{d} \,, \qquad
\f s_i(x) \deq \frac{\f 1_{S_i(x)}}{\norm{\f 1_{S_i(x)}}}\,,
\end{equation*}
where the coefficients $(u_i(x))_{i \geq 0}$ are determined by
\begin{equation*}
u_1(x) = \frac{\sigma(\fra u)}{\sqrt{\alpha_x}}\, u_0(x) \,,  \qquad \qquad u_k(x) = \pbb{\frac{2}{\sigma(\fra u) + \sqrt{\sigma(\fra u)^2 - 4}}}^{k - 1} \, u_1(x) \quad (k \geq 2)\,,
\end{equation*}
and $\f 1_{S_i(x)}$ is the vector equal to $1$ on $S_i(x)$ and $0$ elsewhere.
\end{remark}

\begin{remark}
Finally, we comment on some straightforward extensions of Theorem \ref{thm:localisation}.
\begin{enumerate}[label=(\roman*)]
\item
The error bound $(d \fra u)^{-8}$ in Theorem \ref{thm:localisation} and Remark \ref{rem:localization} can be easily improved to any constant power of $(d \fra u)^{-1}$, by the same proof.
\item
Theorem \ref{thm:localisation} also applies to the $d^{1/4 - 1 / \zeta - 5 \xi}$ smallest eigenvalues of $H$, with minor modifications that we omit; see also Remark \ref{rem:left_edge_intro}.
\item
The number of eigenvalues covered by Theorem \ref{thm:localisation} is $\sqrt{\cal K}$, where $\cal K$ is the number of eigenvalues covered by Theorem \ref{thm:point_process}.
This stronger constraint arises from a union bound needed for the simultaneous statement of Theorem \ref{thm:localisation}. By the same proof, we also find that localization holds for all of the largest $\cal K$ eigenvalues contained in any deterministic interval $I$ satisfying $\rho(I) = O(1)$.
\end{enumerate}
\end{remark}

We conclude this section with an overview of the structure of the paper. In Section \ref{sec:proof_overview}, we introduce basic definitions used throughout the paper and give an overview of the proof. Section \ref{sec:block_diagonal} contains the statement of our main rigidity bounds, which are proved in Sections \ref{sec:IntermediateRigidity}--\ref{sec:block_diagonal_approximation}. Theorem~\ref{thm:point_process} is proved in Section \ref{sec:ev_process}, and Theorem~\ref{thm:localisation} in Section \ref{sec:proof_thm_localization}. Section \ref{sec:graph_structure} establishes some properties of  the graph $\bb G$ which are used throughout Sections \ref{sec:IntermediateRigidity} and \ref{sec:fine_rigidity}.

\section{Basic definitions and overview of proof} \label{sec:proof_overview}

In this preliminary section we introduce some basic notations and definitions that are used throughout the paper, and give an overview of the proofs of Theorems \ref{thm:point_process} and \ref{thm:localisation}.

\subsection{Basic definitions} \label{sec:basic_definitions} 

We write $\N = \{0,1,2,\dots\}$ and $\N^* = \{1,2,3, \dots\}$. We set $[n] \defeq \{1, \ldots, n\}$
for any $n \in \N^*$ and $[0] \defeq \emptyset$.
We write $\abs{X}$ for the cardinality of a finite set $X$. 
We use $\ind{\Omega}$ as symbol for the indicator function of the event $\Omega$. 

Vectors in $\R^N$ are denoted by boldface lowercase Latin letters such as $\f u$, $\f v$ and $\f w$. We use the notation $\f v = (v_x)_{x \in [N]} \in \R^N$ for the entries of a vector. We denote by $\supp \f v \deq \{x \in [N] \col v_x \neq 0\}$ the support of a vector $\f v$. We denote by $\scalar{\f v}{\f w} = \sum_{x \in [N]} v_x w_x$ the Euclidean scalar product on $\R^N$ and by $\norm{\f v} = \sqrt{\scalar{\f v}{\f v}}$ the induced Euclidean norm. 
For a matrix $M \in \R^{N \times N}$, $\norm{M}$ is its operator norm induced by the Euclidean norm on $\R^N$.
For any $x \in [N]$, we define the standard basis vector $\f 1_x \defeq (\delta_{xy})_{y \in [N]} \in \R^N$.
To any subset $S \subset [N]$ we assign the vector $\f 1_S\in \R^N$ given by $\f 1_S \defeq \sum_{x \in S} \f 1_x$. 
In particular, $\f 1_{\{ x\}} = \f 1_x$.

For $r \in \N$ and $x \in [N]$, we denote by $B_r(x)$ the ball of radius $r$ around $x$, i.e.\ the set of vertices whose graph distance from $x$ in $\bb G$ is at most $r$. For $r \in \N^*$, we denote by $S_r(x) \deq B_r(x) \setminus B_{r-1}(x)$ the sphere of radius $r$  around $x$.

For $X \subset [N]$, we denote by $\bb G |_X$ the restriction of the graph $\bb G$ to the vertex set $X$, i.e.\ the set of edges of $\bb G$ whose incident vertices are both in $X$. Similarly, we denote by the $H |_X$ the restriction of $H$ to the set $X$, so that $(H|_X)_{xy} = H_{xy} \ind{x \in X} \ind{y \in X}$.
 
The degree of a vertex $x \in [N]$ is $D_x \deq \abs{S_1(x)}$. We also define
\begin{equation} \label{eq:def_alpha_x_beta_x} 
\alpha_x \deq \frac{\abs{S_1(x)}}{d}\,, \qquad \beta_x \deq \frac{\abs{S_2(x)}}{d \abs{S_1(x)}}\,,
\end{equation}
which have the interpretation of the normalized degree of $x$ and the normalized size of the $2$-sphere, $S_2(x)$, 
around $x$, respectively.

\subsection{Overview of proof} \label{sec:overview_proof_subsect}
Throughout this subsection, we use $c > 0$ to denote a small positive constant.
The central phenomenon underlying our proof is that in the regime $d\leq (b_* - c) \log N$ the extreme eigenvalues of $H$ are associated with distinct neighbourhoods of vertices of large degree. Thus, we identify a family of random variables, which approximate the extreme eigenvalues of $H$ with a better precision than the scale of the eigenvalue fluctuations, and  whose joint distribution can be explicitly identified. In fact, variables of this family can be regarded as the approximate top and bottom eigenvalues of the graph restricted to small balls around the vertices of large degree.

A much simpler and less precise instance of this phenomenon was identified in \cite{ADK19,ADK20, tikhomirov2021outliers} (see also \cite{BBK2}), where it was proved that each extreme eigenvalue $\lambda$ of $H$ satisfies a rigidity estimate of the form $\lambda=\Lambda(\alpha_{x})+\epsilon_{x}$ for some vertex $x \in [N]$ of large normalized degree $\alpha_x$, with $\abs{\epsilon_{x}}\ll 1$ and $\Lambda(\alpha) = \frac{\alpha}{\sqrt{\alpha -1}}$. However, the estimates obtained in \cite{ADK19,ADK20,tikhomirov2021outliers} are far from being able to capture the actual fluctuations of the eigenvalues. This is made particularly obvious in the case $d \ll \log N$, where it is well known \cite{Bol01} that the top degree is typically almost surely deterministic.

Hence, in order to determine the fluctuations of the extreme eigenvalues, we need a much more refined analysis, which entails identifying a suitable local function of the neighbourhoods of vertices of large degree, and establishing rigidity bounds on a scale much smaller than the eigenvalue fluctuations. It turns out that, in order to reach the required precision, we have to establish rigidity bounds on three different scales, which we call \emph{fine}, \emph{intermediate}, and \emph{rough}, in decreasing order of precision and increasing number of vertices to which they apply.

The three-scale approach is required because of a competition between the precision of an estimate and the strength of the corresponding probability bounds. The latter prevents a simultaneous statement for many vertices. Thus, the fine rigidity bounds are sufficiently precise to identify the fluctuations, but only hold with weak probability bounds. They can hence only be applied to a rather small set of vertices. To ensure that the remaining eigenvalues do not lie in the region of the spectrum near the largest eigenvalue, we need less precise rigidity estimates on the complementary set of vertices. These estimates are not precise enough to identify the fluctuations, but they hold with stronger probability bounds. It turns out that such rigidity estimates have to be established on two levels, the weakest of which, rough rigidity, is essentially on the level of \cite{ADK20} and holds with very high probability.

In the following we give more details on this argument, focusing for simplicity on the largest scale $d \asymp \log N$, where the scale of the fluctuations of the extreme eigenvalues of $H$  is of order $d^{-1}$ (by Theorem \ref{thm:point_process}). How to deal with sparser graphs is remarked on at the end of this subsection. Moreover, as a further simplification, throughout this subsection we ignore the stray eigenvalue (see Remark~\ref{rem:stray}).

\paragraph{Fine rigidity}
The fine rigidity estimates are established on the set of vertices of large degree
\[
\cal W = \{ x \in [N] \colon \alpha_x \geq \fra u - \eta_{\cal W}\}\,,
\] 
where $\fra u$ is the typical normalized maximal degree defined in \eqref{def_u}, and $\eta_{\cal W}$ is an appropriately chosen cutoff parameter\footnote{The actual choice of $\eta_{\cal W}$ will be made precise in the precise definition of $\cal W$; see \eqref{eq:def_cal_W} below.}. The main work is to construct, for each $x \in \cal W$, a random variable $\Lambda_x$ and a normalized vector $\f w(x)$ supported in a small neighbourhood of $x$ such that, with high probability, $\norm{(H-\Lambda_x) \f w(x)} = O( d^{-1-c})$ for each $x \in \cal W$.

Our construction of $\Lambda_x$ and $\f w(x)$ proceeds by an analysis of the graph $\mathbb{G} |_{B_r(x)}$, the restriction of $\bb G$ to the ball $B_r(x)$ of an appropriately chosen radius $r$ around a vertex $x\in \cal W$. The guiding principle of our estimates is to approximate $\mathbb{G} |_{B_r(x)}$ with a rooted tree that is \emph{regular} in the sense that the degree of a vertex depends only on its distance to the root. As we shall see, however, this approximation is too crude to capture the correct fluctuations, and an important element of our proof is a precise analysis of the deviations of $\mathbb{G} |_{B_r(x)}$ from such a regular tree.

We start by establishing some graph-theoretic properties of $\mathbb{G} |_{B_r(x)}$, which are collected in Proposition~\ref{pro:graphProperty} below. We show that with high probability, $\mathbb{G} |_{B_r(x)}$ is a tree satisfying a host of estimates stating, informally, that all vertices except the root have approximately degree $d$.  These properties hold with high probability simultaneously for all $x \in \cal W$. We note that this latter fact is crucial for our argument, and its validity determines the maximal size of $\cal W$ through $\eta_{\cal W}$. Our choice for the size of $\cal W$ is essentially optimal, as the concentration bounds in Proposition~\ref{pro:graphProperty} rely on near-sharp large deviation estimates and the maximal size of $\cal W$ is obtained from a union bound for almost independent events.

Using this information about the structure of $\mathbb{G} |_{B_r(x)}$, we analyse the spectrum of $M \deq H|_{B_r(x)}$ for $x \in \cal W$. The starting point of this analysis is the tridiagonalization of $M$ around the vertex $x$ (see e.g.\ \cite[Appendix A]{ADK19}). This amounts to writing $M$ in the basis $\f h_0, \f h_1, \f h_2, \dots$ obtained from orthogonalizing the sequence $\f 1_x, M \f 1_x, M^2 \f 1_x, \dots$. Unfortunately, this simpleminded approach, as used e.g.\ in \cite{ADK19}, faces a major obstacle arising from the irregularity of the tree $\mathbb{G} |_{B_r(x)}$. To understand it, we note that if $\mathbb{G} |_{B_r(x)}$ were a regular tree, where the degree of a vertex depends only on its distance to the root, then we would simply have $\f h_i = \f 1_{S_i(x)}$. However, owing to the irregularity of $\mathbb{G} |_{B_r(x)}$, this is not true for $\mathbb{G} |_{B_r(x)}$ and in fact the basis $(\f h_i)$ is very complicated. Its unwieldy form makes it very difficult to obtain precise enough estimates on the spectrum of $M$ from its tridiagonal form.

A possible way to overcome this issue is to simply ignore the irregularity $\mathbb{G} |_{B_r(x)}$ by writing $M$ in the basis $(\f 1_{S_i(x)})$. In this basis, $M$ is no longer tridiagonal, but it is almost tridiagonal, in the sense that its off-tridiagonal entries are small. Unfortunately, it turns out that these entries are too large to actually obtain the eigenvalue fluctuations.

An important ingredient of our proof, therefore, is an \emph{approximate tridiagonalization} of $M$, through the construction of a suitable basis that is in some sense intermediate between the simple basis $(\f 1_{S_i(x)})$ and the basis $(\f h_i)$ in which $M$ is tridiagonal. These basis vectors are denoted by $\f f_i$, and they are explicit enough to admit precise error estimates, while at the same time being close enough to $\f h_i$ to ensure that the off-tridiagonal entries are small enough. To explain their construction, we note that each basis vector $M^i \f 1_x$ can be naturally decomposed into a sum indexed by simple walks on $\N$ of length $i$ starting at $0$, whereby a step of the walk to the left/right indicates that we keep only the terms of $M$ that decrease/increase the distance from the root by one (recall that $\mathbb{G} |_{B_r(x)}$ is a tree). Keeping the contribution of all walks results in the basis $(\f h_i)$, while only keeping walks with no steps to the left results in the basis $(\f 1_{S_i(x)})$. By definition, the basis $(\f f_i)$ is obtained by keeping the contribution of all walks \emph{with at most one step to the left}. It is sufficiently explicit to be amenable to a detailed analysis, and close enough to $(\f h_i)$ to yield an accurate enough tridiagonal approximation. For instance,
\begin{equation*}
\f f_3 = \f 1_{S_3(x)} + \sum_{y \in S_1(x)} (D_y - F) \f 1_y\,,
\end{equation*}
where $F$ is some constant determined by the orthogonality of $(\f f_i)$. The higher-order vectors $\f f_i$, $i > 3$, are constructed analogously, with values on $S_{i-2}(x)$ determined by the average degree of the vertices on the geodesic back to the root $x$.

The most involved analytical part of our proof is to show that, in the basis $(\f f_i)$, the matrix $M$ has an approximate tridiagonal form such that the off-tridiagonal matrix has a suitable structure and operator norm bounded by $O(d^{-1 - c})$, thus yielding estimates of sufficient precision to capture the eigenvalue fluctuations. We compare the upper-left $(r+1) \times (r+1)$ block of the tridiagonal part with that of
the matrix $Z_{\fra d}(\alpha_x,\beta_x)$, where $\fra d = 1 + \frac{1}{d}$ (as in \eqref{def_tau_etc}) and
\begin{equation} \label{eq:def_Zd} 
Z_{\omega}(\alpha, \beta) \deq
\begin{pmatrix}
0 & \sqrt{\alpha} 
\\
\sqrt{\alpha} & 0  & \sqrt{\beta}
\\
& \sqrt{\beta} & 0 & \sqrt{\omega}
\\
&& \sqrt{\omega} & 0 & \sqrt{\omega}
\\
&&& \sqrt{\omega} & 0 & \ddots
\\
&&&& \ddots & \ddots
\end{pmatrix}\,.
\end{equation}
The matrix $Z_{\fra d}(\alpha,\beta)$ has a simple interpretation, which also gives a heuristic explanation of its role in our analysis. Let $\bb T_{p,q,s}$ be the infinite $(p,q,s)$-regular rooted tree, whose root has $p$ children, which themselves each have $q$ children, and all other vertices have $s$ children. Then it is easy to see that the tridiagonalization 
of the adjacency matrix of 
$\bb T_{d\alpha,d\beta,d+1}$ around the root is $\sqrt{d} Z_{\fra d}(\alpha,\beta)$. See Appendix \ref{app:spectral_analysis} below. Hence, somewhat surprisingly, we have to compare the graph $\mathbb{G} |_{B_r(x)}$ to the regular tree $\bb T_{d \alpha_x, d \beta_x, d+1}$, where vertices further than $2$ from the root have $d+1$ children instead of the expected $d$. This slight excess degree is a trace of the irregularity of $\mathbb{G} |_{B_r(x)}$ in its approximation by a regular tree.

The largest eigenvalue of $Z_{\fra d}(\alpha,\beta)$ can be explicitly computed as
\begin{equation} \label{eq:def_Lambda_d} 
\Lambda_{\fra d}(\alpha,\beta) \deq \sqrt{\fra d} \Lambda \big( {\alpha}/{\fra d}, {\beta}/{\fra d} \big) 
\end{equation} 
with $\Lambda$ defined in \eqref{eq:def_Lambda} (see Corollary~\ref{cor:Lambda_expansion} below). 
Therefore, we choose the approximate eigenvalue $\Lambda_x$ from above as $\Lambda_{\fra d}(\alpha_{x},\beta_{x})$. 
We construct a corresponding approximate eigenvector $\f v(x)$ for $M$ out of the top eigenvector $(u_i)_{i \in \N}$ of $Z_{\fra d}(\alpha_x,\beta_x)$ (see Corollary~\ref{cor:Lambda_expansion} below) and the orthonormal family $(\f f_i)_i$ by setting $\f v(x) \deq \sum_{i=0}^r u_i \frac{\f f_i}{\|\f f_i\|}$.
The proof of the closeness of the largest eigenvalue of $M$ and $\Lambda_{\fra d}(\alpha_x,\beta_x)$ 
is based on a perturbation theory argument summarized in Lemma~\ref{lem:perturbationEV}. 
This requires precise bounds on $\scalar{\f v(x)}{(M - \Lambda_{\fra d}(\alpha_x,\beta_x))\f v(x)}$
and $\norm{(M - \Lambda_{\fra d}(\alpha_x,\beta_x))\f v(x)}$ as well as control on the  
spectral gap of $M$ (see the explanations in the proof of Proposition~\ref{pro:fine_rigidity} 
and the proof of Proposition~\ref{pro:extreme_in_a_ball}). As a conclusion, we obtain that the largest eigenvalue of $M$ is $\Lambda_{\fra d}(\alpha_x,\beta_x) + O(d^{-1 - c})$.

Finally, we let $\f w(x)$ be the top eigenvector of $M$. We establish exponential decay of $\f w(x)$, using resolvent estimates and a spectral gap for $M$, which follows from the local tree approximation of $\bb G \vert_{B_r(x)}$ and concentration estimates of the degrees in $B_r(x) \setminus \{x\}$.
Then we deduce, embedding $\f w(x)$ in the original graph $\bb G$, that if $r$ is chosen large enough then $\norm{(H-\Lambda_{\fra d}(\alpha_x,\beta_x))\f w(x)} = O(d^{-1 - c})$, as desired. 
Since the balls $(B_r(x))_{x \in \cal W}$ are disjoint with high probability for small enough $r$ (see Proposition~\ref{pro:graphProperty}), the vectors $(\f w(x))_{x \in \cal W}$ are orthogonal. 
We conclude that there are $(\epsilon_{x})_{x\in{\cal W}}$ such that $\max_{x\in{\cal W}}|\epsilon_{x}| = O(d^{-1-c})$
and we have the inclusion
\begin{equation} \label{eq:W_eigenvalues_subset_spec_H} 
\{\Lambda_{\fra d}(\alpha_{x},\beta_{x})+\epsilon_{x} \colon x\in{\cal W}\}\cap [\sigma(\fra u) - \chi, \infty) \subset\spec(H)\cap [\sigma(\fra u) - \chi, \infty)\,,
\end{equation}
where $\sigma(\fra u)$, defined in \eqref{def_tau_etc}, is the typical value of the largest eigenvalue of $H$ 
and $\chi$ is a parameter satisfying $d^{-1} \ll \chi \ll \eta_{\cal W}$.

\paragraph{Intermediate and rough rigidity}
The next step of the proof is to show that \eqref{eq:W_eigenvalues_subset_spec_H} is not merely an inclusion but an equality. This entails showing that the only eigenvalues in the interval $[\sigma(\fra u) - \chi, \infty)$ are precisely the ones arising from vertices $x \in \cal W$ described above: each eigenvalue in $[\sigma(\fra u)- \chi, \infty)$ can be written as $\Lambda_{\fra d}(\alpha_x,\beta_x) + \eps_x$ for some $x \in \cal W$ and $\eps_x =O(d^{-1 - c})$.

To exclude other eigenvalues of $H$ in $[\sigma(\fra u)-\chi, \infty)$, we first consider eigenvalues of the form $\Lambda_{\fra d}(\alpha_x,\beta_x) + \eps_x$ for $x \notin \cal W$ and $\alpha_x \geq 2 + o(1)$. 
In this case, we do not obtain the precision $\eps_x  = O(d^{-1-c})$ as above, but we also do not need it. All that we need is to ensure that such eigenvalues cannot pollute the interval $[\sigma(\fra u) - \chi, \infty)$. Essentially, we prove that
\begin{equation} \label{eq:upper_bound_Lambda_d} 
\Lambda_{\fra d}(\alpha_x, \beta_x) + \eps_x < \sigma(\fra u) - \chi
\end{equation}
for all $x \notin \cal W$ and $\alpha_x \geq 2 + o(1)$. 

As outlined above, this step has to be split into two scales, since the precision of the rigidity estimates on $\epsilon_x$ that are valid simultaneously for all vertices $x$ satisfying $\alpha_x \geq 2 + o(1)$ is not sufficient to ensure \eqref{eq:upper_bound_Lambda_d} if $\alpha_x$ is below but close to the threshold $\fra u - \eta_{\cal W}$. Thus, we introduce the sets 
\[ 
\cal V = \{ x \colon \alpha_x \geq \fra u - \eta_{\cal V} \}, \qquad \qquad 
\cal U = \{ x \colon \alpha_x \geq 2 + o(1) \} 
\] 
with cutoff parameters $\eta_{\cal V}$ and $2 + o(1)$ satisfying $\eta_{\cal W} \ll \eta_{\cal V} \ll \fra u -2 - o(1)$. 
Thus, we obtain the three-scale hierarchy $\cal W \subset \cal V \subset \cal U$. In the largest set $\cal U$ we prove the rough rigidity result $\epsilon_x = O(d^{- 1/2 + c})$, while in the intermediate set $\cal V$ we prove the intermediate rigidity result $\epsilon_x = O(d^{- 1 + c})$.

Owing to the monotonicity of $\Lambda_{\fra d}$, the high-probability estimate $\beta_x \approx 1$, and $\sigma(\fra u) = \Lambda_{\fra d}(\fra u,1)$, by linearization of $\Lambda_{\fra d}$ in $\alpha$, we find that the bound \eqref{eq:upper_bound_Lambda_d} essentially reduces to showing $\epsilon_x < c \eta_{\cal W} - \chi$ for all $x \in \cal V \setminus \cal W$, and similarly $\epsilon_x < c \eta_{\cal V} - \chi$ for all $x \in \cal U \setminus \cal V$, where $c > 0$ is some positive constant. At this point the need for a three-scale approach is apparent: for the fine rigidity step, we have to choose $d^{-1} \ll \eta_{\cal W} \ll d^{-1/2}$, so that the condition $\epsilon_x < c \eta_{\cal W} - \chi$ is clearly never going to be satisfied with the rough rigidity estimate $\epsilon_x = O(d^{- 1/2 + c})$. We refer to Figure \ref{fig:scales} for an illustration of the three-scale rigidity structure.

\begin{figure}[!ht]
\begin{center}
{\small 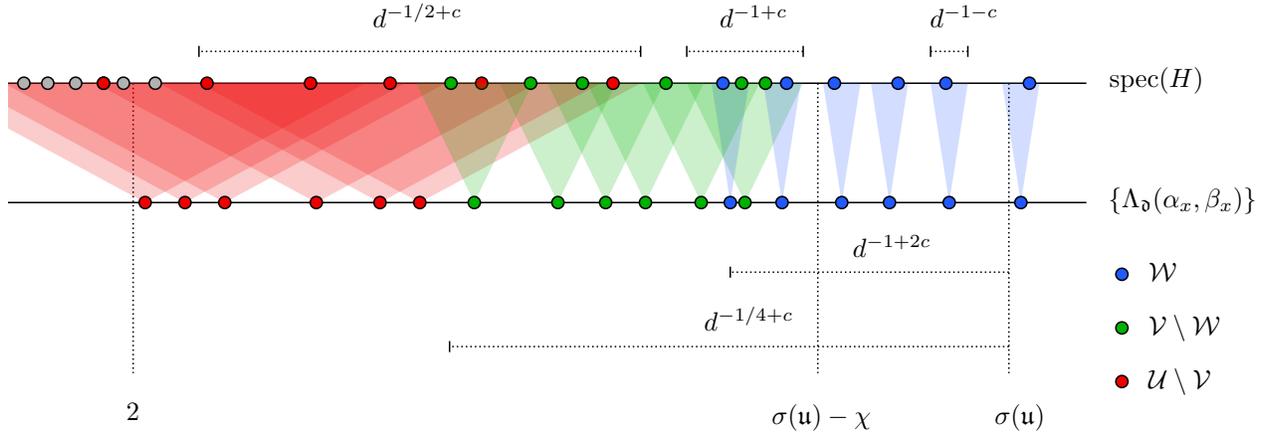}
\end{center}
\caption{A schematic illustration of the three-scale rigidity estimate. Eigenvalues of $H$ are plotted on the top line, while the variables $\Lambda_{\fra d}(\alpha_x, \beta_x)$ with $\alpha_x \geq 2 + o(1)$ are plotted on the bottom line. The approximate bijection between these sets is indicated with the coloured triangular regions. We use blue for points associated with $x \in \cal W$, green for points associated with $x \in \cal V \setminus \cal W$, red for points associated with $x \in \cal U \setminus \cal V$, and grey for all other eigenvalues. The precision of the rigidity estimates is indicated above the diagram: $d^{-1-c}$ for $x \in \cal W$ (fine rigidity), $d^{-1 + c}$ for $x \in \cal V \setminus \cal W$ (intermediate rigidity), and $d^{-1/2 + c}$ for $x \in \cal U \setminus \cal V$ (rough rigidity). The extent of the blue and green regions is indicated below the diagram: $\eta_{\cal W} \asymp d^{-1 + 2c}$ for the region of fine rigidity, and $\eta_{\cal V} \asymp d^{-1/4 + c}$ for the region of intermediate rigidity. See Proposition \ref{prop:blockMatrix} below for a detailed statement.\label{fig:scales}}
\end{figure}

The intermediate rigidity estimate $\epsilon_x = O(d^{- 1 + c})$ for $x \in \cal V \setminus \cal W$ follows analogously to the fine rigidity argument sketched above, except that we require less precision but stronger high-probability bounds to accommodate the larger set $\cal V$. In particular, it suffices to perform an approximate tridiagonalization in terms of the simple basis $(\f 1_{S_i(x)})$ instead of $(\f f_i)$; see Proposition~\ref{pro:intermediate_rigidity} as well as its proof. Finally, the rough rigidity estimate $\epsilon_x = O(d^{- 1/2 + c})$ for $x \in \cal U \setminus \cal V$ follows similarly to the rough bounds from \cite{ADK19,ADK20}. In particular, we also need to understand the graph structure of $\mathbb{G}$ in the vicinity of any vertex $x \in \cal U\setminus \cal W$ to obtain an approximate eigenvector 
$\f w(x)$  
for $H$ with approximate eigenvalue $\Lambda_{\fra d}(\alpha_x,\beta_x)$; see Proposition~\ref{prop:RoughCandidatEigenvector} as well as its proof.

We remark that it is crucial that all the approximate eigenvectors $\f w(x)$ constructed above, for $x \in \cal W$, for $x \in \cal V \setminus \cal W$, and for $x \in \cal U \setminus \cal V$, are orthogonal. This is ensured by requiring that their supports be disjoint. It turns out that, with high probability, this is true for $x \in \cal W$ and for $x \in \cal V \setminus \cal W$, but not for $x \in \cal U \setminus \cal V$. Indeed, the set $\cal U \setminus \cal V$ is large enough that with high probability the balls of even a small radius $r$ around its vertices overlap. This problem was remedied in \cite{ADK19} by introducing the pruned graph, obtained from $\bb G$ by removing a small number of edges. We use this construction in the proof of the rough rigidity in the neighbourhood of vertices in $\cal U \setminus \cal V$. It is essential, however, that the neighbourhoods of all vertices in $\cal W$ and $\cal V \setminus \cal W$ be left intact, as in general the removal of even a single edge gives rise to a shift of order $d^{-1/2}$ in the eigenvalues of $H$ (which is also the precision of the rough rigidity estimate). Our argument ensures that, owing to the structure of the eigenvectors, the shift in the extreme eigenvalues caused by a pruning near the vertices of $\cal U \setminus \cal V$ is smaller than $d^{-1-c}$.

\paragraph{Block diagonal matrix}

The three-scale rigidity result outlined above is best formulated in terms of a block diagonal matrix, which also incorporates information about the approximate eigenvectors of $H$. We note that for each $x \in \cal U$, the arguments sketched above also yield an approximate eigenvector $\f w_-(x)$ 
of $H$ with approximate eigenvalue $-\Lambda_{\fra d}(\alpha_x,\beta_x)$. 
We set $\f w_+(x) \deq \f w(x)$, note that $\f w_-(x) \perp \f w_+(x)$, and extend $(\f w_\sigma(x))_{x \in \cal U, \, \sigma = \pm}$ to an orthonormal basis constituting the columns of an orthogonal matrix $U$. Then we have
\begin{equation} \label{sketch_block}
U^{-1}HU=\begin{pmatrix}\text{diag}((\sigma\Lambda_{\fra d}(\alpha_{x},\beta_{x}))_{x\in{\cal U},\sigma=\pm}) & 0\\
0 & X
\end{pmatrix}+E\,, 
\end{equation}
where $E$ is a matrix containing all error terms, and whose individual blocks are controlled precisely using the above three-scale rigidity estimates and the fact that all balls of radius $r+1$ around vertices of $x \in \cal U$ can be chosen to be disjoint. See Proposition \ref{prop:blockMatrix} below. Moreover, by adapting an estimate on the spectral radius of the non-backtracking matrix of $H$ and Ihara-Bass-type formulas from \cite{ADK19,ADK20,BBK1} we show that the lower-right block  $X$ satisfies $\norm{X}\leq 2  + o(1)$.

Together with the results sketched above, these bounds on $X$ and $E$ imply in particular that 
all eigenvalues of $H$ in $[\sigma(\fra u) - \chi, \infty)$ are of the form $\Lambda_{\fra d}(\alpha_x,\beta_x) + \eps_x$ 
for some $x \in \cal W$ and $\eps_x = O(d^{-1 - c})$. That is, the inclusion in \eqref{eq:W_eigenvalues_subset_spec_H} is an equality.  This concludes the sketch of the proof of the rigidity estimates, which are summarized in Corollary \ref{Cor:SpecMax} below.

\paragraph{Convergence to a Poisson point process}
The rigidity result establishes an approximate bijection between the extreme eigenvalues of $H$ and random variables $(\Lambda_{\fra d}(\alpha_x, \beta_x))_{x \in \cal W}$. By standard extreme value theory, the eigenvalue point process near the edge would be Poisson if the random variables $\Lambda_{\fra d}(\alpha_x, \beta_x)$, $x \in \cal W$, were independent. 
In our case, therefore, the main work to prove Poisson statistics, Theorem \ref{thm:point_process}, is a decorrelation result, stating that the correlation functions of the appropriately rescaled correlation measures of the eigenvalue process factorize asymptotically. First we prove that the finite-dimensional joint distributions of $(\alpha_x, \beta_x)_{x \in \cal W}$ factorize asymptotically. To that end, we use that $(\alpha_x, \beta_x)$ is a function of $H$ restricted to the ball $B_2(x)$, and estimate the correlations by bounding the contribution of geodesics of length shorter than $4$ connecting vertices of $\cal W$. Then we deduce a factorization result for the correlation measures using a truncated inclusion-exclusion formula. We remark that the intensity measure $\rho$ from \eqref{def_rho} arises from an asymptotic analysis of the quantity $N \P(\Lambda_{\fra d}(\alpha_x, \beta_x) \geq s)$ in the regime of $s$ where this quantity is of order one. Essentially, this analysis amounts to an expansion of $f$ from \eqref{def_f} around $\fra u$ and of $\Lambda_{\fra d}$ around $(\fra u, 1)$, combined with an extreme value analysis of $\alpha_x$ and a Gaussian approximation of $\beta_x - 1$. This analysis also determines the parameters $\sigma(\fra u)$ and $\tau(\fra u)$.

\paragraph{Localization}

The proof of eigenvector localization, Theorem~\ref{thm:localisation}, is relatively straightforward using two main ingredients: Poisson eigenvalue statistics with an intensity that has a bounded density, and exponential localization of the vectors $\f w(x)$, $x \in \cal W$, used to construct the block diagonal approximation. Indeed, using that the two-point correlation measure of the eigenvalue process near the edge is approximately a product of the intensity $\rho$ with a bounded density, we deduce a uniform spectral gap between all extreme eigenvalues of $H$. This allows us to conclude that for $x \in \cal W$, $\|(H-\Lambda_{\fra d}(\alpha_{x},\beta_{x}))\f w(x)\|$ is much smaller than $\min_{x\neq y\in{\cal W}}|\Lambda_{\fra d}(\alpha_{x},\beta_{x})-\Lambda_{\fra d}(\alpha_{y},\beta_{y})|$. Hence, we can apply perturbation theory to the eigenvectors of $H$ in terms of the eigenvectors $\f w(x)$ of $H - UEU^{-1}$ (in the notation of \eqref{sketch_block}). Together with the exponential localization of the top eigenvectors $\f w(x)$ of $H |_{B_r(x)}$, we conclude the proof.

\paragraph{Extension to very sparse graphs} 

As advertised at the beginning of this subsection, the preceding discussion focused on the critical regime $d \asymp \log N$. Our results are in fact valid for much sparser graphs, down to the scale $d \gg (\log \log N)^4$. The arguments sketched above carry over with a few complications. Now one has to keep track of the $\fra u$-dependence of the errors, as $\fra u \gg 1$ if $d \ll \log N$. With some additional care, we obtain the error bound $\epsilon_x = O(d^{-1 - c} / \fra u)$, with optimal $\fra u$-dependence. The main new ingredient is a new proof of the rough rigidity estimate, since we cannot rely on the estimates from \cite{ADK19,ADK20}, which were restricted to $d \gg \sqrt{\log N}$. The key observation here is that, in the very sparse regime $d \lesssim \sqrt{\log N}$, the eigenvectors $\f w(x)$ have very fast exponential decay, and hence for the rough rigidity estimates it suffices to consider balls of radius $r = 2$.

\section{Block diagonal approximation} \label{sec:block_diagonal}

In this section we state the three-scale rigidity estimates and the block diagonal approximation of $H$.
First, we reformulate our upper bound on $d$ from \eqref{d_assumptions} in terms of the expected location of
the largest normalized degree. In many arguments it is convenient to replace $\fra u$ from \eqref{def_u} with an approximation, which is denoted by $\am$ and defined through
\begin{equation} \label{eq:def_alpha_max} 
h(\am-1) =  \frac{\log N}{d}\,, 
\end{equation}
where the function $h\colon [0,\infty) \to [0,\infty)$ is defined as
\begin{equation} \label{eq:def_h} 
 h(a) \deq (1 + a) \log (1 + a) - a\,. 
\end{equation}
We note that $\fra a$ is well approximated by $\fra u$; see Lemma \ref{lem:u_a} below for a precise estimate. Moreover, $\fra a \gg 1$ if $d \ll \log N$. 

Throughout the remainder of this section, we shall assume that 
\begin{equation} \label{eq:ConditionAlphaMaxMin}
 \am -2 \geq (\log d) \Big( d^{-\frac{1}{20} + \frac{3}{2}\gamma} \vee d^{-\frac{2}{3}\gamma} \Big)\,.
\end{equation}

Here, and throughout the proof, $\gamma$ is a constant exponent assumed to lie in $(0,1/6)$. At the end of the proof (see Section \ref{sec:proof_thm_poisson_statistics}), it will be chosen depending on the exponents $\zeta$ and $\xi$ from Theorem \ref{thm:point_process}.

The three-scale rigidity estimates are formulated in terms of the three sets of vertices
\begin{align} 
\label{eq:def_cal_W} 
 \cal W & \defeq \biggl\{ x \in [N] \colon \alpha_x \geq \am - \frac{c_* d^{2\gamma-1}}{ \log \am} \biggr\}\,, 
\\ 
\label{eq:VsetDef}
\mathcal{V}& \deq \begin{cases} 
\Bigl\{x \in [N] \colon \alpha_x \geq \am  - c_* \frac{\am^{3/2}}{\am - 2} \frac{(\log N)^{1/4}}{\sqrt{d}} \log d \Bigr\} 
& \text{ if } d > (\log N)^{3/4} \\
\bigl\{x \in [N] \colon \alpha_x \geq \am  -c_*\am^{1/2} \bigr\} 
& \text{ if } d \leq (\log N)^{3/4}\,, 
\end{cases} 
\\ 
\label{eq:UsetDef}
\cal U & \deq \begin{cases} 
\bigl\{x \in [N] \colon  \alpha_x \geq 2 +  (\frac{(\log N)^{1/2}\log d}{d})^{1/4}  \bigr\} & \text{ if }d> (\log N)^{3/4}
\\ 
\bigl\{x \in [N] \colon \alpha_x \geq \frac{\am }{5} \bigr\} & \text{ if }d\leq (\log N)^{3/4}\,. 
\end{cases}
\end{align} 
Here, $\gamma >0$ and 
the constant $c_*>0$ in the definitions of $\cal W$ and $\cal V$ will be chosen in the following results. 
We remark that, owing to \eqref{eq:ConditionAlphaMaxMin}, we have $\cal W \subset \cal V \subset \cal U$.

\begin{proposition}[Block diagonal approximation and rigidity estimates] \label{prop:blockMatrix}
Let $\gamma \in (0,1/6)$ be a constant. 
Let $d$ be such that \eqref{eq:ConditionAlphaMaxMin} and 
\begin{equation} \label{eq:condition_d_blockMatrix} 
d^{2\gamma} \geq K \log \log N 
\end{equation} 
with a large enough $K$ are satisfied. 
Then there is a constant $c_*>0$ such that, with high probability, the following holds. 
There exists an orthogonal matrix $U$ such that 
\begin{equation}
\label{eq:LargeBlock} U^{-1}  H U  =
\begin{pmatrix}\nu_{\r s} & 0 &0 &0 & E^*_{\r s} \\ 0
& \mathcal{D}_{\cal W} & 0 & 0 & E_{\cal W}^* 
\\ 0 & 0 & \mathcal{D}_{ \cal V\setminus \cal W} & 0 & E_{\cal V\setminus \cal W}^* 
\\ 0 & 0 & 0 &  \mathcal{D}_{ \cal U \setminus \cal V} + \cal E_{\cal U\setminus \cal V}& E_{\cal U \setminus \cal V}^*
\\E_{\r s} & E_{\cal W} & E_{\cal V \setminus \cal W} & E_{\cal U \setminus \cal V} &  X
\end{pmatrix}\,.   
\end{equation}
Here, the upper left blocks satisfy 
\begin{equation} \label{eq:cal_D_equal_Lambda_plus_eps} 
\cal D_{\#} = \diag(\sigma \Lambda_{\fra d}(\alpha_x,\beta_x)+\epsilon_{x,\sigma})_{x\in \# ,\sigma=\pm},\qquad \qquad \nu_{\r s} \in \mathbb{R}\,,
\end{equation}
for $\# = \cal W, \cal V \setminus \cal W, \cal U \setminus \cal V$, 
and the error terms as well as the other blocks satisfy the estimates
\begin{enumerate}[label=(\roman*)] 
\item 
\label{Item:WsetBound}
$\max_{x\in \cal W, \, \sigma = \pm}|\epsilon_{x,\sigma}| \lesssim \frac{d^{- \frac 1 2 + 3 \gamma}}{d \am} 
\Big( 1 + \Big( \frac{\log d}{\log \am} \Big)^2 \frac{\am^4}{(\am - 2)^4} \Big)$ and $\|E_{\cal W}\| \lesssim (d \am)^{-10}$,
\item \label{Item:VsetBound}
$\max_{x\in \cal V\setminus \cal W, \, \sigma = \pm}|\epsilon_{x,\sigma}| \lesssim  \frac{1}{d \am} \bigg( 1 + \frac{\am^{3/2}\log d }{(\am - 2)^2} \bigg)$
and $\|E_{\cal V \setminus \cal W}\| \lesssim (d \am)^{-10}$,
\item \label{Item:UsetBound}
$\max_{x\in \cal U\setminus \cal V, \, \sigma = \pm}|\epsilon_{x,\sigma}|+\|\cal E_{\cal U \setminus \cal V}\|+\|E_{\cal U \setminus \cal V}\| \lesssim 
\begin{cases} \frac{(\log N)^{1/4} \sqrt{\log d}}{\sqrt{d}} & \text{ if } d>(\log N)^{3/4} \\ 
  \frac{1}{\sqrt{\am}} + \frac{\log \log N}{\sqrt{d}} & \text{ if }  
{d \leq (\log N)^{3/4}},
\end{cases}$ 
\item $\|E_{\r s}\| + |\nu_{\r s} -(d^{1/2}+d^{-1/2}+d^{-3/2})|\lesssim d^{-5/2} + d^{1/2}\ind{d < (\log N)^{1/4}}$,
\label{Item:StrayEstimate}
\item 
\label{Item:bulkBound}
$\|X\|\leq  \begin{cases} 
  2 + O \Big( \frac{(\log N)^{1/4} \sqrt{\log d}}{\sqrt d} \frac{\am^2}{(\am-2)^2} \Big) & \text{ if }d > (\log N)^{3/4} \\ 
\eta \am^{1/2}  & \text{ if } d\leq (\log N)^{3/4},  
 \end{cases}\quad $ 
for any $\eta \in (\sqrt{\frac{4}{5}},1)$.
\end{enumerate}
 \end{proposition}

The proof of Proposition~\ref{prop:blockMatrix} is given in Section~\ref{sec:proof_prop_block_matrix} below.

\begin{remark}[Choice of $U$]  
In the proof of Proposition~\ref{prop:blockMatrix}, the $(x,\sigma)$--column of $U$, for any $x \in \mathcal V$ and $\sigma \in \{ \pm\}$, is chosen as the eigenvector of $H|_{B_r(x)}$ corresponding to its largest ($\sigma = +$) or smallest ($\sigma = - $) eigenvalue, respectively, with an appropriately chosen, large $r$. 
\end{remark} 

The following corollary of Proposition~\ref{prop:blockMatrix} states that the eigenvalue process of $H$ 
and the process of the expected locations $(\Lambda_{\fra d}(\alpha_x,\beta_x))_{x \in \cal W}$ 
(with the possible exception of an eigenvalue close to $d^{1/2} + d^{-1/2} + d^{-3/2}$) 
coincide approximately for all sufficiently large energies. 

\begin{corollary} \label{Cor:SpecMax}
Let $\gamma \in (0,1/6)$ be a constant. 
If $d$ is chosen such that \eqref{eq:ConditionAlphaMaxMin} and 
\begin{equation} \label{eq:condition_d_SpecMax} 
d^{2\gamma} \geq K_* \gamma^{-1} (\log \log N)(\log \log \log N) 
\end{equation} 
with a sufficiently large $K_*$ are satisfied then,   
 with high probability there exist a constant $c>0$, an eigenvalue $\nu$ of $H$, and error terms $\epsilon_x$, $x \in [N]$, such that
\begin{equation*}
\abs{\nu - (d^{1/2} + d^{-1/2} + d^{-3/2})} \lesssim 
d^{-5/2} + d^{1/2}\ind{d < (\log N)^{1/4}}\,,
\qquad
\abs{\epsilon_x}\lesssim \frac{d^{- \frac 1 2 + 3 \gamma}}{d \fra u} 
\bigg( 1 + \bigg( \frac{\log d}{\log \fra u} \bigg)^2 \frac{\fra u^4}{(\fra u - 2)^4} \bigg)\,,
\end{equation*} 
and the processes
\begin{equation*}
\sum_{\lambda \in \spec(H) \setminus \{\nu\}} \delta_{\lambda}, \qquad \text{ and } \qquad 
\sum_{x \in \cal W} \delta_{\Lambda_{\fra d}(\alpha_x, \beta_x) + \epsilon_x} 
\end{equation*}
coincide in the window
$ [\sigma(\fra u)- c \chi ,\infty)$, where $\chi \deq \frac{(\fra u - 2)d^{2\gamma -1}}{\fra u^{3/2} \log \fra u}$. 
\end{corollary}

Corollary~\ref{Cor:SpecMax} is proved in Section~\ref{sec:proof_cor_spec_max} below.

\begin{remark} \label{rem:left_edge}
The same proof at the left spectral edge shows that Corollary \ref{Cor:SpecMax} holds also if one replaces the point process $\sum_{ \lambda \in \spec(H) \setminus \{\nu\}} \delta_{\lambda}$ with $\sum_{ \lambda \in \spec(H)} \delta_{-\lambda}$. In fact, using the notation of Remark \ref{rem:extreme_evs}, with high probability we have the much stronger bound $\abs{\lambda_k - \lambda_{N-k}} \leq (d \fra u)^{-10}$ for all $k$ such that $\lambda_k \in [\sigma(\fra u)- c \chi ,\infty)$. This can be read off from our proof in Section \ref{sec:fine_rigidity}, using, in its notation, that the graph restricted to $B_{r+1}(x)$ is a tree, in particular bipartite, and hence has a symmetric spectrum.
\end{remark}

\begin{remark}
By a trivial extension of our proof (see Section \ref{sec:StrayEigenvalue}), we obtain an asymptotic expansion for $\nu$ in $d^{-1}$ up to any constant power $d^{-k}$; we omit the details as we do not need it here.
\end{remark}

For future use, we note that the definition of $\cal V$ in \eqref{eq:VsetDef} and the lower bound on $\am - 2$ from \eqref{eq:ConditionAlphaMaxMin} imply 
\begin{equation} \label{eq:lower_bound_alpha_x_minus_2_am_minus_2} 
\alpha_x - 2 \gtrsim \am - 2 \asymp \fra u -2
\end{equation} 
for all $x \in \cal V$; and, in particular, for all $x \in \cal W$. 
In the last step we also used 
the following lemma 
whose proof is given in Appendix~\ref{app:proof_auxiliary_estimates} below. 
\begin{lemma} \label{lem:u_a}
For $1 \ll d \leq 3 \log N$ we have
\begin{equation} \label{eq:scaling_alpha_max} 
\am \asymp \frac{\log N}{d \log \frac{4\log N}{d}} \gtrsim 1\,, \qquad \fra u - \fra a = O \pbb{\frac{\log d}{d}}\,.
\end{equation}
\end{lemma}

Moreover, we note that under the assumption \eqref{d_assumptions} we have
\begin{equation} \label{a_2_lower_bound}
\fra u - 2 \gtrsim d^{-\xi}\,,
\end{equation}
as follows by Taylor expansion of \eqref{def_f}.

\section{Intermediate rigidity} \label{sec:IntermediateRigidity}

The main result of this section is the next proposition, which establishes intermediate rigidity for 
eigenvalues induced by large degree vertices. 
We note that its condition \eqref{eq:condition_alpha_x_intermediate_rigidity} is 
satisfied for $x \in \cal V$, as explained in Remark~\ref{rem:relation_cal_V_cal_V_delta} below.

\begin{proposition}[Intermediate rigidity for extreme eigenvalues] \label{pro:intermediate_rigidity} 
Let $C >0$ be an arbitrary constant. Then there are positive constants $K$, $c_*$ and $C_*$ such that 
if $d$ and $\delta$ satisfy 
\begin{equation} \label{eq:condition_delta} 
K \log \log N \leq d \leq 3 \log N, \qquad \qquad  K \frac{(\log d)^{4/3}}{d^{2/3}} \leq  \delta \leq C \am
\end{equation} 
then the following holds with high probability. 
There is $r_\delta \in \N$ such that $r_\delta \ll \frac{d}{\log \log N}$ and, for each $x \in [N]$ satisfying 
\begin{equation} \label{eq:condition_alpha_x_intermediate_rigidity} 
 \alpha_x \geq \bigg( \am - \frac{c_* \delta}{\log \am} \bigg) \vee \big(2 + C_* \delta^{1/4} \big)\,, 
\end{equation}
and each $\ell^2$-normalized eigenvectors $\f w_+(x)$ and $\f w_-(x)$ of $H|_{B_{r_\delta}(x)}$ corresponding to its largest and smallest eigenvalue, respectively, 
we have 
\begin{align} 
\norm{(H - \Lambda_{\fra d}(\alpha_x,\beta_x)) \f w_+(x)}  + \norm{(H + \Lambda_{\fra d}(\alpha_x,\beta_x)) \f w_-(x)} & \lesssim \frac{1}{d \am} \bigg( 1 + \frac{\am^{3/2}}{(\alpha_x - 2)^2}\log d  \bigg)\,, \nonumber \\ 
 \norm{(H - \scalar{\f w_+(x)}{H\f w_+(x)}) \f w_+(x)} + \normb{\big(H-\scalar{\f w_-(x)}{H\f w_-(x)}\big) \f w_-(x)} & \lesssim \frac{1}{(d\am)^{10}}\,.  
\label{eq:bound_H_minus_eigenvalue_ball_intermediate} 
\end{align}  
In particular, 
$\f w_+ (x) \perp \f w_-(x)$, $\supp \f w_\pm(x) \subset B_{r_\delta}(x)$ and $\f w_+(x) \perp H \f w_-(x)$. 
\end{proposition}

\begin{remark} 
Note that $r_\delta$ in Proposition~\ref{pro:intermediate_rigidity} can be chosen throughout the regime \eqref{eq:condition_r} below, which is nonempty due to \eqref{eq:condition_delta} (see the proof of Proposition~\ref{pro:intermediate_rigidity} below). 
\end{remark}

We now collect a few auxiliary results for the proof of Proposition~\ref{pro:intermediate_rigidity}. 
For a constant $c_* >0$ we define the set of vertices 
\begin{equation} \label{eq:def_cal_V} 
\cal V_\delta \deq \bigg\{ x \in [N] \colon \alpha_x \geq \am - \frac{c_* \delta}{\log \am}
\bigg\}\,. 
\end{equation}

\begin{remark}[Relation between $\cal V$ and $\cal V_\delta$] \label{rem:relation_cal_V_cal_V_delta} 
We recall the definition of $\cal V$ from \eqref{eq:VsetDef} and note that $\cal V = \cal V_{\delta_\star}$ 
with the definition of $\cal V_\delta$ from \eqref{eq:def_cal_V} and 
\begin{equation}\label{eq:deltaStarDef}
\delta_{\star} \deq \begin{cases}
\frac{\am^{3/2}\log \am}{\am - 2} \frac{(\log N)^{1/4}}{\sqrt{d}} \log d 
& \text{ if }d>(\log N)^{3/4} \\
\am^{1/2} \log \am & \text{ if } d\leq (\log N)^{3/4}\,.  \end{cases}
\end{equation}
Moreover, \eqref{eq:condition_alpha_x_intermediate_rigidity} is satisfied for each $x \in \cal V = \cal V_{\delta_\star}$ due to \eqref{eq:ConditionAlphaMaxMin}. 
\end{remark}

In the next proposition, we fix $c_*$ and $\cal V_\delta$ is always understood with respect to this $c_*$. 
First, we choose the constant $c_*>0$ such that 
\begin{equation} \label{eq:lower_bound_degree_cal_V} 
 \am - \frac{c_* \delta}{\log \am} \geq (1 + c) \vee \frac{\am}{2} \gtrsim \frac{\log N}{d \log \frac{4 \log N}{d}}\,, 
\end{equation}
where $c\asymp 1$ depends only on the constant $C$ in the upper bound $\delta \leq C \am$; see \eqref{eq:condition_delta}. 
The proof of \eqref{eq:lower_bound_degree_cal_V} is given in Appendix~\ref{app:proof_auxiliary_estimates} below. 
Along the argument, $c_*$ will be reduced a few times.

For $x, y \in [N]$, we denote by 
\begin{equation} \label{eq:def_N_y_x} 
 N_{y}(x) \deq \abs{S_{d(x,y)+1}(x) \cap S_{1}(y)} 
\end{equation}
the number of vertices that have distance $d(x,y) + 1$ from $x$ and distance 1 from $y$.

\begin{proposition} \label{pro:prob_estimates_1} 
Let $C>0$ be an arbitrary constant. 
Then there exist $K>0$ and $c_*>0$ such that if $d$ and $\delta$ satisfy 
\begin{equation} \label{eq:condition_d_delta_graph} 
K \log \log N \leq d \leq 3 \log N, \qquad \qquad K \frac{\log \log N}{d} \leq  \delta \leq C \am\,, 
\end{equation}
and $\cal V_\delta$ is defined as in \eqref{eq:def_cal_V} then, for any $r \in \N$ with 
\begin{equation} \label{eq:upper_bound_r} 
r \leq c_* \bigg( \frac{d}{\log\log N} \wedge \frac{\delta d}{\log d} \bigg)\,,
\end{equation} 
the following statements hold with high probability. 
\begin{enumerate}[label=(\arabic*)] 
\item \label{item:balls_disjoint} The balls of radius $r$ around the vertices in $\cal V_\delta$ are disjoint, i.e.\ $B_r(x) \cap B_r(y) = \emptyset$ for all $x, y \in \cal V_\delta$ with $x \neq y$. 
\item \label{item:ER_locally_tree} For each $x \in \cal V_\delta$, the graph $\mathbb{G}$ restricted to $B_r(x)$ is a tree. 
\item \label{item:concentration_S_i} 
For all $x \in \cal V_\delta$ and $i \in [r]$, we have 
\[ \absbb{\frac{\abs{S_{i+1}(x)}}{d\abs{S_i(x)}} -1} \lesssim \bigg(\frac{\delta}{\abs{S_i(x)}} \bigg)^{1/2}, \qquad \absbb{\frac{\abs{S_i(x)}}{D_x d^{i-1}} -1 } \lesssim  \bigg(\frac{\delta}{D_x } \bigg)^{1/2}\,.  \] 
\item  \label{item:degrees_bounded} 
For any $x \in \cal V_\delta$ and $y \in B_r(x) \setminus \{x \}$  we have $\abs{D_{y}-d}\leq \delta^{1/2} d$.
\item \label{item:Z_i_estimate} 
For any $x \in \cal V_\delta$ and any $i \in [r]$, we have 
\[ \sum_{y \in S_{i}(x)} (N_y(x)-d)^2 \lesssim \abs{S_{i}(x)} d \big ( \log d + \delta^{1/2} \big)  \bigg(1 + \frac{d \delta}{\abs{S_i}} \bigg)\,. \] 
\end{enumerate}
\end{proposition} 

The proof of Proposition~\ref{pro:prob_estimates_1} is deferred to Section~\ref{sec:proof_prob_estimates_1} below.

\begin{lemma} \label{lem:convergence_max_alpha_x_upper_bound_D_x} 
If $1 \leq d \leq 3 \log N$ and $\fra u >2$, then for any $\xi \gg 1$ we have $\max_{x \in [N]} \alpha_x - \fra u = O(\xi/d)$ with high probability.
In particular, $\max_{x \in [N]} D_x \lesssim \am d \lesssim \log N$ with high probability. 
\end{lemma} 

\begin{proof} 
The first claim follows from \cite[Proposition~D.1]{ADK19} with $l=1$ and $\beta_1(d)=\fra u$. The second claim follows from \eqref{eq:scaling_alpha_max}.
\end{proof}

Throughout the remainder of this section, we exclusively work on the intersection of the high-probability events from Proposition~\ref{pro:prob_estimates_1} and Lemma~\ref{lem:convergence_max_alpha_x_upper_bound_D_x}. 
Whenever we say that a statement holds with high probability then it is meant that it holds on this intersection. 

The next proposition provides important properties of the largest (and smallest) eigenvalue of the graph $\mathbb{G}|_{B_{r+1}(x)}$ restricted to a ball around $x \in \cal V_\delta$ as well as the associated eigenvector.  
This proposition is the core of the proof of Proposition~\ref{pro:intermediate_rigidity} and is proved in Section~\ref{subsec:proofs_lemmas_graph_in_a_ball} below.

\begin{proposition}[Extreme eigenvalue and eigenvector of $H|_{B_{r+1}(x)}$]  \label{pro:extreme_in_a_ball} 
Let $K$ be chosen as in Proposition~\ref{pro:prob_estimates_1}. 
Suppose that $d$ and $\delta$ satisfy \eqref{eq:condition_d_delta_graph}. 
Let $x \in [N]$. 
Then there are $c_*>0$ and $C_*>0$ such that if $\alpha_x$ satisfies \eqref{eq:condition_alpha_x_intermediate_rigidity} and $r\in \N$ satisfies 
\begin{equation} \label{eq:condition_r}
41 + \frac{1}{c_*} \frac{\log d}{\log \am } \bigg(\frac{\am}{\alpha_x -2}\bigg)^2 
\leq r \leq c_* 
\bigg( \frac{d}{\log\log N} \wedge \frac{\delta d}{\log d} \bigg)\,, 
\end{equation} 
then, with $H^{(x,r)} \deq H|_{B_{r+ 1}(x)}$, the following holds with high probability. 
\begin{enumerate}[label=(\roman*)]
\item (Eigenvalue rigidity) \label{item:eigenvalue_ball_Lambda} 
The largest eigenvalue $\mu$ of $H^{(x,r)}$ satisfies 
\begin{equation} \label{eq:expansion_eigenvalue_ball_Lambda} 
\mu = \Lambda_{\fra d}(\alpha_x, \beta_x) + O \bigg( \frac{1}{d \am} + \frac{\am^{1/2}\log d }{d (\alpha_x - 2)^2} \bigg)\,. 
\end{equation}
The smallest eigenvalue of $H^{(x,r)}$ is $-\mu$. 
\item (Eigenvector decay) \label{item:eigenvector_ball_decay}  
The eigenspaces of $H^{(x,r)}$ corresponding to $\mu$ and $-\mu$, respectively, are both one-dimensional. If $\f w$ is contained in 
one of these eigenspaces then 
\begin{equation} \label{eq:eigenvector_decay_ball} 
\norm{\f w|_{S_{r + 1}(x)}}
\lesssim \frac{\abs{w_x}}{(1 + \delta^{1/4}) (d \am)^{10}}\,.
\end{equation}
\end{enumerate} 
\end{proposition} 

We note that, for certain choices of $d$ and $\delta$, it might not be possible to find $r\in \N$ satisfying \eqref{eq:condition_r}. 
Proposition~\ref{pro:intermediate_rigidity} and its proof provide conditions on $d$ and $\delta$ that are sufficient for the existence of $r$. 

We remark that the factor $(d\am)^{-10}$ on the right-hand side of \eqref{eq:eigenvector_decay_ball} can be
easily improved to $(d\am)^{-k}$ for any fixed $k$ by shrinking $c_*$ and increasing 41 in the lower bound on $r$ 
in \eqref{eq:condition_r}. Since this does not strengthen our other results, we refrain from establishing this improvement.

\begin{lemma} \label{lem:norm_bound_adjacency_matrix_general_graph} 
Let $\mathbb{T}$ be a graph whose vertices have degree at most $q + 1$ for some $q \geq 1$. Then its adjacency matrix $A^{\bb T}$ satisfies the bound $\norm{A^{\mathbb{T}}} \leq q + 1$. Moreover, if in addition $\mathbb{T}$ is a tree then $\norm{A^{\mathbb{T}}} \leq 2 \sqrt q$. 
\end{lemma} 

\begin{proof} 
See e.g.\ \cite[Lemma~A.4]{ADK20}. 
\end{proof}

\begin{proof}[Proof of Proposition~\ref{pro:intermediate_rigidity}] 
Owing to \eqref{eq:condition_delta} and \eqref{eq:scaling_alpha_max}, it is easy to see by distinguishing the regimes $d \geq c_0 \log N$ and 
$d \leq c_0 \log N$ for some small enough $c_0>0$ that there is $r \in \N$ satisfying \eqref{eq:condition_r} 
for all $\alpha_x$ satisfying \eqref{eq:condition_alpha_x_intermediate_rigidity}.
We set $r_\delta \deq r + 1$. 
Let $\mu$ be the largest eigenvalue of $H^{(x,r)}$ and $\f w_+$ and $\f w_-$ two $\ell^2$-normalized eigenvectors of $H^{(x,r)}$ associated to $\mu$ and $-\mu$, respectively. 
With this choice, we clearly have $\f w_+(x) \perp \f w_-(x)$ and $\supp \f w_\pm(x) \subset B_{r_\delta}(x)$. Moreover, 
\[ \scalar{\f w_+(x)}{H \f w_-(x)} = \scalar{\f w_+(x)}{H^{(x,r)} \f w_-(x)} = - \mu \scalar{\f w_+(x)}{\f w_-(x)} = 0\,. \] 
For the rest of the argument we focus on the positive eigenvalue of $H$ and $\f w = \f w_+$ as the argument for the negative eigenvalue proceeds in the same way. 
We compute 
\begin{equation}\label{eq:proof_intermediate_rigidity_aux} 
 (H-\mu) \f w = (H - H^{(x,r)})\f w = O((1+\delta^{1/4}) \norm{\f w|_{S_{r + 1}(x)}})\,,   
\end{equation}
where, in the last step, we used that $\mathbb{G}|_{B_{r+2}(x) \setminus B_{r +1}(x)}$ is a forest with maximal degree $d (1 + \delta^{1/2})$ by Proposition~\ref{pro:prob_estimates_1} and Lemma~\ref{lem:norm_bound_adjacency_matrix_general_graph} if $r + 2$ is smaller or equal to the right-hand side of \eqref{eq:upper_bound_r}. 
Therefore, using \eqref{eq:eigenvector_decay_ball} to 
estimate the right-hand side of \eqref{eq:proof_intermediate_rigidity_aux} yields \eqref{eq:bound_H_minus_eigenvalue_ball_intermediate}. Consequently, using 
\eqref{eq:expansion_eigenvalue_ball_Lambda} to
replace $\mu$ by $\Lambda_{\fra d}(\alpha_x,\beta_x)$ completes the proof of Proposition~\ref{pro:intermediate_rigidity}. 
\end{proof} 

\subsection{Proof of Proposition~\ref{pro:extreme_in_a_ball}} 
\label{subsec:proofs_lemmas_graph_in_a_ball} 

In this section, we prove Proposition~\ref{pro:extreme_in_a_ball}. 
We first note that it suffices to prove Proposition~\ref{pro:extreme_in_a_ball} with $\Lambda_{\fra d}(\alpha_x,\beta_x)$
replaced by $\Lambda(\alpha_x,\beta_x)$ from \eqref{eq:def_Lambda} due to Lemma~\ref{lem:EstimateLambdad} below.  
An important step in the proof of this version of Proposition~\ref{pro:extreme_in_a_ball} is the next proposition.

\begin{proposition}[Approximate eigenvector for $H^{(x,r)}$] \label{pro:H_minus_Lambda_v_norm_quadratic_form_intermediate} 
Suppose that $d$ and $\delta$ satisfy \eqref{eq:condition_d_delta_graph}. 
Then there are $c_*>0$ and $C_*>0$ such that the following holds with high probability. 
For any $x \in [N]$, if $\alpha_x$ satisfies \eqref{eq:condition_alpha_x_intermediate_rigidity} and $r\in \N$ satisfies \eqref{eq:condition_r} 
then
there is a normalized vector $\f v(x)$ such that 
\begin{subequations} 
\begin{align} 
 \norm{(H^{(x,r)} - \Lambda(\alpha_x,\beta_x))\f v(x) } & 
\lesssim \frac{1}{\sqrt{d \am}} \Big( \log d + \delta^{1/2} \Big)^{1/2} \,, \label{eq:estimate_norm_approximate}  \\ 
\abs{\scalar{\f v(x)}{( H^{(x,r)} - \Lambda(\alpha_x,\beta_x) ) \f v(x)}} & 
\lesssim \frac{1}{d \am} \,.  \label{eq:estimate_overlap} 
\end{align}
\end{subequations} 
\end{proposition} 

We shall prove Proposition~\ref{pro:H_minus_Lambda_v_norm_quadratic_form_intermediate} in Section~\ref{subsec:proof_Proposition_H_minus_Lambda_intermediate} below. 
After the next corollary, we shall use Proposition~\ref{pro:H_minus_Lambda_v_norm_quadratic_form_intermediate} to establish Proposition~\ref{pro:extreme_in_a_ball}.

\begin{corollary}[Concentration of $\beta_x$]
\label{cor:concentration_beta_x} 
Let $d$ and $\delta$ satisfy \eqref{eq:condition_d_delta_graph}. 
On the high-probability event from Proposition~\ref{pro:prob_estimates_1}, for all $x \in \cal V_\delta$, we have
\begin{equation}
 \beta_x   = 1 + O\bigg(\frac{\sqrt \delta}{\sqrt{d \am}} \bigg). \label{eq:concentration_beta_x}  
\end{equation} 
\end{corollary} 

\begin{proof} 
The expansion \eqref{eq:concentration_beta_x} follows directly from the definition of $\beta_x$ in \eqref{eq:def_alpha_x_beta_x} 
and Proposition~\ref{pro:prob_estimates_1} \ref{item:concentration_S_i}, 
since $D_x \gtrsim d \am$ for all $x \in \cal V_\delta$ by \eqref{eq:lower_bound_degree_cal_V}. 
\end{proof} 

The last missing ingredient for the proof of Proposition~\ref{pro:extreme_in_a_ball} 
is the following simple perturbation estimate for approximate eigenvalues and eigenvectors.

\begin{lemma}\label{lem:perturbationEV}
Let $M$ be a self-adjoint matrix. Let $\epsilon, \Delta > 0$ satisfy $5 \epsilon \leq \Delta$. Let $\lambda \in \R$ and suppose that $M$ has a unique eigenvalue, $\mu$, in $[\lambda - \Delta, \lambda + \Delta]$, with corresponding normalized eigenvector $\f w$. If there exists a normalized vector $\f v$ such that $\norm{(M - \lambda) \f v} \leq \epsilon$ then
\begin{equation*}
\mu - \lambda = \scalar{\f v}{(M - \lambda) \f v} + O \pbb{\frac{\epsilon^2}{\Delta}}\,, \qquad
\norm{\f w- \f v} = O \pbb{\frac{\epsilon}{\Delta}}\,.
\end{equation*}
\end{lemma}

\begin{proof}
Define the orthogonal projections $\Pi \deq \f v\f v ^*$ and $\ol{\Pi} \deq 1 - \Pi$, and decompose $M = W + E$, where
\[ 
W \deq \lambda \Pi + \ol \Pi M \ol \Pi\,, \qquad E \deq \Pi (M - \lambda) \Pi + \Pi M \ol \Pi + \ol \Pi M \Pi\,.\] 
From $E \f v = (M - \lambda)\f v$ and the assumption $\norm{(M - \lambda) \f v} \leq \epsilon$ we easily conclude that $\norm{E} \leq 2 \epsilon$. Since $\lambda$ is an eigenvalue of $W$, we conclude that all other eigenvalues of $W$ are separated from $\lambda$ by at least $\Delta - 2 \epsilon$. The claim now follows from perturbation theory for the isolated eigenvalue $\lambda$ of $W$ with associated eigenvector $\f v$.
\end{proof}

\begin{proof}[Proof of Proposition~\ref{pro:extreme_in_a_ball}] 
We start the proof of \ref{item:eigenvalue_ball_Lambda} by remarking that $-\mu$ is the smallest eigenvalue of $H^{(x,r)}$ if $\mu$ is the largest eigenvalue of $H^{(x,r)}$. 
This follows from the symmetry of the spectrum of $H^{(x,r)}$ around zero which holds as $\mathbb{G}|_{B_{r + 1}(x)}$ is a tree by Proposition~\ref{pro:prob_estimates_1} \ref{item:ER_locally_tree} and, hence, bipartite. 

The next step is applying perturbation theory, more precisely, Lemma~\ref{lem:perturbationEV} with $M = H^{(x,r)}$, $\lambda = \Lambda(\alpha_x,\beta_x)$ and $\f v = \f v(x)$ from Proposition~\ref{pro:H_minus_Lambda_v_norm_quadratic_form_intermediate}. 
To check the conditions of Lemma~\ref{lem:perturbationEV}, we start by determining the spectral gap of $H^{(x,r)}$. 
By removing the vertex $x$ from $\mathbb{G}|_{B_{r +1}(x)}$ and using eigenvalue interlacing, we see 
that $H^{(x,r)}$ has at most one eigenvalue larger than $\norm{Q H^{(x,r)} Q}$ with the projection 
$Q$ defined through 
\begin{equation} \label{eq:norm_QHQ} 
Q \deq \sum_{y \in B_{r + 1}(x)\setminus \{ x\}} \scalar{\f 1_y}{\,\cdot\,} \f 1_y, 
\qquad \qquad \quad 
 \norm{QH^{(x,r)}Q} \leq 2 (1 + \delta^{1/2})^{1/2}\,, 
\end{equation} 
where we used Lemma~\ref{lem:norm_bound_adjacency_matrix_general_graph} and Proposition~\ref{pro:prob_estimates_1} \ref{item:degrees_bounded} to obtain the bound in \eqref{eq:norm_QHQ}. 

Thus, $H^{(x,r)}$ has at most one eigenvalue in $[\Lambda(\alpha_x,\beta_x) - \Delta,\Lambda(\alpha_x,\beta_x) + \Delta]$, where $\Delta \deq \Lambda(\alpha_x,\beta_x) - \norm{Q H^{(x,r)} Q}$. 
Moreover, 
there is a constant $ C_*  >0$ such that if $\alpha_x$ satisfies \eqref{eq:condition_alpha_x_intermediate_rigidity} then 
\begin{equation} \label{eq:lower_bound_spectral_gap} 
 \Delta = \Lambda(\alpha_x, \beta_x) - \norm{QH^{(x,r)}Q} \geq \Lambda(\alpha_x,\beta_x) - 2 (1 + \delta^{1/2})^{1/2} \gtrsim \frac{ (\alpha_x - 2)^2}{\am^{3/2}}\,. 
\end{equation}
The first inequality in \eqref{eq:lower_bound_spectral_gap} is trivial since $\norm{QH^{(x,r)}Q} \leq 2(1 +\delta^{1/2})^{1/2}$ by \eqref{eq:norm_QHQ}.   
The second estimate follows from \eqref{Lambda_geq_2_est} with $\alpha = \alpha_x$ and $\beta = \beta_x$ 
as well as $\alpha_x \gtrsim \am$ due to \eqref{eq:lower_bound_degree_cal_V}  
by distinguishing the cases $\alpha_x \geq C_0$ and $\alpha_x \leq C_0$ for a sufficiently large constant. 
The condition \eqref{assumption_beta_pr} required for \eqref{Lambda_geq_2_est} is satisfied 
due to \eqref{eq:condition_alpha_x_intermediate_rigidity}, \eqref{eq:lower_bound_degree_cal_V} 
and \eqref{eq:concentration_beta_x}.

We set $\eps \deq \norm{(H^{(x,r)} - \Lambda(\alpha_x,\beta_x)) \f v}$ and conclude from 
 \eqref{eq:estimate_norm_approximate} in Proposition~\ref{pro:H_minus_Lambda_v_norm_quadratic_form_intermediate} 
that 
\begin{equation} \label{eq:5_eps_leq_Delta} 
5 \eps \leq \Delta\,, 
\end{equation} which can be checked by distinguishing the regimes $\am \geq C_0$ and $\am \leq C_0$ 
for a sufficiently large $C_0>0$ and using the conditions on $d$ and $\delta$ from \eqref{eq:condition_d_delta_graph}.

Hence, Lemma~\ref{lem:perturbationEV} for 
the largest eigenvalue $\mu$ of $H^{(x,r)}$ yields  
\[ \mu = \Lambda(\alpha_x,\beta_x) + \scalar{\f v}{(H^{(x,r)}-\Lambda(\alpha_x,\beta_x))\f v} + O \bigg ( \frac{\eps^2}{\Delta} \bigg) = 
\Lambda(\alpha_x,\beta_x) + O \bigg( \frac{1}{d \am} + \frac{\am^{1/2} \log d}{d (\alpha_x - 2)^2} \bigg)\,. \] 
Here, in the last step, we used \eqref{eq:estimate_norm_approximate} and \eqref{eq:estimate_overlap} as well as, in the last step, the bound on $\Delta$ from \eqref{eq:lower_bound_spectral_gap} and   
$\frac{(\log d + \delta^{1/2})\am^{3/2}}{(\alpha_x - 2)^2} \lesssim 1 + \frac{ \am^{3/2} \log d}{(\alpha_x - 2)^{2}}$ (this can be obtained easily from $\alpha_x \geq 2 + C_* \delta^{1/4}$ by distingushing the 
two regimes $\am \geq C_0$ and $\am \leq C_0$ for a large enough constant $C_0$). 
Owing to Lemma~\ref{lem:EstimateLambdad}, this completes the proof of Proposition~\ref{pro:extreme_in_a_ball}~\ref{item:eigenvalue_ball_Lambda}.

For the proof of \ref{item:eigenvector_ball_decay}, we focus on eigenvectors $\f w$ corresponding to the eigenvalue $\mu$. The argument works completely analogously for eigenvectors corresponding to $-\mu$.  
We first decompose $\f w = w_x \f 1_x + Q \f w$ in the eigenvector relation $H^{(x,r)} \f w = \mu \f w$, where $Q$ denotes the orthogonal projection onto 
the coordinates in ${B_{r + 1}(x) \setminus \{x \}}$ (compare \eqref{eq:norm_QHQ}). 
Hence, we obtain $(QH^{(x,r)} Q -\mu)Q\f w = - w_x Q H^{(x,r)}  \f 1_x = -w _x \f 1_{S_1(x)} /\sqrt{d}$, i.e.\  
\[ Q \f w = - \frac{w_x}{\sqrt{d}} (Q H^{(x,r)}Q - \mu)^{-1} \f 1_{S_1(x)}\,. \] 
Thus, if $P$ denotes the projection onto the coordinates in $S_{r + 1}(x)$ then we get 
\begin{equation} \label{eq:P_applied_to_w} 
P \f w = \frac{w_x}{\sqrt{d}\, \mu} P \sum_{k=0}^\infty \bigg( \frac{Q H^{(x,r)}Q }{\mu} \bigg)^k \f 1_{S_1(x)} = \frac{w_x}{\sqrt{d}\, \mu}   P \sum^\infty_{k = r} \bigg( \frac{ QH^{(x,r)}Q}{\mu} \bigg)^k \f 1_{S_1(x)} 
\end{equation}
since $(Q H^{(x,r)}Q )^k_{y_1,y_2} = 0$ for all $y_1 \in S_{r + 1}(x)$ and $y_2 \in S_1(x)$ if $k \leq r-1$.

Applying $\norm{\,\cdot\,}$ to the previous identity and summing up the geometric series, we see that to show the estimate on $\norm{\f w|_{S_{r+1}(x)}} = \norm{P \f w}$ in \ref{item:eigenvector_ball_decay}, it suffices to prove that 
\begin{equation} \label{eq:eigenvector_decay_proof_aux} 
 \frac{\mu}{\mu - \norm{QH^{(x,r)}Q}} \bigg( \frac{\norm{QH^{(x,r)}Q}}{\mu} \bigg)^{r} \lesssim \frac{1}{(1 + \delta^{1/4})(d\am)^{10}} 
\end{equation}
since $\norm{\f 1_{S_1(x)}} = \sqrt{D_x} \lesssim \sqrt{\am d}$ by Lemma~\ref{lem:convergence_max_alpha_x_upper_bound_D_x} 
and $\mu \asymp \sqrt{\am}$ by \eqref{eq:expansion_eigenvalue_ball_Lambda},  
\eqref{eq:approx_Lambda_alpha_x_beta_x_by_Lambda_alpha_x}, \eqref{eq:condition_d_delta_graph} and 
\eqref{eq:condition_alpha_x_intermediate_rigidity}.  
Note that the condition \eqref{assumption_beta_pr} required for \eqref{eq:approx_Lambda_alpha_x_beta_x_by_Lambda_alpha_x} is satisfied here 
as explained after \eqref{eq:lower_bound_spectral_gap}. 

For the proof of \eqref{eq:eigenvector_decay_proof_aux}, we distinguish two cases. First, we assume $\am \geq C_0$ for some sufficiently large constant $C_0>0$. 
In this case, we use \eqref{eq:norm_QHQ}, $\mu \asymp \sqrt{\am}$ and $\delta \lesssim \am$ by \eqref{eq:condition_d_delta_graph} and conclude 
\begin{equation} \label{eq:eigenvector_decay_proof_aux2} 
 \frac{\norm{QH^{(x,r)}Q}}{\mu} \lesssim \bigg(\frac{1+ \delta^{1/2}}{\am} \bigg)^{1/2} \lesssim \am^{-1/4}\,. 
\end{equation}
Thus, if $r \geq 41 +  \frac{1}{c} \frac{\log d}{\log \am}$ for some sufficiently small $c>0$ then we conclude $\big( \frac{\norm{QH^{(x,r)}Q}}{\mu} \big)^{ r} \lesssim \frac{1}{d^{10}\am^{41/4}}$. 
Note that this condition on $r$ is satisfied due to the lower bound on $r$ in \eqref{eq:condition_r} since $\frac{\am^2}{(\alpha_x-2)^2} \asymp 1$ due to $\alpha_x \gtrsim \am \geq C_0$ by \eqref{eq:lower_bound_degree_cal_V}.  
Moreover, \eqref{eq:eigenvector_decay_proof_aux2} with large enough $C_0$ implies that $\norm{QH^{(x,r)}Q}/\mu \leq 1/2$, hence, $\frac{\mu}{\mu- \norm{QH^{(x,r)}Q}} \leq 2$.  
This proves \eqref{eq:eigenvector_decay_proof_aux} if $\am \geq C_0$ for a large enough $C_0$ as $\delta^{1/4} \lesssim \am^{1/4}$. 

Next, we suppose that $\am \leq C_0$. Since $\am \gtrsim 1$ by \eqref{eq:scaling_alpha_max}, we have $\am \asymp 1$ in this regime. 
Hence, since $\abs{\mu - \Lambda(\alpha_x,\beta_x)} \leq \eps$ by the definition of $\eps$, the bounds \eqref{eq:5_eps_leq_Delta} and \eqref{eq:lower_bound_spectral_gap} imply 
$\mu - \norm{QH^{(x,r)}Q} \gtrsim \Delta \gtrsim (\alpha_x -2)^2$. 
Thus, as $\mu \asymp \sqrt{\am} \asymp 1$ and $1 - t \leq \ee^{-t}$ for all $t \in \R$, we get 
\[  \bigg( \frac{\norm{QH^{(x,r)}Q}}{\mu} \bigg)^{r}  = \bigg( 1 - \frac{\mu - \norm{QH^{(x,r)}Q}}{\mu} \bigg)^{r} \leq \exp\big( - c r (\alpha_x -2)^2 \big) \] 
for some $c \asymp 1$. 
This proves \eqref{eq:eigenvector_decay_proof_aux} in the remaining regime due to the lower bound on $r$ in \eqref{eq:condition_r} for small enough $c_*$ depending on $C_0$ and $\delta \lesssim \am \lesssim 1$ 
since $\mu - \norm{QH^{(x,r)}Q} \gtrsim \delta^{1/2}\gtrsim \big(\frac{\log \log N}{d}\big)^{1/2}$ 
by the lower bound $\alpha_x \geq 2 + C_* \delta^{1/4}$ in \eqref{eq:condition_alpha_x_intermediate_rigidity} and the lower bound on $\delta$ in \eqref{eq:condition_d_delta_graph}.  
This completes the proof of \eqref{eq:eigenvector_decay_proof_aux} and, thus, the ones of \ref{item:eigenvector_ball_decay} and Proposition~\ref{pro:extreme_in_a_ball}. 
\end{proof}

\begin{remark}[Eigenvector decay for $H^{(x,r)}$] 
Let $\f w$ be the eigenvector of $H^{(x,r)}$ corresponding to its largest eigenvalue $\mu$; see Proposition~\ref{pro:extreme_in_a_ball}. 
We note that 
\begin{equation} \label{eq:decay_eigenvector_in_ball} 
\norm{\f w|_{S_i(x)}} \lesssim\frac{\mu}{\mu - \norm{QH^{(x,r)}Q}} \biggl( \frac{\norm{QH^{(x,r)}Q}}{\mu} \biggr)^{i - 1}  \abs{\scalar{\f 1_x}{\f w}} 
\end{equation}
for all $i \in \qq{1,r + 1}$, where $Q$ is the orthogonal projection onto the coordinates in $B_{r+1}(x)\setminus \{x \}$.

The inequality \eqref{eq:decay_eigenvector_in_ball} can be obtained from the proof of Proposition~\ref{pro:extreme_in_a_ball} \ref{item:eigenvector_ball_decay} with minor modifications. 
If $P_i$ denotes the projection onto the coordinates in $S_i(x)$ then replacing $P$ by $P_i$ in 
\eqref{eq:P_applied_to_w} 
and arguing as after \eqref{eq:P_applied_to_w} and \eqref{eq:eigenvector_decay_proof_aux}  
yields \eqref{eq:decay_eigenvector_in_ball} as $\norm{\f w|_{S_i(x)}} = \norm{P_i \f w}$. 
\end{remark}

\subsection{Proof of Proposition~\ref{pro:H_minus_Lambda_v_norm_quadratic_form_intermediate}} \label{subsec:proof_Proposition_H_minus_Lambda_intermediate}

We first introduce $\f v = \f v(x)$, defined in terms of an eigenvector $(u_i)_{i \in \N}$ of $Z(\alpha_x,\beta_x) \deq Z_1(\alpha_x, \beta_x)$ associated with the eigenvalue
$\Lambda(\alpha_x,\beta_x)$ (see \eqref{eq:def_Zd}, \eqref{Z_scaling}, and Proposition~\ref{prop:Lambda_ab}).  
We rescale $(u_i)_{i \in \N}$ such that $u_i >0$ for all $i \in \N$ and $\sum_{i=0}^r u_i^2 = 1$. 
Define 
\begin{equation} \label{eq:def_v_intermediate_rigidity} 
\f v \deq \sum_{i=0}^r u_i \f s_i, \qquad \qquad \f s_i \defeq \frac{\f 1_{S_i(x)}}{\norm{\f 1_{S_i(x)}}}\,. 
\end{equation}
Clearly, $\norm{\f 1_{S_i(x)}} = \sqrt{\abs{S_i(x)}}$. 

For the proof of \eqref{eq:estimate_norm_approximate}, we note that 
\begin{equation} \label{eq:decomposition_H_minus_Lambda_v} 
 ( H^{(x,r)} - \Lambda(\alpha_x,\beta_x)) \f v  = \f w_2 + \f w_3 + \f w_4\,, 
\end{equation}
where, using the shorthands $S_i \deq S_i(x)$ and $N_y \deq N_y(x)$,  
we introduced the error terms 
\[ \begin{aligned} 
\f w_2 & \deq \frac{u_2}{\sqrt{d \abs{S_2}}} \sum_{y \in S_1} \f 1_y \bigg( N_y - \frac{\abs{S_2}}{\abs{S_1}} \bigg) + \sum_{i=3}^{r} \frac{u_i}{\sqrt{d\abs{S_i}}} \sum_{y \in S_{i-1}} \f 1_y \big( N_y  - d \big)\,,  \\ 
\f w _3 & \deq  \sum_{i=2}^{r-1} \f s_{i +1} u_i \bigg( \frac{\sqrt{\abs{S_{i+1}}}}{\sqrt{d\abs{S_i}}} - 1 \bigg)
+ \sum_{i=3}^r \f s_{i-1} u_i \bigg( \frac{\sqrt{d\abs{{S_{i-1}}}}}{\sqrt{\abs{{S_{i}}}}} - 1 \bigg)\,, \\ 
\f w_4 & \deq \f s_r  u_{r+1} + 
 \frac{\sqrt{\abs{S_{r+1}}}}{\sqrt{d\abs{S_r}}} u_r \f s_{r+1}\,. 
\end{aligned} \] 
The computation proving \eqref{eq:decomposition_H_minus_Lambda_v} is detailed in Appendix~\ref{app:proof_auxiliary_estimates} below. 

We shall show that, with high probability, we have
\begin{subequations} \label{eq:estimates_w_i} 
\begin{align} 
\norm{\f w_2}^2 & \lesssim \frac{\log d + \delta^{1/2}}{d \am} \bigg( 1 + \frac{d\delta}{D_x} \bigg)\,, \label{eq:estimate_w_2} \\ 
\norm{\f w_3}^2 & \lesssim  \frac{1}{d\am} \frac{\delta}{D_x}  \,, \label{eq:estimate_w_3} \\ 
\norm{\f w_4}^{\phantom{2}} &\lesssim (d \am)^{-10} \label{eq:estimate_w_4} \,.  
\end{align}
\end{subequations} 
Before proving \eqref{eq:estimates_w_i}, we note that \eqref{eq:estimates_w_i} and \eqref{eq:decomposition_H_minus_Lambda_v} imply \eqref{eq:estimate_norm_approximate}, 
since 
$D_x = d \alpha_x\gtrsim d\am$ by \eqref{eq:lower_bound_degree_cal_V} as $x \in \cal V_\delta$ 
and $\delta \lesssim \am$ by \eqref{eq:condition_d_delta_graph}.  
We shall use $\alpha_x \gtrsim \am$ frequently throughout this proof.  

Hence, to complete the proof of \eqref{eq:estimate_norm_approximate}, it suffices to show 
\eqref{eq:estimates_w_i}. 
In the following, we shall use that $u_{i+1} \leq u_i$ for $i \geq 2$ by \eqref{Zd_vect_components}. 
Moreover, we shall apply Lemma~\ref{lem:eigenvector_fine_properties} 
with $\alpha = \alpha_x$ and $\beta = \beta_x$ 
whose condition, \eqref{assumption_beta_pr}, is satisfied due to 
\eqref{eq:condition_alpha_x_intermediate_rigidity}, \eqref{eq:lower_bound_degree_cal_V} 
and \eqref{eq:concentration_beta_x}. 
To verify \eqref{eq:estimate_w_2}, we use the Pythagorean theorem and estimate  
\[ \begin{aligned} 
\norm{\f w_2}^2  & \lesssim \sum_{i=2}^{r} \frac{u_i^2}{d\abs{S_i}} \sum_{y \in S_{i-1}} \big( N_y  - d \big)^2 
+ \frac{u_2^2}{d \abs{S_2}} \abs{S_1} \bigg( d - \frac{\abs{S_2}}{\abs{S_1}} \bigg)^2 \\ & \lesssim \sum_{i=2}^r u_i^2\frac{\abs{S_{i-1}}}{\abs{S_i}} (\log d + \delta^{1/2}) \bigg( 1 + \frac{d \delta}{D_x} \bigg) + \frac{u_2^2 d \delta}{\abs{S_2}} \lesssim \frac{\log d + \delta^{1/2}}{d\am} \bigg( 1 + \frac{d \delta}{D_x} \bigg)\,. 
\end{aligned} \] 
Here, we used Proposition~\ref{pro:prob_estimates_1} \ref{item:Z_i_estimate} as well as the concentration of $\abs{S_2}/\abs{S_1}$ and $\abs{S_i} \geq D_x$ by Proposition~\ref{pro:prob_estimates_1} \ref{item:concentration_S_i} in the second step
and $\abs{S_{i-1}}/\abs{S_i} \lesssim 1/d$ by Proposition~\ref{pro:prob_estimates_1} \ref{item:concentration_S_i} as well as $u_2^2 \leq \sum_{i=2}^r u_i^2 \lesssim \am^{-1}$ due to \eqref{eq:relations_sum_u_i_u_2}    
in the third step. This completes the proof of \eqref{eq:estimate_w_2}.  

Using Proposition~\ref{pro:prob_estimates_1} \ref{item:concentration_S_i} and $\sum_{i=2}^r u_i^2 \lesssim \am^{-1}$ from \eqref{eq:relations_sum_u_i_u_2},   
we get 
\[ \norm{\f w _3}^2 \leq 2   \sum_{i=2}^{r-1} u_i^2 \bigg( \frac{\sqrt{\abs{S_{i+1}}}}{\sqrt{d\abs{S_i}}} - 1 \bigg)^2
+ 2\sum_{i=3}^r u_i^2 \bigg( \frac{\sqrt{d\abs{{S_{i-1}}}}}{\sqrt{\abs{{S_{i}}}}} - 1 \bigg)^2
\lesssim \frac{1}{\am}  \frac{\delta}{\abs{S_{2}}}   \lesssim \frac{1}{d\am}  \frac{\delta}{D_x}\,,  \] 
which is \eqref{eq:estimate_w_3}. 

The definition of $\f w_4$, $u_{r+1} \leq u_r$ and $\abs{S_{r+1}}\leq 4 d \abs{S_r}$ by Proposition~\ref{pro:prob_estimates_1} \ref{item:concentration_S_i} directly yield
$\norm{\f w_4} \leq 3 u_r$. 
Thus, \eqref{eq:estimate_w_4} follows from $u_2 \leq 1$ and \eqref{eq:u_r_leq_u_2_small}, 
which completes the proof of \eqref{eq:estimates_w_i}.

We now show \eqref{eq:estimate_overlap}. 
Since $\scalar{\f s_j}{H^{(x,r)}\f s_i} = \scalar{\f s_j}{H \f s_i}$ for $i,j \in [r]$ by definition 
of $H^{(x,r)}$, writing $\Lambda = \Lambda(\alpha_x,\beta_x)$ yields
\begin{equation} \label{eq:scalar_v_H_minus_Lambda_v} 
\scalar{\f v}{(H^{(x,r)} - \Lambda)\f v} 
 = \sum_{i,j=0}^r u_i u_j \scalar{\f s_j}{(H- \Lambda) \f s_i} 
= 2\sum_{i=2}^{r-1} u_i u_{i+1} \bigg( \frac{\sqrt{\abs{S_{i+1}}}}{\sqrt{d\abs{S_i}}} - 1 \bigg) - u_r u_{r+1}\,. 
\end{equation}
The details of the proof are given in Appendix~\ref{app:proof_auxiliary_estimates} below. 

By \eqref{eq:scalar_v_H_minus_Lambda_v},  
Proposition~\ref{pro:prob_estimates_1} \ref{item:concentration_S_i}, $u_{i+1} \leq u_i$, $\sum_{i=2}^r u_i^2 \lesssim \am^{-1}$ due to \eqref{eq:relations_sum_u_i_u_2},  and \eqref{eq:u_r_leq_u_2_small}, we obtain 
\[ \abs{\scalar{\f v}{(H^{(x,r)} - \Lambda)\f v}}\lesssim \sum_{i=2}^r u_i^2  \sqrt{\frac{\delta}{d D_x}} \lesssim \frac{1}{\am}   \sqrt{\frac{\delta}{d D_x}}  \,.  \]  
Since $D_x \gtrsim d \am$ and $\delta \lesssim \am$ by \eqref{eq:condition_d_delta_graph},  
this completes the proof of \eqref{eq:estimate_overlap} and, thus, the one of Proposition~\ref{pro:H_minus_Lambda_v_norm_quadratic_form_intermediate}.  \qed

\section{Fine rigidity} \label{sec:fine_rigidity}

Throughout this section, we assume that \eqref{eq:ConditionAlphaMaxMin} is satisfied. 
The main result of this section is Proposition~\ref{pro:fine_rigidity} below. 
It refines the rigidity estimates shown in Proposition~\ref{pro:intermediate_rigidity} 
for vertices $x \in \cal W$. 
Recalling the definition of $\cal W$ from \eqref{eq:def_cal_W}, we see that they satisfy a stronger lower bound than \eqref{eq:condition_alpha_x_intermediate_rigidity}.

\begin{proposition}[Fine rigidity for extreme eigenvalues] \label{pro:fine_rigidity} 
Let $\gamma \in (0,1/6)$ be a constant. 
Then there are positive constants $K$ and $c_*$ such that 
if $d$ satisfies 
\begin{equation} \label{eq:condition_d_approximate_eigenvector_fine_rigidity} 
( K \log \log N)^{\frac{1}{2\gamma}} \leq d \leq 3\log N 
\end{equation}
then there is $r_{\cal W} \in \N$ such that $r_{\cal W}  \ll \frac{d}{\log \log N}$
 and, with high probability, for any $x \in \cal W$, 
the following holds. 
Any two $\ell^2$-normalized eigenvectors $\f w_+(x)$ and $\f w_-(x)$ of $H|_{B_{r_{\cal W}}(x)}$ corresponding to its largest and smallest eigenvalue, respectively, satisfy 
\begin{equation} \label{eq:fine_rigidity} 
\norm{(H - \Lambda_{\fra d}(\alpha_x,\beta_x)) \f w_+(x)}  + \norm{(H + \Lambda_{\fra d}(\alpha_x,\beta_x)) \f w_-(x)} \lesssim \frac{d^{- \frac 1 2 + 3\gamma}}{d \am} \bigg( 1 + \bigg(\frac{\log d}{\log \am} \bigg)^2 \bigg( \frac{\am}{\am - 2} \bigg)^4\bigg)\,.
\end{equation}  
 Moreover, $\f w_+(x) \perp \f w_-(x)$, $\f w_+(x) \perp H\f w_-(x)$, $\supp \f w_\pm(x) \subset B_{r_{\cal W}}(x)$ and 
\begin{equation}  
\norm{(H-\scalar{\f w_+(x)}{H \f w_+(x)})\f w_+(x)} + \norm{(H - \scalar{\f w_+(x)}{H \f w_+(x)})\f w_-(x)} \lesssim (d \am)^{-10}\,. 
\label{eq:bound_H_minus_eigenvalue_ball_fine} 
\end{equation}
\end{proposition}

The proof of Proposition \ref{pro:fine_rigidity} will follows exactly as of Proposition \ref{pro:intermediate_rigidity} once we have improved Proposition \ref{pro:H_minus_Lambda_v_norm_quadratic_form_intermediate}  to Proposition \ref{prop:eigenError}. 
We recall the definition $H^{(x,r)} = H|_{B_{r+1}(x)}$. 

\begin{proposition}[Approximate eigenvector for $H^{(x,r)}$] \label{prop:eigenError}
Let $d$, $r \in \N$ and $\gamma\in (0,1/6]$ satisfy 
$r d^{- \frac 1 2 + 3 \gamma} \leq c$, 
 \eqref{eq:condition_r} with $\delta = d^{2\gamma - 1}$ as well as  
\eqref{eq:condition_d_approximate_eigenvector_fine_rigidity} 
for sufficiently small constants $c>0$ and $c_*>0$ as well as a large enough constant $K>0$. 
Then, with high probability,
for any $x\in \cal W$, there exists a normalized vector $\f v{(x)}$ supported on $B_r(x)$ such that  
\begin{subequations} 
\begin{align} 
\normb{\big(H^{(x,r)}-\Lambda_{\fra d}(\alpha_x,\beta_x)\big) \f v{(x)}}^2 & \lesssim 
\am^{-1} r^2 d^{-2 + 4\gamma} \label{eq:bestCandidateNorm}\,, \\ 
\absb{\scalar{\f v{(x)}}{(H^{(x,r)}-\Lambda_{\fra d}(\alpha_x,\beta_x)) \f v{(x)}}} & \lesssim \am^{-1} r^2d^{-\frac 3 2 + 3\gamma}\,. 
\label{eq:bestCandidateOrthogonal}
\end{align} 
\end{subequations} 
\end{proposition}

The last tool for the proof of Proposition~\ref{pro:fine_rigidity} is the next lemma 
which will be proved in Appendix~\ref{app:proof_auxiliary_estimates} below. 

\begin{lemma} \label{lem:beta_x_assumption_beta_pr_checked} 
If $1 \ll d \leq 3 \log N$ then,
with high probability, for all $x \in \cal U$, we have 
\begin{equation} \label{eq:beta_x_equal_1_error_term_on_cal_U} 
\beta_x = \begin{cases} 1 + O \Big(\frac{\sqrt{\log N}}{d} \Big) & \text{ if } d > (\log N)^{3/4} \\ 
1 + O\Big( \Big( \frac{\log \log N}{d} \Big)^{1/2} \Big) & \text{ if } d \leq (\log N)^{3/4}\,. \end{cases} 
\end{equation} 
Moreover, \eqref{assumption_beta_pr} with $\alpha = \alpha_x$ and $\beta = \beta_x$ is satisfied for all $x \in \cal U$.  
\end{lemma}

\begin{proof}[Proof of Proposition \ref{pro:fine_rigidity}]
Denoting by $\lfloor\,\cdot\, \rfloor$ the integer part of a real number,  
we introduce 
\begin{equation} \label{eq:choice_r} 
r_{\cal W} \deq 42 + \biggl\lfloor \frac{1}{c} \frac{\am^2}{(\am - 2)^{2}} \frac{\log d}{\log \am}\biggr\rfloor
\end{equation}
for sufficiently small $c>0$, such that $r = r_{\cal W} - 1$ satisfies 
\eqref{eq:condition_r} with $\delta = d^{2\gamma -1 }$ due to \eqref{eq:condition_d_approximate_eigenvector_fine_rigidity} and \eqref{eq:ConditionAlphaMaxMin}. 
Note that $r_\cal W \ll \frac{d}{\log \log N}$. 
Moreover, $r d^{-\frac 1 2 + 3 \gamma} =o(1)$ for all $\gamma \in (0,1/6)$ 
by \eqref{eq:ConditionAlphaMaxMin} and \eqref{eq:condition_d_approximate_eigenvector_fine_rigidity}, i.e.\ $r$ satisfies all 
assumptions of Proposition~\ref{prop:eigenError}.

We now follow step by step the proof of Proposition \ref{pro:intermediate_rigidity}
with $\delta = d^{2\gamma - 1}$ 
and use Proposition~\ref{pro:graphProperty} below instead of Proposition~\ref{pro:prob_estimates_1} 
as well as an analogue of Proposition~\ref{pro:extreme_in_a_ball}. 
This analogue is obtained by following the proof of Proposition~\ref{pro:extreme_in_a_ball} 
and applying Proposition~\ref{prop:eigenError} instead of Proposition~\ref{pro:H_minus_Lambda_v_norm_quadratic_form_intermediate}. 
This yields 
\begin{equation} \label{eq:mu_Lambda_fine_rigidity} 
\mu = \Lambda_{\fra d}(\alpha_x,\beta_x) + O \bigg( \frac{r^2 d^{- \frac 3 2 + 3 \gamma}}{\am} \bigg( 1  + \frac{\am^{3/2} d^{-\frac 1 2 + \gamma}}{(\alpha_x - 2)^2} \bigg) \bigg) = \Lambda_{\fra d}(\alpha_x,\beta_x) + O \bigg( \frac{r^2 d^{-\frac 3 2 + 3 \gamma}}{\am} \bigg)\,, 
\end{equation}
where we used Lemma~\ref{lem:beta_x_assumption_beta_pr_checked} and $\cal W \subset \cal U$
to justify  \eqref{assumption_beta_pr} and, in the last step, $\alpha_x - 2 \gtrsim \am - 2 \gtrsim \am^{3/4} d^{-1/4 + \gamma/2}$ due to \eqref{eq:def_cal_W} and \eqref{eq:ConditionAlphaMaxMin} (see also \eqref{eq:lower_bound_alpha_x_minus_2_am_minus_2}). 
Estimating $r = r_{\cal W} - 1$ in \eqref{eq:mu_Lambda_fine_rigidity} using the definition of $r_{\cal W}$ in \eqref{eq:choice_r} completes the proof of \eqref{eq:fine_rigidity}. 
The remaining statements of Proposition~\ref{pro:fine_rigidity} follow as in the proof of Proposition~\ref{pro:intermediate_rigidity}. 
\end{proof}

\subsection{Proof of Proposition \ref{prop:eigenError}}
Since $\supp \f v(x) \subset B_r(x)$ and $H^{(x,r)} = H|_{B_{r+1}(x)}$, we have that 
$H^{(x,r)} \f v(x) = H\f v(x) = A \f v(x) /\sqrt{d}$. 
Therefore, for notational convenience, we express everything in terms of $A$ or $H$ (instead of $H^{(x,r)}$). 

The next proposition collects a few key properties of the graph $\mathbb{G}$ in the vicinity of vertices in $\mathcal{W}$. 
For every $x \in [N]$, we introduce 
\begin{equation} \label{eq:def_N_k_z_x} 
N_{z}^{(k)}(x) = \abs{S_k(z) \cap S_{k + d(x,z)}(x)}\,.
\end{equation}
Note that $N_z^{(1)}(x) = N_z(x)$ by the definition from \eqref{eq:def_N_y_x}. 
An interpretation of $N_z^{(k)}(x)$ will be given just after the next proposition. 

\begin{proposition}[Structure of $\mathbb{G}$ in vicinity of $\cal W$] \label{pro:graphProperty}
Let $\gamma \in (0,1/2)$ and $d$ satisfy \eqref{eq:condition_d_approximate_eigenvector_fine_rigidity} 
 for some sufficiently large $K >0$.  
Then there exists $c_* >0$ such that if $\cal W$ is defined as in \eqref{eq:def_cal_W} and  
 $r \in \N$ satisfies 
\begin{equation} \label{eq:upper_bound_r_fine_rigidity} 
r \leq c_* \bigg( \frac{d}{\log\log N} \wedge \frac{d^{2\gamma}}{\log d} \bigg)\,,
\end{equation} 
then the following statements hold with high probability. 
\begin{enumerate}[label=(\arabic*)] 
\item \label{item:disjoint_balls} The balls of radius $r$ around the vertices in $\cal W$ are disjoint, i.e.\ $B_r(x) \cap B_r(y) = \emptyset$ for all $x, y \in \cal W$ with $x \neq y$. 
\item \label{item:tree_in_balls} For all $x \in \cal W$, the graph $\mathbb{G}$ restricted to $B_r(x)$ is a tree.
\item \label{item:concentration_S_i_Omega} 
For all $x \in \cal W$ and $i \in [r]$, we have  
\[ \absbb{\frac{\abs{S_{i+1}(x)}}{d\abs{S_i(x)}} -1} \lesssim \frac{d^{-\frac 1 2+ \gamma}}{\abs{S_i(x)}^{\frac 1 2}} , \qquad \absbb{\frac{\abs{S_i(x)}}{D_x d^{i-1}} -1 } \lesssim d^{-\frac 1 2 + \gamma}D_x^{-\frac 1 2}\,. \] 
\item \label{item:concentration_degrees} For any $x \in \cal W$ and $y \in B_r(x) \setminus \{x \}$ we have
$\abs{D_y - d} \leq d^{\frac 1 2 + \gamma}$.  
In particular, 
\[ \absb{N_y(x) -d}\leq d^{\frac 1 2 + \gamma} + 1 \,. \] 
\item \label{item:Concentration} 
Let $i \in [r]$. If $z \in B_r(x)\setminus\{x\}$ for some $x \in \cal W$ and $S_i(z) \subset B_r(x)$ then we have 
\[\absbb{\sum_{y\in S_i(z)\cap S_{i+d(x,z)}(x)}(N_y(x) -d)} \lesssim d^{\frac 1 2(i+1) + \gamma}, 
\qquad \qquad \absbb{\frac{N_{z}^{(i)}(x)}{D_zd^{i-1}} -1} \lesssim 
d^{-1 + \gamma}\,. \] 
\item \label{item:MeanQuadratic} If $\gamma \leq 1/6$ then for any $x \in \cal W$ and any $i \in [r]$, 
we have 
\[ \sum_{y\in S_{i}(x)}(N_{y}(x) -d)^{2} = D_{x}d^{i}\bigl(1 +O\bigl( D_x^{-\frac 1 2}{d^{\gamma}}\bigr)\bigr)\,. \] 
\end{enumerate}
\end{proposition} 

The proof of Proposition~\ref{pro:graphProperty} is given in Section~\ref{sec:proof_graphProperty} below. 
We remark that Proposition~\ref{pro:graphProperty} \ref{item:disjoint_balls} --  
\ref{item:concentration_degrees} coincide with Proposition~\ref{pro:prob_estimates_1} \ref{item:balls_disjoint} --
\ref{item:degrees_bounded} for $\delta = d^{2 \gamma -1}$. 
Moreover, \eqref{eq:upper_bound_r_fine_rigidity} is \eqref{eq:upper_bound_r} with $\delta = d^{2\gamma - 1}$. 
We restate them here to collect all properties of $\mathbb{G}$ used in this section in one place.

In the remainder of this section, we work on the high-probability event from Proposition~\ref{pro:graphProperty} (or a subset of it that occurs with high probability) 
and choose $r\in \N$ such that $r + 1$ is smaller or equal to the right-hand side of \eqref{eq:upper_bound_r_fine_rigidity}.  
Moreover, throughout, we fix a vertex $x \in \cal W$. 
We suppress the $x$-dependence of various quantities from our notation; e.g.\ we set $B_i\deq B_i(x)$, $S_i\deq S_i(x)$ and $N_y^{(i)} \deq N_y^{(i)}(x)$ (cf.\ \eqref{eq:def_N_k_z_x}) and denote the graph distance of a vertex $z$ from $x$ by $|z|=d(x,z)$.

Owing to Proposition~\ref{pro:graphProperty} \ref{item:tree_in_balls}, we will always view $\mathbb{G}|_{B_{r+1}}$ as a tree rooted on $x$ and will say that $z$ is a \emph{child} of $y$ and that $y$ is a \emph{parent} of $z$ if $z\in S_1(y)\cap S_{d(x,y)+1}$. Note that $N_y$ is the number of children of $y$.  We also denote
\begin{itemize}
\item  $N_y^{(i)}=| S_{i}(y)\cap S_{i + d(y,x)}|$,  
the number of children of $y$ after $i$ generations, 
\item $(x,y]$ the path from $x$ to $y$ in the graph, $y$ included and $x$ not included. Since $\mathbb{G}|_{B_{r+1}}$ is a tree, the path is uniquely defined.
\item $z_1 \leq z_2$ if $z_1\in (x,z_2]$ as well as $z_1 < z_2$ if $z_1\leq z_2$ and $z_1\neq z_2$.
\end{itemize}

The main improvement in Proposition \ref{prop:eigenError} compared to 
Proposition~\ref{pro:H_minus_Lambda_v_norm_quadratic_form_intermediate} 
is a more precise choice of the approximate eigenvector $\f v(x)$ (see \eqref{eq:bestcandidate} below). 
The vectors $\f 1_{S_k}$ in the definition \eqref{eq:def_v_intermediate_rigidity} of the approximate eigenvector for Proposition~\ref{pro:H_minus_Lambda_v_norm_quadratic_form_intermediate} 
are replaced by different vectors, which we call $\f f_k$. They are defined by
\begin{subequations} \label{eq:def-fi}
\begin{align} 
\f f_{0}  & \deq \f 1_{x}\,,\qquad\qquad 
\f f_{1}  \deq \f 1_{S_{1}}\,,\qquad\qquad 
\f f_{2}  \deq \f 1_{S_{2}}\,,\\ 
\f f_{i} & \deq \f 1_{S_{i}}+\sum_{y\in S_{i-2}}\f 1_{y}\sum_{z\in(x,y]}(N_{z}-F_{|z|})\,,
\end{align} 
\end{subequations}  
for all $3\leq i\leq r$, where the $F_{|z|}\in \mathbb{R}$ are defined such that the $\f f_{i}$ are orthogonal 
(the existence of such $F_i$ follows from \eqref{eq:FDefinition} below). 
By convention, we set $F_0 \deq D_x$.

To explain the intuition behind the definition of the basis $(\f f_i)$, which also serves as a guiding principle in the proof of Proposition \ref{Prop:Yerror} below, we recall that the tridiagonalization of $A$ amounts to writing $A$ in the basis $\f 1_x, A \f 1_x, A^2 \f 1_x, \dots$. Because $\bb G$ is a tree in the neighbourhood of $x$, we note that $A \f 1_y$ for some $y$ can be split into the \emph{outer terms} $\sum_{z \in S_1(y)} \ind{z > y} \f 1_z$ and the \emph{inner terms} $\sum_{z \in S_1(y)} \ind{z < y} \f 1_z$. Hence, we can decompose $A^i \f 1_x$ into a sum of terms encoded by simple walks $w$ on $\N$ of length $i$ starting from $0$, whereby a step of $w$ to the right corresponds to selecting the outer terms in applying $A$, and a step to the left to selecting the inner terms. More explicitly, denoting by $W_i$ the set of simple walks $w = w_0 \cdots w_i$ on $\N$ of length $i$ starting from $w_0 = 0$, we have the splitting $A^i \f 1_x = \sum_{w \in W_i} \f b(w)$,  where $\f b(w) = \sum_{y} b_y(w) \f 1_y$ for some coefficients $b_y(w)$. Then we have inductively, for any $w \in W_{i-1}$,
\begin{equation*}
A \f b(w) = \sum_{z} b_{z}(w) A \f 1_{z} = \sum_y \f 1_y \sum_{z \in S_1(y)} b_z(w) (\ind{y > z} + \ind{y < z})  \eqd
\f b(w^+) + \f b(w^-)\,,
\end{equation*}
where $w^\pm \in W_i$ is the walk of length $i$ obtained from $w$ by making a step $\pm 1$ in the last step. By definition, $\f b(w)$ is supported on the sphere $S_{w_i}$. Note that $i - w_i$ is twice the number of steps to the left in $w$. The simple basis vectors $\f 1_{S_i}$ used in Section \ref{sec:IntermediateRigidity} arise from considering only the term $\f b(w)$ for the walk $w \in W_i$ with no steps to the left. Moreover, if the tree around $x$ were regular, in the sense that the degree of a vertex $y$ depends only on $\abs{y}$, then it is easy to see that the contribution of all other walks vanishes after orthogonalization of the vectors $A^i \f 1_x$. As explained in Section \ref{sec:overview_proof_subsect}, the vector $\f f_i$ corresponds to the contribution of \emph{all walks with at most one step to the left}. These vectors constitute an intermediate basis between $\f 1_{S_i}$ and $A^i \f 1_x$. They are sufficiently simple to admit an effective analysis of the tridiagonal matrix, and sufficiently close to the true tridiagonal basis to result in small enough off-tridiagonal entries.

The next proposition shows that $A$ restricted to $\mathrm{span}\{ \f f_0,\ldots, \f f_r\}$ is approximately tridiagonal in the basis $(\f f_i)_{i=0}^r$.

\begin{proposition}\label{Prop:Yerror}
Let $\gamma \in (0,1/6]$ and $r \in \N$ satisfy \eqref{eq:upper_bound_r_fine_rigidity} for a sufficiently small $c_*>0$. 
Then, on the high-probability event from Proposition~\ref{pro:graphProperty}, we have
\begin{equation} \label{eq:A_f_i_relations}
A \f f_0 = \f f_1\,, \qquad
A\f f_{i} = \f f_{i+1}+F_{i-1}\f f_{i-1}+\f g_{i} \quad (1 \leq i \leq r)\,,
\end{equation}
for vectors $\f g_1$, \ldots, $\f g_r$ satisfying 
 $\f g_{i} = 0$ for $i\leq 3$ and 
\begin{enumerate}[label=(\roman*)] 
\item $\supp \f g_i\subset S_{i-3}$, \label{Item:YSupport}
\item $\langle \f g_i, \f f_{{i-3}}\rangle=0$, \label{Item:Yorthogonal}
\item \label{Item:Ynorm} $\|\f g_i\|^2 \leq 4|S_{i-3}|d^{2+4\gamma}i^2 $,
\end{enumerate}
for $4\leq i\leq r$. 
\end{proposition}

The proof of Proposition~\ref{Prop:Yerror} is given in Section~\ref{subsec:proof_prop_Yerror} below. 

\noindent Motivated by Proposition~\ref{Prop:Yerror}, we define the matrix $M\in\mathbb{R}^{(r+1) \times (r+1)}$ with entries 
\begin{equation}\label{eq:Mdefinition}
M_{ij}\deq\begin{cases}
\displaystyle \frac{\langle A \f f_{i},\f f_{j}\rangle}{\sqrt{d} \|\f f_{i}\|\|\f f_{j}\|} & \text{if }|i-j|=1\\
0 & \text{otherwise}\,.
\end{cases}
\end{equation}
The symmetry of $A$ and the construction of $M$ imply that $M$ is a symmetric tridiagonal matrix. 
The next result proves that $M$ is close to the upper-left $(r+1) \times (r+1)$ block $Z_{\fra d}(\alpha_x,\beta_x)_{[0,r]}$ of the (infinite) tridiagonal matrix $Z_{\fra d}(\alpha_x,\beta_x)$ defined in \eqref{eq:def_Zd}.

\begin{proposition}\label{Prop:TriMatrixf}
Let $r \in \N$ and $\gamma\in (0,1/6]$ satisfy \eqref{eq:upper_bound_r_fine_rigidity} and 
$rd^{-\frac 1 2 + 3 \gamma} \leq c$ for a sufficiently small constant $c>0$. 
Then, on the high-probability event from Proposition~\ref{pro:graphProperty}, the matrix  $M$ satisfies
\begin{equation} \label{eq:M_Z_identities} 
M_{ii} = Z_{\fra d}(\alpha_x, \beta_x)_{ii} = 0 \,, \quad M_{01} = Z_{\fra d}(\alpha_x, \beta_x)_{01}=\sqrt{\alpha_x}\,, \quad M_{1 2}=Z_{\fra d}(\alpha_x,\beta_x)_{12}=\sqrt{\beta_x}
\end{equation} 
and  
\begin{equation} \label{eq:comparison_M_Z} 
M=Z_{\fra d}(\alpha_x,\beta_x)_{[0,r]}+ O\big(r^2d^{-\frac 3 2 + 3\gamma}\big)\,. 
\end{equation}
\end{proposition}

Proposition~\ref{Prop:TriMatrixf} is proved in Section~\ref{subsec:proof_prop_TriMatrixf} below. 

\noindent Let $(u_i)_{i=0}^{r}\in \mathbb{R}^{r+1}$ be the first $r+1$ components of the eigenvector of $Z_{\fra d}(\alpha_x,\beta_x)$ associated with its largest eigenvalue  $\Lambda_{\fra d}(\alpha_x,\beta_x)$ (see
Corollary~\ref{cor:Lambda_expansion} \ref{item:Zd_evect}). 
We choose $u_0$ such that $\sum_{i=0}^r u_i^2 = 1$. 
Since $M$ and $Z_{[0,r]}(\alpha_x,\beta_x)$ are close by Proposition~\ref{Prop:TriMatrixf}, 
we consider the candidate eigenvector 
\begin{equation}  \label{eq:bestcandidate}
\f v(x)=\f v\deq \sum_{i=0}^r u_i \frac{\f f_i}{\|\f f_i\|} 
\end{equation}
in analogy to the definition of the approximate eigenvector for Proposition~\ref{pro:H_minus_Lambda_v_norm_quadratic_form_intermediate} in \eqref{eq:def_v_intermediate_rigidity}. 

After the following lemma, we shall estimate $(H- \Lambda_{\fra d}(\alpha_x,\beta_x))\f v(x)$ 
and prove Proposition \ref{prop:eigenError}. 

\begin{lemma}\label{Rem:MFraction}
If $r \in \N$ satisfies \eqref{eq:upper_bound_r_fine_rigidity} for $c_*$ from Proposition~\ref{Prop:Yerror} then $M_{i\,i+1}=\frac{\|\f f_{i+1}\|}{\sqrt{d}\|\f f_i\|}$ for all $i \in [0,r]\cap \N$. 
\end{lemma}

\begin{proof}[Proof of Lemma \ref{Rem:MFraction}] 
The claim is obvious by Proposition~\ref{Prop:Yerror} if $i = 0$. 

If $i \geq 1$ then Proposition~\ref{Prop:Yerror} yields   
\[M_{i\,i+1}=
\frac{1}{\sqrt{d}\|\f f_{i}\|\|\f f_{i+1}\|}\langle A\f f_{i},\f f_{i+1}\rangle=\frac{1}{\sqrt{d}\|\f f_{i}\|\|\f f_{i+1}\|}\langle \f f_{i+1}+F_{i-1}\f f_{i-1}+\f g_i,\f f_{i+1}\rangle=\frac{\|\f f_{i+1}\|}{\sqrt{d}\|\f f_i\|}\,.\]
For the last equality we used that $\f f_{i+1}$ and $\f f_{i-1}$ are orthogonal and the supports of $\f g_i$ and $\f f_{i+1}$ are disjoint because of Proposition \ref{Prop:Yerror} \ref{Item:YSupport}.
\end{proof}

\begin{proof}[Proof of Proposition \ref{prop:eigenError}]
We note that \eqref{eq:condition_r} with a sufficiently small $c_*>0$ implies \eqref{eq:upper_bound_r_fine_rigidity}. 
Throughout this proof, we work on the intersection of the high-probability events from Proposition~\ref{pro:graphProperty} and Lemma~\ref{lem:convergence_max_alpha_x_upper_bound_D_x}. 
Note that $H^{(x,r)} \f v = H \f v$ since $\supp \f v \subset B_r(x)$. 
Owing to \eqref{eq:A_f_i_relations} in Proposition~\ref{Prop:Yerror}, we have 
\begin{equation} \label{eq:proof_Yerror_aux1} 
A \f f_i = \f f_{i+1} + \left(F_{i-1}+\frac{\langle \f g_{i},\f f_{i-1}\rangle}{\|\f f_{i-1}\|^{2}}\right)\f f_{i-1}+\left(\f g_{i}-\frac{\langle \f g_{i},\f f_{i-1}\rangle \f f_{i-1}}{\|\f f_{i-1}\|^{2}}\right)
\end{equation}
for $1\leq i \leq r$.  
Since $\f f_{i-1}$ and $\f g_{i}-\frac{\langle \f g_{i},\f f_{i-1}\rangle \f f_{i-1}}{\|\f f_{i-1}\|^{2}}$ are orthogonal, taking the scalar product of \eqref{eq:proof_Yerror_aux1} and $\f f_{i-1}$ yields  
$\scalar{\f g_i}{\f f_{i-1}} = M_{i\, i-1} \sqrt{d} \norm{\f f_i}\norm{\f f_{i-1}} - F_{i-1} \norm{\f f_{i-1}}^2$ due to \eqref{eq:Mdefinition}. 
Therefore, as $M_{i\,i+1}=\frac{\|\f f_{i+1}\|}{\sqrt{d}\|\f f_{i}\|}$ by Lemma \ref{Rem:MFraction}, 
we obtain from \eqref{eq:proof_Yerror_aux1} that 
\[\frac{1}{\sqrt{d}}A \frac{\f f_i}{\| \f f_i\|}= M_{i\,i+1}\frac{\f f_{i+1}}{\|\f f_{i+1}\|} + M_{i\,i-1} \frac{\f f_{i-1}}{\|\f f_{i-1}\|}+\frac{1}{\sqrt{d} \| \f f_i\|}\left(\f g_{i}-\frac{\langle \f g_{i},\f f_{i-1}\rangle \f f_{i-1}}{\|\f f_{i-1}\|^{2}}\right)\] 
for $1 \leq i \leq r$. 
Therefore, using $A \f f_0 = \f f_1$ by \eqref{eq:A_f_i_relations}, the definition of $\f v$ in \eqref{eq:bestcandidate} and the convention $M_{0\, -1}=u_{-1} = 0$, we obtain 
\begin{equation}\label{eq:proof_Yerror_aux2} 
H \f v  =\sum_{i=0}^{r-1}(M_{i\,i+1}u_{i+1}+M_{i\,i-1}u_{i-1})\frac{\f f_{i}}{\|\f f_{i}\|}+\frac{u_{r-1}\f f_{r}}{\sqrt{d}\|\f f_{r-1}\|}+\frac{u_r\f f_{r+1}}{\sqrt{d}\|\f f_r\|}  + \f w_2\,, 
\end{equation}
where we used that $\f g_1 = \f g_2 = \f g_3 = 0$ by Proposition~\ref{Prop:Yerror} and introduced the error term 
\[ \f w_2 \deq \sum_{i=4}^r\frac{u_{i}}{\sqrt{d} \|\f f_{i}\|}\left(\f g_{i}-\frac{\langle \f g_{i},\f f_{i-1}\rangle   \f f_{i-1}}{\|\f f_{i-1}\|^{2}}\right)\,. \]

To shorten the notation, we set $Z_{ij} \deq Z_{\fra d}(\alpha_x,\beta_x)_{ij}$ with the convention $Z_{0\,\, -1} = 0$. 
Owing to the definitions of $\Lambda_{\fra d}(\alpha_x,\beta_x)$, $Z_{ij}$ and $u_i$, we have $\Lambda_{\fra d}(\alpha_x,\beta_x)u_i = Z_{i\,i+1}u_{i+1}+Z_{i\,i-1}u_{i-1}$
for $0 \leq i \leq r$. 
Hence, \eqref{eq:proof_Yerror_aux2} and the identities \eqref{eq:M_Z_identities} from Proposition \ref{Prop:TriMatrixf} imply
\begin{equation} \label{eq:NiceRoughErrors}
H \f v = \Lambda_{\fra d}(\alpha_x,\beta_x) \f v + \f w_2 + \f w_3 + \f w_4
\end{equation} 
with the error terms $\f w_3$ and $\f w_4$ defined through 
\begin{align*} 
\f w_3 & \deq 
\sum_{i=3}^{{r -1}} \big((M_{i\,i+1}-Z_{i\,i+1})u_{i+1}+(M_{i\,i-1}-Z_{i\,i-1})u_{i-1}\big)\frac{\f f_{i}}{\|\f f_{i}\|} + (M_{23} - Z_{23}) u_3 \frac{\f f_2}{\norm{\f f_2}}\,, 
\\ 
\f w_4 & \deq 
 \left(\frac{u_{r-1}}{\sqrt{d}\|\f f_{r-1}\|}-\frac{\Lambda_{\fra d}(\alpha_x,\beta_x)u_r}{\|\f f_r\|}\right)\f f_{r}+\frac{u_r\f f_{r+1}}{\sqrt{d}\|\f f_r\|}\,.
\end{align*} 
Note that the error terms $\f w_2$, $\f w_3$ and $\f w_4$ are the analogues of the corresponding error terms 
in \eqref{eq:decomposition_H_minus_Lambda_v}.

We now control these three error terms. 
Owing to the definition of $\cal W$ in \eqref{eq:def_cal_W} and the upper bound in Lemma~\ref{lem:convergence_max_alpha_x_upper_bound_D_x}, we have $\alpha_x \asymp \am$. 
In the following, we shall use 
 \eqref{eq:relations_sum_u_i_u_2} and \eqref{eq:u_r_leq_u_2_small} from  
Lemma~\ref{lem:eigenvector_fine_properties} and always apply them to $Z_{\fra d}(\alpha_x,\beta_x)$ and $\Lambda_{\fra d}(\alpha_x,\beta_x)$. 
The condition~\eqref{assumption_beta_pr} required for Lemma~\ref{lem:eigenvector_fine_properties} is satisfied due to Lemma~\ref{lem:beta_x_assumption_beta_pr_checked} and $\cal W \subset \cal U$. 
Moreover, Lemma~\ref{lem:EstimateLambdad} and \eqref{eq:approx_Lambda_alpha_x_beta_x_by_Lambda_alpha_x} are applicable and yield $\Lambda_{\fra d}(\alpha_x,\beta_x) \asymp \sqrt{\alpha_x} \asymp \sqrt{\am}$.

For $\f w_2$, we use that $\supp \f g_i$ and $\supp \f g_j$ are disjoint for $i \neq j$ 
and that $(\f f_i)_{i=4}^r$ is orthogonal to obtain 
\begin{align}
\norm{\f w_2}^{2} 
\leq \frac{2}{d} \left(\sum_{i=4}^r |u_i|^2 \frac{\|\f g_{i}\|^2}{\|\f f_{i}\|^2} +\sum_{i=4}^r |u_{i}|^2\frac{|\langle \f g_{i},\f f_{i-1}\rangle |^2}{\|\f f_{i-1}\|^2\|\f f_i\|^2}\right)
\leq \frac{4}{d} \max_i \frac{\|\f g_i\|^2}{\|\f f_i\|^2} \sum_{j\geq 4} |u_j|^2
\lesssim \frac{r^2 d^{-2 + 4 \gamma}}{\am}\,. \label{eq:NiceRoughtErrorT3}
\end{align}
Here, in the last step, we used Proposition~\ref{Prop:Yerror} \ref{Item:Ynorm} and $\|\f f_i\|^2\geq |S_i|\geq \frac{d^3}{2}|S_{i-3}|$ by \eqref{eq:def-fi} and Proposition \ref{pro:graphProperty} \ref{item:concentration_S_i_Omega}
to estimate the maximum and \eqref{eq:relations_sum_u_i_u_2} to estimate the sum. 

The Pythagorean theorem, $u_3 \leq u_2$, \eqref{eq:relations_sum_u_i_u_2} and \eqref{eq:comparison_M_Z} in Proposition~\ref{Prop:TriMatrixf} imply 
\begin{equation} \label{eq:NiceRoughtErrorT1} 
 \norm{\f w_3}^2 \lesssim \abs{M_{23} - Z_{23}}^2u_3^2 + \sum_{i=3}^{r-1} \big(\abs{M_{i\, i+1} -Z_{i\, i+1}}^2 u_{i+1}^2 + \abs{M_{i\, i-1} - Z_{i\, i-1}}^2 u_{i-1}^2 \big)  \lesssim \am^{-1} r^4 d^{-3+ 6\gamma} \,. 
\end{equation}

By Lemma \ref{Rem:MFraction} and Proposition \ref{Prop:TriMatrixf}, we have $\frac{\|\f f_{r+1}\|}{\| \f f_r \|\sqrt{d}} \lesssim 1$.
Thus, \eqref{eq:u_r_leq_u_2_small}, \eqref{eq:relations_sum_u_i_u_2} and $\Lambda_{\fra d}(\alpha_x,\beta_x) \asymp \sqrt{\am}$  
imply 
\begin{equation} \label{eq:NiceRoughtErrorT2}
 \norm{\f w_4} \leq \normbb{\frac{u_r \f f_{r+1}}{\sqrt{d}\norm{\f f_r}}} + \normbb{\frac{u_{r-1} \f f_r}{\sqrt{d} \norm{\f f_{r-1}}}} + \normbb{\frac{\Lambda_{\fra d}(\alpha_x,\beta_x)u_r \f f_r}{\norm{\f f_r}}} \lesssim u_2 \sqrt{\am} (d \am)^{-10} \lesssim (d\am)^{-10}\,. 
\end{equation} 
Finally, the identity \eqref{eq:NiceRoughErrors} together with the estimates \eqref{eq:NiceRoughtErrorT3}, \eqref{eq:NiceRoughtErrorT1} and \eqref{eq:NiceRoughtErrorT2}
yields 
\eqref{eq:bestCandidateNorm} 
since $r^2 d^{-1 + 2 \gamma} \lesssim 1$ by \eqref{eq:condition_r} with $\delta = d^{2\gamma - 1}$ and $\gamma \leq 1/6$.

We now prove \eqref{eq:bestCandidateOrthogonal}. 
Owing to \eqref{eq:NiceRoughErrors}, we have $\scalar{\f v}{(H-\Lambda_{\fra d}(\alpha_x,\beta_x))\f v} = \scalar{\f v}{\f w_2} + \scalar{\f v}{\f w_3} + \scalar{\f v}{\f w_4}$. 
We estimate each of these terms separately. 
We get
\[ 
\scalar{\f v}{\f w_2} = \sum_{j=0}^r u_j \scalarbb{\frac{\f f_j}{\norm{\f f_j}}}{\sum_{i=4}^r \frac{u_i}{\norm{\f f_i}} \bigg( \f g_i - \frac{\scalar{\f g_i}{\f f_{i-1}} \f f_{i-1}}{\norm{\f f_{i-1}}^2} \bigg)} 
= \sum_{i=4}^r \frac{u_{i-3}}{\norm{\f f_{i-3}}} \frac{u_i}{\norm{\f f_i}} \scalar{\f f_{i-3}}{\f g_i} = 0 
\] 
as $\supp \f f_j \subset S_j \cup S_{j-2}$ by \eqref{eq:def-fi}, $\supp \f g_i \subset S_{i-3}$ and $\scalar{\f f_{i-3}}{\f g_i} = 0 $ by 
 Proposition \ref{Prop:Yerror} \ref{Item:YSupport} and \ref{Item:Yorthogonal}, respectively. 
Moreover, we compute 
\[ 
\scalar{\f v}{\f w_3} = u_2 u_3 (M_{23} - Z_{23}) + \sum_{i = 3}^{r-1} u_i  \big((M_{i\,i+1}-Z_{i\,i+1})u_{i+1}+(M_{i\,i-1}-Z_{i\,i-1})u_{i-1}\big)\,, 
\] 
which is bounded by the right-hand side of \eqref{eq:bestCandidateOrthogonal} due to 
\eqref{eq:comparison_M_Z} in Proposition~\ref{Prop:TriMatrixf}, $u_{i+1} \leq u_i$ and 
\eqref{eq:relations_sum_u_i_u_2}.  
Finally, we trivially have 
$\scalar{\f v}{\f w_4} = O( ( d\am)^{-10})$  by \eqref{eq:NiceRoughtErrorT2}, 
which completes the proof of \eqref{eq:bestCandidateOrthogonal} and, thus, the one 
of Proposition~\ref{prop:eigenError}.   
\end{proof}

\subsection{Proof of Proposition~\ref{Prop:Yerror}} \label{subsec:proof_prop_Yerror} 
We now prove Proposition \ref{Prop:Yerror}. 
Throughout this section, we work on the high-probability event from Proposition~\ref{pro:graphProperty}. 
First, we show that the coefficients $F_i$ from \eqref{eq:def-fi} are close to $d$.

\begin{lemma} \label{lem:Festimate}
On the high-probability event from Proposition~\ref{pro:graphProperty}, the following holds. 
We have 
\begin{equation} \label{eq:F_1_explicit} 
F_1 = \abs{S_2}/\abs{S_1}\,. 
\end{equation} 
Moreover, if $r\in \N$ satisfies \eqref{eq:condition_alpha_x_intermediate_rigidity} and $\gamma \in (0,1/6]$ then, for all $1 \leq i\leq r-2$, we have
\begin{equation} \label{eq:F_i_minus_d_estimate} 
|F_i-d|\lesssim i d^{2\gamma}\,.
\end{equation}
\end{lemma}

\begin{proof}[Proof of Lemma~\ref{lem:Festimate}]
First, we shall see that $F_{i}$ satisfies the recursive relation
\begin{equation}\label{eq:FDefinition}
F_{i}   =\frac{|S_{i+1}|}{|S_{i}|}+\frac{1}{|S_{i}|}\sum_{y\in S_{i-1}}(N_y-d)\sum_{z\in(x,y]}(N_{z}-F_{|z|})
\end{equation}
for any $1 \leq i \leq r-2$. 
This proves \eqref{eq:F_1_explicit} since the second term on the right-hand side of \eqref{eq:FDefinition} vanishes if $i = 1$. 
Note that $1 \leq \abs{z} \leq i-1$ in \eqref{eq:FDefinition}. 

Indeed, the definition of $\f f_i$ in \eqref{eq:def-fi} and their orthogonality imply
\begin{align*}
0 & = \langle \f f_{i+2},\f f_{i} \rangle = \left\langle  \sum_{y\in S_{i}}\f 1_{y}\sum_{z\in(x,y]}(N_{z}-F_{|z|}) , \f 1_{S_{i}} \right\rangle  =  \sum_{y\in S_{i}}\sum_{z\in(x,y]}(N_{z}-F_{|z|}) 
\\ & = \sum_{z\in S_{i}}N_z-|S_{i}| F_{i} + \sum_{y\in S_{i-1}} N_y \sum_{z\in(x,y]}(N_{z}-F_{|z|})\,. 
\end{align*}
Since $\sum_{z\in S_{i}}N_z = |S_{i+1}|$ and $\sum_{y\in S_{i-1}}\sum_{z\in(x,y]}(N_{z}-F_{|z|}) = \langle \f f_{i+1},\f f_{i-1} \rangle =0$, 
this shows \eqref{eq:FDefinition}.

Next, we prove \eqref{eq:F_i_minus_d_estimate} in Lemma~\ref{lem:Festimate} by induction. 
Since the second term on the right-hand side of \eqref{eq:FDefinition} vanishes for $ i = 1$, 
we have $\abs{F_1 - d} \lesssim d^{\gamma}$ by Proposition~\ref{pro:graphProperty} \ref{item:concentration_S_i_Omega} and $D_x \gtrsim \am d \geq d$. 

We assume that \eqref{eq:F_i_minus_d_estimate} is valid up to $i-1$. By \eqref{eq:FDefinition} we have
\begin{align*}
 |F_{i} - d | & \leq \left|\frac{|S_{i+1}|}{|S_{i}|}-d\right|  +\frac{1}{|S_{i}|}\sum_{y\in S_{i-1}} \bigg( |N_y-d|\sum_{z\in(x,y]}(|N_z-d|+|d-F_{|z|}|) \bigg)  \\ 
&  \lesssim \frac{d^{\frac 1 2+\gamma}}{\sqrt{D_xd^{i-1}}}+\frac{|S_{i-1}|}{|S_{i}|} (i-2) d^{\frac 1 2+\gamma}(d^{\frac 1 2+\gamma}  + (i-1) d^{2 \gamma})
 \lesssim d^\gamma + (i-2) d^{2\gamma} \left(1+O(i d^{\gamma- \frac 1 2})\right) 
\end{align*}
where we used Proposition \ref{pro:graphProperty} \ref{item:concentration_S_i_Omega}, \ref{item:concentration_degrees}, $D_x \gtrsim d$ by the definition of $\cal W$ in \eqref{eq:def_cal_W} and the induction hypothesis. 
As $i \leq r \lesssim d^{2\gamma}$ and $\gamma \leq 1/6$, we have $i d^{\gamma - 1/2}\lesssim 1$, which completes the proofs of \eqref{eq:F_i_minus_d_estimate} and Lemma~\ref{lem:Festimate}. 
\end{proof}

\begin{proof}[Proof of Proposition \ref{Prop:Yerror}]
This is a direct calculation. 
By slightly decreasing $c_*$ in \eqref{eq:upper_bound_r_fine_rigidity}, 
we can assume due to Proposition~\ref{pro:graphProperty} \ref{item:tree_in_balls} 
that $\mathbb{G}|_{B_{r+1}}$ is a tree. 
Therefore, we have $A \f f_0 = \f f_1$ and 
\begin{equation} \label{eq:A_1_S_i_tree} 
A \f 1_{S_i} = \f 1_{S_{i+1}} + \sum_{y\in S_{i-1}} N_y \f 1_y
\end{equation}
for $i \in [r]$. 
From \eqref{eq:A_1_S_i_tree} and the definitions of $\f f_0$, $\f f_1$, $\f f_2$ and $\f f_3$ in \eqref{eq:def-fi}, we conclude $A \f f_1 = \f f_2 + F_0 \f f_0$ 
and $A \f f_2 = \f f_3 + F_1 \f f_1$ since $F_0 = D_x$ by definition and $F_1 = \abs{S_2}/\abs{S_1}$ by \eqref{eq:F_1_explicit}. 
Hence, \eqref{eq:A_f_i_relations} for $i \leq 2$ with $\f g_1 = \f g_2 = 0$. 

Let $3 \leq i \leq r$ and $y\in S_{i-2}$. 
If $t$ denotes the parent of $y$ in the tree, i.e.\ $\{t\}= S_{i-3}\cap S_{1}(y)$, and $y' \in S_{i-1} \cap S_1(y)$ are the children of $y$ then  
\begin{equation} \label{eq:A_1_y_tree} 
A \f 1_{y} = \f 1_{t} + \sum_{y'\in S_{i-1}\cap S_1(y)}  \f 1_{y'}\,. 
\end{equation}
If $2 \leq i \leq r$ then, using the definition of $\f f_i$ in \eqref{eq:def-fi}, the relations in \eqref{eq:A_1_S_i_tree} and \eqref{eq:A_1_y_tree} imply 
\begin{align*}
A\f f_{i} & =A\bigg(\f 1_{S_{i}}+\sum_{y\in S_{i-2}}\f 1_{y}\sum_{z\in(x,y]}(N_{z}-F_{|z|})\bigg)\\
 & =\f 1_{S_{i+1}}+\sum_{y\in S_{i-1}}\f 1_{y} \bigg(N_{y}+\sum_{z\in(x,y)}(N_{z}-F_{|z|})\bigg) +\sum_{t\in S_{i-3}}\f 1_{t}\sum_{y\in S_1(t), \, y> t}\bigg((N_{y}-F_{|y|})+\sum_{z\in(x,t]}(N_{z}-F_{|z|})\bigg)\\
&  =\f 1_{S_{i+1}}+F_{i-1}\f 1_{S_{i-1}}+\sum_{y\in S_{i-1}}\f 1_{y} \bigg(\sum_{z\in(x,y]}(N_{z}-F_{|z|})\bigg) 
\\ & \phantom{=} +\sum_{t\in S_{i-3}}\f 1_{t}\bigg(N_{t}\sum_{z\in(x,t]}(N_{z}-F_{|z|})+\sum_{y\in S_1(t), \, y> t}(N_{y}-F_{|y|})\bigg)
\\ & = \f 1_{S_{i+1}}+\sum_{y\in S_{i-1}}\f 1_{y} \bigg(\sum_{z\in(x,y]}(N_{z}-F_{|z|})\bigg)
+F_{i-1}\f 1_{S_{i-1}} +F_{i-1}\sum_{t\in S_{i-3}}\f 1_{t}\sum_{z\in(x,t]}(N_{z}-F_{|z|}) 
\\ & \phantom{=} +\sum_{t\in S_{i-3}}\f 1_{t}\bigg((N_{t}-F_{i-1})\sum_{z\in(x,t]}(N_{z}-F_{|z|})+\sum_{y\in S_1(t), \, y> t}(N_{y}-F_{|y|})\bigg)
\\ & =\f f_{i+1}+F_{i-1}\f f_{i-1}+\f g_{i}\,. 
\end{align*}
Here, in the second step, we used that $(x,y')=(x,y]=(x,t]\cup \{y\}$ if $t$ is the parent of $y$ and $y'$ is a child of $y$.
In the last step, we used the definitions of $\f f_{i-1}$ and $\f f_{i+1}$ (cf.\ \eqref{eq:def-fi}) as well as introduced
\begin{equation}\label{Eq:YErrorDef}
\f g_{i}\deq \sum_{t\in S_{i-3}}\f 1_{t}\bigg((N_{t}-F_{i-1})\sum_{z\in(x,t]}(N_{z}-F_{|z|})+\sum_{y \in S_1(t), \, y > t}(N_{y}-F_{|y|})\bigg)\,.
\end{equation}
This shows the desired relation, \eqref{eq:A_f_i_relations}, for $i \geq 3$ and some vector $\f g_i$. 
We now verify the additional properties of $\f g_i$. From \eqref{Eq:YErrorDef}, we immediately deduce $\f g_3 = 0$ and \ref{Item:YSupport}, i.e.\ $\supp \f g_i \subset S_{i-3}$ for all $i$.  

For the proof of \ref{Item:Yorthogonal}, we use the symmetry of $A$ and the orthogonality of $(\f f_i)_i$ to obtain 
\[
\langle \f f_{i-3},\f g_i\rangle= \langle \f f_{i-3},A\f f_i-\f f_{i+1}-F_{i-1}\f f_{i-1}\rangle=\langle A \f f_{i-3} , \f f_i \rangle = \langle \f f_{i-2}, \f f_i \rangle
=0\,. 
\]
Finally, for \ref{Item:Ynorm}, we compute the norm in \eqref{Eq:YErrorDef} and obtain 
\begin{align*}
\|\f g_{i}\|^{2} & =\sum_{t\in S_{i-3}}\Big((N_{t}-F_{i-1})\sum_{z\in(x,t]}(N_{z}-F_{|z|})+\sum_{y\geq t,y\in S_1(t)}(N_{y}-F_{|y|})\Big)^{2}\\
& \leq \sum_{t\in S_{i-3}}\Big( \sum_{z\in(x,t]}(|N_{t}-d|+|F_{i-1}-d|)(|N_{z}-d|+|F_{|z|}-d|) 
\\ & \qquad 
+\sum_{y\geq t,y\in S_1(t)}|F_{|y|}-d|+\Big|\sum_{y\geq t,y\in S_1(t)}(N_{y}-d)\Big|\Big)^{2}\\
 & \leq \sum_{t\in S_{i-3}}\Big((i-3) (d^{\frac 1 2+\gamma}+(i-1)d^{2\gamma})^2+(d+d^{\frac 1 2+\gamma})(i-2)d^{2\gamma} + d^{1+\gamma}\Big)^2 \\
 & \lesssim  |S_{i-3}|d^{2+4\gamma} i^2\,, 
\end{align*}
where we used, for the second inequality, \eqref{eq:F_i_minus_d_estimate} and   
 Proposition~\ref{pro:graphProperty} \ref{item:concentration_degrees} to estimate the first term, that $\left|\{y\in S_{1}(t),t<y\}\right|=N_t$, \eqref{eq:F_i_minus_d_estimate} and   
 Proposition~\ref{pro:graphProperty} \ref{item:concentration_degrees} to estimate the second term, and  Proposition~\ref{pro:graphProperty} \ref{item:Concentration} for the last term.
For the last inequality, we used $i^2 d^{2\gamma} \lesssim d$ by \eqref{eq:upper_bound_r_fine_rigidity} and $\gamma \leq 1/6$. 
\end{proof}

\subsection{Proof of Proposition~\ref{Prop:TriMatrixf}} \label{subsec:proof_prop_TriMatrixf}

We now prove Proposition \ref{Prop:TriMatrixf}. It will be a consequence of the following result.

\begin{lemma} \label{lem:normF} 
Let $\gamma \in (0,1/6]$ and $r\in \N$ satisfy \eqref{eq:upper_bound_r_fine_rigidity}. 
Then, on the high-probability event from Proposition~\ref{pro:graphProperty}, we have 
\[
\|\f f_{i}\|^{2} = |S_{i}|+(i-2)D_{x}d^{i-2}\Big(1+O\big(id^{-\frac{1}{2}+3\gamma}+ i^3 d^{-1 + 4\gamma} \big)\Big)
\]
uniformly for any $i \in [2,r] \cap \N$. 
\end{lemma}

\begin{proof}[Proof of Proposition \ref{Prop:TriMatrixf}]
From \eqref{eq:def-fi}, we deduce $\norm{\f f_0}^2 = 1$, $\norm{\f f_1}^2 = \abs{S_1} =  D_x$ and $\norm{\f f_2}^2 = \abs{S_2}$. Hence, Lemma \ref{Rem:MFraction} implies 
\[
M_{0 1}^2 = \frac{\|\f f_{1}\|^{2}}{d\|\f f_{0}\|^{2}}=\frac{D_{x}}{d} = \alpha_x,
\qquad M_{1 2}^2 = \frac{\|\f f_{2}\|^{2}}{d\|\f f_{1}\|^{2}}=\frac{|S_{2}|}{d D_{x}} = \beta_x\,.
\]
For all $i\geq 2$, we conclude from Lemma~\ref{Rem:MFraction} that 
\begin{align}
d M_{i\,i+1}^2 & = \frac{\|\f f_{i+1}\|^{2}}{\|\f f_{i}\|^{2}} = \frac{|S_{i+1}|+(i-2)D_{x}d^{i-1}+D_{x}d^{i-1}}{|S_{i}|+(i-2)D_{x}d^{i-2}}+O\big(i^2d^{-\frac 1 2 + 3\gamma} \big)
\notag\\ \label{eq:M_k_k_plus_1_squared_times_d} 
& = d+\frac{D_{x}d^{i-1}}{|S_{i}|}+O\big(i^2d^{-\frac 1 2 + 3\gamma} \big)  =  d+1+O\big(i^2d^{-\frac 1 2 + 3\gamma} \big)\,.
\end{align}
Here, in the second step, we used Lemma~\ref{lem:normF}, 
$i d^{- \frac 1 2 + 3 \gamma} + i^3 d^{-1 + 4 \gamma} \lesssim i d^{- \frac 1 2 + 3 \gamma} \leq c$ 
as $r d^{-2\gamma} \lesssim 1$ by \eqref{eq:upper_bound_r_fine_rigidity} and $r d^{-\frac 1 2 + 3 \gamma} \leq c$ by assumption,
and chose the constant $c>0$ sufficiently small. 
In the third and fourth step, we used 
Proposition~\ref{pro:graphProperty} \ref{item:concentration_S_i_Omega}. 
As $i \leq r$, dividing \eqref{eq:M_k_k_plus_1_squared_times_d} by $d$ and taking the square root of the result 
prove \eqref{eq:comparison_M_Z} in the entrywise matrix norm. Since $M - Z_{\fra d}(\alpha_x,\beta_x)_{[0,r]}$ is tridiagonal, the bound \eqref{eq:comparison_M_Z} in operator norm follows. 
\end{proof}

\begin{proof}[Proof of Lemma~\ref{lem:normF}]
Owing to the definition of $\f f_2$ in \eqref{eq:def-fi}, $\norm{\f f_2}^2 = \abs{S_2}$, which implies the claim for $i = 2$. Therefore, we shall assume $i \geq 3$ in the following. 
The definition of $\f f_i$ in \eqref{eq:def-fi} yields
\begin{equation} \label{eq:proof_normF_aux1} 
\|\f f_{i}\|^{2}  = |S_{i}|+ \sum_{y\in S_{i-2}}\bigg(\sum_{z\in(x,y]}(N_{z}-F_{|z|})\bigg)^2\,.
\end{equation}
Since $\sum_{y\in S_{i-2}}\sum_{z\in(x,y]}(N_{z}-F_{|z|}) = \langle \f f_{i},\f f_{i-2} \rangle =0$, we have 
\[
\sum_{y\in S_{i-2}}\bigg(\sum_{z\in(x,y]}(N_{z}-F_{|z|})+\delta \bigg)^2 = \sum_{y\in S_{i-2}}\bigg(\sum_{z\in(x,y]}(N_{z}-F_{|z|})\bigg)^2 + |S_{i-2}|\delta^2
\]
for any $\delta\in \mathbb{R}$. We set $\delta \deq \sum_{j=1}^{i-2} (F_{j}-d)$ and obtain
\begin{equation} \label{eq:proof_normF_aux2} 
 \sum_{y\in S_{i-2}}\bigg(\sum_{z\in(x,y]}(N_{z}-F_{|z|})\bigg)^2 = \sum_{y\in S_{i-2}}\bigg(\sum_{z\in(x,y]}(N_{z}-d)\bigg)^2  + O(D_x d^{i-3} i^4 d^{4\gamma})\,,  
\end{equation}
where we used that $|S_{i-2}|\lesssim  D_x d^{i-3}$ by Proposition \ref{pro:graphProperty} \ref{item:concentration_S_i_Omega}  
and $ \delta = O(i^2 d^{2\gamma})$ by \eqref{eq:F_i_minus_d_estimate} in Lemma~\ref{lem:Festimate}.  
We now calculate 
\begin{equation}\label{eq:MeanTermFNorm}
\sum_{y\in S_{i-2}}\bigg(\sum_{z\in(x,y]}(N_{z}-d)\bigg)^2   = \sum_{y\in S_{i-2}}\bigg(\sum_{z\in(x,y]}(N_{z}-d)^2 + 2 \sum_{z_1,z_2\in(x,y],\, z_1<z_2}  (N_{z_1}-d)(N_{z_2}-d) \bigg)\,. 
\end{equation}
As $\mathbb{G}|_{B_r}$ is a tree by Proposition~\ref{pro:graphProperty} \ref{item:tree_in_balls}, we have
\begin{align}
 \sum_{y\in S_{i-2}}\sum_{z\in(x,y]} (N_{z}-d)^2  &  =\sum_{z\in B_{i-2}\setminus \{x\}}\sum_{y\in S_{i-2},\, z\leq y} (N_{z}-d)^2 
\nonumber \\  &   = \sum_{j = 1}^{i-2}\sum_{z\in S_{j}}(N_{z}-d)^2N_{z}^{(i-2-j)}  = (i-2)D_xd^{i-2}\big(1 + O\big( d^{-\frac 1 2+\gamma}\big) \big)\,.\label{eq:proof_normF_aux3} 
\end{align}
Here, in the last step, we used 
\begin{equation} \label{eq:N_z_i_2_j_expansion} 
 N_z^{(i-2-j)} = d^{i - 2 - j} (1 + O(d^{-\frac 1 2 + \gamma})) ( 1 + O(d^{-1 + \gamma})) = d^{i -2 -j}(1 + O(d^{-\frac 1 2 + \gamma}))\,, 
\end{equation}
Proposition~\ref{pro:graphProperty} \ref{item:MeanQuadratic} and $D_x \gtrsim d$ by the definition of $\cal W$ in \eqref{eq:def_cal_W}. 
The expansion \eqref{eq:N_z_i_2_j_expansion} follows from Proposition~\ref{pro:graphProperty} \ref{item:Concentration} and \ref{item:concentration_degrees}.

Next, we estimate the second term of \eqref{eq:MeanTermFNorm}. Owing to \eqref{eq:N_z_i_2_j_expansion}, we get
\begin{align}
&\mspace{-10mu}\sum_{y\in S_{i-2}}\sum_{\substack{z_1,z_2\in(x,y] \nonumber \\ z_1<z_2}}  (N_{z_1}-d)(N_{z_2}-d) = \sum_{z_1\in B_{i-3}\setminus \{x\}}\sum_{\substack{z_2 \in B_{i-2}\\z_2>z_1}}\sum_{\substack{y\in S_{i-2}\\ y \geq z_2}} (N_{z_1}-d)(N_{z_2}-d) \nonumber
\\ & \quad = \sum_{j=1}^{i-3}\sum_{z_1\in S_{j}} \sum_{\substack{z_2 \in B_{i-2}\\z_2>z_1}}(N_{z_1}-d)(N_{z_2}-d) N_{z_2}^{(i-2-\abs{z_2})} \nonumber
\\ & \quad = \sum_{j=1}^{i-3}\sum_{z_1\in S_{j}} \sum_{\substack{z_2 \in B_{i-2}\\z_2>z_1}} \left((N_{z_1}-d)(N_{z_2}-d) d^{i-2-|z_2|}+O\big(\abs{N_{z_1}-d}\abs{N_{z_2}-d}d^{i-\abs{z_2}- \frac 5 2 + \gamma}\big) \right)\,. \label{eq:lastSumErrorDivision}
\end{align}
The first term on the right-hand side of \eqref{eq:lastSumErrorDivision} can be bounded by 
\begin{align}
\absBB{\sum_{j=1}^{i-3}\sum_{z_1\in S_{j}} \sum_{\substack{z_2 \in B_{i-2}\\z_2>z_1}} (N_{z_1}-d)(N_{z_2}-d) d^{i-2-|z_2|}} 
& \lesssim \sum_{j=1}^{i-3}\sum_{z_1\in S_{j}} d^{\frac 1 2+\gamma}\sum_{j_2=j+1}^{i-2}d^{i-2-j_2}\absbb{\sum_{\substack{z_2\in S_{j_2}\\z_2>z_1}}(N_{z_2}-d)}
\notag \\ & \lesssim \sum_{j=1}^{i-3}|S_{j}| d^{\frac 1 2+\gamma}\sum^{i-2}_{j_2=j+1}d^{i-2-j_2}d^{\frac 1 2(j_2-j+1)+\gamma} 
\notag \\ & \leq  \sum_{j=1}^{i-3} D_x d^{j-1} d^{i-1-j}  \sum_{j_2=j+1}^{i-2} d^{2\gamma+\frac 1 2(j-j_2)}
\notag \\ \label{eq:proof_normF_aux4} 
&  \lesssim (i-3) D_x d^{i-2} d^{2\gamma-\frac 1 2}\,, 
\end{align}
where we used Proposition~\ref{pro:graphProperty} \ref{item:concentration_degrees} for the first inequality, Proposition~\ref{pro:graphProperty} \ref{item:Concentration} for the second inequality and 
Proposition~\ref{pro:graphProperty} \ref{item:concentration_S_i_Omega} for the third inequality. 

For the second term on the right-hand side of \eqref{eq:lastSumErrorDivision}, we estimate 
\begin{align} 
\sum_{j=1}^{i-3}\sum_{z_1\in S_{j}} \sum_{\substack{z_2 \in B_{i-2}\\z_2>z_1}} \abs{N_{z_1}-d}\abs{N_{z_2}-d}d^{i-\abs{z_2} - \frac 5 2 + \gamma} 
& \leq  \sum_{j=1}^{i-3}\sum_{z_1\in S_{j}} \sum_{\substack{z_2 \in B_{i-2}\\z_2>z_1}} d^{3\gamma+\frac 3 2+i-3-|z_2|} \notag \\ \label{eq:proof_normF_aux5} 
& \lesssim (i-2) D_x d^{i-2} (i-3) d^{3\gamma - \frac 1 2}\,,
\end{align}
where the first step follows from Proposition~\ref{pro:graphProperty} \ref{item:concentration_degrees} 
and, for the second step, we first used 
$\sum_{z_2 \in B_{i-2},\, z_2 > z_1} d^{-\abs{z_2}} \lesssim (i-2) d^{-j}$ if $z_1 \in S_j$ by Proposition~\ref{pro:graphProperty} \ref{item:Concentration} as well as \ref{item:concentration_degrees} 
and then \ref{item:concentration_S_i_Omega}. 

Finally, by plugging \eqref{eq:proof_normF_aux4} and \eqref{eq:proof_normF_aux5} into \eqref{eq:lastSumErrorDivision} 
and then the result together with \eqref{eq:proof_normF_aux3} into \eqref{eq:MeanTermFNorm}, we get 
\begin{equation} \label{eq:proof_normF_aux6} 
\sum_{y\in S_{i-2}}\bigg(\sum_{z\in(x,y]}(N_{z}-d)\bigg)^2 = (i-2)D_x d^{i-2}\left(1+ O\left(i d^{3\gamma - \frac 1 2} \right)\right)\,. 
\end{equation}
Using \eqref{eq:proof_normF_aux6} in \eqref{eq:proof_normF_aux2} and the result then in \eqref{eq:proof_normF_aux1} completes the proof of Lemma~\ref{lem:normF}. 
\end{proof}

\section{Proof of block diagonal approximation}\label{sec:block_diagonal_approximation} 

In this section we prove Proposition \ref{prop:blockMatrix} and Corollary \ref{Cor:SpecMax}. Throughout this section, we assume that \eqref{eq:ConditionAlphaMaxMin} is satisfied. 

\subsection{Proof of Proposition \ref{prop:blockMatrix}} \label{sec:proof_prop_block_matrix} 

For the proof of Proposition \ref{prop:blockMatrix}, we use the following result on the stray eigenvalue whose proof is deferred to Section~\ref{sec:StrayEigenvalue} below.
Throughout the rest of the paper, we write $\f e \deq N^{-1/2} \f 1_{[N]}$ 
with the notation $\f 1_{[N]}$ from Section~\ref{sec:basic_definitions}. 

\begin{proposition} \label{prop:stray}
For $d \gg 1$ and $r \ll \frac{d}{\log \log N}$ there exists a normalized vector $\f q$ supported in the complement of $\bigcup_{x \in \cal U} B_{r + 1}(x)$ such that with high probability
\begin{equation*}
\normb{\big(H -  (d^{1/2} + d^{-1/2} + d^{-3/2})\big) \f q} \lesssim d^{-5/2} + d^{1/2} \ind{d < (\log N)^{1/4}}  \,, \qquad \norm{\f q - \f e} \lesssim d^{-1/2}\,.
\end{equation*}
\end{proposition}

For the proof of Proposition \ref{prop:blockMatrix}, we separately consider the regimes $d > (\log N)^{3/4}$ and $d \leq (\log N)^{3/4}$. We start with the former regime. 

\begin{proof}[Proof of Proposition~\ref{prop:blockMatrix} for $d > (\log N)^{3/4}$] 
First, we recall $r_{\cal W}$ from Proposition~\ref{pro:fine_rigidity}. 
Moreover, we introduce 
\begin{equation} \label{eq:def_radii} 
 r_{\cal V}\deq r_{\delta_\star}, \qquad \qquad  r_{\cal U} \deq \floor{c \sqrt{\log N}}, \qquad \qquad r_\star \deq \max \{ r_{\cal W}, r_{\cal V}, r_{\cal U} \}\,, 
\end{equation}
where $r_{\delta}$ is as in Proposition~\ref{pro:intermediate_rigidity}, $\delta_\star$ as in \eqref{eq:deltaStarDef} and the constant $c>0$ is as in \cite[eq.~(1.8)]{ADK20}. 
Note that $r_\star \ll \frac{d}{\log \log N}$ due to Proposition~\ref{pro:intermediate_rigidity}, Proposition~\ref{pro:fine_rigidity} and $d >(\log N)^{3/4}$. 

We choose the vectors $\f w_\sigma(x)$ as follows. 
\begin{enumerate}

\item 
If $x\in \cal W$ then $\f w_\sigma(x)\deq \f w_{\sigma}(x) \text{ as  in Proposition \ref{pro:fine_rigidity}}$.
\item 
If $x\in \cal V \setminus \cal W$ then $\f w_\sigma(x) \deq \f w_\sigma(x) \text{ as in Proposition \ref{pro:intermediate_rigidity}}$. 
\item If $x\in \cal U \setminus \cal V$ then $\f w_\sigma(x) \deq \f v^{\tau}_\sigma(x)$ from  \cite[eq.~(3.5)]{ADK20} with 
\begin{equation} \label{eq:choice_tau} 
 \tau \deq 1 + \frac{(\log N)^{1/4}}{\sqrt{d}} \sqrt{\log d} \, . 
\end{equation}
To guide the reader, we recall that $\f v^{\tau}_\sigma(x)$ from  \cite[eq.~(3.5)]{ADK20} is given by 
\begin{equation} \label{eq:def_v_tau_sigma} 
\f v_\sigma^{\tau}(x) \deq \sum_{i=0}^{r_{\cal U}} u_i \frac{\f 1_{S_i^\tau(x)}}{\norm{\f 1_{S_i^\tau(x)}}} 
\end{equation} 
(see also the definition of $\f v$ in \eqref{eq:def_v_intermediate_rigidity}). 
Here, the vector 
$(u_i)_{i \in \N}$ is an eigenvector of $Z_1(\alpha_x,1)$ associated with the eigenvalue $\Lambda(\alpha_x,1)$ 
(see Proposition~\ref{prop:Lambda_ab} below) such that $\sum_{i=0}^{r_{\cal U}} u_i^2 = 1$, 
and $S_i^\tau(x)$ is the sphere of radius $i$ around $x$ in the \emph{pruned graph} $\mathbb G^{\tau}$, 
which is a subgraph of the Erd{\H o}s-R\'enyi graph $\mathbb G$. 
The detailed definition of the pruned graph $\mathbb G^{\tau}$ can be found in \cite[Proposition~3.1]{ADK20}. 
Here, $\tau$ is chosen as in \eqref{eq:choice_tau}.   
\end{enumerate}

\begin{lemma}\label{lem:Disjointfamily}
Let $d >(\log N)^{3/4}$ and $r_\ast$ as in \eqref{eq:def_radii}. 
 Then, with high probability, the balls  
$B_{r_\star + 10}(x)$ and $B_{r_\star+ 10}(y)$ are disjoint for all $x \in \cal V$ and $y \in\cal U$ with $x \neq y$.  
\end{lemma}

\begin{proof}[Proof of Lemma \ref{lem:Disjointfamily}]
We apply Proposition~\ref{pro:prob_estimates_1} with $\delta = 1$ and $r = \frac{c_* d}{\log \log N}$ 
for a sufficiently small constant $c_*>0$. 
By  Proposition \ref{pro:prob_estimates_1} \ref{item:degrees_bounded}, for all $x\in \cal V_1$ and $y\in [N]\setminus \{x\}$ with $\alpha_y>2$ we have $\dist(x,y)> \frac{c_* d}{\log \log N} \gg  2 r_\star + 21 $ in the regime $d>(\log N)^{3/4}$, where $\dist$ denotes the distance in the graph $\mathbb{G}$. 
Since $\cal V \subset \cal V_1$ and $\cal U \subset \{ y \colon \alpha_y > 2\}$, this proves Lemma~\ref{lem:Disjointfamily}. 
\end{proof}

Let $\f q $ be the vector from Proposition \ref{prop:stray} for $r = r_\star = \max \{ r_{\cal W}, r_{\cal V}, r_{\cal U} \}$.

\begin{corollary} \label{cor:orthonormal} 
The family $(\f q, (\f w_{\sigma}(x))_{x\in \cal U, \, \sigma \in \{\pm\}})$ is orthonormal and $\supp \f w_\sigma(x) \subset B_{r_\star}(x)$ for all $x \in \cal U$ 
and $\sigma \in \{\pm\}$. 
\end{corollary} 

\begin{proof}[Proof of Corollary~\ref{cor:orthonormal}] 
Owing to their definitions, these vectors are all normalized. 
For $x \in \cal U\setminus \cal V$, we have $\f w_+(x) = \f v_+^\tau(x) \perp \f v_-^\tau(x) = \f w_-(x)$ by 
\cite[Remark~3.3]{ADK20}.  
The supports of $\f v_\sigma^\tau(x)$ and $\f v_{\sigma'}^\tau (y)$ are disjoint for all $x$, $y \in \cal U \setminus \cal V$ with $x \neq y$ and $\sigma$, $\sigma' \in \{ \pm\}$ by \cite[Remark~3.3]{ADK20}. 
For any $x \in\cal U\setminus \cal V$ and $\sigma \in \{\pm\}$, we have 
\begin{equation} \label{eq:supp_v_B_r_cal_U} 
\supp \f v_\sigma^\tau(x) \subset B_{r_{\cal U}}(x)
\end{equation} 
 by \cite[eq.~(3.5), (1.8) and Proposition~3.1 (iv)]{ADK20}.  
Moreover, $\supp \f w_\sigma(x) \subset B_{r_\star}(x)$ for all $x \in \cal V$ by Proposition~\ref{pro:intermediate_rigidity} and Proposition~\ref{pro:fine_rigidity}. 
Hence, by Proposition~\ref{prop:stray}, the supports of $\f q$ and $\f w_\sigma(x)$ are disjoint for each $x \in \cal U$ and $\sigma \in \{ \pm\}$. 
From Lemma~\ref{lem:Disjointfamily}, we thus conclude that 
the supports of $\f w_\sigma(x)$ and $\f w_{\sigma'}(y)$ are disjoint for all $x \in \cal V$ and $ y \in \cal U$ with $x \neq y$. 
Finally, $\f w_+(x) \perp \f w_-(x)$ for all $x \in \cal V$ by Proposition~\ref{pro:intermediate_rigidity} and Proposition~\ref{pro:fine_rigidity}. 
\end{proof}

We expand the orthonormal family $(\f q, (\f w_{\sigma}(x))_{x\in \cal U, \sigma \in \{\pm\}})$ by $(\f t_i)_{i\in [N-1-2|\cal U|]}$ to an orthonormal basis of $\R^N$. We choose $U$ as the matrix with columns $(\f q, (\f w_{\sigma}(x))_{x\in \cal U, \sigma \in \{\pm\}},(\f t_i)_{i\in [N-1-2|\cal U|]})$.
Note that $U$ is orthogonal. 

We view $U^{-1} H U$ as a $5 \times 5$-block matrix (see \eqref{eq:LargeBlock}) whose blocks, on the level of $H$, are induced by 
the families $\f q$, $(\f w_\sigma(x))_{x \in \cal W, \sigma \in \{\pm\}}$, $(\f w_\sigma(x))_{x \in \cal V \setminus \cal W, \sigma \in \{\pm\}}$, $(\f w_\sigma(x))_{x \in \cal U\setminus\cal V, \sigma \in \{\pm\}}$ and $(\f t_i)_{i \in [N-1- 2\abs{\cal U}}$, 
respectively. 
We now determine the structure of these blocks.  
Let $x\in \cal V$, $y\in \cal U$ with $x\neq y$ and $\sigma,\sigma'\in \{\pm \}$. 
Since $\supp H \f w_{\sigma'} (y) \subset B_{r_\star+1}(y)$ and $\supp \f w_\sigma(x) \subset B_{r_\star}(x)$, 
we conclude from Lemma~\ref{lem:Disjointfamily} that $\langle \f w_{\sigma}(x),H \f w_{\sigma'}(y)\rangle =0$. 
Therefore, the $(2,3)$-, $(2,4)$- and $(3,4)$-block of $U^{-1}HU$ as well as their counterparts  
induced by the Hermitian symmetry vanish.  
Since $(\supp \f q)^c \supset \bigcup_{x \in \cal U} B_{r_\star+1}(x)$ by Proposition~\ref{prop:stray}, the $(1,2)$-,  $(1,3)$-,  $(1,4)$-block (and their counterparts)  
are also zero.  

By Proposition~\ref{pro:intermediate_rigidity} and Proposition~\ref{pro:fine_rigidity}, we also have 
$\langle \f w_{+}(x), H \f w_{-}(x)\rangle =\langle \f w_{-}(x), H \f w_{+}(x)\rangle =0$ for all $x\in \cal V$. 
Therefore, the $(2,2)$- and the $(3,3)$-block of $U^{-1} HU$ are diagonal matrices whose entries are $\langle \f w_{\sigma}(x),H \f w_{\sigma}(x)\rangle $ indexed by $x$ and $\sigma$.
This proves the structure given in \eqref{eq:LargeBlock}.

\begin{proof}[Proof of Proposition \ref{prop:blockMatrix} \ref{Item:WsetBound}, \ref{Item:VsetBound}, \ref{Item:UsetBound} for $d>(\log N)^{3/4}$]
For any $x\in \cal W$, \eqref{eq:fine_rigidity} in Proposition \ref{pro:fine_rigidity} implies\[|\eps_{x,\sigma}|=\left|\langle \f w_\sigma(x), H \f w_\sigma(x)\rangle -\sigma\Lambda_{\fra d}(\alpha_x,\beta_x)\right| \lesssim  \frac{d^{-\frac{1}{2} + 3\gamma}}{ d\am} \bigg( 1 + \bigg(\frac{\log d}{\log \am} \bigg)^2\frac{\am^4}{(\am - 2)^4}   \bigg)\,.\]
Let $\f a=\sum_{x\in \cal W}\sum_{\sigma=\pm} a_{x,\sigma} \f w_\sigma(x)$. 
As the balls $(B_{r_\star+10}(x))_{x \in \cal W}$ are disjoint by Lemma~\ref{lem:Disjointfamily}, we have
 \[ 
\|E_{\cal W}\f a\|^2 \leq 2 \sum_{x\in \cal W}\sum_{\sigma=\pm} |a_{x,\sigma}|^2 \big\|\big(H-\scalar{\f w_\sigma(x)}{
H \f w_\sigma(x)} \big)\f w_\sigma(x) \big\|^2\lesssim (d \am)^{-20} \sum_{x,\sigma}|a_{x,\sigma}|^2\,,
 \]
where we used 
\eqref{eq:bound_H_minus_eigenvalue_ball_fine} from
Proposition \ref{pro:fine_rigidity} in the last step. 
This proves Proposition~\ref{prop:blockMatrix} \ref{Item:WsetBound}.  

The estimates on $\abs{\eps_{x,\sigma}}$ for $x \in \cal V\setminus \cal W$ and on $\|E_{\cal V \setminus \cal W}\|$ are obtained in the same way using Proposition \ref{pro:intermediate_rigidity} instead of Proposition~\ref{pro:fine_rigidity} as well as \eqref{eq:lower_bound_alpha_x_minus_2_am_minus_2}.

We now bound $\abs{\eps_{x,\sigma}}$ for $x \in \cal U\setminus \cal V$, $\|E_{\cal U \setminus \cal V}\|$  and $\|\cal E_{\cal U \setminus \cal V}\|$. 
We make use of the matrices $H^\tau$ from \cite[Definition~3.6]{ADK20} and $\wh{H}^\tau$ from \cite[Defintion~3.10]{ADK20} whose definitions we recall now. 
For $\tau$ as in \eqref{eq:choice_tau}, let $A$ and $A^\tau$ be the adjacency matrices of the Erd{\H o}s-R\'enyi graph $\mathbb G$ and the pruned graph $\mathbb G^\tau$, 
respectively. Then we define
\begin{align} 
\chi^\tau &\deq  \text{orthogonal projection onto } \op{span}\{ \f 1_y \colon y \notin B_{2r_{\cal U}}^\tau(x) \text{ for any } x \text{ such that } \alpha_x \geq \tau \}, \label{eq:def_chi_tau} \\  
H^\tau &\deq  (A^\tau - \chi^\tau (\E A) \chi^\tau)/\sqrt{d}\,.\label{eq:def_H_tau}  
\end{align} 
For the definition of $\wh{H}^\tau$, we first introduce the orthogonal projections 
\begin{equation} \label{eq:def_Pi_tau} 
 \Pi^\tau \deq \sum_{x \in \cal U, \, \sigma = \pm} \f v_\sigma^\tau(x)\f v_\sigma^\tau(x)^*,  
\qquad \overline{\Pi}^\tau \deq \id - \Pi^\tau\,, 
\end{equation} 
where $\f v_\sigma^\tau(x)$ is as in \eqref{eq:def_v_tau_sigma}. 
These definitions of $\Pi^\tau$ and $\overline{\Pi}^\tau$ coincide with those from \cite[eq.~(3.15)]{ADK20}. 
As in \cite[eq.~(3.16)]{ADK20}, we set 
\begin{equation} \label{eq:def_wh_H_tau} 
 \wh{H}^\tau \deq \sum_{x \in \cal U} \sum_{\sigma = \pm} \sigma \Lambda(\alpha_x,1) \f v_\sigma^{\tau}(x) 
\f v_\sigma^{\tau}(x)^* + \ol{\Pi}^\tau H^\tau \ol{\Pi}^\tau. 
\end{equation} 
By definition, we have 
$\wh{H}^\tau \f v_\sigma^\tau(x) = \sigma \Lambda(\alpha_x,1) \f v_\sigma^\tau(x)$ for all $x \in \cal U$ 
and $\sigma \in \{ \pm\}$. 
To simplify the notation in the following, we introduce the control parameters 
\begin{equation} \label{eq:def_xi}  
 \xi \deq \frac{\sqrt{\log N}}{d}  \log d\,, \qquad \qquad \xi_u \deq \frac{\sqrt{\log N}}{d} \frac{1}{u}  
\end{equation}  
for $u>0$ (see also \cite[eq.~(1.9)]{ADK20}). 
Owing to \eqref{eq:beta_x_equal_1_error_term_on_cal_U} and $\alpha_x \geq 2 +\xi ^{1/4}$ for $x \in \cal U$, the conditions of  
Lemma~\ref{lem:EstimateLambdad} below are satisfied and we obtain 
$\Lambda_{\fra d}(\alpha_x,\beta_x) = \Lambda(\alpha_x,1) + O(\sqrt{\log N}/d)$  
from Lemma~\ref{lem:EstimateLambdad}, $\alpha_x \geq 2$, \eqref{Lambda_expansion} and \eqref{eq:beta_x_equal_1_error_term_on_cal_U}. 
Hence, $\norm{(\Lambda -\wh{H}^\tau)\f v} \lesssim \sqrt{\log N}/d$ for any normalized $\f v \in  \op{span}\{ \f v_\sigma^\tau(x) \colon x \in \cal U, \, \sigma \in \{ \pm\}\}$, 
where $\Lambda$ is the matrix satisfying $\Lambda \f v_\sigma^\tau(x) = \sigma \Lambda_{\fra d}(\alpha_x,\beta_x) \f v_\sigma^\tau(x)$ for all $x \in \cal U$ and $\sigma \in \{\pm\}$. 
On the other hand, for any such $\f v$, we have $\norm{(H-\wh{H}^\tau)\f v} \leq \norm{(\E H)\f v} + 
\norm{\E H - \chi^\tau (\E H) \chi^\tau} +  \norm{(H - \chi^{\tau}(\E H)\chi^\tau)-H^{\tau}} + \norm{H^\tau - \wh{H}^\tau} \lesssim \xi^{1/2}$.  
Here, we used \eqref{eq:supp_v_B_r_cal_U} and Lemma~\ref{lem:B_k_estimate} below with $r = r_{\cal U}$ to estimate the first term, 
\cite[eq.~(3.13)]{ADK20} with $\chi^\tau$ from \eqref{eq:def_chi_tau} for the second term, 
\cite[Lemma~3.8]{ADK20} and  
$\xi_{\tau - 1} \lesssim \xi^{1/2}$ by our choice $\tau = 1 + \xi^{1/2}$ from \eqref{eq:choice_tau}  
for the third term and \cite[Lemma~3.11]{ADK20} for the last term. 
In particular, for later use, we note that 
\begin{equation} \label{eq:rough_rigidity_large_d} 
\norm{(H - \sigma \Lambda_{\fra d}(\alpha_x,\beta_x)) \f v _\sigma^\tau(x)} \lesssim \xi^{1/2} 
\end{equation} 
for all $x \in \cal U$ and $\sigma \in \{ \pm\}$. This proves Proposition~\ref{prop:blockMatrix} \ref{Item:UsetBound} for $d > (\log N)^{3/4}$. 
\end{proof}

Proposition~\ref{prop:blockMatrix} \ref{Item:StrayEstimate} follows from Proposition \ref{prop:stray} and we now prove \ref{Item:bulkBound}.

\begin{proof}[Proof of Proposition \ref{prop:blockMatrix} \ref{Item:bulkBound} for $d>(\log N)^{3/4}$]
We introduce the orthogonal projections 
 \[\Pi \deq \sum_{x\in \cal U,\sigma=\pm} \f w_\sigma(x) \f w_\sigma(x)^*\,, \qquad \qquad \overline{\Pi} \deq \id -\Pi\,. \] 
By \cite[Theorem 1.7]{ADK20} and Lemma~\ref{lem:convergence_max_alpha_x_upper_bound_D_x}, there exists at most one eigenvalue of $H$ larger than $(1 + o(1))\sqrt{\am} \ll \nu_\mathrm{s}$ if $d > (\log N)^{3/4}$ and therefore the same for $\overline{\Pi}H\overline{\Pi}$. 
Hence, denoting by $\lambda_2[M]$ the second eigenvalue of the matrix $M$, we have
\begin{equation*}
\|X\|  = \lambda_2 \qBB{\begin{pmatrix}
\nu_{\mathrm{s}} & 0 \\ 0 & X
\end{pmatrix}} \leq \lambda_2 \qBB{\begin{pmatrix}
\nu_{\mathrm{s}} & E_{\mathrm{s}}^* \\ E_{\mathrm{s}} & X
\end{pmatrix}} + \|E_{\mathrm{s}}\|
 \leq \|\overline{\Pi}H\overline{\Pi}-\f f \f f^*\| + \|E_{\mathrm{s}}\|
\end{equation*}
for any $\f f\in \mathbb{R}^{N} $ by eigenvalue interlacing. 
Hence, by choosing $\f f \f f^* = \overline{\Pi}\chi^\tau (\mathbb{E}H) \chi^\tau\overline{\Pi}$, we obtain  
\begin{equation} \label{eq:upper_bound_X_Pi_tau_H_tau} 
\|X\| 
 \leq \|\overline{\Pi}^\tau H^\tau \overline{\Pi}^\tau \|+ 2 \|H^{\tau}\| \| \overline{\Pi}-\overline{\Pi}^\tau\| + \|H^\tau -(H- \chi^\tau (\mathbb{E}H) \chi^\tau)\|+\| E_{\mathrm{s}}\|, 
\end{equation}
where $\chi^\tau$, $\ol{\Pi}^\tau$ and $H^\tau$ are defined in \eqref{eq:def_chi_tau}, \eqref{eq:def_Pi_tau} and \eqref{eq:def_H_tau}, respectively. 

We now estimate the four terms on the right-hand side of \eqref{eq:upper_bound_X_Pi_tau_H_tau}. 
For the first one, we obtain 
$\|\overline{\Pi}^\tau H^\tau \overline{\Pi}^\tau \|\leq 2\tau + O(\xi^{1/2}) = 2 + O(\xi^{1/2})$ 
from \cite[Proposition~3.12]{ADK20} and our choice of $\tau$ in \eqref{eq:choice_tau}. 
For the second one, we note that 
\[
\| \Pi^\tau -\Pi\| = \normbb{\sum_{x\in \cal V,\sigma=\pm} \Big(\f w_\sigma(x) \f w_\sigma(x)^* - \f v^\tau_\sigma(x) \f v^\tau_\sigma(x)^*\Big)}  \leq 4 \max_{x\in \cal V,\sigma=\pm} \| \f w_\sigma(x) - \f v^\tau_\sigma(x)\|
\]
as $\f w_\sigma(x) = \f v^\tau_\sigma(x)$ for $x\in \cal U\setminus \cal V$ and $(\f v^\tau_\sigma(x),\f w_{\sigma'}(x))_{x\in \cal V}$ are supported on disjoint balls by Corollary~\ref{cor:orthonormal}, 
\eqref{eq:supp_v_B_r_cal_U} and Lemma~\ref{lem:Disjointfamily}. 
Below we shall see that for any $x \in \cal V$, we have 
\begin{equation} \label{eq:bound_w_minus_v} 
\| \f w_\sigma(x) - \f v^\tau_\sigma(x)\| \lesssim \xi^{1/2} (\am-2)^{-2}\am^{3/2}\,, 
\end{equation} 
and therefore
$\|H^{\tau}\| \| \overline{\Pi}-\overline{\Pi}^\tau\| \lesssim \xi^{1/2} (\am-2)^{-2}\am^{2}$,  
where we used that $\|H^\tau\|\lesssim \Lambda(\am) = \frac{\am}{\sqrt{\am - 1}} \lesssim \am^{1/2}$ by
the definition of $\wh{H}^\tau$ in \eqref{eq:def_wh_H_tau}, Lemma~\ref{lem:convergence_max_alpha_x_upper_bound_D_x} as well as \cite[Lemma~3.11, Proposition~3.12]{ADK20}. 
The third term satisfies $\|H^\tau -(H-  \chi^\tau (\mathbb{E} H) \chi^\tau)\|\lesssim \xi^{1/2}$ due to \cite[Lemma~3.8]{ADK20}. 
Proposition~\ref{prop:blockMatrix} \ref{Item:StrayEstimate} yields  $\norm{E_{\mathrm{s}}}\lesssim \xi^{1/2}$. 

Therefore, to prove Proposition~\ref{prop:blockMatrix} \ref{Item:bulkBound}, it remains to establish \eqref{eq:bound_w_minus_v}. 
For its proof, we focus on the case $\sigma = + $ and apply Lemma~\ref{lem:perturbationEV} to the matrix $H^{(x,r)} = H|_{B_r(x)}$, 
where $r = r_{\cal W}$ if $x \in \cal W$ and $r = r_{\cal V}$ if $x \in \cal V \setminus \cal W$ (cf.\ \eqref{eq:def_radii}), 
the approximate eigenvalue $\lambda = \Lambda(\alpha_x,\beta_x)$ and the approximate eigenvector $\f v = \f v_+^\tau(x)$. 

If $x \in \cal V \setminus \cal W$ then we see from the proof of Proposition~\ref{pro:intermediate_rigidity} 
that $r_{\cal V} = r_{\delta_\star}$ can be chosen 
bigger than $\floor{c\sqrt{\log N}}$ for any constant $c>0$, i.e.\ $\supp \f v_+^\tau(x) \subset B_{r_{\cal V}}(x)$ 
by \eqref{eq:supp_v_B_r_cal_U}. 
Thus, the spectral gap estimate from \eqref{eq:lower_bound_spectral_gap} together 
with \eqref{eq:lower_bound_alpha_x_minus_2_am_minus_2} and the approximate eigenvector estimate in \eqref{eq:rough_rigidity_large_d} 
combined with Lemma~\ref{lem:EstimateLambdad} imply the conditions of 
Lemma~\ref{lem:perturbationEV}, which yields \eqref{eq:bound_w_minus_v}. 

For $x \in \cal W$, we see from the choice of $r_{\cal W}$ in the proof of Proposition~\ref{pro:fine_rigidity} that 
$\supp \f v_+^\tau(x) \not\subset B_{r_{\cal W}}(x)$. 
As for $x \in \cal V\setminus \cal W$, we obtain that $\norm{(\f w_\sigma(x) - \f v^\tau_\sigma(x))|_{B_{r_{\cal W}}(x)}} \lesssim \xi^{1/2} (\am - 2)^{-2} \am^{3/2}$. 
Owing to \cite[eq.'s~(3.4) and (3.5)]{ADK20} 
and the choice of $r_{\cal W}$ in the proof of Proposition~\ref{pro:fine_rigidity}, it is easy to show that 
$\norm{\f v_+^\tau(x)|_{[N] \setminus B_{r_{\cal W}}(x)}}\lesssim \xi^{1/2} (\am - 2)^{-2} \am^{3/2}$.
Therefore, $\supp \f w_{\sigma}(x) \subset B_{r_{\cal W}}(x)$ 
implies \eqref{eq:bound_w_minus_v} and completes the proof of Proposition~\ref{prop:blockMatrix} \ref{Item:bulkBound} for $d > (\log N)^{3/4}$. 
\end{proof} 
This completes the proof of Proposition~\ref{prop:blockMatrix} if $d > (\log N)^{3/4}$. 
\end{proof} 

\begin{proof}[Proof of Proposition~\ref{prop:blockMatrix} for $d \leq (\log N)^{3/4}$] 
In this regime, the proof proceeds analogously to the one for $d > (\log N)^{3/4}$. However, we need to include a few additional arguments as the results from \cite{ADK20} are not applicable throughout the entire regime $d \leq (\log N)^{3/4}$. 

We start with rigidity estimates that replace results from \cite{ADK20} used for $d > (\log N)^{3/4}$ and hold for all $x \in \cal U$ in the next proposition whose proof is given in Section~\ref{sec:RoughRigidity} below. 
\begin{proposition}[Rough rigidity for extreme eigenvalues] \label{prop:RoughCandidatEigenvector}
There is a constant $K>0$ such that if $K \log\log N \leq d \leq (\log N)^{3/4}$ then, 
with high probability, there exists a family of orthonormal vectors $(\f w_{\sigma}(x))_{x\in \cal U, \sigma \in \{ \pm\}}$ such that 
\begin{enumerate}[label=(\roman*)] 
\item \label{Item:rough_rigidity_supports_disjoint} The sets $(\supp \f w_+(x)\cup \supp \f w_-(x))_{x \in \cal U}$ are disjoint.
\item \label{Item:UcandidateInBall}
For any $x \in \cal U$ and $\sigma \in \{ \pm\}$, the vector $\f w_{\sigma}(x)$ is supported in $B_2(x)$ and  
\begin{equation} \label{eq:cal_U_approx_w_sigma} 
\f w_{\sigma}(x)=\frac{1}{\sqrt{2}}\bigg(\f 1_{x}+\sigma\frac{\f 1_{S_1(x)}}{\|\f 1_{S_{1}(x)}\|}\bigg)+O(\am^{-1/2})\,. 
\end{equation} 
\item \label{Item:UcandidateEstimate}
For any normalized $\f v \in \op{span}\{ \f w_\sigma(x) \colon x \in \cal U, \, \sigma \in \{ \pm \} \}$, we have 
\[ \norm{ ( H - D) \f v} \lesssim
\frac{{\log \log N}}{\sqrt{d}} + \frac{1}{\sqrt{\am}}\,, 
\] 
where the matrix $D$ is defined through $D \f w_\sigma(x) = \sigma \Lambda_{\fra d}( \alpha_x,\beta_x) \f w_\sigma(x)$ 
for all $x \in \cal U$ and $\sigma \in \{\pm \}$. 
\end{enumerate}
\end{proposition}

To define the matrix $U$ in Proposition~\ref{prop:blockMatrix} for $d \leq (\log N)^{3/4}$, 
 we choose the vectors $\f w_\sigma(x)$ as follows. 
\begin{enumerate}
\item 
If $x\in \cal W$ then $\f w_\sigma(x)$ is chosen as in Proposition \ref{pro:fine_rigidity}.
\item 
If $x\in \cal V \setminus \cal W$ then $\f w_\sigma(x)$ is chosen as in Proposition \ref{pro:intermediate_rigidity}. 
\item If $x\in \cal U \setminus \cal V$ then $\f w_\sigma(x)$ is chosen as in Proposition \ref{prop:RoughCandidatEigenvector}. 
\end{enumerate}

With $r_{\cal W}$ from Proposition~\ref{pro:fine_rigidity}, we introduce the radii 
\begin{equation} \label{eq:def_radii_small_d}  
r_{\cal V} \deq r_{\delta_\star}, \qquad \qquad r_\star \deq \max\{ r_{\cal W}, r_{\cal V}, 2\}\,, 
\end{equation} 
where $r_\delta$ is from Proposition~\ref{pro:intermediate_rigidity} and $\delta_{\star}$ was defined in \eqref{eq:deltaStarDef}. 
From Proposition~\ref{pro:intermediate_rigidity} and Proposition~\ref{pro:fine_rigidity}, 
we get $r_\star \ll \frac{d}{\log \log N}$.

The next lemma and the next corollary are the analogues of Lemma~\ref{lem:Disjointfamily} 
and Corollary~\ref{cor:orthonormal}, respectively, in the regime $d \leq (\log N)^{3/4}$.

\begin{lemma} \label{lem:Disjointfamily_small_d} 
Let $d \leq (\log N)^{3/4}$ and $r_\star$ as in \eqref{eq:def_radii_small_d}. 
There is a constant $K>0$ such that if $d \geq K \log \log N$ then, with high probability, the balls  
$B_{r_\star + 10}(x)$ and $B_{r_\star+ 10}(y)$ are disjoint for all $x \in \cal V$ and $y \in\cal U$ with $x \neq y$.  
\end{lemma}

\begin{proof}[Proof of Lemma~\ref{lem:Disjointfamily_small_d}] 
Let $x \in\cal V$. From Proposition \ref{pro:prob_estimates_1} \ref{item:degrees_bounded} with $\delta = \am^{1/2} \log \am$ and $r = \frac{c_*d}{\log \log N}$ for a sufficiently 
small $c_*>0$, we conclude that $\alpha_y \leq 1 + \am^{1/4} (\log \am)^{1/2}$ if $\dist(x,y) \leq \frac{c_* d}{\log \log N}$. 
Since $y \in \cal U$ implies $\alpha_y \geq \am/5$ and $\frac{c_* d}{\log \log N} > 2r_\star + 21$ for $r_\star$ from \eqref{eq:def_radii_small_d}, this proves Lemma~\ref{lem:Disjointfamily_small_d}. 
\end{proof}

Let $\f q$ be the vector from Proposition~\ref{prop:stray} for $r = r_\star= \max\{ r_{\cal W}, r_{\cal V}, 2\}$ (cf.\ \eqref{eq:def_radii_small_d}). 

\begin{corollary} \label{cor:orthonormal_small_d} 
The family $(\f q, (\f w_\sigma(x))_{x\in \cal U, \sigma \in \{ \pm\}})$ is orthonormal and $\supp \f w_\sigma(x) \subset B_{r_\star}(x)$ for all $x \in \cal U$ and $\sigma \in \{ \pm\}$. 
Moreover, the sets $(\supp\f w_+(x)\cup \supp \f w_-(x))_{x \in \cal U}$ are disjoint. 
\end{corollary}

\begin{proof}[Proof of Corollary~\ref{cor:orthonormal_small_d}] 
Given the normalization of all vectors, it suffices to show that they are orthogonal. 
By Propositions~\ref{pro:intermediate_rigidity}, \ref{pro:fine_rigidity} and \ref{prop:RoughCandidatEigenvector}, we have $\supp \f w_\sigma(x) \subset B_{r_\star}(x)$ 
and $\f w_+(x) \perp \f w_-(x)$ 
for all $x \in \cal U$ and $\sigma \in \{ \pm\}$.  
Hence, Proposition~\ref{prop:stray} implies that the supports of $\f q$ and $\f w_\sigma(x)$ are disjoint for each $x\in \cal U$ and $\sigma \in \{ \pm\}$. 
Moreover, we conclude that the sets $(\supp \f w_+(x)\cup \supp \f w_-(x))_{x \in \cal U}$ are disjoint by Proposition~\ref{prop:RoughCandidatEigenvector} \ref{Item:rough_rigidity_supports_disjoint} and Lemma~\ref{lem:Disjointfamily_small_d}. 
By Proposition~\ref{prop:RoughCandidatEigenvector}, we also know that $\f w_\sigma(x)$ and $\f w_{\sigma'}(y)$ are orthogonal if $x$, $y \in\cal U\setminus \cal V$ and $x \neq y$. 
If $x \in \cal V$, $y \in \cal U$ and $x \neq y$, the orthogonality of $\f w_\sigma(x)$ and $\f w_{\sigma'}(y)$ follows from Lemma~\ref{lem:Disjointfamily_small_d}, 
$\supp \f w_\sigma(x) \subset B_{r_\star}(x)$ and $\supp \f w_{\sigma'}(y) \subset B_{r_\star}(y)$. 
\end{proof} 

Owing to Corollary~\ref{cor:orthonormal_small_d}, we can extend $(\f q, (\f w_{\sigma}(x))_{x\in \cal U,\sigma \in \{\pm\}})$ by some 
vectors $(\f t_i)_{i\in [N-1-2|\cal U|]}$ to obtain an orthonormal basis of $\R^N$. We denote the matrix $(\f q, (\f w_{\sigma}(x))_{x\in \cal U,\sigma \in \{\pm\}},(\f t_i)_{i\in [N-1-2|\cal U|]})$ by $U$. By construction, $U$ is orthogonal. 

In the same way as in the regime $d>(\log N)^{3/4}$, parts \ref{Item:WsetBound}, \ref{Item:VsetBound} and 
\ref{Item:StrayEstimate} of Proposition \ref{prop:blockMatrix} 
follow from Proposition~\ref{pro:fine_rigidity}, Proposition~\ref{pro:intermediate_rigidity} and 
Proposition~\ref{prop:stray} using Lemma~\ref{lem:Disjointfamily_small_d} and 
Corollary~\ref{cor:orthonormal_small_d} instead of Lemma~\ref{lem:Disjointfamily} and Corollary~\ref{cor:orthonormal}, respectively. 
In the present regime, \ref{Item:UsetBound} of Proposition~\ref{prop:blockMatrix} is a consequence of Proposition~\ref{prop:RoughCandidatEigenvector} \ref{Item:UcandidateEstimate}. 

For the proof of Proposition \ref{prop:blockMatrix}  \ref{Item:bulkBound}, 
we shall use the following proposition whose formulation uses 
the diagonal matrix $Q$ with the normalized degrees on the diagonal, i.e.\ $Q\deq \text{diag}(\alpha_x\colon x\in [N])$.
\begin{proposition}[Ihara-Bass type inequality] \label{pro:ihara_bass_bound} 
Let $4 \leq d \leq 3 \log N$. Then there is $C>0$ such that, 
for all $t \geq 1 + Cd^{-1/2}$, with high probability, we have 
\[ \abs{H- \E H} \leq \bigg(t + O\bigg(\frac{\log N}{t N} \bigg) \bigg)\id + \frac{1}{t} \bigg(1 + O\bigg(\frac{1}{t\sqrt{d}}\bigg)\bigg)Q\,,  \]
as an inequality of positive definite matrices.
\end{proposition}

\begin{proof}[Proof of Proposition~\ref{pro:ihara_bass_bound}]
We adjust the proof of \cite[Proposition~3.13]{ADK20}. We first remark that $H$ in \cite{ADK20} is centered, i.e.\ 
it corresponds to $\ul{H\!}\, \defeq H - \E H$ in the notation of this paper. 

We denote the entries of $\ul{H\!} \,$ by ${\ul{H\!}\,}_{xy}$. 
We introduce the diagonal matrix $M(t) \deq \diag( m_x(t) \colon x \in [N])$ and estimate its entries  
\[ 
m_x(t) \defeq 1 + \sum_{y \in [N]} \frac{{\ul{H\!}\,}^2_{xy}}{t^2 - {\ul{H\!}\,}_{xy}^2} \leq 1 + \frac{1}{t^2} \sum_{y \in [N]} 
\bigg( H_{xy}^2 + \frac{d}{N^2}\bigg)  \bigg( 1 + O\bigg( \frac{1}{t^2 d}\bigg)\bigg),  
\] 
where we used $\abs{H_{xy}}^2 \lesssim 1/d$. 
Together with the definition of $\alpha_x$ and $d \lesssim \log N$, this bound yields a refinement of \cite[eq.~(3.47)]{ADK20} given by  
\begin{equation} \label{eq:upper_bound_M_t} 
 M(t) \leq \bigg(1 + O\bigg( \frac{\log N}{t^2N} \bigg) \bigg) \id + \frac{1}{t^2} \bigg( 1 + O\bigg( \frac{1}{t^2 d} \bigg) \bigg) Q\,. 
\end{equation}
Combining \eqref{eq:upper_bound_M_t} with \cite[eq.~(3.45)]{ADK20} and \cite[eq.~(3.46)]{ADK20} 
yields Proposition~\ref{pro:ihara_bass_bound}. 
\end{proof}

\begin{proposition}\label{prop:rought_delocalization}
Let $\tilde{\mathcal{U}}\subset\mathcal{U}$. 
For any vector $\f w=(w_{y})_{y\in[N]}$ satisfying $\f w \perp \f w_{\sigma}(x)$ for all $x\in\tilde{\mathcal{U}}$ and $\sigma=\pm$, we have $\sum_{x\in\tilde{\mathcal{U}}}\alpha_{x}|w_{x}|^{2}\lesssim\|\f w\|^{2}$.
\end{proposition}

\begin{lemma} \label{lem:coarse_eigve_approx}
For any $x\in{\cal V}$, we have 
\[\f w_{\sigma}(x)=\frac{1}{\sqrt{2}}\bigg(\f 1_{x}+\sigma\frac{\f 1_{S_{1}(x)}}{\|\f 1_{S_{1}(x)}\|}\bigg)+O(\am^{-1/2})\,. \]
\end{lemma}
\begin{proof}[Proof of Lemma \ref{lem:coarse_eigve_approx}]
We apply Lemma~\ref{lem:perturbationEV} with $M = H^{(x,r)} = H|_{B_{r}(x)}$, where $r = r_{\cal W}$ if 
$x \in \cal W$ and $r = r_{\cal V}$ if $x \in \cal V \setminus \cal W$, and $\lambda= \sqrt{\alpha_x}$. 
To check its conditions, we first compute 
\[
(H^{(x,r)}-\sqrt{\alpha_x})\left(\f 1_{x}+\frac{ \f 1_{S_{1}(x)}}{\|\f 1_{S_{1}(x)}\|}\right)=\frac{\f 1_{S_{2}(x)}}{\sqrt{d}\|\f 1_{S_{1}(x)}\|}=O(1)\,. 
\]
Here, the last step follows from $\norm{\f 1_{S_2(x)}} = \abs{S_2(x)}^{1/2} \lesssim \sqrt{d} \abs{S_1(x)}^{1/2} = \sqrt{d} \norm{\f 1_{S_1(x)}}^{1/2}$ 
by Proposition~\ref{pro:prob_estimates_1} \ref{item:concentration_S_i} with $\delta = \delta_\star$. 
For $x \in \cal V \setminus \cal W$, the second largest eigenvalue of $H^{(x,r)}$ is smaller than $\norm{QH^{(x,r)}Q} \leq 2 (1 + \delta_\star^{1/2})^{1/2}\lesssim \am^{1/8} (\log \am)^{1/4}$ 
by eigenvalue interlacing and \eqref{eq:norm_QHQ}. 
Similarly, if $x \in \cal W$ then the second largest eigenvalue is bounded by $\norm{QH^{(x,r)}Q} \lesssim 1$. 
Hence, $\Delta = \sqrt{\alpha_x} - \norm{QH^{(x,r)}Q} \gtrsim \am^{1/2}$ as $\alpha_x \asymp \am$ and, thus, Lemma \ref{lem:coarse_eigve_approx} follows from Lemma \ref{lem:perturbationEV}. 
\end{proof}

\begin{proof}[Proof of Proposition \ref{prop:rought_delocalization}]
For any $x\in\tilde{\mathcal{U}}$, we have 
\[0=\langle \f w,\frac{1}{\sqrt{2}}(\f w_{+}(x)+\f w_{-}(x))\rangle=w_{x}+O(\am^{-1/2}\|\f w|_{\supp \f w_{+}(x)\cup \supp \f w_-(x)}\|)\,, \] 
where in the second step we used Lemma~\ref{lem:coarse_eigve_approx} if $x \in \cal V$ and 
\eqref{eq:cal_U_approx_w_sigma} if $x \in \cal U \setminus \cal V$. 
Thus, $\alpha_x \lesssim \am$ due to Lemma~\ref{lem:convergence_max_alpha_x_upper_bound_D_x} implies 
\[\sum_{x\in\tilde{\mathcal{U}}}\alpha_x|w_{x}|^2\lesssim\sum_{x\in\tilde{\mathcal{U}}}\|\f w|_{\supp \f w_{+}(x)\cup \supp \f w_-(x)}\|^{2}\leq\|\f w\|^{2}\]
because the sets $(\supp \f w_{+}(x)\cup \supp \f w_-(x))_{x \in\cal U}$ are disjoint by Corollary~\ref{cor:orthonormal_small_d}. 
\end{proof}

\begin{proof}[Proof of Proposition \ref{prop:blockMatrix}  \ref{Item:bulkBound} for $d\leq (\log N)^{3/4}$]
We introduce the orthogonal projections $\Pi$ and $\overline{\Pi}$ defined through 
\[ \Pi \deq \f q \f q^* + \sum_{x \in \cal U} \sum_{\sigma = \pm} \f w_{\sigma}(x) \f w_{\sigma}(x)^*, \qquad \overline{\Pi} \deq \id - \Pi \] 
and note that 
owing to $\norm{X} = \norm{\overline{\Pi}H \overline{\Pi}}$,  it suffices to estimate $\norm{\overline{\Pi}H \overline{\Pi}}$. 

By contradiction, assume
that there exists a normalized eigenvector $\f w$ of $\overline{\Pi}H \overline{\Pi}$
associated to an eigenvalue $\lambda  > \eta \am^{1/2}$. 
Since $\overline{\Pi} \f w = \f w$ by the eigenvalue relation, we have $\f w \perp \f q$. Hence, $\mathbb{E}H = \sqrt{d}\f e\f e^*$  
implies 
\begin{equation} \label{eq:quadratic_form_eigenvector_bulk} 
\langle \f w,(H-\mathbb{E}H) \f w\rangle\geq \langle \f w,\overline{\Pi}H \overline{\Pi} \f w\rangle - \sqrt{d}\|\f e - \f q  \|^2 \geq \lambda - o(1)\,,  
\end{equation} 
where the last step follows from $\overline{\Pi}H \overline{\Pi} \f w = \lambda \f w$ and Proposition \ref{prop:stray}.

We 
conclude from 
\eqref{eq:quadratic_form_eigenvector_bulk} and Proposition~\ref{pro:ihara_bass_bound} that $\lambda \leq t + o(1) + \frac{1}{t} (1 + o(1)) \sum_{x \in [N]} \alpha_x \abs{w_x}^2$
with high probability for any $t \geq 1 + Cd^{-1/2}$, where $C$ is the constant from Proposition~\ref{pro:ihara_bass_bound}. In the following, we will increase the value of $C$ a few times. The final value of $C$ is irrelevant for the validity of the argument.

Because $\lambda> \eta \am^{1/2}$ we can assume $\sum_{x\in [N]}\alpha_{x}|w_x|^{2}> (1+Cd ^{-1/2})^2$ and the choice $t:=\sqrt{\sum_{x\in [N]}\alpha_{x}|w_x|^{2}}$ yields 
\begin{equation} \label{eq:upper_bound_eigenvalues_bulk} 
\lambda\leq (2+o(1))\sqrt{\sum_{x\in [N]}\alpha_{x}|w_x|^{2}}\,.
\end{equation}
We decompose the sum on the right-hand side and use  
Proposition \ref{prop:rought_delocalization} 
with $\wt{\cal U} = \cal U $ as well as $\sum_x |w_x|^2\leq 1$ to obtain 
\[
\sum_{x\in [N]}\alpha_{x}|w_x|^{2}=\sum_{x\in\mathcal{U}}\alpha_{x}|w_x|^{2}+\sum_{x\in[N],\alpha_{x}< \am/5}\alpha_{x}|w_x|^{2}
\leq C + \frac{1}{5} \am\,.
\]
Here, we used that $\f w \perp \f w_\sigma(x)$ for all $x \in\cal U$ and $\sigma \in \{\pm\}$ as $\f w$ is an eigenvector of 
$\overline{\Pi}H \overline{\Pi}$. 
Therefore, with \eqref{eq:upper_bound_eigenvalues_bulk}, we get $\lambda\leq (2+o(1))\sqrt{\frac{1}{5} \am} = \big(\sqrt{\frac{4}{5}} + o(1) \big) \sqrt{\am}$,
which contradicts $\lambda > \eta \am^{1/2}$ as $\eta > \sqrt{4/5}$.

With the same argument, we can exclude eigenvalues $\lambda$ of $\overline{\Pi}H\overline{\Pi}$ 
with $\lambda < - \eta \am^{1/2}$.
This shows $\norm{\overline{\Pi}H\overline{\Pi}} \leq \eta \am^{1/2}$ and completes the proof. 
\end{proof}

This completes the proof of Proposition~\ref{prop:blockMatrix} in the missing regime, $d \leq (\log N)^{3/4}$.
\end{proof} 

\subsection{Proof of Corollary \ref{Cor:SpecMax}} \label{sec:proof_cor_spec_max} 

Corollary \ref{Cor:SpecMax} will follow from the next result. 
We recall the function $\Lambda(\alpha) = \frac{\alpha}{\sqrt{\alpha - 1}}$. 
\begin{proposition}\label{prop:SpecMax}
If \eqref{eq:ConditionAlphaMaxMin} and \eqref{eq:condition_d_blockMatrix} with a large enough $K$ are satisfied 
then there exists $c>0$ such that, with high probability, 
\begin{align} \label{eq:estimate_B_def_B} 
\left\| Y \right\|  \leq  \Lambda(\am) - \frac{c(\am-2) d^{2\gamma-1}}{\am^{3/2}\log \am}, \qquad \qquad 
Y \deq \begin{pmatrix}   \mathcal{D}_{\cal V \setminus \cal W} & 0 & E_{\cal V \setminus \cal W}^* 
\\   0 &  \mathcal{D}_{\cal U \setminus \cal V} + \cal E_{\cal U\setminus \cal V} &  E_{\cal U\setminus \cal V}^*
\\ E_{\cal V\setminus \cal W} & E_{\cal U\setminus \cal V} &  X 
\end{pmatrix}\,.
\end{align}
Moreover, if $d < (\log N)^{1/4}$ then 
\begin{align} \label{eq:SpecMax_small_d} 
\normbb{\begin{pmatrix} \nu_{\mathrm{s}} & E_{\mathrm{s}}^* \\ E_{\mathrm{s}} & Y \end{pmatrix}} 
\leq  \Lambda(\am) - \frac{c(\am-2) d^{2\gamma-1}}{\am^{3/2}\log \am}\,, 
\end{align}
where $E_\mathrm{s}^*$ was identified with $\pb{\begin{matrix} 0 &  0 &  E_\mathrm{s}^*\end{matrix}}$ and an analogous identification for $E_\mathrm{s}$ was used. 
\end{proposition}

\begin{proof}[Proof of Corollary \ref{Cor:SpecMax}]
First, we conclude from \eqref{eq:scaling_alpha_max} that $\fra u \asymp \am$. 
Moreover, $\am - 2 \asymp \fra u - 2$ by \eqref{eq:ConditionAlphaMaxMin} and \eqref{eq:scaling_alpha_max}. 
In particular, $\tilde \chi \gtrsim \chi$ with $\tilde \chi \deq \frac{(\am - 2) d^{2\gamma - 1}}{\am^{3/2} \log \am}$. 
Moreover, $\sigma(\mathfrak u) = \Lambda_{\fra d}(\mathfrak u,1)$ by definition, 
\eqref{Lambda_estimates}, \eqref{eq:scaling_alpha_max} and Lemma \ref{lem:EstimateLambdad} 
imply 
\[
|\Lambda(\am) - \sigma(\fra u)|\leq |\Lambda(\am)-\Lambda(\fra u)|+|\Lambda(\fra u,1)-\Lambda_{\fra d}(\fra u,1)| \lesssim \frac{\log d }{d\am^{1/2}} \leq \frac{1}{K_*} \tilde\chi\,, 
\] 
where the last step followed from \eqref{eq:condition_d_SpecMax} for small $d$ 
and from \eqref{eq:ConditionAlphaMaxMin} for large $d$. 
Therefore, it suffices to show Corollary~\ref{Cor:SpecMax} with $\sigma(\fra u)- c \chi$ replaced by 
$\Lambda(\am) - \tilde c \tilde \chi$ as well as $\fra u$ by $\am$ in 
the bound on $\abs{\eps_x}$ 
and to choose 
 $K_*$ in \eqref{eq:condition_d_SpecMax} large enough. 

Let $Y$ be the matrix defined as in \eqref{eq:estimate_B_def_B}. 
When $ d\geq (\log N)^{1/4}$, we compare the matrix $U^{-1} H U$ in Proposition \ref{prop:blockMatrix} with the block matrix $\pbb{\begin{smallmatrix} \nu_\mathrm{s} & 0 & 0 \\ 0 & \cal D_{\cal W} & 0\\
 0 & 0 &  Y
\end{smallmatrix}}$.
By perturbation theory, the eigenvalue process $\sum_{\lambda \in \spec(H)}\delta_{\lambda}$ of $H$ coincides with the process
\[\delta_{\nu}+\sum_{x\in \cal W,\sigma \in \{\pm\}}\delta_{\sigma\Lambda_{\fra d}(\alpha_x,\beta_x)+\epsilon_{x,\sigma}+\epsilon_{x,\sigma}'}+\sum_{\tilde{\lambda}\in \spec(Y)}\delta_{\tilde{\lambda}+\epsilon_{\tilde{\lambda}}}\,, \]
where $\eps_{x,\sigma}$ is from \eqref{eq:cal_D_equal_Lambda_plus_eps} and controlled in Proposition~\ref{prop:blockMatrix} \ref{Item:WsetBound}, $\nu = d^{1/2} + d^{-1/2} + d^{-3/2} + \eps_{\mathrm{s}}$ is an eigenvalue of $H$ and the error terms $\eps_{x,\sigma}'$, $\eps_\mathrm{s}$ 
and $\eps_{\tilde{\lambda}}$ satisfy 
  $|\epsilon_{x,\sigma}'| \leq \|E_{\cal W}\| \lesssim (\am d)^{-10}$ and $|\epsilon_{\tilde{\lambda}}| + |\epsilon_{\mathrm{s}}|\lesssim \| E_{\cal W}\| + \|E_\mathrm{s}\|\lesssim d^{-5/2} $ by Proposition \ref{prop:blockMatrix} \ref{Item:WsetBound} \ref{Item:StrayEstimate}. 
Hence, 
\eqref{eq:ConditionAlphaMaxMin}, \eqref{eq:scaling_alpha_max} and $d \geq (\log N)^{1/4}$ imply 
$\abs{\eps_{\tilde{\lambda}}} \ll \tilde \chi$ 
and we conclude from \eqref{eq:estimate_B_def_B} in Proposition \ref{prop:SpecMax}  
that the process $\sum_{\tilde{\lambda}\in \spec(Y)}\delta_{\tilde{\lambda}+\epsilon_{\tilde{\lambda}}}$ does not appear in the window $[\Lambda(\am)- \tilde c \tilde \chi,\infty)$.
Since $-\Lambda_{\fra d}(\alpha_x,\beta_x) + \eps_{x,-} + \eps_{x,-}' < 0$ for all $x \in \cal W$, this proves 
Corollary~\ref{Cor:SpecMax} for $d \geq (\log N)^{1/4}$ by Proposition~\ref{prop:blockMatrix} \ref{Item:StrayEstimate}.   

In the regime $d<(\log N)^{1/4}$, we apply perturbation 
theory, Proposition~\ref{prop:blockMatrix} \ref{Item:WsetBound} 
and \eqref{eq:SpecMax_small_d} in Proposition~\ref{prop:SpecMax}  to the right-hand side of 
\[U^{-1}HU= \begin{pmatrix} 0 & 0 & 0 \\ 0 & \cal D_{\cal W} & 0\\
 0 & 0 &  0
\end{pmatrix} + \begin{pmatrix} \nu_\mathrm{s} & 0 & E_\mathrm{s}^* \\ 0 &  0 & 0\\
 E_\mathrm{s}^* & 0 &  Y
\end{pmatrix}+O( \|E_{\cal W} \|)\,.  
\]
We thus obtain that the processes $\sum_{\lambda \in \spec(H)} \delta_{\lambda}$ and $\sum_{x \in \cal W} \delta_{\Lambda_{\fra d}(\alpha_x,\beta_x) + \eps_x}$  
coincide on $[\Lambda(\am) - \tilde c \tilde \chi, \infty)$ with sufficiently small $\eps_x$, $x \in \cal W$. 
Since $\nu_{\mathrm{s}} + O(d^{1/2}) \ll \am^{1/2} \lesssim \Lambda(\am) - \tilde c \tilde \chi$, 
setting $\nu$ to be a sufficiently small eigenvalue of $H$  completes the proof of Corollary~\ref{Cor:SpecMax}. 
\end{proof}

For the proof of Proposition~\ref{prop:SpecMax} we shall need the following estimate, whose proof is given after the proof of Proposition~\ref{prop:SpecMax}.

\begin{lemma} \label{lem:BoundVeigenvalue}
There exists $c>0$ such that with high probability
\[
\max_{x\in \cal V \setminus \cal W}\Lambda_{\fra d}(\alpha_x,\beta_x)\leq \Lambda(\am) - \frac{c(\am -2) d^{2\gamma-1}}{ \am^{3/2}\log \am}\,.  
\]
\end{lemma}

\begin{proof}[Proof of Proposition~\ref{prop:SpecMax}] 
We have 
\begin{align} 
\norm{Y}  
& \leq  \max\left\{ \| \mathcal{D}_{\cal V \setminus \cal W}\|, \left\|\begin{pmatrix}  
   \mathcal{D}_{\cal U \setminus \cal V} +  \cal E_{\cal U\setminus \cal V} &    E_{\cal U\setminus \cal V}^*
\\   E_{\cal U\setminus \cal V} &  X 
\end{pmatrix} \right\|\right\} + \|  E_{\cal V\setminus \cal W}\|
\notag
\\ \label{eq:Y_leq_maximum} 
& \leq  \max\left\{ \| \mathcal{D}_{\cal V \setminus \cal W}\|, 
  \| \mathcal{D}_{\cal U \setminus \cal V}\| + \| \cal E_{\cal U\setminus \cal V}\| + \|  E_{\cal U\setminus \cal V}\|,
  \|X \| +\|  E_{\cal U\setminus \cal V} \| 
\right\} + \|  E_{\cal V\setminus \cal W} \|\,. 
\end{align}
By Proposition \ref{prop:blockMatrix} \ref{Item:VsetBound} we have $\|E_{\cal V \setminus \cal W}\|\lesssim (\am d)^{-10}\ll \frac{(\am -2) d^{2\gamma-1}}{\am^{3/2}\log \am} $ and we will therefore neglect this term in \eqref{eq:Y_leq_maximum}.

We now estimate each argument of the maximum in \eqref{eq:Y_leq_maximum}. 
By definition of $\cal D_{\cal V \setminus \cal W}$, 
 Lemma \ref{lem:BoundVeigenvalue}, and Proposition \ref{prop:blockMatrix} \ref{Item:VsetBound}, we obtain  
\begin{equation} \label{eq:proof_SpecMax_aux1}
\|\mathcal{D}_{\cal V \setminus \cal W}\| \leq \max_{x\in \cal V \setminus \cal W} \Lambda_{\fra d}(\alpha_x,\beta_x) + 
O\bigg( \frac{1}{d \am} + \frac{\am^{1/2}\log d }{d (\am - 2)^2} \bigg)
   \leq \Lambda(\am) - \frac{c(\am -2) d^{2\gamma-1}}{2 \am^{3/2}\log \am}\,, 
\end{equation}
where we used that $(\am-2)^3 d^{2\gamma}\gg \log d$ by \eqref{eq:ConditionAlphaMaxMin}. 

Before continuing, we record some auxiliary results. 
Lemma~\ref{lem:beta_x_assumption_beta_pr_checked} implies the applicability of 
Lemma~\ref{lem:EstimateLambdad} and \eqref{Lambda_expansion} with 
$\alpha = \alpha_x$ and $\beta = \beta_x$ in the following. 
Moreover, as $\frac{c_* \delta_\star}{(\am - 2)\log \am} \leq \frac{1}{2}$ by \eqref{eq:deltaStarDef} 
and \eqref{eq:ConditionAlphaMaxMin}, for all $\delta \in [0, \delta_\star]$, the mean-value theorem 
and the second identity in \eqref{Lambda_identities} imply
\begin{equation} \label{eq:lower_bound_derivative_Lambda} 
\Lambda\bigg(\am - \frac{c_* \delta}{\log \am} \bigg) \geq \Lambda(\am) - \frac{(\am - 2)}{4\am^{3/2}}  \frac{c_*\delta}{\log \am}\,. 
\end{equation} 
Finally, $\Lambda(\alpha,\beta)$ is monotonously increasing in $\alpha$ and $\beta$ by Proposition~\ref{prop:Lambda_ab} \ref{item:Z_eval}. 

In the regime $d>(\log N)^{3/4}$, Lemma~\ref{lem:EstimateLambdad}, the monotonicity of $\Lambda$ and Proposition~\ref{prop:blockMatrix} \ref{Item:UsetBound} yield
\begin{align}
 \| \mathcal{D}_{\cal U \setminus \cal V}\|+\| \cal E_{\cal U\setminus \cal V} \|+\|  E_{\cal U\setminus \cal V} \|
& \leq \Lambda\left(\am - \frac{c_*\delta_\star}{\log \am},\max_{x \in \cal U} \beta_x \right)+ O \big( \xi^{1/2} \big) 
\nonumber \\ & \leq \Lambda(\am ) - \frac{(\am -2)}{4\am^{3/2}}\frac{c_*\delta_\star}{\log \am } + O \bigg( \frac{(\am-2)}{\am^{3/2}}\frac{\sqrt{\log N}}{d} + \xi^{1/2} \bigg)
\nonumber \\ & \leq \Lambda(\am ) - \frac{c_*(\am -2) d^{2\gamma-1}}{4\am^{3/2}\log \am}\,, 
\label{eq:proof_SpecMax_aux4} 
\end{align}
where, recalling the definition of $\xi$ from \eqref{eq:def_xi}, we used  \eqref{Lambda_expansion}, \eqref{eq:lower_bound_derivative_Lambda},   
\eqref{eq:beta_x_equal_1_error_term_on_cal_U} and 
 $ \delta_\star \gg \frac{\am^{3/2} \log \am}{\am - 2} \xi^{1/2} + \frac{\log \am \sqrt{\log N}}{d} + d^{2\gamma-1}$  
because of \eqref{eq:deltaStarDef} and \eqref{eq:ConditionAlphaMaxMin}.
By \ref{Item:UsetBound} and \ref{Item:bulkBound} of Proposition~\ref{prop:blockMatrix}, we also have
\begin{align}
 \| X \|+\|  E_{\cal U\setminus \cal V} \|  \leq 2 + O \bigg( \frac{\xi^{1/2} \am^2}{(\am-2)^2} \bigg)
\leq \Lambda(\am) - \frac{(\am-2) d^{2\gamma-1}}{\am^{3/2}\log \am}
\label{eq:proof_SpecMax_aux5} 
\end{align}
where we used that in the regime $(\am-2) <1$,   $\Lambda(\am)-2\geq c (\am-2)^2 \gg (\am -2)^{-2}\xi^{1/2}+(\am-2)d^{2\gamma -1}$ by \eqref{Lambda_geq_2_est} and \eqref{eq:ConditionAlphaMaxMin}. 
Plugging \eqref{eq:proof_SpecMax_aux1}, \eqref{eq:proof_SpecMax_aux4} and \eqref{eq:proof_SpecMax_aux5} into \eqref{eq:Y_leq_maximum} proves 
\eqref{eq:estimate_B_def_B} for $d > (\log N)^{3/4}$. 

If $d\leq (\log N)^{3/4}$ then Lemma~\ref{lem:EstimateLambdad}, the monotonicity of $\Lambda$  and Proposition~\ref{prop:blockMatrix} \ref{Item:UsetBound} imply
\begin{align}
 \| \mathcal{D}_{\cal U \setminus \cal V}\|+\| \cal E_{\cal U\setminus \cal V} \|+\|  E_{\cal U\setminus \cal V} \|
& \leq \Lambda\left(\am - c_* \am^{1/2},\max_{x\in \cal U}\beta_x\right)+  
O\bigg( \frac{1}{\sqrt\am} + \frac{\log \log N}{\sqrt d} \bigg) 
\nonumber \\ 
& \leq \Lambda(\am) -\frac{c_*}{5}+ 
O\bigg( \frac{1}{\sqrt\am} + \frac{\log \log N}{\sqrt d} \bigg) 
\nonumber \\ 
& \leq \Lambda(\am ) - \frac{c_*}{6}\,,\label{eq:proof_SpecMax_aux2} 
\end{align} 
where, in the second step, we also used \eqref{Lambda_expansion}, \eqref{eq:lower_bound_derivative_Lambda} 
with $\delta = \delta_\star=\am^{1/2} \log \am$ 
and \eqref{eq:beta_x_equal_1_error_term_on_cal_U}. 
Finally, \ref{Item:UsetBound} and \ref{Item:bulkBound} of Proposition~\ref{prop:blockMatrix} yield 
\begin{equation} \label{eq:proof_SpecMax_aux3} 
 \| X \|+\|  E_{\cal U\setminus \cal V} \| \leq \eta \am^{1/2} +  O(1) < \Lambda(\am)-1\,. 
\end{equation}
Inserting \eqref{eq:proof_SpecMax_aux1}, \eqref{eq:proof_SpecMax_aux2} and \eqref{eq:proof_SpecMax_aux3} into \eqref{eq:Y_leq_maximum} shows 
\eqref{eq:estimate_B_def_B} for $d \leq (\log N)^{3/4}$. 

For the proof of \eqref{eq:SpecMax_small_d}, we assume $d<(\log N)^{1/4}$ and estimate
\[ 
\normbb{ \begin{pmatrix} \nu_{\mathrm{s}} & E_\mathrm{s}^* \\ E_\mathrm{s} & Y \end{pmatrix} } 
\leq \max \bigg\{ \norm{\cal D_{\cal V\setminus \cal W}}, \norm{\cal D_{\cal U \setminus \cal V}} + \norm{\cal E_{\cal U\setminus \cal V}} + \norm{
E_{\cal U \setminus \cal V}},
\abs{\nu_\mathrm{s}} + \norm{X} + \norm{E_{\mathrm{s}}} + \norm{E_{\cal U \setminus \cal V}}
\bigg\} + \norm{E_{\cal V\setminus \cal W}}\,. 
\] 
Using \eqref{eq:proof_SpecMax_aux1},  \eqref{eq:proof_SpecMax_aux2} and \eqref{eq:proof_SpecMax_aux3} 
thus 
completes the proof of \eqref{eq:SpecMax_small_d}
as $\abs{\nu_\mathrm{s}} + \norm{E_\mathrm{s}} \lesssim d^{1/2}\ll \am^{1/2}$ by Proposition~\ref{prop:blockMatrix} \ref{Item:StrayEstimate} and \eqref{eq:scaling_alpha_max}. 
\end{proof}

\begin{proof}[Proof of Lemma \ref{lem:BoundVeigenvalue}]
By Corollary \ref{cor:concentration_beta_x}, there exists a constant $K>0$ such that for any $\delta>0$ satisfying \eqref{eq:condition_d_delta_graph}, we have 
\begin{equation}\label{eq:betaBounddelta}
\max_{x\in \cal V_\delta }|\beta_x-1|\leq K\left(\frac{\delta}{d\am}\right)^{1/2}
\end{equation}  with high probability. 
Let $k\in \mathbb{N}$ be the smallest positive integer such that $2^k\geq \frac{1}{\gamma}$ and define $\delta_0 \deq d^{2\gamma-1}$, $\delta_i \deq (C d^{2^i\gamma-1}) \wedge \delta_{\star}$ for all $i\leq k$ and $\delta_{k+1}\deq \delta_{\star} $. Here we can choose $C$ such that 
\begin{equation} \label{eq:DeltaI_DeltaIp1}
\delta_{i+1} < \frac{d\am c_*^2\delta_i^2}{64 K^2\log ^2 \am}
\end{equation}
for all $0\leq i\leq k$, where $c_*$ is chosen as in Proposition~\ref{pro:intermediate_rigidity}. 
Lemma~\ref{lem:EstimateLambdad} and \eqref{Lambda_expansion} with $\alpha = \alpha_x$ and $\beta = \beta_x$ 
are applicable in the following due to Lemma~\ref{lem:beta_x_assumption_beta_pr_checked}. 
We decompose 
$\cal V \setminus \cal W  = \cal V_{\delta_{k+1}} \setminus \cal V_{\delta_{0}} =  \bigcup_{0\leq i \leq k} \left(\cal V_{\delta_{i+1}}\setminus \cal V_{\delta_i} \right)$ 
 and use Lemma~\ref{lem:EstimateLambdad} with $\alpha = \alpha_x \gtrsim \am$ by \eqref{eq:lower_bound_degree_cal_V}, as well as the monotonicity of $\Lambda$ (see Proposition~\ref{prop:Lambda_ab} \ref{item:Z_eval}) to obtain 
\begin{align*}
 \max_{x\in \cal V\setminus \cal W} \Lambda_{\fra d}(\alpha_x,\beta_x)  &  = \max_{0\leq i\leq k} \max_{x\in \cal V_{\delta_{i+1}}\setminus \cal V_{\delta_i}} \Lambda_{\fra d}(\alpha_x,\beta_x)
\\ &  
\leq \max_{0\leq i\leq k} \Lambda\left(\am - \frac{ c_*\delta_i}{\log \am}, \max_{x\in \cal V_{\delta_{i+1}}} \beta_x \right) + O \bigg( \frac{1}{d\am} \bigg) 
\\ &  \leq \max_{0\leq i\leq k} \left(\Lambda(\am ) - \frac{c_*(\am -2) \delta_i}{4 \am^{3/2}\log \am}+ K\frac{(\am-2)}{\am^{3/2} }\left(\frac{\delta_{i+1}}{d\am }\right)^{1/2}  + O\left( \frac{K^2 \delta_{i+1}}{d\am^{3/2}} + \frac{1}{d\am} \right)  \right) 
\\ &\leq \max_{0\leq i\leq k} \left(\Lambda(\am ) - \frac{ c_*(\am -2) \delta_i}{9 \am^{3/2}\log \am} \right)\,. 
\end{align*}
Here, in the third step, we used \eqref{Lambda_expansion}, \eqref{eq:lower_bound_derivative_Lambda}  
and \eqref{eq:betaBounddelta}. 
In the fourth step, we used \eqref{eq:DeltaI_DeltaIp1} to absorb the third term into the second term as well as 
$d^{2\gamma -1} \leq \delta_i \leq \delta_{i+1} \leq \delta_\star$ and \eqref{eq:deltaStarDef} as well as 
\eqref{eq:ConditionAlphaMaxMin} to absorb the fourth term into the second term. 
Since $\min_i \delta_i = d^{2\gamma - 1}$, this completes the proof of Lemma~\ref{lem:BoundVeigenvalue}. 
\end{proof}

\subsection{Stray eigenvalue -- proof of Proposition \ref{prop:stray}} \label{sec:StrayEigenvalue}

This section is devoted to the proof of Proposition \ref{prop:stray}. Abbreviate $\cal L \deq \h{x \col \alpha_x \geq 2}$, and for $k \geq 0$ define $\cal B_k \deq \bigcup_{x \in \cal L} B_{r+k + 1}(x)$.
Throughout this section, $r$ satisfies the condition $r \ll \frac{d}{\log \log N}$ from Proposition~\ref{prop:stray}, and we assume $d \gg 1$.

\begin{lemma} \label{lem:B_k_estimate}
There exist a positive constant $c$ such that for $r,k \ll \frac{d}{\log \log N}$ we have  $\abs{\cal B_k} \leq N \ee^{-c d}$ with high probability.
\end{lemma}

\begin{proof}
By Bennett's inequality (see Lemma~\ref{lem:Bennett} below) we have $\E \abs{\cal L}  = N \P(\alpha_x \geq 2) \leq N \ee^{-d h(1)}$,
where $h$ was defined in \eqref{eq:def_h}. Hence, by Chebyshev's inequality we have $\abs{\cal L} \leq N \ee^{-cd}$ with high probability, for some universal constant $c > 0$. Moreover, by 
 Lemma~\ref{lem:convergence_max_alpha_x_upper_bound_D_x} we have $\max_x D_x \leq C \log N$ with high probability. Thus we get, with high probability, $\abs{\cal B_k} \leq N \ee^{-cd} (C \log N)^{r + k}$, and the claim follows by replacing $c$ with $c/2$.
\end{proof}

Define the restriction $\tilde H \deq H \vert_{\mathcal{B}_0^c}$. Define also $\hat{\f e}_k \deq \abs{\cal B_k^c}^{-1/2} \f 1_{\cal B_k^c}$.
Note that, by locality of the matrix $H$, we have $H^l \hat{\f e}_k = \tilde H^l \hat{\f e}_k$ for all $l \leq k$.

\begin{lemma}\label{prop:StrayCutGraph}
We have 
$\| (H- \E H)|_{\mathcal{B}_0^c}\| \leq 4$ with high probability.
\end{lemma}
\begin{proof}
Let $Q \deq \diag(\alpha_x \col x\in [N])$. Since $\norm{Q \vert_{\mathcal{B}_0^c}} \leq 2$ 
by definition of $\cal B_0$, we deduce the claim by choosing $t = 2$ in Proposition~\ref{pro:ihara_bass_bound}.
\end{proof}

\begin{corollary} \label{cor:norm_H_tilde}
For $r \ll \frac{d}{\log \log N}$, with high probability, $\tilde H$ has a unique eigenvalue $\tilde \nu = \sqrt{d} + O(1)$, and all other eigenvalues are bounded in absolute value by $4$.
\end{corollary}

\begin{proof}
We observe that $(\E H + \sqrt{d} / N) \vert_{\mathcal{B}_0^c}$ has rank one, with nonzero eigenvalue equal to $\frac{\sqrt{d}}{N} \abs{\cal B_0^c} = \sqrt{d} + O(\ee^{-c d})$, where the latter step holds with high probability by Lemma \ref{lem:B_k_estimate}. The claim then follows from Lemma \ref{prop:StrayCutGraph}, by eigenvalue interlacing of rank-one projections as well as first order perturbation theory.
\end{proof}

Define $\tilde{\f e}_k \deq \tilde H^k \hat{\f e}_k$ and  $\nu_k \deq \frac{\scalar{\tilde{\f e}_k}{\tilde H \tilde{\f e}_k}}{\scalar{\tilde{\f e}_k}{\tilde{\f e}_k}}$.
In particular, $\tilde{\f e}_k$ is supported in $\cal B_0^c$.

\begin{lemma}\label{lem:StrayCandidate} Under the assumptions of Lemma \ref{lem:B_k_estimate}, there is a constant $C > 0$ such that with high probability we have
\begin{equation*}
|\nu_k-\tilde \nu|= O \pbb{\frac{4^{2k}}{d^k}} \,, 
\qquad \qquad
\normbb{(\tilde H - \nu_k) \frac{\tilde{\f e}_k}{\norm{\tilde{\f e}_k}}} = O \pbb{ \frac{4^k}{d^{k/2 - 1/4}}}\,.
\end{equation*}
\end{lemma}

\begin{proof}

Use the notation 
$\tilde H \tilde {\f u} = \tilde \nu \tilde {\f u}$
for the normalized eigenvector $\tilde {\f u}$ associated with $\tilde \nu$, which is supported in $\cal B_0^c$.
We decompose $\hat{\f e}_k = \sqrt{1 - \eta^2} \, \tilde {\f u} + \eta \tilde {\f w}$,
where $\tilde {\f w}$ is normalized and orthogonal to $\tilde {\f u}$.

We begin by estimating $\eta$. To that end, we define $M \deq (\E H + \sqrt{d} / N) \vert_{\cal B_0^c} - \sqrt{d} \hat{\f e}_k \hat{\f e}_k^*$,
and use Lemma \ref{lem:B_k_estimate} to estimate $\abs{M_{ij}}$
with high probability by $\ee^{-cd} / N$ if $i,j \in \cal B_k^c$ and by $\sqrt{d}/N$ otherwise. We then estimate the norm $\norm{M}$ by applying the Schur test to the entries $M_{ij}$ with $i,j \in \cal B_k^c$, and estimating the Hilbert-Schmidt norm of the remaining entries using Lemma \ref{lem:B_k_estimate}. This yields $\norm{M} \leq \ee^{-cd}$ for some constant $c > 0$. With Lemma \ref{prop:StrayCutGraph} we conclude $\norm{\tilde H - \sqrt{d} \hat{\f e}_k \hat{\f e}_k^*} \leq 5$ with high probability. Perturbation theory for eigenvectors (see e.g.\ Lemma \ref{lem:perturbationEV}) therefore yields $\tilde {\f u} = \hat{\f e}_k + O(d^{-1/2})$, from which we conclude $\eta = O(d^{-1/4})$.

Next,
\begin{equation} \label{e_k_decomp}
\tilde{\f e}_k = \tilde H^k \hat{\f e}_k = \sqrt{1 - \eta^2} \tilde \nu^k \tilde{\f u} + \eta \tilde H^k \tilde {\f w} = \sqrt{1 - \eta^2} \tilde \nu^k \tilde{\f u} + O (\eta \, 4^k)\,,
\end{equation}
where in the last step we used that on the orthogonal complement of $\tilde {\f u}$ the matrix $\tilde H$ has norm bounded by $4$, by Corollary \ref{cor:norm_H_tilde}. We calculate
\begin{equation*}
\scalar{\tilde{\f e}_k}{\tilde{\f e}_k} = (1 - \eta^2) \tilde \nu^{2k} + O (\eta^2 4^{2k})\,, \qquad \scalar{\tilde{\f e}_k}{\tilde H \tilde{\f e}_k} = (1 - \eta^2) \tilde \nu^{2k + 1} + O (\eta^2 4^{2k+1})\,,
\end{equation*}
from which we obtain
\begin{equation*}
\tilde \nu = \nu_k + O \pbb{\frac{\eta^2}{1 - \eta^2} \frac{4^{2k}}{\tilde \nu^{2k - 1}}} = 
\nu_k + O \pbb{\frac{4^{2k}}{d^k}}\,.
\end{equation*}
From \eqref{e_k_decomp} we get $\norm{(\tilde H - \tilde \nu) \tilde{\f e}_k} = O \pb{4^k d^{1/4}}$,
from which we conclude that 
\begin{equation*}
\normbb{(\tilde H - \nu_k) \frac{\tilde{\f e}_k}{\norm{\tilde{\f e}_k}}} = O \pbb{ \frac{4^k}{d^{k/2 - 1/4}}} \,,
\end{equation*}
which concludes the proof.
\end{proof}

To compute $\nu_k$, for $l \leq 2k+1$ we define $a_l \deq \scalar{{\f e}}{H^l {\f e}}$ and  $\tilde a_l \deq \scalar{\hat{\f e}_k}{\tilde H^l \hat{\f e}_k}$.

\begin{lemma} \label{lem:a_tilde_a}
Under the assumptions of Lemma \ref{lem:B_k_estimate} we have for any $l \leq 2k+1$ with high probability $\abs{a_l  - \tilde a_l} \leq (\log N)^l \ee^{-cd}$
for some positive constant $c > 0 $.
\end{lemma}

\begin{proof}
By locality of $H$ and the definitions of $\hat{\f e}_k$ and $\tilde H$ we have $\tilde a_l = \scalar{\hat{\f e}_k}{H^l \hat{\f e}_k}$ for $l \leq 2k+1$. Thus,
\begin{equation*}
\tilde a_l = \frac{1}{\abs{\cal B^c_k}} \sum_{x_1, \ldots, x_l} H_{x_1 x_2} \cdots H_{x_{l-1} x_l} (1 - \ind{x_1 \in \cal B_k}) (1 - \ind{x_l \in \cal B_k})\,.
\end{equation*}
The claim now easily follows from the estimate $\sup_x \sum_y H_{xy} = O( \log N / \sqrt{d}) \leq \log N$
with high probability, by Lemma \ref{lem:convergence_max_alpha_x_upper_bound_D_x}, as well as Lemma \ref{lem:B_k_estimate}.
\end{proof}

We can replace $a_l$ by its expectation using the following large deviation estimate, which is an immediate consequence of \cite[Lemma~6.5]{EKYY1} (or more precisely from its proof).

\begin{lemma} \label{lem:a_exp}
For any $l \geq 1$ we have $\abs{a_l - \E a_l} \leq \frac{(C l \log N)^l}{\sqrt{N}}$
with high probability.
\end{lemma}

Hence, it suffices to compute $\E a_l$, which is a combinatorial exercise. We summarise a result that is sufficiently precise for our purposes. (The same method yields an asymptotic expansion in $1/d$, which we shall however not need here.)

\begin{lemma} \label{lem:path_computation}
For fixed $l \geq 4$ we have
\begin{equation*}
d^{l/2}\E a_l = d^l + (l - 1) d^{l - 1} + \frac{l^2 - l - 4}{2} d^{l - 2}+ O(d^{l - 3})
\end{equation*}
\end{lemma}
\begin{proof}
We write $d^{l/2} \E a_l = \sum_{x_0, \dots, x_l} \E A_{x_0 x_1} \cdots A_{x_{l - 1} x_l}$, and partition the sum according to the coincidences of the vertices $x_0, \dots, x_l$. We use $\E A_{xy} = \frac{d}{N} (1 - \delta_{xy})$ and the independence of the entries of $A$ to evaluate each term. The term of order $d^l$ arises from the restriction that $x_0, \dots, x_l$ be all distinct. The term of order $d^{l - 1}$ arises from the restriction that $x_{i - 1} = x_{i+1}$ for $1 \leq i \leq l - 1$ and all other vertices are distinct. The term of order $d^{l-2}$ arises from the restriction that $x_{i - 1} = x_{i+1}$ and $x_{j - 1} = x_{j+1}$ for $1 \leq i < j \leq l -1$, or $x_{i - 1} = x_{i+1}$ and $x_{i - 2} = x_{i+2}$ for $2 \leq i \leq l-2$, and all other vertices are distinct. All other terms are of order $O(d^{-3})$.
\end{proof}

\begin{proof}[Proof of Proposition \ref{prop:stray}.] We consider the two regimes $d \geq (\log N)^{1/4}$ and $d < (\log N)^{1/4}$ separately. Let first $d \geq (\log N)^{1/4}$. We set $\f q \deq \tilde{\f e}_6 / \norm{\tilde{\f e}_6}$. Then, by construction, the support of $\f q$ is contained in $\cal B_0^c$. Moreover, since $\tilde H \f q = H \f q$, from Lemma \ref{lem:StrayCandidate} we get $\norm{(H - \nu_6) \f q} \lesssim d^{-5/2}$. Moreover, using Lemmas \ref{lem:a_tilde_a}--\ref{lem:path_computation} we get, setting $k = 6$,
\begin{equation*}
\nu_6 = \frac{\scalar{\hat{\f e}_k}{\tilde{H}^{2k+1} \hat{\f e}_k}}{\scalar{\hat{\f e}_k}{\tilde{H}^{2k} \hat{\f e}_k}} = \frac{\tilde{a}_{2k+1}}{\tilde{a}_{2k}} = d^{1/2} + d^{-1/2} + d^{-3/2} + O(d^{-5/2})
\end{equation*}
with high probability, where we used that $d \geq (\log N)^{1/4}$.
Finally, setting again $k = 6$ we have
\begin{equation*}
\scalar{\f e}{\f q} = \frac{\scalar{\f e}{H^k \hat {\f e}_k}}{\scalar{\hat{\f e}_k}{H^{2k} \hat{\f e}_k}^{1/2}} = \frac{a_k}{a_{2k}^{1/2}} \pb{1 + O(1/d)} = 1 + O(1/d)\,,
\end{equation*}   
with high probability, where the second step follows from Lemma \ref{lem:a_tilde_a} and a trivial modification of its proof, and the last step follows from Lemmas \ref{lem:a_exp} and \ref{lem:path_computation}. We conclude that $\norm{\f q - \f e} = O(d^{-1/2})$. This concludes the proof in the regime $d \geq (\log N)^{1/4}$.

If $d < (\log N)^{1/4}$, we set $\f q \deq \f e \vert_{\cal B_1^c} / \norm{\f e \vert_{\cal B_1^c}}$. By Lemma \ref{lem:B_k_estimate}, $\norm{\f q - \f e} \lesssim \ee^{- c d}$, and we estimate with high probability, using the locality of $H$,
\begin{equation*}
\normb{\big(H - (d^{1/2} + d^{-1/2} + d^{-3/2})\big) \f q} \lesssim d^{1/2} + \norm{H \vert_{\cal B_0^c}} \lesssim d^{1/2} + \norm{(H - \E H) \vert_{\cal B_0^c}} \lesssim d^{1/2}\,,
\end{equation*}
where in the last step we used Lemma \ref{prop:StrayCutGraph}.
\end{proof}

\subsection{Rough rigidity -- proof of Proposition~\ref{prop:RoughCandidatEigenvector}} \label{sec:RoughRigidity}

\begin{proof}[Proof of Proposition \ref{prop:RoughCandidatEigenvector}] 
We recall $\mathcal{U}=\big\{x:\frac{1}{5}\am\leq \alpha_x \big\}$ from \eqref{eq:UsetDef} and set $r \deq 2$.

In the following proposition, we introduce the \emph{pruned graph} $\wh{\mathbb{G}}$, a subgraph of $\mathbb{G}$, with a number of nice properties. The construction of $\wh{\mathbb{G}}$ goes back to \cite[Lemma 7.2]{ADK19}. 
In analogy with the notations on $\mathbb{G}$, we introduce the spheres $\wh{S}_i(x)$, the balls $\wh{B}_i(x)$ 
and the number $\wh{N}_y(x)$ of children of $y$ (relative to $x$) for $x,y \in [N]$, $x \neq y$ defined through 
\[ \wh{S}_i(x) \deq \{ z \in [N] \colon \wh{d}(x,z) = i \}, \qquad 
\wh{B}_i(x) \deq \bigcup_{j = 0}^{i} \wh{S}_j(x), \qquad \wh{N}_y(x) \deq \abs{\wh{S}_1(y) \cap \wh{S}_{\wh{d}(x,y)+1}(x)}\,,  
\] 
where $\wh{d}(\cdot,\cdot)$ denotes the graph metric on $[N]$ induced by $\wh{\mathbb{G}}$. 
Moreover, for any $x \in [N]$, we define 
\[ \wh{\alpha}_x \deq \frac{\abs{\wh{S}_1(x)}}{d}, \qquad \wh{\beta}_x \deq \frac{\abs{\wh{S}_2(x)}}{d \abs{\wh{S}_1(x)}}\,. \] 

\begin{proposition}[Pruned graph $\wh{\mathbb{G}}$]\label{prop:(The-cut-graph)}
There are constants $K>0$ and $C\in \mathbb{N}$ such that
if $K \log \log N \leq d \leq (\log N)^{3/4}$ then, with high probability, there exists a subgraph $\wh{\mathbb{G}}$ of $\mathbb{G}$, which has the following properties. 
\begin{enumerate} [label=(\roman*)] 
\item  \label{Item:balls_disjoint_rough} 
The balls $(\wh{B}_3(x))_{x\in \cal U}$ of radius $3$ in the graph $\wh{\mathbb{G}}$ are disjoint. 
\item \label{Item:roughGtree}
The induced subgraph $\wh{\mathbb{G}}|_{\wh{B}_3(x)}$ is a tree for each $x \in \cal U$. 
\item \label{Item:roughGdegree}
The maximal degree of $\mathbb{G}\setminus \wh{\mathbb{G}}$ is bounded by $C$. In particular, 
\[ \abs{\wh{\alpha}_x - \alpha_x} \leq Cd^{-1}, \qquad \qquad \abs{\wh{\beta}_x - \beta_x} \lesssim d^{-1}\,. \]
\item \label{Item:roughGConcertration}
For all $x\in \cal U$, we have 
\[\left|\frac{|\wh{S}_2(x)|}{d |\wh{S}_1(x)|}-1\right|\lesssim \left(\frac{\log \log N}{ d}\right)^{1/2}\quad \text{and} \quad\left|\frac{|\wh{S}_3(x)|}{d |\wh{S}_2(x)|}-1\right|\lesssim \left(\frac{\log \log N}{ d}\right)^{1/2}\,.\]
\item \label{Item:roughGw2}
For all $x\in \cal U$, we have 
\[
\sum_{y\in \wh{S}_1(x)} (\wh{N}_y(x) - d\wh{\beta}_x)^2 \lesssim   (\log N)^2\,. 
\] 
\end{enumerate}  
\end{proposition}

\begin{proof}[Proof of Proposition \ref{prop:(The-cut-graph)}] 
We set $\wh{\mathbb{G}} \deq \mathbb{G}_\tau$ with $\tau = \frac 1 5 \am$, where the latter graph is 
introduced in \cite[Lemma 7.2]{ADK19}. 
Then properties \ref{Item:balls_disjoint_rough}, 
\ref{Item:roughGtree} and \ref{Item:roughGdegree} follow directly from \cite[Lemma~7.2 (i), (ii) and (vi)]{ADK19}, 
respectively, as we have $\tau \asymp \am$, $h((\tau-1)/2) \gtrsim \am$ and \eqref{eq:scaling_alpha_max}.
In particular, $\abs{\wh{\alpha}_x - \alpha_x} \leq Cd^{-1}$, which implies the bound on $\abs{\beta_x - \wh{\beta}_x}$ due to the definitions of $\beta_x$ and $\wh{\beta}_x$ as well as $\abs{\wh{S}_2 (x) - S_2(x)} \leq C \abs{S_1(x)}$. This proves \ref{Item:roughGdegree}. 

The first bound in \ref{Item:roughGConcertration} is a direct consequence of 
$\wh{\beta}_x = \beta_x + O(d^{-1})$ by \ref{Item:roughGdegree} 
and \cite[Lemma~5.4 (i)]{ADK19}. 
Here, we used that   
$\abs{S_1(x)} = D_x \geq \frac{\am d }{5}\gtrsim \frac{\log N }{\log \log N}\geq \mathcal{K} \frac{\log N}{d}$  
due to \eqref{eq:scaling_alpha_max} and $ K \log \log N \leq d \leq (\log N)^{3/4}$ for some large enough $K$ 
as well as $D_x \lesssim \am d$ with high probability by Lemma~\ref{lem:convergence_max_alpha_x_upper_bound_D_x}. 
By \ref{Item:roughGdegree}, we also have $\abs{\wh{S}_3(x) - S_3(x)} \leq C \abs{S_2(x)}$. 
Hence, arguing as above as well as using the first estimate in \ref{Item:roughGConcertration} 
and \cite[Lemma~5.4 (i)]{ADK19} complete the proof of \ref{Item:roughGConcertration}. 

For the proof of \ref{Item:roughGw2}, we note that 
 $\abs{\wh{N}_y(x) - N_y(x)} \leq C$ by \ref{Item:roughGdegree} 
as $\abs{\wh{N}_y(x) - N_y(x)}$ is bounded by the maximal degree of $\mathbb{G} \setminus \wh{\mathbb{G}}$. 
Since $\wh{S}_1(x) \subset S_1(x)$, $\wh{N}_y(x) = N_y(x) + O(1)$ and $\wh{\beta}_x = \beta_x + O(d^{-1})$ by \ref{Item:roughGdegree}, we have  
\[ \sum_{y \in \wh{S}_1(x)} (\wh{N}_y(x) - d \wh{\beta}_x)^2 \leq 2\sum_{y \in S_1(x)} (N_y(x) - d \beta_x)^2 
+ O(\abs{\wh{S}_1(x)}) \lesssim (\log N)^2\,.  
\] 
Here, we also used that $\abs{\wh{S}_1(x)} \lesssim \abs{S_1(x)} \lesssim \log N$ with high probability 
by Lemma~\ref{lem:convergence_max_alpha_x_upper_bound_D_x} and \eqref{eq:scaling_alpha_max} 
as well as 
\[ \sum_{y \in S_1(x)} (N_y(x) - d \beta_x)^2 \lesssim d ( \abs{S_1(x)} + \log N) \bigg( \log d + \frac{\log N}{d} \bigg) 
\lesssim d \log N \frac{\log N}{d} 
\lesssim (\log N)^2\,, 
\] 
whose first step can be read off from the proof of \cite[eq.~(5.9c)]{ADK19}, while the other steps follow from $\abs{S_1(x)} \lesssim \log N$ and $\log d \ll \frac{ \log N}{d}$ 
due to $d \leq (\log N)^{3/4}$. 
\end{proof}

For each $x\in\mathcal{U}$, we introduce the two vectors 
\begin{equation}
\f w_{\pm}(x)\deq\frac{1}{\sqrt{2}} \bigg( \frac{\sqrt{\wh{\alpha}_{x}}}{\sqrt{\wh{\alpha}_{x}+\wh{\beta}_{x}}} \f 1_{x} \pm \frac{1}{\| \f 1_{\wh{S}_{1}(x)}\|} \f 1_{\wh{S}_{1}(x)}+\frac{\sqrt{\wh{\beta}_{x}}}{\sqrt{\wh{\alpha}_{x}+\wh{\beta}_{x}}\|\f 1_{\wh{S}_{2}(x)}\|} \f 1_{\wh{S}_{2}(x)}\bigg)\,.\label{eq:UcandidateDef}
\end{equation}

By Proposition~\ref{prop:(The-cut-graph)} \ref{Item:balls_disjoint_rough}, the balls 
$(\wh{B}_2(x))_{x \in \cal U}$ in $\wh{\mathbb{G}}$ are disjoint. Therefore, the sets $(\supp \f w_+(x)\cup \supp \f w_-(x))_{x \in \cal U}$ are disjoint. 
This shows \ref{Item:rough_rigidity_supports_disjoint}. 

Moreover, as $\langle\f w_+(x),\f w_-(x)\rangle=0$ for any $x\in \cal U$, the family $(\f w_\sigma(x))_{x \in \cal U, \sigma \in \{ \pm\}}$ is orthonormal. 
Let $x \in\cal U$ and $\sigma \in \{\pm\}$. 
Since $\wh{S}_i(x) \subset S_i(x)$ by Proposition~\ref{prop:(The-cut-graph)}, 
we conclude that $\supp \f w_\sigma(x) \subset B_2(x)$. 
For the proof of \eqref{eq:cal_U_approx_w_sigma}, we note that 
$\frac{\wh{\alpha}_x}{\wh{\alpha}_x + \wh{\beta}_x} = 1 + O(\am^{-1/2})$ as 
$\wh{\beta}_x \lesssim 1$ by Proposition~\ref{prop:(The-cut-graph)} \ref{Item:roughGConcertration} 
and $\wh{\alpha}_x \gtrsim \am$ by Proposition~\ref{prop:(The-cut-graph)} \ref{Item:roughGdegree} 
and $\alpha_x \gtrsim \am$ for $x \in \cal U$. 
Therefore, \eqref{eq:cal_U_approx_w_sigma} follows directly from the definition of $\f w_\sigma(x)$ 
in \eqref{eq:UcandidateDef}. 
This proves \ref{Item:UcandidateInBall} in Proposition~\ref{prop:RoughCandidatEigenvector}.

For the proof of \ref{Item:UcandidateEstimate}, in analogy with $H$, 
we denote by $\wh{H}$ the adjacency matrix of $\wh{\mathbb{G}}$ divided by $\sqrt{d}$.

\begin{proposition}\label{prop:roughtEstimateCandidate} 
There is a constant $K>0$ such that if $K \log \log N \leq d \leq (\log N)^{3/4}$ then, 
\[ \|(\wh H- \wh{\Lambda})\f v\|\lesssim \frac{\log \log N}{\sqrt{d}} + \frac{1}{\sqrt{\am}}\,,  \] 
for all normalized $\f v \in \op{span} \{ \f w_\sigma(x) \colon x \in \cal U, \,\sigma \in \{ \pm\}\} $ 
on the high-probability event from Proposition~\ref{prop:(The-cut-graph)}.  
Here, the matrix $\wh{\Lambda}$ is defined through $\wh{\Lambda} \f w_\sigma(x) = \sigma \sqrt{\wh{\alpha}_x + \wh{\beta}_x} \f w_\sigma(x)$ for all $x \in \cal U$ and $\sigma \in \{ \pm\}$. 
\end{proposition}

For the proof of \ref{Item:UcandidateEstimate}, 
we use $\alpha_x \gtrsim \am$, Proposition~\ref{prop:(The-cut-graph)} \ref{Item:roughGdegree} and \ref{Item:roughGConcertration} to obtain 
\[ \absB{\Lambda_{\fra d}(\alpha_x,\beta_x)-\sqrt{\wh{\alpha}_x+\wh{\beta}_x}}\leq \abs{\Lambda_{\fra d}(\alpha_x,\beta_x) - \Lambda(\alpha_x,\beta_x)} + \abs{\Lambda(\alpha_x,\beta_x) -\sqrt{\alpha_x+\beta_x}}+ d^{-1} \am^{-1/2}  
\lesssim \am^{-1/2}\,.  \] 
Here, in the second step,  the estimate on the first term is a consequence of Lemma~\ref{lem:EstimateLambdad}, whose conditions are satisfied by Propsition~\ref{prop:(The-cut-graph)} \ref{Item:roughGdegree} and \ref{Item:roughGConcertration} as well as $d \leq (\log N)^{3/4}$ and \eqref{eq:scaling_alpha_max}. For the second term, we used  \eqref{Lambda_identities}, $\alpha_x \gtrsim \am$ and \eqref{eq:beta_x_equal_1_error_term_on_cal_U}. 
This shows $\norm{D - \wh{\Lambda}} \lesssim \am^{-1/2}$. 
As $\norm{H - \wh{H}} \lesssim d^{-1/2}$ by Proposition \ref{prop:(The-cut-graph)} \ref{Item:roughGdegree}, 
the bound in \ref{Item:UcandidateEstimate}, thus, 
 follows from Proposition \ref{prop:roughtEstimateCandidate}. 
This completes the proof of Proposition~\ref{prop:RoughCandidatEigenvector}. 
\end{proof}

\begin{proof}[Proof of Proposition \ref{prop:roughtEstimateCandidate}]
To simplify the notation in this proof, we set $\lambda_x \deq \sqrt{\wh{\alpha}_x + \wh{\beta}_x}$. 
As $\|\f 1_{\wh{S}_1(x)}\| = \sqrt{|\wh{S}_1(x)|} = \sqrt{\wh{\alpha}_x d}$ and $\|\f 1_{\wh{S}_2(x)}\| = \sqrt{|\wh{S}_2(x)|} = \sqrt{\wh{\alpha}_x\wh{\beta}_x}d$, for $x \in \cal U$ and $\sigma \in \{\pm\}$, we obtain
\begin{align*}
\left(\wh{H} - \sigma \lambda_x \right)\sqrt{2} \f w_{\sigma}(x) 
& = \sigma \f 1_x\left(\frac{\sqrt{|\wh{S}_1(x)|}}{\sqrt{d}} - \sqrt{\wh{\alpha}_x}\right) + \sum_{y\in \wh{S}_1(x)}\left(\frac{\sqrt{\wh{\beta}_x}\wh{N}_y}{\lambda_x \sqrt{d}\sqrt{\abs{\widehat{S}_2}}}+\frac{\sqrt{\wh{\alpha}_{x}}}{\lambda_x \sqrt{d}} -\frac{\lambda_x}{\sqrt{d\wh{\alpha}_x}}\right)\f 1_y \\ 
& \phantom{ = } 
  + \sigma \bigg(\frac{1}{\sqrt{d}\sqrt{\abs{\wh{S}_1(x)}}} -  \frac{\sqrt{\wh{\beta}_x}}{\sqrt{\abs{\wh{S}_2(x)}}} \bigg) \f 1_{\wh{S}_2}(x)
+\frac{\sqrt{\wh{\beta}_x} \f 1_{\wh{S}_{3}(x)}}{\lambda_x \sqrt{d}\sqrt{\abs{\wh{S}_{2}(x)}}} 
\\ & = \frac{1}{\lambda_x\sqrt{\wh{\alpha}_x} d^{3/2}}\sum_{y\in \wh{S}_1(x)}\left(\wh{N}_y(x)-d \wh{\beta}_x \right)\f 1_y +\frac{\sqrt{\wh{\beta}_x} \f 1_{\wh{S}_{3}(x)}}{\lambda_x \sqrt{d} \sqrt{\abs{\wh{S}_2(x)}}}\,. 
\end{align*}
Thus, for $x$, $y \in \cal U$ and $\sigma_1$, $\sigma_2 \in \{ \pm\}$, we get 
\begin{align*}
\scalar{(\wh{H}-\sigma_1\lambda_x)\f w_{\sigma_1}(x)}{(\wh{H}-\sigma_2 \lambda_y) \f w_{\sigma_2}(y)}  
& = \delta_{xy} \big\|\big(\wh{H} - \lambda_x \big)\f w_{+}(x) \big\|^2 \\ 
 & = \delta_{xy} \bigg( \frac{1}{2 \lambda_x^2  \wh{\alpha}_xd^3 }\sum_{y\in \wh{S}_1(x)}\left( \wh{N}_y(x) -d \wh{\beta}_x \right)^2  +\frac{\wh{\beta}_x}{ 2\lambda_x^2}
\frac{\abs{\wh{S}_{3}(x)}}{d\abs{\wh{S}_{2}(x)}}  \bigg)\,,
\end{align*}
since $\norm{\f 1_{\wh{S}_3(x)}}^2 = \abs{\wh{S}_3(x)}$ and $\wh{B}_3(x)$ and $\wh{B}_3(y)$ are disjoint by Proposition~\ref{prop:(The-cut-graph)} \ref{Item:balls_disjoint_rough} if $x \neq y$. 

Therefore, for $\f v = \sum_{x, \sigma} a_{x,\sigma} \f w_\sigma(x)$, we obtain  
\begin{align*}
 \norm{(\wh{H}-\wh{\Lambda})\f v}^2 & = \sum_{x,y, \sigma_1, \sigma_2} \overline{a_{x,\sigma_1}} a_{y,\sigma_2} \scalar{(\wh{H} - \sigma_1 \lambda_x)\f w_{\sigma_1}(x)}{(\wh{H} - \sigma_2 \lambda_y)\f w_{\sigma_2}(y)}\\ 
& = \sum_x \norm{(\wh{H} - \lambda_x)\f w_+(x)}^2 \sum_{\sigma_1,\sigma_2} \overline{a_{x,\sigma_1}} a_{x,\sigma_2}\\ 
& \lesssim \sum_x \norm{(\wh{H} - \lambda_x)\f w_+(x)}^2 \bigg( \abs{a_{x,+}}^2 + \abs{a_{x,-}}^2 \bigg)\,, 
\end{align*} 
which implies Proposition~\ref{prop:roughtEstimateCandidate} due to  
Proposition~\ref{prop:(The-cut-graph)} \ref{Item:roughGdegree}, \ref{Item:roughGConcertration} and \ref{Item:roughGw2} 
as well as 
 $\lambda_x^2 \geq \wh{\alpha}_x \gtrsim \am$ and $\wh{\alpha}_x d \gtrsim \frac{\log N}{\log \log N}$ by $d \leq (\log N)^{3/4}$ (see the proof of Proposition~\ref{prop:(The-cut-graph)}). 
\end{proof}

\section{Eigenvalue process at the edge -- proof of Theorem~\ref{thm:point_process}} \label{sec:ev_process}

In this section we prove Theorem \ref{thm:point_process}. Throughout, we denote by $\cal P_\mu$ a Poisson random variable with parameter $\mu \geq 0$.

\subsection{Decorrelation}

In this subsection we prove a decorrelation result on the joint distribution of the pairs $(\alpha_x, \beta_x)$. To state it, we introduce the $d$-dependent function
\begin{equation} \label{def_Q}
Q(v,w) \deq \P\pb{\cal P_{dv} - dv \geq w \sqrt{dv}}\,,
\end{equation}
which one should think of as a discrete approximation of the tail distribution function of the standard normal in the argument $w$, which is almost independent of $v$ provided that $dv \gg 1$.

\begin{proposition}[Decorrelation]\label{prop:poisson-normal}
Suppose that $1 \leq d \leq N^{1/12}$ and $k \leq N^{1/12}$. Let $v_1, \dots, v_k \in \N$ satisfy $2 \leq v_1, \dots, v_k \leq N^{1/4}$ and $w_1, \dots, w_k \in \R$.  Then
\begin{multline*}
\P\pBB{ \bigcap_{i \in [k]} \hB{ d \alpha_i = v_i, d \sqrt{\alpha_i} (\beta_i - 1) \geq w_i}}
\\
= \prod_{i \in [k]} \P(\cal P_d = v_i) Q(v_i,w_i)
+ O \pbb{N^{-1/3} \prod_{i \in [k]} \P(\cal P_d = v_i) + N^{-k-1}}\,.
\end{multline*}
\end{proposition}

\begin{proof}
By setting $u_i = d v_i + w_i \sqrt{d v_i}$, it suffices to prove
\begin{multline} \label{main_decorrelation}
\P\pBB{ \bigcap_{i \in [k]} \{\abs{S_1(i)} = v_i, \abs{S_2(i)} \geq u_i\}}
\\
= \prod_{i \in [k]} \P(\cal P_d = v_i) \P(\cal P_{d v_i} \geq u_i)
+ O \pbb{N^{-1/3} \prod_{i \in [k]} \P(\cal P_d = v_i) + N^{-k-1}}
\end{multline}
for any $u_1, \dots, u_k \in \N$.

For $V \subset [N]$ we use the notation $A_V \deq (A_{xy} \col x,y \in V)$ and $A_{(V)} \deq (A_{xy} \col x \in V \text{ or } y \in V)$.
For $\ell \geq 1$ we define $\Xi_\ell$ as the event that there is no geodesic in $\bb G$ of length $\ell$ connecting two distinct vertices of $[k]$. We define $\cal M_0, \dots, \cal M_3$ and $\cal E_1, \dots, \cal E_3$ through
\begin{equation*}
\cal M_\ell \deq \E \qBB{ \prod_{i \in [k]} \ind{\abs{S_1(i)} = v_i, \abs{S_2(i)} \geq u_i} \, \ind{\Xi_\ell} \cdots \ind{\Xi_1}}\,, \qquad
\cal E_\ell \deq \E \qBB{ \prod_{i \in [k]} \ind{\abs{S_1(i)} = v_i, \abs{S_2(i)} \geq u_i} \, \ind{\Xi_\ell^c} \ind{\Xi_{\ell - 1}} \cdots \ind{\Xi_1}}\,.
\end{equation*}
The strategy of the proof is to use a telescoping argument to analyse $\cal M_0$ in terms of $\cal M_\ell$ for $\ell \geq 1$, using the trivial splitting $\cal M_\ell = \cal M_{\ell+1} + \cal E_{\ell + 1}$ for each $\ell \geq 0$. 
 
To estimate $\cal E_1$, we condition on $A_{[k]}$ and decompose $\Xi_1^c = \bigsqcup_{\bb U} \h{\bb G |_{[k]} = \bb U}$, where the union ranges over the set of nonempty graphs $\bb U$ on $[k]$. Thus we estimate
\begin{equation} \label{calE1_0}
\cal E_1 \leq \sum_{\bb U} \E \qBB{\E \qBB{ \prod_{i \in [k]} \ind{\abs{S_1(i)} = v_i} \,\bigg\vert\, A_{[k]}}  \ind{\bb G |_{[k]} = \bb U}}\,.
\end{equation}
For a given graph $\bb U$ we denote by $l = \abs{\bb U}$ its number of edges and by $l_i$ the degree of vertex $i$ in $\bb U$, so that $\sum_{i \in [k]} l_i = 2l$. Denoting by $\cal B_{n,p}$ the binomial random variable with parameters $n,p$,
we have on the event $\ind{\bb G |_{[k]} = \bb U}$
\begin{align*}
\E \qBB{ \prod_{i \in [k]} \ind{\abs{S_1(i)} = v_i} \,\bigg\vert\, A_{[k]}} &= \prod_{i \in [k]} \P \pb{\cal B_{N - k+1,d/N} = v_i - l_i}
\\
&= \prod_{i \in [k]} \qbb{\P \pb{\cal P_{(N - k+1)d/N} = v_i - l_i} \pbb{1 + O \pbb{\frac{v_i^2 + d^2}{N}}}}
\\
&= \prod_{i \in [k]} \qbb{\P \pb{\cal P_{d} = v_i - l_i} \pbb{1 + O \pbb{\frac{v_i^2 + d^2 + v_i k + kd}{N}}}}
\\
&\leq \prod_{i \in [k]} \qbb{\P \pb{\cal P_{d} = v_i} \pbb{\frac{v_i}{d}}^{l_i} \pbb{1 + O \pbb{\frac{v_i^2 + d^2 + v_i k + kd}{N}}}}
\\
&\lesssim \pbb{\frac{N^{1/4}}{d}}^{2l} \prod_{i \in [k]} \P \pb{\cal P_{d} = v_i}\,,
\end{align*}
where the second step follows from Lemma \ref{lem:poisson_approx}, the third step from Lemma \ref{lem:Poisson_perturbation}, and the last step from the assumptions on $k,d,v_i$. Plugging this into \eqref{calE1_0}  yields
\begin{equation} \label{calE1}
\cal E_1 \lesssim \sum_{l = 1}^{k(k-1)/2} \binom{k(k - 1)/2}{l} \pbb{\frac{d}{N}}^l \, \pbb{\frac{N^{1/4}}{d}}^{2l} \prod_{i \in [k]} \P \pb{\cal P_{d} = v_i} \leq N^{-1/3} \prod_{i \in [k]} \P \pb{\cal P_{d} = v_i}\,.
\end{equation}

In order to estimate $\cal E_2$, we set $\Xi_{2,xy}$ as the event that there is no geodesic of length $2$ connecting the vertices $x \neq y$. Hence, $\Xi_2 = \bigcap_{1 \leq x < y \leq k} \Xi_{2,xy}$. By a union bound, we estimate
\begin{multline} \label{calE2_estimate}
\cal E_2 \leq \E \qBB{ \prod_{i \in [k]} \ind{\abs{S_1(i)} = v_i} \ind{\Xi_2^c} \ind{\Xi_1}}
\leq
\sum_{1 \leq x < y \leq k}\E \qBB{ \prod_{i \in [k]} \ind{\abs{S_1(i)} = v_i} \, \ind{\Xi_{2,xy}^c} \ind{\Xi_1}}
\\
= \sum_{1 \leq x < y \leq k } \pBB{\prod_{i \in [k] \setminus \{x,y\}} \P\pb{\abs{S_1(i)} = v_i, \Xi_1}} \E \qb{\ind{\abs{S_1(x)} = v_x} \ind{\abs{S_1(y)} = v_y} \ind{\Xi^c_{2,xy}} \ind{\Xi_1}}\,,
\end{multline}
where in the last step we used that the sets $(S_1(i))_{i \in [k]}$ are independent conditioned on $\Xi_1$.
We estimate
\begin{equation*}
\E \qb{\ind{\abs{S_1(x)} = v_x} \ind{\abs{S_1(y)} = v_y} \ind{\Xi^c_{2,xy}} \ind{\Xi_1}}
\leq \E \qBB{ \sum_{z \in S_1(x)} \E \qB{\ind{\abs{S_1(y)} = v_y}  A_{yz}  \,\Big\vert\, S_1(x)} \ind{\abs{S_1(x)} = v_x} \ind{A_{xy} = 0}}\,,
\end{equation*}
and use that on the event $\{A_{xy} = 0\}$ and for any $z \in S_1(x)$ we have
\begin{equation*}
\E \qB{\ind{\abs{S_1(y)} = v_y}  A_{yz}  \,\Big\vert\, S_1(x)} = \frac{d}{N} \, \P(\cal B_{N-3,d/N} = v_y - 1) \lesssim \frac{d}{N} \, \P(\cal P_{d} = v_y - 1) \lesssim \frac{v_y}{N} \, \P(\cal P_d = v_y)\,,
\end{equation*}
where in the second step we used Lemmas \ref{lem:poisson_approx} and \ref{lem:Poisson_perturbation}. Hence,
\begin{align*}
\E \qb{\ind{\abs{S_1(x)} = v_x} \ind{\abs{S_1(y)} = v_y} \ind{\Xi^c_{2,xy}} \ind{\Xi_1}}
&\lesssim 
\frac{v_y}{N} \, \P(\cal P_d = v_y) \, \E \qB{ \abs{S_1(x)} \ind{\abs{S_1(x)} = v_x} }
\\
&\lesssim
 N^{-1/2} \P(\cal P_d = v_y) \P(\cal P_d = v_x) \,,
\end{align*}
where we used $v_x, v_y \leq N^{1/4}$  as well as Lemmas \ref{lem:poisson_approx} and \ref{lem:Poisson_perturbation}. Applying Lemmas \ref{lem:poisson_approx} and \ref{lem:Poisson_perturbation} to $\P\pb{\abs{S_1(i)} = v_i, \Xi_1} \leq \P\pb{\abs{S_1(i)} = v_i}$ in \eqref{calE2_estimate} and estimating the sum over $x < y$ by $k^2$, we conclude that
\begin{equation} \label{calE2}
\cal E_2 \lesssim N^{-1/3} \prod_{i \in [k]} \P\pb{\cal P_d = v_i}\,.
\end{equation}

Next, we estimate $\cal E_3$.
We condition on $A_{([k])}$, note that $\Xi_1$ and $\Xi_2$ are $A_{([k])}$-measurable, and estimate
\begin{equation} \label{Xi3_estimate}
\P(\Xi_3^c \,|\, A_{([k])}) \ind{\Xi_1} \ind{\Xi_2} \leq \E \qa{\sum_{1 \leq i < j \leq k} \sum_{x \in S_1(i)} \sum_{y \in S_1(j)} A_{xy} \, \Bigg\vert\, A_{([k])} } \ind{\Xi_1} \ind{\Xi_2}
\leq \frac{d}{N} \pBB{\sum_{i \in [k]} \abs{S_1(i)}}^2 \,.
\end{equation}
We estimate the right-hand side using
\begin{equation} \label{S_1_estimate}
\abs{S_1(i)} \leq d + C k \log N
\end{equation}
with probability at least $1 - N^{-k-2}$ for some universal constant $C$, as follows from Lemma \ref{lem:Bennett}. 
Using Lemmas~\ref{lem:poisson_approx} and \ref{lem:Poisson_perturbation} and the independence of $(S_1(i))_i$ on $\Xi_1$, we therefore get 
\begin{equation} \label{calE3}
\cal E_3 \leq \E \qBB{\P \p{\Xi_3^c \, \vert\, A_{([k])}}  \ind{\Xi_1} \ind{\Xi_2} \prod_{i \in [k]} \ind{\abs{S_1(i)} = v_i}}
\leq N^{-1/3} \prod_{i \in [k]} \P\pb{\cal P_d = v_i} + N^{-k-1}\,.
\end{equation}
Writing $S_1([k]) = \bigcup_{i \in [k]} S_1(i)$, it therefore suffices to analyse
\begin{align*}
\cal M_3 &=
\E \qBB{ \E \qBB{\prod_{i \in [k]} \ind{\abs{S_2(i)} \geq u_i}  \,\bigg\vert\, A_{([k])}, A_{S_1([k])}} \prod_{i \in [k]} \ind{\abs{S_1(i)} = v_i} \ind{\Xi_1} \ind{\Xi_2} \ind{\Xi_3}}
\\
&=
\E \qBB{ \prod_{i \in [k]} \P \pb{\abs{S_2(i)} \geq u_i \,\big\vert\, A_{([k])},A_{S_1([k])}} \prod_{i \in [k]} \ind{\abs{S_1(i)} = v_i} \ind{\Xi_1} \ind{\Xi_2} \ind{\Xi_3}}\,,
\end{align*}
where we used that the family $(\abs{S_2(i)})_{i \in [k]}$ is independent conditioned on $A_{([k])}$ and $A_{S_1([k])}$, on the event $\Xi_3 \cap \Xi_2 \cap \Xi_1$.

For $i \in [k]$, conditioned on $A_{[k]}$ and on the event $\Xi_3 \cap \Xi_2 \cap \Xi_1$, we have $\abs{S_2(i)} \eqdist \cal B_{\abs{S_1(i)} (N - B_i), d/N}$, where $B_i \deq \sum_{j \in [k] \setminus \{i\}} (\abs{S_1(j)} + 1)$.
Using \eqref{S_1_estimate} to estimate $B_i \leq N^{1/5}$ with probability at least $1 - N^{-k-2}$ by the assumptions on $d$ and $k$, we therefore deduce from  Lemma \ref{lem:poisson_approx2} that
\begin{equation*}
\cal M_3 = \E \qa{\prod_{i \in [k]} \P \pB{\cal P_{\abs{S_1(i)} d (1 - B_i/N)} \geq u_i} \prod_{i \in [k]} \ind{\abs{S_1(i)} = v_i} \ind{\Xi_1} \ind{\Xi_2} \ind{\Xi_3}} \pb{1 + O \p{N^{-1/3}}} + O(N^{-k-1})\,.
\end{equation*}
Here we estimated the error terms of \eqref{Poisson_approx} using $n = \abs{S_1(i)} (N - B_i) \geq N$, since $v_i \geq 2$ by assumption. Hence, invoking \eqref{S_1_estimate}, Lemma \ref{lem:Poisson_perturbation} with $\abs{\epsilon} \leq N^{-4/5}$, and a simple estimate of the tails $\P \pb{\cal P_{\abs{S_1(i)} d (1 - B_i/N)} \geq N^{1/5}}$ using Lemma \ref{lem:Bennett}, we conclude
\begin{equation*}
\cal M_3 = \prod_{i \in [k]} \P \pB{\cal P_{d v_i} \geq u_i} \, \E \qa{ \prod_{i \in [k]} \ind{\abs{S_1(i)} = v_i} \ind{\Xi_1} \ind{\Xi_2} \ind{\Xi_3}} \pb{1 + O \p{N^{-1/3}}} + O(N^{-k-1})\,.
\end{equation*}

Next, we remove the indicator functions $\ind{\Xi_1}, \ind{\Xi_2}, \ind{\Xi_3}$ in $\cal M_3$. The indicator function $\ind{\Xi_3}$ is removed exactly as in  \eqref{Xi3_estimate} and \eqref{calE3}, and the resulting error term is bounded by the right-hand side of \eqref{calE3}. The indicator function $\ind{\Xi_2}$ is removed exactly as in \eqref{calE2_estimate}, and the resulting error term is bounded by the right-hand side of \eqref{calE2}. Since $(\abs{S_1(i)})_{i \in [k]}$ is independent conditioned on $\Xi_1$, we deduce from the splitting $\cal M_0 = \cal M_3+\cal E_3+\cal E_2+\cal E_1$ that
\begin{equation*}
\cal M_0 =
\prod_{i \in [k]} \pbb{\P \pb{\cal P_{d v_i} \geq u_i} \P \pb{  \abs{S_1(i)} = v_i \, \big\vert \, \Xi_1}} \P(\Xi_1)
\pb{1 + O(N^{-1/3})}
+ \pb{\text{RHS of \eqref{calE1}, \eqref{calE2}, \eqref{calE3}}}\,.
\end{equation*}
On $\Xi_1$ we have $\abs{S_1(i)} \eqdist \cal B_{N - k+1,d/N}$, and $\P(\Xi_1) = 1 + O(dk^2/N)$. Invoking Lemmas \ref{lem:poisson_approx} and \ref{lem:Poisson_perturbation} and recalling the estimates \eqref{calE1}, \eqref{calE2}, and \eqref{calE3}, we conclude the proof of \eqref{main_decorrelation}.
\end{proof}

\subsection{Convergence of the point process -- proof of Theorem \ref{thm:point_process}} \label{sec:proof_thm_poisson_statistics} 
Throughout this subsection we make the following assumptions. We choose $\gamma \in (0,1/6)$ and a small enough constant $\epsilon > 0$, and assume that
\begin{equation} \label{d_assumptions_2}
d \geq (\log \log N)^{\frac{1}{1 - 6 \gamma - 2 \epsilon}}\,, \qquad \fra u - 2 \geq d^{\frac{6 \gamma - 1}{10} + \epsilon}\,.
\end{equation}
Moreover, we assume \eqref{eq:condition_d_SpecMax} for some large enough constant $K_*$.
Finally, we suppose that $\cal K$ satisfies
\begin{equation} \label{calK_condition}
C \leq \cal K \leq d^{\frac{1}{2} - 3 \gamma - \epsilon} \frac{((\fra u - 2)^5 \wedge 1)}{\sqrt{\log \fra u}}
\end{equation}
for some large enough universal constant $C \geq 1$ (see the remark after \eqref{def_kappa}).

We use the notations $E_s \deq [s,\infty)$ and
\begin{equation} \label{def_Gg}
g(s) \deq \frac{1}{\sqrt{2 \pi}} \ee^{-\frac{1}{2} s^2}\,, \qquad G(s) \deq \int_s^\infty g(t) \, \dd t\,.
\end{equation}
We shall use the following simple upper bound on $\kappa$ defined in \eqref{def_kappa}. 

\begin{lemma} \label{lem:kappa_condition}
Recalling the width of the window $\chi$ from Corollary \ref{Cor:SpecMax}, we have
\begin{equation} \label{kappa_condition}
 \kappa \ll d \tau(\fra u) \chi \asymp \frac{\sqrt{\fra u}}{\log \fra u} d^{2 \gamma}\,.
\end{equation}
\end{lemma}
\begin{proof}
For $n \in \N$ define $s_n \deq - \theta(\fra u) n - \theta(\fra u) \ang{d \fra u}$, where $\ang{x} \in [-1/2, 1/2)$ the $1$-periodic representative of $x \in \R$. Using $G(0) = 1/2$ we find
\begin{equation*}
\rho(E_{s_n}) = \sum_{\ell \in \Z} \fra u^{\ang{d \fra u} + \ell} G(\theta(\fra u) (\ell - n)) \geq \frac{\fra u^{n - 1}}{2}\,.
\end{equation*}
We deduce that if $n \geq 1 + \frac{\log d}{\log \fra u}$ then $\rho(E_{s_n}) \geq d/2$, and hence $\kappa \leq - s_n$  (using that $\cal K \leq d/2$). Using $\theta(\fra u) \asymp \sqrt{\fra u}$ we deduce that
\begin{equation*}
\kappa \lesssim \sqrt{\fra u} \pbb{\frac{\log d}{\log \fra u} + 1} \ll \frac{\sqrt{\fra u}}{\log \fra u} d^{2 \gamma}\,,
\end{equation*}
where in the last step we used the assumption $d^{2 \gamma} \gg \log \log N$ as well as both estimates of \eqref{eq:scaling_alpha_max}.
\end{proof}

We define
\begin{equation*}
q \deq \frac{c_* d^{2 \gamma - 1}}{\log \fra a} - \fra a + \fra u\,,
\end{equation*}
with $c_* \leq 1$ as in \eqref{eq:def_cal_W}.
Using Lemma \ref{lem:u_a}, it is not hard to see that for large enough $K_*$ the condition \eqref{eq:condition_d_SpecMax} 
implies $q \geq \frac{\log d}{d}$. Hence, by Lemma \ref{lem:convergence_max_alpha_x_upper_bound_D_x}, 
with high probability we have
\begin{equation} \label{W_q}
\cal W = \h{x\col \abs{\alpha_x - \fra u} \leq q}\,, \qquad q \leq d^{2 \gamma - 1}\,.
\end{equation}
We define the reference process
\begin{equation} \label{def_Z_x}
\Sigma \deq \sum_{x \in [N]} \delta_{Z_x}\,, \qquad
Z_x \deq
\begin{cases}
\theta(\fra u) (d \alpha_x - d \fra u) + d \sqrt{\alpha_x} (\beta_x - 1) & \text{if }
\abs{\alpha_x - \fra u} \leq q
\\
-\infty & \text{otherwise}\,.
\end{cases}
\end{equation}
Here $-\infty$ acts as a graveyard state for those points which are outside the range of interest.

Recall the stray eigenvalue $\nu$ from Corollary \ref{Cor:SpecMax}. Its rescaled version is denoted by
\begin{equation} \label{def_tilde_nu}
\tilde \nu \deq d \tau (\fra u) (\nu - \sigma(\fra u))\,.
\end{equation}
We define the error parameter
\begin{equation} \label{def_eta}
\eta \deq d^{3 \gamma + \frac{\epsilon}{2} - \frac{1}{2}} \frac{\fra u^5}{(\fra u - 2)^5}\,.
\end{equation}

\begin{lemma} \label{lem:Sigma_Phi}
Suppose that the assumptions of Corollary \ref{Cor:SpecMax} hold. Then with high probability, for all $s \geq -\kappa$ we have
\begin{equation*}
 \Sigma(E_{s+\eta}) \leq \Phi(E_s) - \delta_{\tilde \nu}(E_s) \leq  \Sigma(E_{s-\eta})\,.
\end{equation*}
\end{lemma}
\begin{proof}
The proof proceeds in two steps, passing by the intermediate process
\begin{equation} \label{def_SIgma_tilde}
\tilde \Sigma \deq \sum_{x \in [N]} \delta_{\tilde Z_x}\,,
\qquad
\tilde Z_x \deq
\begin{cases}
d\tau(\fra u) \pb{\Lambda_{\fra d}(\alpha_x, \beta_x) - \sigma(\fra u)} & \text{if }
\abs{\alpha_x - \fra u} \leq q
\\
-\infty & \text{otherwise}\,.
\end{cases}
\end{equation}
We now compare $\Phi - \delta_{\tilde \nu}$ and $\tilde \Sigma$. Note first that we have
\begin{equation} \label{eta_lower_bound}
\eta \gg d \tau(\fra u) \, \frac{d^{- \frac 1 2 + 3 \gamma}}{d \fra u} 
\pbb{ 1 + \pbb{ \frac{\log d}{\log \fra u} }^2 \frac{\fra u^4}{(\fra u - 2)^4} }\,,
\end{equation}
where the right-hand side is the error bound on $\epsilon_x$ in Corollary \ref{Cor:SpecMax} multiplied by $d \tau(\fra u)$. Hence, by Corollary \ref{Cor:SpecMax} and Lemma \ref{lem:kappa_condition}, with high probability, for all $s \geq -\kappa$ we have
\begin{equation} \label{tildeSigma_Phi}
\tilde \Sigma(E_{s + \eta/2}) \leq \Phi(E_s)  - \delta_{\tilde \nu}(E_s)  \leq  \tilde \Sigma(E_{s - \eta/2})\,.
\end{equation}

Next, we compare $\tilde \Sigma$ and $\Sigma$.
We expand $\Lambda_{\fra d}(\alpha, \beta)$ using Corollary \ref{cor:Lambda_expansion}. We use that with high probability, for all $x \in \cal W$ we have
\begin{equation} \label{alpha_beta_bounds}
\abs{\alpha_x - \fra u} \leq d^{2 \gamma - 1}\,, \qquad \abs{\beta_x - 1} \lesssim \frac{d^{\gamma - 1/2}}{\sqrt{d \fra u}}\,,
\end{equation}
where the first estimate follows from \eqref{W_q} and the second from Corollary~\ref{cor:concentration_beta_x} with $\delta = d^{2\gamma - 1}$. 
Thus, abbreviating $\alpha = \alpha_x$ and $\beta = \beta_x$ for $x \in \cal W$, we find that the assumptions of Corollary \ref{cor:Lambda_expansion} are satisfied by \eqref{eq:ConditionAlphaMaxMin} and \eqref{alpha_beta_bounds}, 
so that applying \eqref{eq:Lambda_d_expansion} with $\sigma(\fra u) = \Lambda_{\fra d}(\fra u, 1)$ (see \eqref{eq:def_Lambda_d}) yields 
\begin{align}
\Lambda_{\fra d}(\alpha, \beta) &= \sigma(\fra u) + \frac{\theta(\fra u)}{\tau(\fra u)} (\alpha - \fra u) +  \frac{\sqrt{\fra u}}{\tau(\fra u)} (\beta - 1) + O \pbb{\frac{d^{4 \gamma}}{\fra u d^2}}
\nonumber \\
&= \sigma(\fra u) + \frac{\theta(\fra u)}{\tau(\fra u)} (\alpha - \fra u) + \frac{1}{\tau(\fra u)} \sqrt{\alpha} (\beta - 1) + O \pbb{\frac{d^{4 \gamma}}{\fra u d^2}}\,.
\label{eq:expansion_Lambda_d_sigma_u} 
\end{align}
Since $d \tau (\fra u) \frac{d^{4\gamma}}{\fra u d^2} \ll \eta$, we conclude
\begin{equation} \label{Sigma_tilde_Sigma}
\Sigma(E_{s + \eta/2}) \leq \tilde \Sigma(E_s) \leq \Sigma(E_{s - \eta/2})\,,
\end{equation}
and the claim follows together with \eqref{tildeSigma_Phi}. 
\end{proof}

By Lemma \ref{lem:Sigma_Phi}, to analyse $\Phi - \delta_{\tilde \nu}$ it suffices to analyse $\Sigma$. Indeed, denoting by $\Omega$ the high-probability event from Lemma \ref{lem:Sigma_Phi}, we have
\begin{align*}
\P \pBB{\bigcap_{i \in [n]} \{\Phi(E_{s_i})  - \delta_{\tilde \nu}(E_{s_i}) \leq k_i\}} &= \P \pBB{\bigcap_i \{\Phi(E_{s_i})  - \delta_{\tilde \nu}(E_{s_i}) \leq k_i\} \cap \Omega} + O(\P(\Omega^c))
\\
&\leq \P \pBB{\bigcap_{i \in [n]} \{\Sigma(E_{s_i + \eta}) \leq k_i\} \cap \Omega} + O(\P(\Omega^c))
\\
&= \P \pBB{\bigcap_{i \in [n]} \{\Sigma(E_{s_i + \eta}) \leq k_i\}} + O(\P(\Omega^c))\,.
\end{align*}
Together with an analogous lower bound, we therefore conclude that
\begin{equation} \label{comparison_Phi_Sigma}
\P \pBB{\bigcap_{i \in [n]} \{\Sigma(E_{s_i - \eta}) \leq k_i\}} \leq \P \pBB{\bigcap_{i \in [n]} \{\Phi(E_{s_i}) - \delta_{\tilde \nu}(E_{s_i}) \leq k_i\}} + o(1) \leq \P \pBB{\bigcap_{i \in [n]} \{\Sigma(E_{s_i + \eta}) \leq k_i\}}\,.
\end{equation}
Thus, abbreviating $t_i = s_i \pm \eta$, it suffices to analyse
\begin{equation} \label{sum_r_k}
\P \pBB{\bigcap_{i \in [n]} \{\Sigma(E_{t_i}) \leq k_i\}} = \sum_{r_1 = 0}^{k_1} \cdots \sum_{r_n = 0}^{k_n}
\P \pBB{\bigcap_{i \in [n]} \{\Sigma(E_{t_i}) = r_i\}}\,.
\end{equation}
We shall do so by comparing $\Sigma$ to a reference Poisson process with intensity $\tilde \rho$, defined by
\begin{equation} \label{def_rho_tilde}
\tilde \rho(E_s) \deq \sum_{v \in \N} \ind{\abs{v - d \fra u} \leq dq} \, N \P(\cal P_d = v) \, Q\pb{v, s - \theta(\fra u) \p{v - d \fra u}}\,,
\end{equation}
where we recall the definition \eqref{def_Q}.
The measure $\tilde \rho$ approximates $\rho$ in the following sense.

\begin{lemma} \label{lem:rho_tilde_rho}
If $\kappa$ satisfies \eqref{kappa_condition} then for any $s \geq -2 \kappa$  we have 
\begin{equation*}
\tilde \rho(E_s) = \rho(E_s) \pb{1 + O \pb{\p{d \sqrt{\fra u}}^{4 \gamma - 1}}} + O(\ee^{-c d^{2 \gamma}})\,,
\end{equation*}
for some small enough constant $c > 0$.
\end{lemma}
\begin{proof}
For $\abs{v - d \fra u} \leq dq$ we have, by \eqref{Poisson_f},
\begin{equation} \label{NP_Taylor}
N \P(\cal P_d = v) = N \ee^{-d f(v/d)} \pbb{1 + O \pbb{\frac{1}{d \fra u}}} = \fra u^{d \fra u - v} \, \pbb{1 + O \pbb{\frac{d^{4 \gamma}}{d \fra u}}}\,,
\end{equation}
where the second step follows from a Taylor expansion of order one $f$ around $\fra u$, using that $q \leq d^{2\gamma - 1}$ by \eqref{W_q}.
Next, setting $\mu = dv \asymp d^2 \fra u$ for $\abs{v - d \fra u} \leq dq$, we find from Lemmas \ref{lem:Bennett} and \ref{lem:ml_approx}, for any $0 \leq \xi \leq 1/6$,
\begin{equation*}
Q(v,w) =
\begin{cases}
G(w) (1 + O(\mu^{3 \xi - 1/2})) & \text{if } w \leq \mu^{\xi}
\\
O(\ee^{-(w^2 \wedge \mu)/3}) & \text{otherwise}\,.
\end{cases}
\end{equation*}
Choosing $\xi \deq \gamma / 2$ and recalling \eqref{NP_Taylor}, these two cases yield the splitting
\begin{equation} \label{rho_tilde_split}
\tilde \rho(E_s) = \sum_{v \in \N} \ind{\abs{v - d \fra u} \leq dq} \, \fra u^{d \fra u - v} \, G(s - \theta(\fra u) \p{v - d \fra u})\, \pb{1 + O \pb{\p{d \sqrt{\fra u}}^{4 \gamma - 1}}} + O(\ee^{-c d^{2 \gamma} \fra u^\gamma})\,,
\end{equation}
where the error term arises by estimating $G(w)$ and $Q(v,w)$ for $w = s - \theta(\fra u)(v -d \fra u)\geq \mu^\xi$ using $w^2 \wedge \mu \geq \mu^{\gamma} \geq (d^2 \fra u)^\gamma / 4$ combined with
\begin{equation*}
\sum_{v \in \N} \ind{\abs{v - d \fra u} \leq dq} \, \fra u^{d \fra u - v} \, \ee^{-(d^2 \fra u)^\gamma / 12} \lesssim d \exp \pbb{dq \log \fra u - \frac{(d^2 \fra u)^\gamma}{12}} \lesssim \exp \pbb{-\frac{(d^2 \fra u)^\gamma}{24}} \leq \ee^{-c d^{2 \gamma} \fra u^\gamma}\,,
\end{equation*}
after choosing $c_*$ in \eqref{eq:def_cal_W} small enough.

What remains is to remove the condition $\abs{v - d \fra u} \leq dq$ in \eqref{rho_tilde_split}. The contribution of the terms $v - d \fra u > dq$ is easy to estimate by $O(\fra u^{-d q}) = O(\ee^{-c d^{2 \gamma}})$, where we used \eqref{eq:scaling_alpha_max}. The contribution of the terms $v - d \fra u < -dq$ is estimated by
\begin{equation} \label{sum_dq_large}
\sum_{v \in \Z} \ind{d \fra u - v > dq} \, \fra u^{d \fra u - v} \, G(s + \theta(\fra u) \p{d \fra u - v})\,.
\end{equation}
From \eqref{eq:scaling_alpha_max}, $\theta(\fra u) \asymp \sqrt{\fra u}$,  and \eqref{kappa_condition}, we conclude that $2 \kappa \leq \frac{\theta(\fra u)}{2} dq$.
Hence, for $d \fra u - v > dq$ we have $G(s + \theta(\fra u) \p{d \fra u - v}) \leq \exp (-c \fra u (d \fra u - v)^2)$ (see \eqref{G_g_estimates} below) and we estimate \eqref{sum_dq_large} by
\begin{equation*}
\sum_{v \in \Z} \ind{d \fra u - v > dq} \exp \pB{-(d \fra u - v) \pb{c \fra u (d \fra u - v) - \log \fra u}} \leq
\sum_{v \in \Z} \ind{d \fra u - v > dq} \exp \pB{- c \fra u (d \fra u - v)^2} \leq \ee^{-c d^{4 \gamma}}
\end{equation*}
for some constant $c > 0$.
This concludes the proof.
\end{proof}

Next, we observe that the intensity measure $\rho$ is Lipschitz continuous. The proof of the following result is given in Appendix~\ref{sec:lipschitz}.

\begin{lemma} \label{lem:rho_Lipschitz}
For any $s \in \R$ we have $- \frac{\dd}{\dd s} \rho(E_s) \lesssim \sqrt{\log \fra u} \, \rho(E_s)$.
\end{lemma}

\begin{corollary} \label{cor:Lipschitz}
For any $s \in \R$ and $\eta > 0$ we have
$\rho(E_{s - \eta}) \leq \ee^{C \eta \sqrt{\log \fra u}} \rho(E_s)$.
\end{corollary}

Note that, by \eqref{def_eta}, \eqref{calK_condition}, and Corollary \ref{cor:Lipschitz}, $\rho(E_{-\kappa - \eta}) \leq 2 \cal K$. Using Lemma \ref{lem:rho_tilde_rho} we therefore deduce that
\begin{equation} \label{tilde_rho_estimate}
\tilde \rho(E_{-\kappa - \eta}) \leq 3 \cal K\,,
\end{equation}
since $\cal K \geq 1$.

Next, we show that the joint law of the variables $(Z_x)$ factorizes asymptotically.
Recall the definition of $Z_x$ from \eqref{def_Z_x}.

\begin{lemma} \label{lem:PZ_fact}
If $\ell \leq c \frac{\log N}{d^{2 \gamma}}$ for some small enough constant $c > 0$, and $a_1, \dots, a_\ell \in \R$, then
\begin{equation*}
N^\ell \P \pBB{\bigcap_{i \in [\ell]} \{Z_i \geq a_i\}} = \pBB{\prod_{i \in [\ell]} \tilde \rho(E_{a_i})} + O(N^{-2/7})\,.
\end{equation*}
\end{lemma}

\begin{proof}
From Proposition \ref{prop:poisson-normal} we get
\begin{align*}
\P \pBB{\bigcap_{i \in [\ell]} \{Z_i \geq a_i\}} &= \sum_{v_1, \dots, v_\ell \in \N} \pBB{\prod_{i \in [\ell]} \ind{\abs{v_i - d \fra u} \leq dq}} \P \pBB{\bigcap_{i \in [\ell]} \hB{d \alpha_i = v_i, d \sqrt{\alpha_i} (\beta_i - 1) \geq a_i - \theta(\fra u) \p{v_i - d \fra u}}}
\\
&=  \prod_{i \in [\ell]} \pBB{\sum_{v \in \N} \ind{\abs{v - d \fra u} \leq dq} \, \P(\cal P_d = v) \, Q\pb{v, a_i - \theta(\fra u) \p{v - d \fra u}}}
\\
&\qquad
+ O \pB{N^{-1/3} \P(\cal P_d \geq d \fra u - dq)^\ell + (2 d q)^\ell N^{-\ell-1}}\,.
\end{align*}
Using that $\P(\cal P_d \geq v) \leq C \P(\cal P_d = v)$ for $v \geq \frac{3}{2} d$, the estimate \eqref{Poisson_f}, and a Taylor expansion of $f$ around $\fra u$ (see \eqref{NP_Taylor}), we deduce that
\begin{equation*}
N \P(\cal P_d \geq d \fra u - dq) \lesssim \fra u^{dq} \leq \ee^{d^{2 \gamma}}\,,
\end{equation*}
where in the last step we used \eqref{eq:scaling_alpha_max}. The claim now easily follows from the assumption on $\ell$.
\end{proof}

Next, we recall the general definition of correlation measures of point processes. Let $\Phi$ be a random point process on some measurable space $\cal Z$. We can represent $\Phi = \sum_{x \in \cal X} \delta_{Z_x}$,
where $\cal X$ is an index set and $(Z_x)_{x \in \cal X}$ is an exchangeable family of random variables in $\cal Z$.  For $k \in \N^*$ we define the $k$-point correlation measure of the process $\Phi$ as the measure $q_{\Phi,k}$ on $\cal Z^k$ satisfying
\begin{equation} \label{def_correlation_measure}
q_{\Phi, k}(F) \deq \sum_{x_1, \dots, x_k \in \cal X}^* \P ((Z_{x_1}, \dots, Z_{x_k}) \in F)
\end{equation}
for measurable $F \subset \cal Z^k$, where the sum ranges over distinct indices of $\cal X$. In our applications, $\cal X = [N]$ is a deterministic set of size $N$, in which case
\begin{equation} \label{def_correlation_measure_2}
q_{\Phi, k}(F) = N (N-1) \cdots (N - k + 1) \, \P ((Z_1, \dots, Z_k) \in F)\,.
\end{equation}
We frequently write $q_{\Phi,k} \equiv q_\Phi$ when there is no risk of confusion.

The following result is an inclusion-exclusion formula relating finite-dimensional distributions of a point process to its correlation measures. It is standard, except some extra care is taken to avoid overexpanding, which turns out to be important for its application in the proof of Lemma~\ref{lem:cor_fct1} below. Its proof is given in Appendix \ref{sec:inclusion_exclusion}.

\begin{lemma} \label{lem:incl_excl}
For any $n,m \in \N^*$, $k_1, \dots, k_n \in \N$, and disjoint measurable $I_1, \dots, I_n \subset \cal Z$, we have
\begin{multline*}
\P(\Phi(I_1) = k_1, \dots, \Phi(I_n) = k_n)
= \frac{1}{k_1 ! \cdots k_n !} \sum_{\ell_1, \dots, \ell_n \in \N} \ind{\sum_i \ell_i \leq m}  \frac{(-1)^{\sum_i \ell_i}}{\ell_1! \cdots \ell_n!} \, q_{\Phi}(I_1^{k_1 + \ell_1} \times \cdots \times I_n^{k_n + \ell_n})
\\
+ O \pBB{\frac{1}{k_1 ! \cdots k_n !} \sum_{\ell_1, \dots, \ell_n \in \N} \ind{\sum_i \ell_i = m+1}  \frac{1}{\ell_1! \cdots \ell_n!} \, q_{\Phi}(I_1^{k_1 + \ell_1} \times \cdots \times I_n^{k_n + \ell_n})}\,.
\end{multline*}
\end{lemma}

We now show that the $\ell$-point correlation measure factorizes asymptotically, which will imply Poisson statistics.

\begin{lemma} \label{lem:correlation_factorized}
Let $\ell \leq c \frac{\log N}{d^{2 \gamma}}$, for some small enough constant $c > 0$. For all $i \in [\ell]$, let $I_i = [a_i, b_i)$ with $-\kappa - \eta \leq a_i < b_i$. Then
\begin{equation*}
q_{\Sigma}(I_1 \times \cdots \times I_\ell) = \qbb{1 - O\pbb{\frac{\ell^2}{N}}} \, \prod_{i \in [\ell]} \tilde \rho(I_i) + O \p{N^{-1/4}}\,.
\end{equation*}
\end{lemma}

\begin{proof}
By \eqref{def_correlation_measure_2} we have $q_{\Sigma}(I_1 \times \cdots \times I_\ell) = N (N - 1) \cdots (N - \ell+1) \P(Z_1 \in I_1, \dots, Z_\ell \in I_\ell)$, 
and the right-hand side can be calculated using
\begin{align*}
N^\ell \P(Z_1 \in I_1, \dots, Z_\ell \in I_\ell)
&= \sum_{U \subset [\ell]} (-1)^{\abs{U}} \, N^\ell\, \P \pBB{\bigcap_{i \in U} \{Z_i \geq b_i\} \cap \bigcap_{i \in [\ell] \setminus U} \{Z_i \geq a_i\}}
\\
&= \sum_{U \subset [\ell]} (-1)^{\abs{U}} \prod_{i \in U} \tilde \rho(E_{b_i}) \prod_{i \in [\ell] \setminus U} \tilde \rho(E_{a_i}) + O \p{2^\ell N^{-2/7}}
\\
&= \prod_{i \in [\ell]} \tilde \rho([a_i, b_i)) + O \p{N^{-1/4}}\,,
\end{align*}
where in the second step we used Lemma \ref{lem:PZ_fact}.
\end{proof}

Using the factorization of the correlation functions from Lemma~\ref{lem:correlation_factorized} and the inclusion-exclusion  formula from Lemma~\ref{lem:incl_excl}, we deduce the joint Poisson distribution for the finite-dimensional distributions of $\Sigma$. 

\begin{lemma} \label{lem:cor_fct1}
Suppose that $n \in \N^*$ satisfies $n \ll \frac{\log N}{\cal Kd^{2 \gamma}}$. Let $I_1, \dots, I_n$ be disjoint intervals of the form $I_i = [a_i, b_i)$ with $-\kappa - \eta \leq a_i < b_i$.
Then for all $k_1, \dots, k_n \in \N$ we have
\begin{equation*}
\P \pBB{\bigcap_{i \in [n]} \{\Sigma(I_i) = k_i\}} = \prod_i \frac{\tilde \rho(I_i)^{k_i} \ee^{-\tilde \rho(I_i)}}{k_i !} + \cal E(k_1, \dots, k_n)\,,
\end{equation*}
where the error term satisfies, for some small enough $c > 0$, 
\begin{equation} \label{error_sum}
\sum_{k_1, \dots, k_n \in \N} \abs{\cal E(k_1, \dots, k_n)} \lesssim \ee^{-c \frac{\log N}{d^{2\gamma}}}\,.
\end{equation}
\end{lemma}
\begin{proof}
Choose $m \deq c \frac{\log N}{d^{2 \gamma}}$ for some constant $c > 0$, which will be chosen small enough in the following. Suppose first that $\sum_i k_i \leq m$.
Then we get from Lemma \ref{lem:incl_excl} that $\P \pb{\bigcap_{i \in [n]} \hb{\Sigma(I_i) = k_i}} = \cal A_0 + \cal E_0$, where
\begin{align*}
\cal A_0
&\deq \frac{1}{k_1 ! \cdots k_n !} \sum_{\ell_1, \dots, \ell_n \in \N}  \ind{\sum_i \ell_i \leq m}  \frac{(-1)^{\sum_i \ell_i}}{\ell_1 ! \cdots \ell_n !} \, q_{\Sigma}(I_1^{k_1 + \ell_1} \times \cdots \times I_n^{k_n + \ell_n})\,,
\\
\abs{\cal E_0} &\lesssim \frac{1}{k_1 ! \cdots k_n !} \sum_{\ell_1, \dots, \ell_n \in \N}  \ind{\sum_i \ell_i = m+1}  \frac{1}{\ell_1 ! \cdots \ell_n !} \, q_{\Sigma}(I_1^{k_1 + \ell_1} \times \cdots \times I_n^{k_n + \ell_n})\,.
\end{align*}
Using Lemma \ref{lem:correlation_factorized} with $\ell = k_1+ \ell_1 + \ldots + k_n + \ell_n \leq 2 m$ 
(after choosing $c$ in the definition of $m$ small enough) and \eqref{tilde_rho_estimate}, we find $\cal A_0 = \cal A_1 + \cal E_1$, where
\begin{align*}
\cal A_1 &= \biggl( \prod_i \frac{\tilde \rho(I_i)^{k_i}}{k_i !} \biggr) 
\sum_{\ell_1, \dots, \ell_n \in \N}  \ind{\sum_i \ell_i \leq m} \, \prod_{i} \frac{(-\tilde \rho(I_i))^{\ell_i}}{\ell_i!}\,,
\\
\abs{\cal E_1} &\lesssim \frac{(\log N)^2}{N} \biggl( \prod_i \frac{(3 \cal K)^{k_i}}{k_i !} \biggr)
\sum_{\ell_1, \dots, \ell_n \in \N}  \prod_{i} \frac{(3 \cal K)^{\ell_i}}{\ell_i!} + \frac{1}{N^{1/4}} \prod_i \frac{\ee}{k_i!}\,.
\end{align*}
By \eqref{calK_condition}, we find that $\ind{\sum_i k_i \leq m} \abs{\cal E_1}$ satisfies \eqref{error_sum}.
Moreover, we find $\cal A_1 = \cal A_2 + \cal E_2$, where
\begin{align*}
\cal A_2 &\deq \prod_i \frac{\tilde \rho(I_i)^{k_i} \ee^{-\tilde \rho(I_i)}}{k_i !}\,,
\\
\cal E_2 &\deq  -  \biggl( \prod_i \frac{\tilde \rho(I_i)^{k_i}}{k_i !} \biggr)
\sum_{\ell_1, \dots, \ell_n \in \N}  \ind{\sum_i \ell_i \geq m + 1} \, \prod_{i} \frac{(-\tilde \rho(I_i))^{\ell_i}}{\ell_i!} = 
O \pBB{ \biggl( \prod_i \frac{\tilde \rho(I_i)^{k_i}}{k_i !} \biggr) \sum_{\ell \geq m+1} \frac{(3n\cal K)^\ell}{\ell!}}
\end{align*}
by the multinomial theorem and \eqref{tilde_rho_estimate}. 
By \eqref{calK_condition} we have $m \geq 3 \ee^2 n \cal K$, and hence $\sum_{\ell \geq m+1} \frac{(3n\cal K)^\ell}{\ell!} \leq \ee^{-m}$ by Stirling's approximation, and thus it is easy to see that $\ind{\sum_i k_i \leq m} \abs{\cal E_2}$ satisfies \eqref{error_sum}.
The error term $\cal E_0$ is estimated in exactly the same way. This concludes the proof in the case $\sum_i k_i \leq m$. To remove this restriction, for $\# = \leq, >$, we write
\begin{equation*}
X_\# \deq \sum_{k_1, \dots, k_n \in \N} \ind{\sum_i k_i \# m} \P \pBB{\bigcap_{i \in [n]} \{\Sigma(I_i) = k_i\}}\,, \qquad Y_\# \deq  \sum_{k_1, \dots, k_n \in \N} \ind{\sum_i k_i \# m} \prod_i \frac{\tilde \rho(I_i)^{k_i} \ee^{-\tilde \rho(I_i)}}{k_i !}
\end{equation*}
and estimate
\begin{equation*}
 X_> = 1 - X_\leq = 1 - Y_\leq + O(\ee^{-c \frac{\log N}{d^{2\gamma}}}) =  Y_> + O(\ee^{-c \frac{\log N}{d^{2\gamma}}}) = O(\ee^{-c \frac{\log N}{d^{2\gamma}}})\,,
\end{equation*}
where the second step follows from the previous argument, and the last step is an
elementary exercise using the assumption \eqref{calK_condition}, the definition \eqref{def_kappa}, and the assumption on $I_i$.
This concludes the proof.
\end{proof}

From now on we fix $n$, in which case under the condition \eqref{calK_condition} the assumptions of Lemma \ref{lem:cor_fct1} are satisfied. Let $\tilde \Psi$ be the Poisson process with intensity $\tilde \rho$ from \eqref{def_rho_tilde}.
For any $t_1 \geq t_2 \geq \cdots \geq t_n \geq - \kappa - \eta$ and $r_1 \leq r_2 \leq \cdots \leq r_n$ with $r_i \in \N$, we set $t_0 = +\infty$ and $r_0 = 0$, and use Lemma \ref{lem:cor_fct1} to write
\begin{multline*}
\P \pBB{\bigcap_{i \in [n]} \{\Sigma(E_{t_i}) = r_i\}} = \P \pBB{\bigcap_{i \in [n]} \hb{\Sigma([t_i, t_{i-1})) = r_i - r_{i-1}}}
\\
=
\P \pBB{\bigcap_{i \in [n]} \hb{\tilde \Psi([t_i, t_{i-1})) = r_i - r_{i-1}}} + \cal E(r_1, \dots, r_n)
= \P \pBB{\bigcap_{i \in [n]} \{\tilde \Psi(E_{t_i}) = r_i\}} + \cal E(r_1, \dots, r_n)\,,
\end{multline*}
where $\sum_{r_1, \dots, r_n \in \N} \abs{\cal E(r_1, \dots, r_n)} \ll 1$. Recalling \eqref{sum_r_k}, we deduce, for any $t_1, \dots, t_n \geq -\kappa - \eta$
and $k_1, \dots, k_n \in \N$, that
\begin{equation} \label{Sigma_Psi_comparison}
\P \pBB{\bigcap_{i \in [n]} \{\Sigma(E_{t_i}) \leq k_i\}} = \P \pBB{\bigcap_{i \in [n]} \{\tilde \Psi(E_{t_i}) \leq k_i\}} + o(1)\,.
\end{equation}

Next, we use the comparison results from \eqref{comparison_Phi_Sigma} and Lemma \ref{lem:rho_tilde_rho}, as well as the Lipschitz continuity of $\rho$ from Lemma \ref{lem:rho_Lipschitz} to deduce the following result.

\begin{lemma} \label{lem:cont_Psi}
Fix $n \in \N^*$. Then
\begin{equation*}
\P \pBB{\bigcap_{i \in [n]} \{\tilde \Psi(E_{t_i}) \leq k_i\}} = \P \pBB{\bigcap_{i \in [n]} \{\Psi(E_{s_i}) \leq k_i\}} + o(1)
\end{equation*}
uniformly for $k_1, \dots, k_n \in \N$, $s_1, \dots, s_n \geq -\kappa$, and $t_1, \dots, t_n$ satisfying $\abs{t_i - s_i} \leq \eta$.
\end{lemma}

\begin{proof}
A simple exercise on Poisson processes shows that it suffices to establish $\tilde \rho(E_{t_i}) = \rho(E_{s_i}) + o(1)$.
To that end, we use the assumption \eqref{calK_condition}, the estimate \eqref{tilde_rho_estimate}, and Lemma \ref{lem:rho_tilde_rho} to show that $\tilde \rho(E_{t_i}) = \rho(E_{t_i}) + o(1)$. Moreover, \eqref{calK_condition} implies $\eta \ll \frac{1}{\cal K \sqrt{\log \fra u}}$, so that Lemma \ref{lem:rho_Lipschitz} yields $\rho(E_{t_i}) = \rho(E_{s_i}) + o(1)$. This concludes the proof.
\end{proof}

\begin{proof}[Proof of Theorem \ref{thm:point_process}]
We now have everything needed to prove Theorem \ref{thm:point_process}. We need to verify that the assumptions \eqref{d_assumptions} and the choice \eqref{calK_choice} imply the conditions \eqref{d_assumptions_2} and \eqref{calK_condition}, as well as the assumptions \eqref{eq:ConditionAlphaMaxMin} and \eqref{eq:condition_d_SpecMax} of Corollary \ref{Cor:SpecMax}. It is easiest to consider the cases $d > \sqrt{\log N}$ and $d \leq \sqrt{\log N}$ separately.

Suppose first that $d > \sqrt{\log N}$.
By \eqref{a_2_lower_bound}, we have $\fra u - 2 \gtrsim d^{-\xi}$.
If $\xi$ is small enough, we can choose $\gamma > 0$ small enough and then $\epsilon > 0$ small enough such that $\xi = \frac{1}{3} \gamma - \epsilon$, and \eqref{eq:ConditionAlphaMaxMin}, \eqref{eq:condition_d_SpecMax}, and \eqref{d_assumptions_2} hold. Moreover, by Lemma \ref{lem:u_a} we deduce $\sqrt{\log \fra u} \leq d^\epsilon$, 
and the right-hand side of \eqref{calK_condition} satisfies
\begin{equation*}
d^{\frac{1}{2} - 3 \gamma - \epsilon} \frac{((\fra u - 2)^5 \wedge 1)}{\sqrt{\log \fra u}} \gtrsim d^{\frac{1}{2} - 14 \xi - 11 \epsilon} \geq d^{\frac{1}{2} - 15 \xi}\,,
\end{equation*}
provided that $\epsilon$ is chosen small enough, depending on $\gamma$. This verifies that \eqref{calK_condition} holds  with the choice \eqref{calK_choice}.

For $d \leq \sqrt{\log N}$, we choose $\gamma = \frac{1}{2 \zeta} + \epsilon$ for small enough $\epsilon > 0$ depending on $\zeta$. Using Lemma~\ref{lem:u_a}, we then easily deduce from \eqref{d_assumptions}  that \eqref{eq:ConditionAlphaMaxMin}, \eqref{eq:condition_d_SpecMax}, and \eqref{d_assumptions_2} hold. Moreover, by Lemma \ref{lem:u_a} we deduce $1 \ll \fra u \leq \log N$, and the right-hand side of \eqref{calK_condition} satisfies
\begin{equation*}
d^{\frac{1}{2} - 3 \gamma - \epsilon} \frac{((\fra u - 2)^5 \wedge 1)}{\sqrt{\log \fra u}} \geq (\log \log N)^{-\frac{1}{2}} d^{\frac{1}{2} - \frac{3}{2 \zeta} - 4 \epsilon} \geq d^{\frac{1}{2} - \frac{2}{\zeta} - 4 \epsilon} \geq d^{\frac{1}{2} - \frac{2}{\zeta} - 10 \xi}\,,
\end{equation*}
provided that $\epsilon$ is chosen small enough. This verifies that \eqref{calK_condition} holds  with the choice \eqref{calK_choice}.

The proof now follows by putting \eqref{comparison_Phi_Sigma}, \eqref{Sigma_Psi_comparison}, and Lemma \ref{lem:cont_Psi} together.
\end{proof}

\section{Eigenvector localization -- proof of Theorem~\ref{thm:localisation}} \label{sec:proof_thm_localization} 

An important ingredient in the proof of Theorem \ref{thm:localisation} is the following estimate on the level spacing of the process $\Sigma$, defined in \eqref{def_Z_x}. Moreover, we recall the definition of $\eta$ from \eqref{def_eta}.

\begin{lemma}[Level spacing for $\Sigma$] \label{lem:level_spacing}
Suppose \eqref{d_assumptions_2} and \eqref{calK_condition}, and let $\kappa$ be given by \eqref{def_kappa}.
Under the assumptions of Corollary \ref{Cor:SpecMax}, for any $a \in \R$ we have
\begin{align} \label{level_repulsion}
\P(\exists x \neq y \col Z_x, Z_y \geq -\kappa, \abs{Z_x - Z_y} \leq \eta) &\lesssim \sqrt{\log \fra u} \, \cal K^2 \, \eta + o(1)
\\
\label{stray_repulsion}
\P(\exists x \col Z_x \geq -\kappa, \abs{Z_x - a} \leq \eta) &\lesssim \sqrt{\log \fra u} \, \cal K \, \eta + o(1)\,.
\end{align}
\end{lemma}

\begin{proof}
We only prove \eqref{level_repulsion}; the proof of \eqref{stray_repulsion} is similar, in fact easier.
By Lemma \ref{lem:convergence_max_alpha_x_upper_bound_D_x}, the second estimate of \eqref{alpha_beta_bounds}, and the estimate $\theta(\fra u) \asymp \sqrt{\fra u}$, we find that, with high probability, $Z_x \leq \sqrt{\fra u} \, d^{2 \gamma}$ for all $x$. Thus,
\begin{align*}
\P(\exists x \neq y \col Z_x, Z_y \geq -\kappa, \abs{Z_x - Z_y} \leq \eta)
&\leq
\P(\exists x \neq y \col Z_x, Z_y \in [-\kappa, \sqrt{\fra u} \, d^{2 \gamma}], \abs{Z_x - Z_y} \leq \eta) + o(1)
\\
&\leq \sum_{x \neq y} \P(Z_x, Z_y \in [-\kappa, \sqrt{\fra u} \, d^{2 \gamma}], \abs{Z_x - Z_y} \leq \eta) + o(1)
\\
&= q_{\Sigma, 2} \pb{\hb{(s,t) \col s,t \in [-\kappa, \sqrt{\fra u} \, d^{2 \gamma}], \abs{s - t} \leq \eta}} + o(1)\,,
\end{align*}
where we used the two-point correlation measure $q_{\Sigma, 2}$ of $\Sigma$ defined in \eqref{def_correlation_measure} below. By covering the set in the argument of $q_{\Sigma,2}$ with squares of the form $[u - \eta, u + \eta]^2$, $u \in \eta \Z$, we find
\begin{align*}
&\mspace{-40mu}\P(\exists x \neq y \col Z_x, Z_y \geq -\kappa, \abs{Z_x - Z_y} \leq \eta)
\\
&\leq \sum_{u \in \eta \Z} \ind{u \in [-\kappa, \sqrt{\fra u} \, d^{2 \gamma}]} q_{\Sigma,2}([u - \eta, u + \eta]^2) + o(1)
\\
&\leq 2 \sum_{u \in \eta \Z} \ind{u \in [-\kappa, \sqrt{\fra u} \, d^{2 \gamma}]} \pbb{\rho([u - \eta, u + \eta])^2 + O(\ee^{-c d^{2 \gamma}} + N^{-1/4})} + o(1)\,,
\end{align*}
where in the second step we used Lemmas \ref{lem:correlation_factorized} and \ref{lem:rho_tilde_rho}. By Lemma \ref{lem:kappa_condition}, the sum over $u$ has $O(\sqrt{\fra u} \, d^{2 \gamma} / \eta)$ terms. Moreover, by Lemma \ref{lem:rho_Lipschitz}, Corollary \ref{cor:Lipschitz}, and the definition \eqref{def_eta} of $\eta$, we find that $\rho([-\kappa - \eta, \infty)) \leq 2 \cal K$,  and the density of $\rho$ in the interval $[-\kappa -\eta, \infty)$ is bounded by $C \sqrt{\log \fra u} \, \cal K$. Using the assumption $d^{2 \gamma} \gg \log \log N$ and the estimate $\fra u \leq \log N$ from Lemma \ref{lem:u_a}, we therefore conclude that
\begin{equation*}
\P(\exists x \neq y \col Z_x, Z_y \geq -\kappa, \abs{Z_x - Z_y} \leq \eta) \lesssim
\sqrt{\log \fra u} \, \cal K \, \eta \, \sum_{u \in \eta \Z} \ind{u \in [-\kappa, \sqrt{\fra u} \, d^{2 \gamma}]} \rho([u - \eta, u + \eta]) + o(1)\,,
\end{equation*}
from which the claim follows.
\end{proof}

From now on we assume that, instead of \eqref{calK_condition}, $\cal K$ satisfies the stronger condition
\begin{equation} \label{calK_condition2}
C \leq \cal K^2 \leq d^{\frac{1}{2} - 3 \gamma - \epsilon} \frac{((\fra u - 2)^5 \wedge 1)}{\sqrt{\log \fra u}}\,,
\end{equation}
obtained by replacing $\cal K$ with $\cal K^2$ in \eqref{calK_condition}. Under \eqref{calK_condition2}, the right-hand sides of \eqref{level_repulsion} and \eqref{stray_repulsion} are $o(1)$.

We conclude that under the condition \eqref{calK_condition2}, with high probability, all points of the process $\Sigma$ in the  region $[-\kappa, \infty)$ are separated by at least $\eta$. By invoking Lemma \ref{lem:Sigma_Phi} and \eqref{Sigma_tilde_Sigma} with a smaller $\epsilon > 0$, we conclude the following result.

\begin{lemma} \label{lem:pert1}
With high probability, each interval of the form
\begin{equation} \label{localization_intervals}
[Z_x - \eta/4, Z_x + \eta/4] \,, \qquad Z_x \geq -\kappa
\end{equation}
contains exactly one point of $\tilde \Sigma$ (see \eqref{def_SIgma_tilde}) and exactly one point of $\Phi - \delta_{\tilde \nu}$ (see \eqref{def_tilde_nu}). Moreover, the complement of the intervals \eqref{localization_intervals} in the region $[-\kappa,\infty)$ contains no point of $\tilde \Sigma$ and no point of $\Phi - \delta_{\tilde \nu}$.
\end{lemma}

We also have to ensure that with high probability the (rescaled) stray eigenvalue $\tilde \nu$ is outside of the intervals \eqref{localization_intervals}. To that end, we observe that if $d \leq (\log N)^{1/2 - \gamma}$ then $\tilde \nu \notin [-\kappa - \eta, \infty)$.  This follows from the definitions of $\tau(\fra u)$ and $\sigma(\fra u)$ from \eqref{def_tau_etc}, Lemma \ref{lem:u_a}, and Lemma \ref{lem:kappa_condition}.
On the other hand, if $d \geq (\log N)^{1/2 - \gamma}$ then by Corollary \ref{Cor:SpecMax}, with $a \deq d \tau(\fra u)(d^{1/2} + d^{-1/2} + d^{-3/2} - \sigma(\fra u))$, we have with high probability $\abs{\tilde \nu - a} \lesssim d \tau(\fra u) d^{-5/2} \ll \eta$,
where we also used Lemmas  \ref{lem:u_a} and \ref{lem:kappa_condition}. By Lemma \ref{lem:level_spacing}, we therefore conclude the following result.

\begin{lemma} \label{lem:pert2}
With high probability, $\tilde \nu$ is not contained in any of the intervals \eqref{localization_intervals}.
\end{lemma}

\begin{proof}[Proof of Theorem~\ref{thm:localisation}] 
We may now conclude the proof of Theorem \ref{thm:localisation} using Proposition \ref{prop:blockMatrix} and the perturbation result from Lemma \ref{lem:perturbationEV}. We work on the intersection of the high-probability events from Proposition \ref{prop:blockMatrix} and Lemmas \ref{lem:level_spacing} and \ref{lem:pert1}. Let $\lambda \neq \nu$ be an eigenvalue of $H$ satisfying $\tilde \lambda \deq d \tau(\fra u)(\lambda - \sigma(\fra u)) \geq -\kappa$. Let $x$ be the unique vertex such that $\tilde \lambda \in [Z_x - \eta/4, Z_x + \eta/4]$ (see Lemma \ref{lem:pert1}). Let $\f v \deq \f w_+(x)$, defined in Proposition \ref{pro:fine_rigidity}, and let $\lambda' \deq \Lambda_{\fra d}(\alpha_x, \beta_x) + \epsilon_{x,+}$ in the notation of Proposition \ref{prop:blockMatrix}.

Then, recalling the construction of the orthogonal matrix $U$ in \eqref{eq:LargeBlock} from the proof of Proposition \ref{prop:blockMatrix} in Section~\ref{sec:proof_prop_block_matrix} we find $\norm{(H - \lambda') \f v} \leq \norm{E_{\cal W}} \lesssim (d \fra u)^{-10}$,
by Proposition \ref{prop:blockMatrix}. Moreover, by Lemmas \ref{lem:pert1} and \ref{lem:pert2}, $\lambda \in [\lambda' - \Delta, \lambda' + \Delta]$ with $\Delta \deq d \tau(\fra u) \eta /4$, and there is no other eigenvalue of $H$ in the same interval. Denoting by $\f w$ the eigenvector associated with $\lambda$, we deduce from Lemma \ref{lem:perturbationEV} that $\norm{\f w - \f v} \lesssim (d \fra u)^{-8}$.

What remains is to bound $\norm{\f v|_{B_i(x)^c}}$. 
Since $\supp \f v \subset B_{r_{\cal W}}(x)$ by Proposition~\ref{pro:fine_rigidity}, 
we conclude from \eqref{eq:decay_eigenvector_in_ball} with $r = r_{\cal W} - 1$ that
\[ 
\norm{\f v|_{B_i(x)^c}} \lesssim \frac{\mu^2}{(\mu - \norm{Q (H|_{B_{r_{\cal W}}(x)})Q})^2} \biggl( \frac{\norm{Q (H|_{B_{r_{\cal W}}(x)})Q}}{\mu} \biggr)^{i} 
\lesssim  \frac{q^i}{(1 - q)^2}, \qquad q = \frac{2 + O(d^{-1/3})}{\sigma(\fra u)}\,, 
\] 
where $Q$ is the orthogonal projection onto the coordinates in $B_{r_{\cal W}}(x)$.  
Here, in the second step, we used that 
$\norm{QH^{(x,r)}Q} \leq 2 ( 1 + d^{\gamma-1/2})^{1/2}\leq 2 + O(d^{-1/3})$ by \eqref{eq:norm_QHQ} with $\delta = d^{2\gamma - 1}$ 
and $\gamma \leq 1/6$, 
 and $\mu = \sigma(\fra u) + O(d^{-1/3})$. 
For the proof of the latter expansion, we note that 
$\mu = \Lambda_{\fra d}(\alpha_x,\beta_x) + O(d^{-1/3})$ if $\xi$ is sufficiently small 
by  \eqref{eq:mu_Lambda_fine_rigidity}, \eqref{eq:choice_r} and \eqref{a_2_lower_bound}. 
Moreover, 
$\Lambda_{\fra d}(\alpha_x,\beta_x) = \sigma(\fra u) + O(d^{-2/3}/\sqrt{\fra u})$ 
follows from using \eqref{alpha_beta_bounds}, which is applicable due to $Z_x \geq - \kappa$, 
 in \eqref{eq:expansion_Lambda_d_sigma_u} and $\gamma \leq 1/6$.  
\end{proof}

\begin{proof}[Proof of Remark~\ref{rem:localization}] 

Using the notation of Proposition~\ref{pro:fine_rigidity}, we choose $r = r_{\cal W}$ and $\f w(x) \deq \f w_+(x)$. 
From the proof of Theorem~\ref{thm:localisation}, we recall that $\norm{\f w - \f w(x)} \leq (d\fra u)^{-8}$, 
which proves the first part of Remark~\ref{rem:localization}.

Let $\f v(x)$ be as in Proposition~\ref{pro:H_minus_Lambda_v_norm_quadratic_form_intermediate}. 
With $\delta = d^{2\gamma - 1}$, we use \eqref{eq:estimate_norm_approximate} and \eqref{eq:lower_bound_spectral_gap} 
to check the conditions of Lemma~\ref{lem:perturbationEV} and arrive at 
$\norm{\f w(x) - \f v(x)} \lesssim \frac{\fra u}{(\fra u - 2)^2} \frac{\sqrt{\log d}}{\sqrt{d}}
 \lesssim \frac{d^{-1/2 + 3\xi}}{\fra u} $ by \eqref{a_2_lower_bound}.
Moreover, the precise choice of $\f v(x)$ in the proof of Proposition~\ref{pro:H_minus_Lambda_v_norm_quadratic_form_intermediate} 
and \eqref{Zd_vect_components} yield 
\[ \f v(x) = \sum_{i = 0}^r u_i(x) \f s_i(x) + O\bigg( \frac{d^{-2/3}}{\sqrt{\fra u}} + \frac{d^{\gamma}}{d \sqrt{\fra u}} + \frac{1}{d} \bigg)\,, \]  
where $(u_i(x))_i$ are as in Remark~\ref{rem:localization} and 
we used that  $\Lambda_{\fra d}(\alpha_x,\beta_x) = \sigma(\fra u) + O(d^{-2/3}/\sqrt{\fra u})$ 
(see the proof of Theorem~\ref{thm:localisation}), 
$\beta_x = 1 +  O(d^{\gamma - 1} \fra u^{-1/2})$ by Corollary~\ref{cor:concentration_beta_x} with $\delta = d^{2\gamma - 1}$ 
and $\fra d = 1 + d^{-1}$ by definition. 
As $\gamma \leq 1/6$, this completes the proof of Remark~\ref{rem:localization}. 
\end{proof}

\section{Local graph structure around large degree vertices} \label{sec:graph_structure}

In this section we prove the properties of the local graph structure around vertices of large degree, as stated in Propositions \ref{pro:prob_estimates_1} and \ref{pro:graphProperty}.

\subsection{Proof of Proposition \ref{pro:prob_estimates_1}}
\label{sec:proof_prob_estimates_1} 

We shall now prove Proposition~\ref{pro:prob_estimates_1} and use that, by Bennett's inequality (see Lemma \ref{lem:Bennett} below), 
\begin{equation} \label{eq:Bennett_for_degree} 
\P \big( \abs{D_x - d} \geq a d \big) \leq 2 \exp ( -d h(a))
\end{equation} 
for any $x \in [N]$ and 
for any $a >0$, where $h$ is defined as in \eqref{eq:def_h} 
We recall from \eqref{S_1_estimate} that if $d \leq 3 \log N$ then for each $\nu>0$ there is $C >0$ such that
\begin{equation} \label{eq:D_x_bounded_very_high_probability} 
 \P \big( D_x \geq C \log N \big) \leq C N^{-\nu}\,. 
\end{equation}

\begin{proof}[Proof of Proposition~\ref{pro:prob_estimates_1}] 
Before going into the proof, we first establish that 
\begin{equation} \label{eq:upper_bound_alpha_d_on_D_x} 
 \P ( D_x \geq \alpha d)\leq 2\ee^{-d h(\alpha- 1)} \leq \frac{2}{N}\ee^{(\am -\alpha) d \log \am} 
\end{equation}
for any $\alpha \geq 1$. 
The bound \eqref{eq:upper_bound_alpha_d_on_D_x} follows from \eqref{eq:Bennett_for_degree} and, for the second step, from 
\begin{equation} \label{eq:lower_bound_exponent_d_h} 
 d h(\alpha-1) \geq d h(\am - 1) + (\alpha - \am) d \log \am = \log N +  (\alpha - \am) d \log \am 
\end{equation}
for all $\alpha \geq 1$, 
which is a consequence of the convexity of $h$, $h'(\alpha-1) =\log \alpha$ and the 
definition of $\am$ in \eqref{eq:def_alpha_max}. 

We now show separately that each of the five events defined by parts \ref{item:balls_disjoint} -- 
\ref{item:Z_i_estimate} occurs with high probability.
We denote the event from \ref{item:balls_disjoint} by $\Xi_1$ and, to estimate its probability, remark that 
\[ \Xi_1^c = \big\{ \exists x \neq y \in [N] \text{ such that } \alpha_x, \alpha_y \geq \tau,~ y \in B_{2r}(x) \big\}  \] 
with $\tau = \am - c_*\delta/\log \am$ for some sufficiently small $c_*>0$. 
We follow the proof of \cite[Lemma~7.3]{ADK19}, note that $\Xi_1^c = \Xi^{(1)}$ in the notation of \cite{ADK19}, choose $k =1$, observe $a_0=a_1 = (\tau - 1)d - 1$ 
and end up with 
\begin{equation} \label{eq:prob_Xi_1_bound} 
 \P (\Xi_1^c) \leq N \frac{d^{2r+2} - 1}{d - 1} \bigg( C \exp \bigg( -2 d h\bigg(\tau - 1 - \frac{n+1}{d} \bigg) \bigg) + 2 \bigg( \frac{2d}{N} \bigg)^n \bigg)\,. 
\end{equation}
Therefore, by choosing $n =10$ and $\tau$ as above in \eqref{eq:prob_Xi_1_bound}, we obtain 
$ \P(\Xi_1^{c}) = N^{-\eps}$ with $\eps >0$ small enough from \eqref{eq:lower_bound_exponent_d_h}
with $\alpha = \tau - (n+1)/d$ since $r\leq c_* d/\log\log N$ for some sufficiently small $c_*>0$ 
and $d \leq 3\log N$. 
Note that $\alpha \geq 1$ since $\tau \geq 1 + c$ by \eqref{eq:lower_bound_degree_cal_V}.

For any $x \in [N]$, we introduce the $x$-dependent events 
\[\begin{aligned} 
 \Xi_2(x) & \deq \{ \mathbb{G}|_{B_r(x)} \text{ is a tree } \}, \\ 
 \Xi_3(x) & \deq \Biggl\{   
 \absbb{\frac{\abs{S_{i+1}(x)}}{d\abs{S_i(x)}} -1} \lesssim  \bigg(\frac{\delta}{\abs{S_i(x)}} \bigg)^{1/2}\quad \text{ and } \quad \absbb{\frac{\abs{S_i(x)}}{D_x d^{i-1}} -1 } \lesssim \bigg(\frac{\delta}{D_x } \bigg)^{1/2} \text{ for all } i \in [r]\Biggr\}, 
\\ 
\Xi_4(x) & \deq \bigl\{  \abs{D_y -d} \leq \delta^{1/2} d \text{ for all } y \in B_r(x) \setminus \{ x\} \bigr\}, \\ 
 \Xi_5(x) & \deq \biggl\{ \sum_{y \in S_i(x)} (N_y(x) - d)^2 \leq \abs{S_i(x)} d ( \log d + \delta^{1/2}) \bigg(1 + \frac{d \delta}{\abs{S_i}} \bigg) \text{ for all } i \in [r] 
 \biggr \}\,. 
\end{aligned}  \]

We shall show below that 
\begin{subequations} 
\begin{align}
 \label{eq:estimates_Xi_j} 
 \P \bigl( \Xi_j(x)^c \bigm| S_1(x) \bigr)\ind{x \in \cal V_\delta}\ind{D_x \leq C \log N} 
& \lesssim  \ee^{-c d \delta} \big( d^{r-1} \log N + r^2\big) + \frac{1}{N}(C\log N)^{2r + 3} \\ 
\P \bigl( \Xi_5(x)^c \bigm| S_1(x) \bigr)\ind{x \in \cal V_\delta}\ind{D_x \leq C \log N} & \lesssim\ee^{-c d \delta}\big( d^{r-1} \log N + r^2 + r \log \log N) + \frac{1}{N}(C\log N)^{2r + 3} 
\label{eq:estimate_Xi_5} 
\end{align}
\end{subequations} 
for all $j = 2, 3,4$ and for any fixed $x \in [N]$. From \eqref{eq:estimates_Xi_j} and \eqref{eq:estimate_Xi_5} we conclude  
\begin{align} 
\P \bigl( \exists x \in \cal V_\delta \colon \Xi_j(x)^c \bigr) 
& \leq \sum_{x \in [N]} \E \Bigr[ \P \bigl(\Xi_j(x)^c \bigm| S_1(x) \bigr) \ind{x \in \cal V_\delta}\ind{D_x \leq C \log N} 
\Bigr] + \sum_{x \in [N]} \P ( D_x > C \log N) \notag \\ \label{eq:combination_Xi_j_cal_V} 
 & \leq \ee^{c_* d \delta} \biggl( \frac{1}{N} \ee^{C (r + 3) \log \log N} + C \ee^{\log \log N} \ee^{r \log d} \ee^{-c d \delta} \biggr) + N^{-2}\,. 
\end{align} 
Here, the second term was estimated via \eqref{eq:D_x_bounded_very_high_probability}. For the first term, we used \eqref{eq:estimates_Xi_j} together with $D_x \leq C \log N$. 
Then, we used $\P(x \in \cal V_\delta)\leq \frac{2}{N} \ee^{c_* d \delta}$, a consequence of 
\eqref{eq:upper_bound_alpha_d_on_D_x} with $\alpha = \am - c_* \delta/(\log \am)$ (see the definition of $\cal V_\delta$ in \eqref{eq:def_cal_V}).  
Note that $\alpha \geq 1$ by \eqref{eq:lower_bound_degree_cal_V}.

Proposition~\ref{pro:prob_estimates_1} follows from \eqref{eq:combination_Xi_j_cal_V} 
provided $c_*$ is chosen sufficiently small 
and $K$ in \eqref{eq:condition_d_delta_graph} is chosen sufficiently large 
since $r \log \log N \leq c_* d \leq c_* 3 \log N$ and $r \log d \leq c_* \delta d$. 

Therefore, what remains is proving \eqref{eq:estimates_Xi_j} and \eqref{eq:estimate_Xi_5}. 
To that end, we fix $x \in [N]$. 
For controlling $\P ( \Xi_2(x)^c | S_1(x))$, we introduce the nonnegative, integer-valued function 
$t_r(x) \deq \abs{E(\mathbb{G}|_{B_r(x)})} - \abs{B_r(x)} + 1$ and remark that $\mathbb{G}|_{B_r(x)}$ is a tree if and only if $t_r(x) =0$. 
Hence, by \cite[Lemma~5.5]{ADK19}, we have 
\[ \P \bigl(\Xi_2(x)^c \bigm| S_1(x) \bigr) = \P\bigl( t_r(x) \geq 1 \bigm| S_1(x) \bigr) \leq \frac{1}{N} (C (d + D_x))^{2r + 1}(2r)^2\,. \] 
This proves \eqref{eq:estimates_Xi_j} for $j = 2$ due to the conditions $D_x \lesssim \log N$, $d \leq 3 \log N$ and $r \leq c\log N/\log \log N$ for a small enough $c>0$.

To prove \eqref{eq:estimates_Xi_j} for $j = 3$, we shall use the auxiliary bound 
\begin{equation} \label{eq:concentration_S_i_auxiliary_estimate} 
\P \bigg( \absbb{ \frac{\abs{S_{i+1}(x)}}{d \abs{S_i(x)}} -1} \geq  \eps + C \cal E_i \biggm| B_i(x) \bigg) \leq 2 \exp\big( - c d \abs{S_i(x)} \eps^2\big)\,, 
\qquad \cal E_i \deq \frac{d \abs{S_i(x)}}{N} + \frac{1}{\sqrt{N}}\,, 
\end{equation}
which holds for any $\eps \in [0,1]$ if $\abs{B_i(x)}\leq \sqrt N$. 
The estimate \eqref{eq:concentration_S_i_auxiliary_estimate} is proved in \cite[eq.~(5.17)]{ADK19}.  

Choosing $\eps = \delta^{1/2}\abs{S_i(x)}^{-1/2}$ in \eqref{eq:concentration_S_i_auxiliary_estimate}, following the induction arguments 
for the proofs of \cite[eq.~(5.12a)]{ADK19} and \cite[eq.~(5.12b)]{ADK19} 
and performing a union bound over $i = 1, \ldots, r$ yield 
\begin{equation} \label{eq:P_Xi_3} 
\P(\Xi_3(x)^c|S_1(x))\ind{x \in \cal V_\delta} \lesssim r \exp\big(-cd \delta\big)\,. 
\end{equation}
Here, we used that $\delta/D_x \leq 1/C$ for any constant $C>0$ since $D_x \gtrsim d\am$ for $x \in \cal V_\delta$ 
by \eqref{eq:lower_bound_degree_cal_V}, $d \delta \leq C \log N$ and $d \geq K \log \log N$. 
This shows \eqref{eq:estimates_Xi_j} for $j = 3$.

 To estimate $\P ( \Xi_4(x)^c | S_1(x))$, 
we use 
\eqref{eq:Bennett_for_degree} as well as $h(a) \geq c a^2 \wedge a$ for all $a >0$ and some $c>0$ to conclude 
\[ \P \big( \Xi_{4,k}(x) \big| B_k(x) \big) \leq 2 \abs{S_k(x)} \ee^{-c d \delta }, 
\qquad \Xi_{4,k}(x) \deq \bigl\{ \exists y \in S_k(x) \colon \abs{D_y -d} \geq \delta^{1/2}d\bigr\}\,, 
\] 
if $\abs{B_k(x)}\leq \sqrt{N}$. Hence, since $\ind{\abs{S_k} \lesssim D_x d^{k-1}}$ is $B_k$-measurable,  
$\{\abs{S_k} \lesssim D_x d^{k-1}\} \supset \Xi_3(x)$ and 
\eqref{eq:P_Xi_3}, we obtain 
\begin{align} 
 \P \big( \Xi_4(x)^c \big|S_1(x) \big) \ind{x \in \cal V_\delta} & \leq \sum_{k=1}^r \E \bigl[ \P \bigl(\Xi_{4,k}(x) \big| B_k(x) \bigr)\ind{\abs{S_k} \lesssim D_x d^{k-1}} \bigr] + r \P\bigl(\Xi_3(x)^c\big|S_1(x)\bigr)\notag \\ \label{eq:bound_deviation_D_y_from_d_1} 
& \lesssim D_x d^{r-1} \ee^{-cd\delta}  + r^2 \ee^{-cd \delta}\,. 
\end{align} 
This completes the proof of \eqref{eq:estimates_Xi_j} in the remaining case $j = 4$.

We now estimate $\P\big( \Xi_5(x)^c \big| S_1(x)\big)$. For brevity, we write $N_y$ instead of $N_y(x)$ throughout this argument. 
We have 
\begin{equation} \label{eq:meanquadratic_first_step} 
 \sum_{y \in S_i} \big(N_y - d\big)^2 \leq 2 \sum_{y \in S_i} \Big(\big(N_y - \E[N_y|B_{i}]\big)^2 + \big(\E[N_y|B_{i}] - d\big)^2 \Big) \lesssim \sum_{y \in S_i} E_y^2 + \abs{S_i}\,. 
\end{equation}
Here, we used that $\abs{d- \E[N_y|B_{i}]} \leq d \abs{B_{i}}/N \leq C$ on $\Xi_3(x)$ and set $E_y \deq N_y - \E[N_y|B_{i}]$ to simplify the notation in the following. 

Using the lower bound on $\delta$ in \eqref{eq:condition_d_delta_graph}, we see that 
$E_y^2 \lesssim \delta d^2$ on $\Xi_2(x) \cap \Xi_4(x)$ as $N_y = D_y - 1$ on $\Xi_2(x)$. 
Hence, on $\Xi_2(x) \cap \Xi_4(x)$, we get 
\[ \sum_{y \in S_i} E_y^2 \leq d \abs{S_i} + \sum_{k = \floor{-\log d}}^{\ceil{\log \delta}} d^2 \ee^{k+1} \abs{\cal N_{i,k}} \] 
where we introduced $\cal N_{i,k} \deq \{ y \in S_i \colon d^2 \ee^k < E_y^2 \leq d^2 \ee^{k+1} \}$. 

We use the bound in \cite[eq.~(5.28)]{ADK19}, apply $\begin{pmatrix} \abs{S_i} \\ \ell \end{pmatrix} \leq \ee^{\abs{S_i}}$ to it and obtain that there is a constant $c>0$ such that 
\[ \P \big( \abs{\{y \in S_i \colon E_y^2 > s^2 d^2 \}} \geq \ell \big| B_i \big) \leq \exp\big ( \abs{S_i} - c d \ell (s \wedge s^2 ) \big) \] 
for any $s >0$. 
Therefore, setting $\ell_{i,k} \defeq \frac{C}{d} (\abs{S_i} + d \delta) ( \ee^{-k/2} \vee \ee^{-k})$ yields $\P \big( \abs{\cal N_{i,k}} \geq \ell_{i,k} \big| S_1(x) \big) \leq \ee^{- c d \delta}$.  
We introduce the event $\wt{\Xi}_5(x) \defeq \{ \abs{\cal N_{i,k}} \leq \ell_{i,k} \text{ for all } i \in [r], k \in \Z \text{ with } \abs{k} \lesssim \log \log N \}$ and conclude 
\begin{equation} \label{eq:prob_wt_Xi_5} 
 \P (\wt{\Xi}_5(x)^c |S_1(x)) \lesssim r \log \log N \ee^{-cd\delta}\,. 
\end{equation}
Thus, on $\wt{\Xi}_5(x)\cap \Xi_2(x) \cap \Xi_4(x)$, we have that 
\[ \sum_{y \in S_i} E_y^2 \lesssim d \abs{S_i} + d (\abs{S_i} + d \delta) \bigg( \sum_{k=0}^{\ceil{\log \delta}} \ee^{k/2} + \log d \bigg) \lesssim d \abs{S_i} (\log d + \delta^{1/2}) \bigg( 1 + \frac{d \delta}{\abs{S_i}} \bigg) \] 
for all $i \in [r]$. Hence, owing to \eqref{eq:meanquadratic_first_step}, we conclude $\Xi_5(x) \supset \wt{\Xi}_5(x) \cap \Xi_2(x) \cap \Xi_3(x) \cap \Xi_4(x)$ 
and \eqref{eq:prob_wt_Xi_5} as well as \eqref{eq:estimates_Xi_j} for $j = 2$, 3, 4 yields \eqref{eq:estimate_Xi_5}, 
completing the proof of Proposition~\ref{pro:prob_estimates_1}. 
\end{proof}

\subsection{Proof of Proposition \ref{pro:graphProperty}} \label{sec:proof_graphProperty} 
We shall show that each subevent corresponding to \ref{item:disjoint_balls} -- \ref{item:MeanQuadratic} of Proposition~\ref{pro:graphProperty} holds with high probability.

We have $\cal W = \cal V_\delta$ if we choose $\delta = d^{2\gamma - 1}$ in the definition of $\cal V_\delta$ in \eqref{eq:def_cal_V}. 
With this choice of $\delta$, the conditions in \eqref{eq:condition_d_delta_graph} are met since $K \log \log N \leq d^{2\gamma} \leq d \leq 3 \log N$ as $\gamma \leq 1/2$. 
Therefore, the events \ref{item:disjoint_balls} -- \ref{item:concentration_degrees} hold with probability $1-o(1)$ by choosing $c_*$ sufficiently small, 
applying Proposition~\ref{pro:prob_estimates_1} 
and noting that 
\eqref{eq:upper_bound_r} with $\delta = d^{2\gamma-1}$ in Proposition~\ref{pro:prob_estimates_1} is satisfied 
due to \eqref{eq:upper_bound_r_fine_rigidity}.

To simplify the notation, we fix $x \in [N]$ and 
introduce the $x$-dependent events 
\[\begin{aligned} 
\Omega_2(x) & \deq \bigl\{ \mathbb{G}|_{B_r(x)} \text{ is a tree} \bigr\}, \\ 
{\Omega}_3(x) & \deq \biggl\{ 
\absbb{ \frac{\abs{S_{k+1}(x)}}{d \abs{S_k(x)}} - 1} \lesssim \frac{d^{-1/2 + \gamma}}{\abs{S_k(x)}^{1/2}} 
\text{ and } \abs{S_k(x)} = D_x d^{k-1} \bigl(1 + O \bigl(d^{-1/2 + \gamma} D_x^{-1/2}\bigr) \bigr) 
\text{ for all } k \in [r] \biggr\}, \\ 
\Omega_4(x) & \deq \bigl\{ \abs{D_y -d} \leq d^{1/2}d^{\gamma} \text{ and } \abs{N_y(x) - d} \leq d^{1/2} d^{\gamma} +1 \text{ for all } y \in B_r(x)\setminus \{x\} \bigr\}, 
\\ \Omega_5(x) & \deq \biggl\{ 
\absbb{\sum_{y\in S_k(z)}(N_y(x) -d)} \lesssim d^{(k+1)/2 + \gamma} 
\text{ for all } 
z \in B_r(x) \setminus \{x \}, ~k \in [r] \text{ such that } S_k(z) \subset B_r(x) 
\biggl\},  
\\ \Omega_6(x) & \deq 
 \biggl\{ 
\sum_{y \in S_k(x)} (N_y(x) - d)^2 = D_x d^{k} \big(1 + O(D_x^{-1/2}d^{\gamma})\big)  \text{ for all } 
k \in [r]
\biggr \}\,. 
\end{aligned} \] 
(These events should be compared to $\Xi_2(x)$, $\Xi_3(x)$ and $\Xi_4(x)$ from the proof of Proposition~\ref{pro:prob_estimates_1}.) 

We shall show below that 
\begin{subequations} \label{eq:estimates_P_Omega_x_all} 
\begin{align} 
\label{eq:estimates_P_Omega_x} 
\P\big( \Omega_j(x)^c \big| S_1(x) \big)\ind{x \in \cal W} \ind{D_x \lesssim \log N} & 
\lesssim \frac{1}{N}(C\log N)^{2r + 3} + \ee^{-c d^{2\gamma}} \big( d^{r-1} \log N + r^2\big) \\ 
\P \big(\Omega_5(x)^c  \big| S_1(x) \big) \ind{x \in \cal W} \ind{D_x \lesssim \log N} & 
 \lesssim 
\frac{r}{N}(C\log N)^{2r + 3} + r \ee^{-c d^{2\gamma}} \big( d^{r-1} \log N + r^2\big) 
\label{eq:estimates_P_Omega_5_x} \\ 
\P \big(\Omega_6(x)^c  \big| S_1(x) \big) \ind{x \in \cal W} \ind{D_x \lesssim \log N} 
& \lesssim r \ee^{C \ee^{-c d^{2\gamma}} D_x d^{r}} \ee^{-c d^{2\gamma}} 
+ \frac{r}{N}(C\log N)^{2r + 3} + r \ee^{-c d^{2\gamma}} \big( d^{r-1} \log N + r^2\big) 
\label{eq:estimates_P_Omega_6_x} 
\end{align} 
\end{subequations} 
for $j = 2, 3 , 4$. 

Given \eqref{eq:estimates_P_Omega_x_all}, arguing as in \eqref{eq:combination_Xi_j_cal_V} completes the proof of Proposition~\ref{pro:graphProperty} 
since, for sufficiently small $c_*>0$ and sufficiently large $K>0$, we have 
$N \P (x \in \cal W) \leq 2\ee^{c_* d^{2\gamma}}$ by \eqref{eq:upper_bound_alpha_d_on_D_x} 
with $\alpha = \am - c_* d^{2\gamma -1}/\log \am$ 
and  
\[ \ee^{c_*d^{2\gamma}} r  D_x d^r \ee^{-c d^{2\gamma}} \ind{D_x \lesssim \log N} \lesssim \ee^{c_*d^{2\gamma}} r \ee^{-cd^{2\gamma}} d^r \log N 
= o(1)
, \qquad \qquad 
r \ee^{c_* d^{2\gamma}} \ee^{-c d^{2\gamma}} = o(1)\,. 
\] 
Here, for the first expression, we used $r \leq c_* d^{2\gamma}/\log d$ and $d^{2\gamma} \geq K \log \log N\to\infty$. 
For the second expression, we chose $c_*$ sufficiently small and used $r \leq c_* d/(\log \log N)$ 
as well as  $d^{2\gamma} \geq K \log \log N \to \infty$. 

Therefore, what remains is proving \eqref{eq:estimates_P_Omega_x_all}. 
The estimate \eqref{eq:estimates_P_Omega_x} follows from \eqref{eq:estimates_Xi_j} in the proof of Proposition~\ref{pro:prob_estimates_1} 
with $\delta = d^{2\gamma - 1}$ and $j=2$, $3$ or 4, respectively. 
Here, we used that $D_y = N_y(x) +1$ on $\Omega_2(x)$. 

We now show \eqref{eq:estimates_P_Omega_5_x}.
We start by introducing the vertex sets $S^{(k)}_z(x) = S_k(z) \cap S_{k+ d(z,x)}(x)$ if $x$ and $z$ are connected 
and the auxiliary event 
\begin{multline*} 
\wt{\Omega}_3(x) \deq \biggl\{ 
\absbb{ \frac{\abs{S^{(k+1)}_z(x)}}{d \abs{S^{(k)}_z(x)}} - 1} \lesssim \frac{d^{-1/2 + \gamma}}{\abs{S^{(k)}_z(x)}^{1/2}} 
\text{ and } \abs{S^{(k)}_z(x)} = D_z d^{k-1} \bigl(1 + O \bigl(d^{-1+ \gamma}\bigr) \bigr) \\ 
\text{ for all } z \in B_r(x),~k \in [r] \text{ such that }S_{k+1}(z) \subset B_r(x) \biggr\}\,. 
\end{multline*} 
Note that $\abs{S^{(k)}_z(x)} = N^{(k)}_z (x)$ with the definition of $N^{(k)}_z(x)$ in \eqref{eq:def_N_k_z_x}. 

Below, we shall show that 
\begin{equation} \label{eq:estimate_P_wt_Omega_3_x} 
\P(\wt{\Omega}_3(x)^c|S_1(x)) \lesssim 
\frac{r}{N}(C\log N)^{2r + 3} + r \ee^{-c d^{2\gamma}} \big( d^{r-1} \log N + r^2\big)\,. 
\end{equation}
The estimate \eqref{eq:estimates_P_Omega_5_x} follows directly from  
\eqref{eq:estimates_P_Omega_x} for $j = 2$, 4, \eqref{eq:estimate_P_wt_Omega_3_x}, 
and the inclusion 
\begin{equation} \label{eq:Omega_5_supset_Omega_2_3_4}
\Omega_5(x)\supset \Omega_2(x)\cap \wt{\Omega}_3(x) \cap {\Omega}_4(x)\,. 
\end{equation}
For the proof of \eqref{eq:Omega_5_supset_Omega_2_3_4}, let $z \in B_r(x)$ and $k \in [r]$ such that $S_k(z) \subset B_r(x)$. 
Then we have  
\begin{equation} \label{eq:proof_prob_estimates_tree_formula} 
\sum_{y \in S^{(k)}_z(x)} (N_y(x) - d) =|S^{(k+1)}_z(x)|-d|S^{(k)}_z(x)|  
\end{equation}
on $\Omega_2(x)$ as $\mathbb{G}|_{B_r(x)}$ is a tree on $\Omega_2(x)$. 
Therefore, on $\Omega_2(x) \cap \wt{\Omega}_3(x)\cap \Omega_4(x)$, the bounds  
\[
\absbb{ \sum_{y\in S^{(k)}_z(x)}(N_y(x)-d)} 
 \leq d \abs{S^{(k)}_z(x)} \absbb{\frac{\abs{S^{(k+1)}_z(x)}}{d \abs{S^{(k)}_z(x)}} -1} 
\lesssim d^{1/2 + \gamma} \abs{S^{(k)}_z(x)}^{1/2} 
\lesssim d^{(k+1)/2 + \gamma} 
\]
hold as $D_z = N_z(x) +1\in [d/2,2d]$ on $\Omega_2(x) \cap {\Omega}_4(x)$ and $\gamma < 1/2$. 
Hence, we have proved \eqref{eq:Omega_5_supset_Omega_2_3_4} and, thus, 
\eqref{eq:estimates_P_Omega_5_x} assuming \eqref{eq:estimate_P_wt_Omega_3_x}.

To prove \eqref{eq:estimate_P_wt_Omega_3_x}, for fixed $x\neq z \in [N]$, we shall use the auxiliary bound 
\begin{equation} \label{eq:concentration_S_i_auxiliary_estimate_2} 
\P \bigg( \absbb{ \frac{\abs{S^{(k+1)}_z(x)}}{d \abs{S^{(k)}_z(x)}} -1} \geq  \eps + C \cal E_k \biggm| B_{k+d(x,z)}(x) \bigg) 
\leq 2 \exp\big( - c d \abs{S^{(k)}_z(x)} \eps^2\big), 
\qquad \cal E_k \deq \frac{d \abs{S^{(k)}_z(x)}}{N} + \frac{1}{\sqrt{N}}  
\end{equation}
for any $\eps \in [0,1]$ if $\abs{B_{k+d(x,z)}(x)}\leq \sqrt N$. 
The estimate \eqref{eq:concentration_S_i_auxiliary_estimate_2} is proved completely analogously to \cite[Eq.~(5.17)]{ADK19}.  

Choosing $\eps = d^{-1/2 + \gamma}\abs{S^{(k)}_z(x)}^{-1/2}$ in \eqref{eq:concentration_S_i_auxiliary_estimate_2} and following the induction arguments 
for the proofs of \cite[eq.~(5.12a)]{ADK19} and \cite[eq.~(5.12b)]{ADK19} yield 
\[ 
\P \bigg( \absbb{ \frac{\abs{S^{(k+1)}_z(x)}}{d \abs{S^{(k)}_z(x)}} -1} \geq  \eps + C \cal E_k \biggm| B_{k+d(x,z)}(x) \bigg) \ind{D_z \gtrsim d} \leq \ee^{-c d^{2\gamma}}\,. 
\]  
Note that $D_z \gtrsim d$ implies that $d^{-1/2 + \gamma}/D_z^{1/2} = o(1)$ since $\gamma < 1/2$. 

By applying $\E[\,\cdot\,|B_j(x)]$ and perform union bounds over $z \in S_j(x)$ and $k \in [r-j]$ to 
the previous bound, we obtain 
\[ 
\P \bigl( \wt{\Omega}_{3,j}(x)^c \big| B_j(x) \bigr) \leq r \abs{S_j(x)} \ee^{-c d^{2\gamma}} 
+ \P \bigl( \{D_z \lesssim d \text{ for all } z \in S_j(x)\}^c\big|B_j(x)\bigr)  
\] 
for all $j \in [r]$, where we introduced the event 
\[ 
\wt{\Omega}_{3,j}(x) \deq \biggl\{ \absbb{ \frac{\abs{S^{(k+1)}_z(x)}}{d \abs{S^{(k)}_z(x)}} -1} \lesssim 
\frac{d^{-1/2 + \gamma}}{\abs{S^{(k)}_z(x)}^{1/2}} \text{ for all } z \in S_j(x),~k \in [r-j] \biggr\}\,. 
\] 
Therefore, as $D_z \gtrsim d$ for all $z \in B_r(x)$ on $\Omega_2(x) \cap \Omega_4(x)$  
and $\{ \abs{S_j(x)} \lesssim D_xd^{j-1}\} \supset \Omega_3(x)$, we get 
\begin{multline*} 
\P \bigl( \wt{\Omega}_3(x)^c\big|S_1(x)\bigr) \leq \sum_{j=1}^r \E \Bigl[ \P \bigl( \wt{\Omega}_{3,j}(x)^c \big| B_j(x) \bigr)\ind{\abs{S_j(x)} \lesssim D_x d^{j-1}} \Big| S_1(x) \Bigr] \\ 
+ r \P \bigl(\Omega_3(x)^c \big|S_1(x)\bigr) + \P \bigl(\Omega_2(x)^c \big|S_1(x)\bigr) + \P \bigl(\Omega_4(x)^c \big|S_1(x)\bigr)\,. 
\end{multline*} 
Here, we used that by iterating the first bound in the definition of $\wt{\Omega}_3(x)$, 
we directly obtain the expansion of $\abs{S_z^{(k)}(x)}$ stated in the definition of $\wt{\Omega}_3(x)$ 
since $D_z \gtrsim d$ on $\Omega_2(x) \cap \Omega_4(x)$. 
This completes the proof of \eqref{eq:estimate_P_wt_Omega_3_x} due to \eqref{eq:estimates_P_Omega_x} 
and, thus, the one of \eqref{eq:estimates_P_Omega_5_x}.

To prove \eqref{eq:estimates_P_Omega_6_x}, we define the random variable $X_y \deq (N_y(x)-d)^2 -d$ for any $y \in [N]\setminus \{x\}$ and the event 
\[ \wt{\Omega}_6(x) \deq \biggl\{ \absbb{\sum_{y \in S_k} X_y } \lesssim \abs{S_k} d^{1+\gamma} D_x^{-1/2} \text{ for all }k \in [r] \biggr\}\,.  
\] 
Here and in the following, we write $S_k = S_k(x)$ and $B_k = B_k(x)$. 
Since $\abs{S_k} = D_x d^{k-1} ( 1 + O(d^{-1/2 + \gamma}D_x^{-1/2}))$ on $\Omega_3(x)$, we have 
$\Omega_6(x)\supset \wt{\Omega}_6(x) \cap \Omega_3(x)$. 
Therefore, $\P(\Omega_6(x)^c|S_1(x)) \leq \P(\wt{\Omega}_6(x)^c|S_1(x)) + \P(\Omega_3(x)^c|S_1(x))$ and it suffices to show that $\P(\wt{\Omega}_6(x)^c |S_1(x))$ is bounded by the right-hand side of \eqref{eq:estimates_P_Omega_6_x} due to \eqref{eq:estimates_P_Omega_x} for $j =4$.

We shall show below that there are $c>0$ and $C>0$ such that 
\begin{equation} \label{eq:estimate_final_sum_X_y} 
 \P\bigg( \absbb{\sum_{y \in S_k} X_y} \geq t \bigg| B_k \bigg)  \leq  
2 \exp\big(C d^{1-2\gamma} \ee^{-c d^{2\gamma}} \abs{S_k} \big) \exp\bigg( - \frac{t^2}{4 Cd^2 \abs{S_k}} \bigg)
 + C \abs{S_k} \ee^{-c d^{2\gamma}} 
\end{equation}
for all $t \in [0, c \abs{S_k} d^{1 -2 \gamma}]$. 

We choose $t = c \abs{S_k} d^{1+ \gamma} D_x^{-1/2}$ on the event $D_x \gtrsim d^{6\gamma}$ in \eqref{eq:estimate_final_sum_X_y}, 
use that $\ind{D_x \lesssim \abs{S_k} \lesssim D_x d^{k-1}}$ is $B_k$-measurable and 
$\{D_x \lesssim \abs{S_k} \lesssim D_x d^{k-1}\} \supset \Omega_3(x)$ to obtain 
\[ \begin{aligned} 
\P \bigl( \wt{\Omega}_6(x)^c\bigm|S_1(x) \bigr) \ind{D_x \gtrsim d^{6\gamma}} & \leq \sum_{k=1}^r \E \biggl[ \P \biggl( \absbb{\sum_{y \in S_k} X_y } \gtrsim \abs{S_k} d^{1+\gamma}D_x^{-1/2} \biggm| B_k \biggr)\ind{D_x \gtrsim d^{6\gamma}}  \ind{D_x \lesssim \abs{S_k} \lesssim D_x d^{k-1}} 
\biggm| S_1 \biggr]\\ 
 & \phantom{\leq} \quad  + r \P \bigl( \Omega_3(x)^c \bigm|S_1(x)\bigr) \\ 
 & \lesssim 
r \ee^{C \ee^{-c d^{2\gamma}} D_x d^{r}} \ee^{-c d^{2\gamma}} + r D_x d^{r-1} \ee^{-c d^{2\gamma}} 
+ r \P \bigl(\Omega_3(x)^c \bigm|S_1(x) \bigr)\,. 
\end{aligned} \] 
Since $D_x \gtrsim d$ on $\{ x \in \cal W\}$ by \eqref{eq:lower_bound_degree_cal_V} with $\delta = d^{2\gamma - 1}$, we have 
$\{D_x \gtrsim d^{6\gamma}\} \supset \{ x \in \cal W\}$ as $\gamma \leq 1/6$.  
Hence, owing to \eqref{eq:estimates_P_Omega_x} for $j = 3$, 
we have proved \eqref{eq:estimates_P_Omega_6_x} provided that \eqref{eq:estimate_final_sum_X_y} 
holds.

We now turn to the proof of \eqref{eq:estimate_final_sum_X_y}. 
For $t>0$, we decompose 
\begin{equation} \label{eq:sum_X_y_split} 
 \P\bigg( \absbb{\sum_{y \in S_k} X_y} \geq t \bigg| B_k \bigg)  \leq 
\P \bigg( \absbb{\sum_{y \in S_k} X_y \ind{X_y \leq d^{1 + 2\gamma}}} \geq t \bigg| B_k \bigg) 
+ \P \bigg( \exists y \in S_k \colon X_y > d ^{1 + 2\gamma} \bigg| B_k \bigg)\,. 
\end{equation}
Next, we shall use the following lemma whose proof is deferred to Section~\ref{sec:proof_lemma_exponential_moment}. 

\begin{lemma} \label{lem:exponential_moment_aux_random_variable} 
Let $x \in [N]$ and $X_y \deq (N_y(x) - d)^2 -d$ for any $x \in [N] \setminus \{y \}$. 
The following holds. 
\begin{enumerate}[label=(\roman*)] 
\item \label{item:X_y_prop_i} For any $\gamma \in (0,1/2)$, there is $C>0$ such that, for any $\alpha \in [-d^{-1-2\gamma},d^{-1-2\gamma}]$, we have 
\begin{equation} \label{eq:upper_bound_exponential_moment} 
 \E [\exp ( \alpha X_y \ind{X_y \leq d^{1 +  2\gamma}})| B_k(x)]   \leq \exp( C \alpha^2 d^2) 
+ C \abs{\alpha} d^2 \ee^{-c d^{2\gamma}} 
\end{equation} 
if $\abs{B_k(x)} \leq \sqrt{N}$. 
\item \label{item:X_y_prop_ii} For any $k$, the random variables $(X_{y})_{y \in S_k(x)}$ are i.i.d.\ conditioned on $B_k(x)$. 
\item \label{item:X_y_prop_iii} For any $\gamma \in (0,1/2)$, there are $c>0$ and $C>0$ such that $\P( X_y \geq d^{1 + 2 \gamma} | B_k(x)) \leq C \ee^{-c d^{2\gamma}}$.
\end{enumerate} 
\end{lemma} 

From Lemma~\ref{lem:exponential_moment_aux_random_variable} \ref{item:X_y_prop_iii} and a union bound, we directly conclude that 
the second term on the right-hand side of \eqref{eq:sum_X_y_split} is 
bounded by the second term on the right-hand side of \eqref{eq:estimate_final_sum_X_y}. 

We now estimate the first term on the right-hand side of \eqref{eq:sum_X_y_split}. 
Using an exponential Chebyshef's inequality, we obtain 
\[\begin{aligned}    
\P \bigg( \absbb{\sum_{y \in S_k} X_y \ind{X_y \leq d^{1 + 2\gamma}}} \geq t \bigg| B_k \bigg) 
& =   
\P \bigg( \sum_{y \in S_k} X_y \ind{X_y \leq d^{1 + 2\gamma}} \geq t \bigg| B_k \bigg) 
+ \P \bigg( \sum_{y \in S_k} X_y \ind{X_y \leq d^{1 + 2\gamma}} \leq - t \bigg| B_k \bigg) \\ 
& \leq \inf_{\alpha \in [0, d^{-1-2\gamma}]} \ee^{-\alpha t} \bigg(
\E\bigg[ \exp \bigg( \alpha  \sum_{y \in S_k} X_y \ind{X_y \leq d^{1 + 2\gamma}}\bigg) \bigg| B_k(x) \bigg] \\
& \qquad \qquad \qquad \qquad\qquad  + \E\bigg[ \exp \bigg( - \alpha  \sum_{y \in S_k} X_y \ind{X_y \leq d^{1 + 2\gamma}}\bigg) \bigg| B_k(x) \bigg] \bigg) \\ 
& \leq 2  \inf_{\alpha \in [0, d^{-1-2\gamma}]}  \ee^{-\alpha t} 
\bigg( \exp ( C \alpha^2 d^2) + C \alpha d^2 \ee^{-cd^{2\gamma}}\bigg)^{\abs{S_k}}  \\ 
& \leq 2 \exp\big(C d^{1-2\gamma} \ee^{-c d^{2\gamma}} \abs{S_k} \big) \inf_{\alpha \in [0,d^{-1-2\gamma}]}
\ee^{-\alpha t} \ee^{C\alpha^2 d^2 \abs{S_k}} \\ 
& \leq 2 \exp\big(C d^{1-2\gamma} \ee^{-c d^{2\gamma}} \abs{S_k} \big) 
\exp\bigg( - \frac{t^2}{4 Cd^2 \abs{S_k}} \bigg)\,. 
\end{aligned} \] 
Here, we used Lemma~\ref{lem:exponential_moment_aux_random_variable} \ref{item:X_y_prop_ii} and \ref{item:X_y_prop_i} in the third step. 
In the fourth step, we pulled out $\ee^{C \alpha^2 d^2}$ from the parenthesis, used $1 + x \leq \ee^x$ and the upper bound on $\alpha$.  
The last step follows from minimizing in $\alpha$ and the upper bound $t\leq c \abs{S_k}d^{1-2\gamma}$ imposed after \eqref{eq:estimate_final_sum_X_y}. 
This completes the proof of \eqref{eq:estimate_final_sum_X_y} and, thus, the one of Proposition~\ref{pro:graphProperty}.  \qed

\subsection{Proof of Lemma~\ref{lem:exponential_moment_aux_random_variable}} 
\label{sec:proof_lemma_exponential_moment} 

We shall now prove Lemma~\ref{lem:exponential_moment_aux_random_variable}. 
To that end, we record the following auxiliary bound. Let $\cal P_\lambda$ have distribution $\mathrm{Poisson}(\lambda)$. 
For any $a>0$, we have 
\begin{equation} \label{eq:poisson_tail_aux_estimate} 
 \P \big( \cal P_\lambda \geq \lambda(1 + a) \big) \leq \exp( -\lambda h(a)) \leq \exp( - c \lambda a^2)\,, 
\end{equation}
where the last step holds only if $a \in (0,1]$. Here, $h$ is defined as in \eqref{eq:def_h} and $c>0$ is a constant. 
The bound \eqref{eq:poisson_tail_aux_estimate} follows from an exponential Chebyshev's inequality and 
minimizing the resulting upper bound. 
Similarly, we get 
\begin{equation} \label{eq:poisson_lower_tail_aux_estimate} 
\P \big( \cal P_\lambda \leq \lambda(1- a) \big) \leq \exp\big( -\lambda h(-a) \big) \leq \exp\big( - c \lambda a^2 \big) 
\end{equation}
for all $a \in [0,1/2]$ and some $c >0$. 

\begin{proof}[Proof of Lemma~\ref{lem:exponential_moment_aux_random_variable}]
For the proof of \ref{item:X_y_prop_i}, we fix $y$ and write $X$ instead of $X_y$. 
For all $t \in [-1,1]$ we have $\ee^t \leq 1 + t +2 t^2$. Therefore, $X \geq - d$ and $\alpha \in [-d^{-1-2\gamma},d^{-1-2\gamma}]$ imply 
\begin{equation}\label{eq:proof_upper_bound_exponential_moment_aux1} 
 \E \big[ \exp(\alpha X \ind{X\leq d^{1 + 2\gamma}})\big| B_k(x) \big] \leq 1 + \alpha \E\big[X \ind{X\leq d^{1 + 2\gamma}}\big| B_k(x) \big] +  2\alpha^2 \E \big[ X^2 \ind{X\leq d^{1 + 2\gamma}}\big| B_k(x) \big] 
\end{equation} 
We shall show below that there are constants $C>0$ and $c>0$ such that 
\begin{equation} \label{eq:intermediate_estimates_X} 
 \absb{\E\big[X \ind{X\leq d^{1 + 2\gamma}} \big| B_k(x) \big]}  \leq C d^2 \ee^{-c d^{2\gamma}}, \qquad \E \big[ X^2 \ind{X\leq d^{1 + 2\gamma}}\big| B_k(x) \big] \leq C d^2\,. 
\end{equation}
Assuming these two estimates, \eqref{eq:proof_upper_bound_exponential_moment_aux1} and $1 + t \leq \ee^t$
for all $t \in \R$ directly yield \eqref{eq:upper_bound_exponential_moment}. 

In the following we shall approximate the distribution of $N_y(x)$ conditioned on $B_k(x)$ by $\mathrm{Poisson}(\lambda)$. 
Since $N_y(x)$ has law $\mathrm{Binom}(N - \abs{B_k(x)},d/N)$ conditioned on $B_k(x)$, we choose 
$\lambda = d (1 - \abs{B_k(x)}/N)$. 
Note that $\lambda \leq d \leq \lambda(1 + O(N^{-1/2}))$ since $\abs{B_k(x)} \leq \sqrt N$. 

Let $\cal P_\lambda$ be a random variable with distribution $\mathrm{Poisson}(\lambda)$. 
We set $Y \deq (\cal P_\lambda-d)^2 - d$. Since $d^{1 + 2\gamma}\leq \sqrt{N - \abs{B_k(x)}}$ by $\abs{B_k(x)} \leq \sqrt N$, we can apply Lemma~\ref{lem:poisson_approx} and obtain
\begin{align*}
\E \big [\ind{X \leq d^{1+ 2\gamma}} X\big| B_k(x) \big] &= \E \big[\ind{Y \leq d^{1 + 2\gamma}} Y |B_k(x) \big] \bigg(1 + O\bigg( \frac{d^2}{N} \bigg) \bigg)
\\
&= \bigg( - \E \big[ \ind{Y > d^{1 + 2 \gamma}}Y | B_k(x) \big] + O\bigg( \frac{d}{\sqrt{N}} \bigg) \bigg) \bigg(1 + O\bigg( \frac{d^2}{N} \bigg) \bigg)\,. 
\end{align*}
Here, in the second step, we used $\lambda \leq d$, $\abs{B_k(x)} \leq \sqrt N$ and $\E[Y |B_k(x)] = - \frac{d \abs{B_k(x)}}{N} + ( \frac{d \abs{B_k(x)}}{N})^2=O\big( \frac{d}{\sqrt{N}}\big)$ due to 
$\E[ \cal P_\lambda | B_k(x)] = \lambda$ and $\E[\cal P_\lambda^2 |B_k(x)] = \lambda(\lambda +1)$. 

Introducing $l_\pm \deq d \pm \sqrt{d + d^{1 + 2\gamma}}$, we have 
\[ 0 \leq \E \big[ \ind{Y > d^{1 + 2 \gamma}} Y | B_k(x) \big] = \sum_{l > l_+}  ((l-d)^2 - d) \frac{\ee^{-\lambda} \lambda^l}{l!}  
 + \sum_{l < l_-}  ((l-d)^2 - d) \frac{\ee^{-\lambda} \lambda^l}{l!}\,. 
\] 
Since $\lambda \leq d\leq \lambda( 1 + o(1))$ and $l_+ \geq \lambda + d^{1/2 + \gamma}$, we have 
\begin{align*}
\sum_{l=l_+}^\infty ((l-d)^2 - d) \frac{\ee^{-\lambda} \lambda^l}{l!} &\leq \ee^{-\lambda} \lambda^2 \sum_{l=\lambda + \lambda^{1/2 + \gamma}}^{\infty} \frac{l \cdot l}{l(l-1)} \frac{\lambda^{l-2}}{(l-2)!}  
\\
&\leq 2 d^2 \P \big( \cal P_\lambda \geq \lambda + \lambda^{1/2 + \gamma} - 2 \big| B_k(x) \big) \leq 2 d^2 \ee^{-c d^{2\gamma}} \,,
\end{align*}
where we used \eqref{eq:poisson_tail_aux_estimate} in the last step. 
Similarly, using \eqref{eq:poisson_lower_tail_aux_estimate}, we get 
\[ \sum_{l=0}^{l_-} ((l-d)^2 - d) \frac{\ee^{-\lambda} \lambda^l}{l!} \leq d^2 \P ( \cal P_\lambda \leq d - d^{1/2 + \gamma}\big| B_k(x)  ) 
\leq d^2 \ee^{-cd^{2\gamma}}\,. 
\] 
Combining these estimates yields the first bound in \eqref{eq:intermediate_estimates_X}.

For the proof of the second bound in \eqref{eq:intermediate_estimates_X}, we use Lemma~\ref{lem:poisson_approx} and obtain 
\begin{align*}
\E \big[ X^2 \ind{X \leq d^{1 + 2\gamma}}\big| B_k(x) \big] &= \E\big[Y^2 \ind{Y\leq d^{1+2\gamma}}\big|B_k(x)] \bigg(1 + O\bigg(\frac{d^2}{N} \bigg)\bigg)
\\
&\leq 2 \bigg(\E \big[ ((\cal P_\lambda-\lambda)^2 -\lambda)^2 \big| B_k(x)\big] + O\bigg(\frac{d^4}{N} \bigg) \bigg) = O(d^2)\,.
\end{align*}
In the second step, we dropped the indicator function and replaced $d$ in the definition of $Y$ by $\lambda$. Finally, we used $\E [ ((\cal P_\lambda-\lambda)^2 - \lambda)^2|B_k(x)] = O(\lambda^2)$ and $\lambda \leq d$. 
This completes the proof of \eqref{eq:intermediate_estimates_X}.

Part \ref{item:X_y_prop_ii} is obvious from the definition of $X_y$ and the conditioning on $B_k(x)$. 

For the proof of \ref{item:X_y_prop_iii}, we choose $\lambda$, $\cal P_\lambda$ and $Y$ as in the proof of \ref{item:X_y_prop_i}. Thus, using  
from Lemma~\ref{lem:poisson_approx} and arguing as before yield  
\begin{align*}
 \P( X_y < d^{1 + 2 \gamma} | B_k(x)) & = \P ( Y < d^{1+ 2\gamma}|B_k(x))\bigg( 1 + O\bigg(\frac{d^2}{N} \bigg)\bigg) \\ 
 & = \big(1  +\P(\cal P_\lambda \geq l_+|B_k(x)) + \P(\cal P_\lambda \leq l_-|B_k(x)) \big) \bigg(1 +  O\bigg(\frac{d^2}{N} \bigg) \bigg) \\
& = 1 + O\big(\ee^{-c d^{2\gamma}}\big)\,. 
\end{align*} 
This completes the proof of Lemma~\ref{lem:exponential_moment_aux_random_variable}. 
\end{proof}

\appendix

\section{Spectral analysis of the $(p,q,s)$-regular tree} \label{app:spectral_analysis} 

In this appendix, we study the extreme eigenvalues and eigenvectors of the infinite tridiagonal matrix
$Z_{\fra d}(\alpha,\beta)$ defined in \eqref{eq:def_Zd} when it is viewed as an operator acting on $\ell^2(\N)$. 
The main results about $Z_{\fra d}$ are collected in Corollary~\ref{cor:Lambda_expansion} below. We have the scaling relation
\begin{equation} \label{Z_scaling}
Z_{\fra d}(\alpha,\beta) = \sqrt{\fra d} Z\p{\alpha / \fra d, \beta / \fra d}\,, \qquad Z(\alpha, \beta) \deq Z_1(\alpha, \beta)\,.
\end{equation}

We shall also prove a few properties of the functions $\Lambda(\alpha,\beta)$ and $\Lambda_{\fra d}(\alpha,\beta)$ defined 
in \eqref{eq:def_Lambda} and \eqref{eq:def_Lambda_d}, respectively. 
They will turn out to be the largest eigenvalues of $Z(\alpha,\beta)$ and $Z_{\fra d}(\alpha,\beta)$, respectively.

\begin{proposition}[Extreme eigenvalues and eigenvectors of $Z(\alpha,\beta)$] \label{prop:Lambda_ab}
Suppose that $\alpha > 2$ and
\begin{equation} \label{assumption_beta_pr}
\abs{\beta - 1} \leq c \pb{1 \wedge (\alpha - 2)}
\end{equation}
for some small enough constant $c$.
\begin{enumerate}[label=(\roman*)]
\item \label{item:Z_eval}
The matrix $Z(\alpha, \beta)$ has a unique eigenvalue $\lambda$  in $(2,\infty)$. It satisfies $\lambda = \Lambda(\alpha, \beta)$. 
 Moreover, the spectrum of $Z(\alpha, \beta)$ is symmetric.
\item \label{item:Z_evect}
The eigenvector $(u_i)_{i \in \N}$ associated with the eigenvalue $\pm \lambda = \pm \Lambda(\alpha, \beta)$ satisfies
\begin{equation} \label{Z_vect_components}
u_1 = \pm \frac{\lambda}{\sqrt{\alpha}}\, u_0 \,, \quad u_2 = \pm \frac{2 \sqrt{\beta}}{\lambda + \sqrt{\lambda^2 - 4}} \, u_1\,, \quad u_k = \pbb{\frac{\pm 2}{\lambda + \sqrt{\lambda^2 - 4}}}^{k - 2} \, u_2 \quad (k \geq 3)\,.
\end{equation}
\end{enumerate}
\end{proposition}

Proposition \ref{prop:Lambda_ab} is proved in Section \ref{sec:pf_Lambda_ab} below.
 There, we also show the following result.

\begin{lemma}[Properties of $\Lambda$] \label{lem:properties_Lambda}
Let $\alpha > 2$. Then 
\begin{equation} \label{Lambda_identities}
\Lambda(\alpha, 1) = \frac{\alpha}{\sqrt{\alpha - 1}}\,, \qquad \partial_\alpha \Lambda(\alpha,1) = \frac{\alpha - 2}{2(\alpha - 1)^{3/2}}\, , \qquad \partial_\beta \Lambda(\alpha, 1) = \frac{\alpha (\alpha - 2)}{2(\alpha - 1)^{5/2}}\,. 
\end{equation}
Moreover, under the assumption \eqref{assumption_beta_pr} for small enough $c$, the following holds. 
The function $\Lambda = \Lambda(\alpha,\beta)$ has positive 
derivatives $\partial_\alpha \Lambda$ and $\partial_\beta \Lambda$ 
and satisfies 
\begin{equation} \label{Lambda_estimates}
\partial_\alpha \Lambda \lesssim \alpha^{-1/2} \,, \quad \partial_\beta \Lambda \lesssim \alpha^{-1/2}\,, \quad
\abs{\partial_\alpha^2 \Lambda} \lesssim \alpha^{-3/2} \, \quad
\abs{\partial_\beta^2 \Lambda} \lesssim \alpha^{-1}
\quad
\abs{\partial_\alpha \partial_\beta \Lambda} \lesssim \alpha^{-3/2} \,,
\end{equation}
the expansion 
\begin{equation} \label{Lambda_expansion}
\Lambda(\alpha, \beta) = \frac{\alpha}{\sqrt{\alpha - 1}} + \frac{\alpha (\alpha - 2)}{2(\alpha - 1)^{5/2}} (\beta - 1) + O\pbb{\frac{(\beta - 1)^2}{\alpha}}
\end{equation}
as well as the estimates
\begin{align} 
\label{eq:approx_Lambda_alpha_x_beta_x_by_Lambda_alpha_x}
\Lambda(\alpha, \beta) & \asymp \sqrt{\alpha}\,, \\ 
\label{Lambda_geq_2_est}
\Lambda(\alpha, \beta) - 2 & \gtrsim \frac{(\alpha - 2)^2}{\alpha^{3/2}}\,.
\end{align} 
\end{lemma}

By expanding \eqref{eq:def_Lambda_d} using Lemma~\ref{lem:properties_Lambda},  we obtain the following result. The expansion is designed so that the error term is negligible in our application, the proof of Lemma \ref{lem:Sigma_Phi},  for the entire regime of $d$ allowed by our results. We note that the optimal bounds on the derivatives obtained in \eqref{Lambda_estimates} for large values of $\alpha$ are crucial. 
 The spectral properties of $Z_{\fra d}$ follow from Proposition~\ref{prop:Lambda_ab}, \eqref{Z_scaling}, and \eqref{eq:def_Lambda_d}.

\begin{corollary} \label{cor:Lambda_expansion}
Suppose that $d \geq 1$, $\alpha \geq 2 + 4/d$, and that \eqref{assumption_beta_pr} holds for some small enough constant $c$.  

\begin{enumerate}[label=(\roman*)]
\item \label{item:Zd_eval}
The matrix $Z_{\fra d}(\alpha, \beta)$ has a unique eigenvalue $\lambda$ in $(2 \sqrt{\fra d},\infty)$.  
It satisfies $\lambda = \Lambda_{\fra d}(\alpha, \beta)$. 
Moreover, the spectrum of $Z_{\fra d}(\alpha,\beta)$ is symmetric. 
\item \label{item:Zd_evect}
The eigenvector $(u_i)_{i \in \N}$ associated with the eigenvalue $\pm \lambda = \pm \Lambda_{\fra d}(\alpha, \beta)$ satisfies
\begin{equation} \label{Zd_vect_components}
u_1 = \pm \frac{\lambda}{\sqrt{\alpha}} \, u_0 \,, \quad u_2 = \pm \frac{2 \sqrt{\beta}}{\lambda + \sqrt{\lambda^2 - 4 \fra d}} \, u_1\,, \qquad
u_k = \pbb{\frac{\pm 2 \sqrt{\fra d}}{\lambda + \sqrt{\lambda^2 - 4 \fra d}}}^{k - 2} \, u_2 \quad (k \geq 3)\,.
\end{equation}
\item 
The derivatives $\partial_\alpha \Lambda_{\fra d}(\alpha,\beta)$ and $\partial_\beta \Lambda_{\fra d}(\alpha,\beta)$ are positive. 
For any $u \geq 2 + 4/d$ satisfying $\abs{\alpha - u} \leq 1$, the function $\Lambda_{\fra d}(\alpha,\beta)$  has the expansion
\begin{multline}\label{eq:Lambda_d_expansion} 
\Lambda_{\fra d}(\alpha, \beta) = \Lambda_{\fra d}(u,1) + \frac{u - 2}{2 (u - 1)^{3/2}} (\alpha - u) + \frac{u (u - 2 )}{2(u - 1)^{5/2}} (\beta - 1)
\\
+ O \pBB{\frac{(\alpha - u)^2}{u^{3/2}} + \frac{\abs{\alpha - u}}{u^{3/2} d} 
 + \frac{\abs{\beta - 1}}{u^{1/2} d} + \frac{(\beta - 1)^2}{u}}\,. 
\end{multline}
\end{enumerate}
\end{corollary} 

\begin{lemma} \label{lem:eigenvector_fine_properties} 
Let $\alpha > 2$.
There is a constant $c>0$ such that
under the condition \eqref{assumption_beta_pr}
the following holds 
for the normalized, nonnegative eigenvector $(u_i)_{i \in \N}$ of $Z(\alpha,\beta)$ corresponding to its largest eigenvalue $\Lambda(\alpha,\beta)$. 
\begin{enumerate}[label=(\roman*)] 
\item 
\begin{equation} 
u_2^2 \leq \sum_{i=2}^\infty u_i^2 \lesssim \frac{1}{\alpha}. \label{eq:relations_sum_u_i_u_2}
\end{equation} 
\item 
There is a constant $c_*>0$ such that, for any $d \geq 1$, if 
 $r \geq 22 + \frac{1}{c_*} \frac{\log d}{\log \alpha } \big(\frac{\alpha}{\alpha -2}\big)^2$ then
\begin{equation}\label{eq:u_r_leq_u_2_small} 
u_r \leq u_2 (d \alpha)^{-10}\,. 
\end{equation}
\end{enumerate}
The same estimates hold if $(u_i)_{i \in \N}$ is the eigenvector of $Z_{\fra d}(\alpha, \beta)$ corresponding to its largest eigenvalue $\Lambda_{\fra d}(\alpha, \beta)$, for any $d \geq 1$ and $\alpha > 2 + 4/d$.
\end{lemma}

\begin{proof}
We distinguish the regimes $\alpha \geq C_0$ and $\alpha \leq C_0$ for some sufficiently large $C_0>0$. 

First, suppose that $\alpha \geq C_0$. 
For large enough $C_0>0$, we get from \eqref{Z_vect_components} and \eqref{eq:approx_Lambda_alpha_x_beta_x_by_Lambda_alpha_x} 
that $(u_i/u_2)^{1/(i-2)} \lesssim  1/\sqrt{\alpha}$ for all $i \geq 3$.
Thus, $\sum_{i\geq 2} u_i^2 \lesssim u_2^2$ if $\alpha \geq C_0$ for $C_0$ large enough. 
Moreover, $u_2^2 \asymp \alpha^{-1}$ if $\alpha \geq C_0$ and $C_0$ is sufficiently large
due to \eqref{eq:approx_Lambda_alpha_x_beta_x_by_Lambda_alpha_x} and 
\eqref{Z_vect_components}. 
This proves \eqref{eq:relations_sum_u_i_u_2} if $\alpha \geq C_0$. 
Since $(u_i /u_2)^{1/(i-2)} \lesssim 1/\sqrt{\alpha}$ for all $i \geq 3$, 
the imposed lower bound on $r$ yields \eqref{eq:u_r_leq_u_2_small} if $\alpha \geq C_0$.

We now assume that $\alpha \leq C_0$. 
Since $\alpha \gtrsim 1$ by \eqref{eq:scaling_alpha_max} below and $(u_i)_{i \in \N}$ is normalized, 
\eqref{eq:relations_sum_u_i_u_2} is trivial in the present regime. 
For the proof of \eqref{eq:u_r_leq_u_2_small}, we also use \eqref{Z_vect_components} to conclude, 
with $\Lambda = \Lambda(\alpha,\beta)$, that  
\[ \frac{u_r^{1/(r-2)}}{u_2^{1/(r-2)}} = \frac{2}{\Lambda + \sqrt{\Lambda^2- 4}} = 1 - \frac{2\sqrt{\Lambda-2}}{\sqrt{\Lambda +2 } + \sqrt{\Lambda -2}} \leq 1 - \tilde{c}\frac{\sqrt{\Lambda - 2}}{\sqrt\Lambda} 
\leq \exp \bigg( - \tilde{c} \frac{\sqrt{\Lambda - 2}}{\sqrt{\Lambda}} \bigg) \] 
for some constant $\tilde{c} > 0 $. 
Hence, \eqref{eq:u_r_leq_u_2_small} holds if $r$ satisfies the imposed lower bound 
due to $\frac{\Lambda}{\Lambda - 2} \geq 1$, \eqref{Lambda_geq_2_est} and \eqref{eq:approx_Lambda_alpha_x_beta_x_by_Lambda_alpha_x}.
Finally, the proof of the claim for the eigenvector of $Z_{\fra d}(\alpha, \beta)$ is analogous, using Corollary \ref{cor:Lambda_expansion} \ref{item:Zd_evect} instead of Proposition \ref{prop:Lambda_ab} \ref{item:Z_evect}. 
\end{proof}

\begin{lemma}\label{lem:EstimateLambdad}
There is a constant $c>0$ such that under the condition \eqref{assumption_beta_pr}, we have for all $\alpha > 2$,
\[ 0 \leq \Lambda_{\fra d}(\alpha,\beta)-\Lambda(\alpha,\beta) \lesssim \frac{1}{d \alpha}\,. \] 
\end{lemma}

\begin{proof}
The inequality $\Lambda_{\fra d}(\alpha,\beta)-\Lambda(\alpha,\beta) \geq 0$ is a standard consequence of the Perron-Frobenius theorem. 
For the other estimate, let $\f u = (u_i)_{i\in \mathbb{N}}$ be a normalized eigenvector of $Z(\alpha,\beta)$ corresponding to its largest eigenvalue $\Lambda(\alpha,\beta)$.
Then 
\begin{equation} \label{eq:proof_estimate_Lambda_d_aux1} 
\left\|\left(Z_{\fra d}(\alpha,\beta)-Z(\alpha,\beta)\right)\f u\right\|^2 \leq 2(\sqrt{\mathfrak{d}}-1)^2\sum_{i\geq 2}u_i^2 \lesssim \frac{1}{\alpha d^2}\,, 
\end{equation}
where we used \eqref{eq:relations_sum_u_i_u_2} from Lemma~\ref{lem:eigenvector_fine_properties} and $\mathfrak{d} 
=1 + 1/d$ 
in the last step. This proves $\Lambda_{\fra d}(\alpha,\beta) = \Lambda(\alpha,\beta) + O((d\sqrt{\alpha})^{-1})$ 
and, thus, the lemma if $\alpha \leq C_0$ for any constant $C_0$. 

We now choose $C_0$ sufficiently big such that $\alpha \geq C_0$ implies 
$\Lambda(\alpha,\beta) > 2 \sqrt{\mathfrak{d}} + 1$.  
Then $Z_{\fra d}(\alpha,\beta)$ has only one eigenvalue in $[\Lambda(\alpha,\beta) - 1, \Lambda(\alpha,\beta) + 1]$ 
by Corollary~\ref{cor:Lambda_expansion} \ref{item:Zd_eval}.  
Hence, we have justified the conditions of Lemma~\ref{lem:perturbationEV} with $M = Z_{\fra d}(\alpha,\beta)$, $\lambda = \Lambda(\alpha,\beta)$, $\Delta = 1$ and 
$\eps = O((d \sqrt{\alpha})^{-1})$ by \eqref{eq:proof_estimate_Lambda_d_aux1}. 
Similary, as in \eqref{eq:proof_estimate_Lambda_d_aux1}, we obtain from \eqref{eq:relations_sum_u_i_u_2} that 
\[ \scalar{\f u}{(Z_{\fra d}(\alpha,\beta)-\Lambda(\alpha,\beta))\f u}  = \scalar{\f u}{(Z_{\fra d}(\alpha,\beta)-Z(\alpha,\beta))\f u} = O \big( (d \alpha)^{-1} \big)\,. 
\] 
Therefore, using Lemma~\ref{lem:perturbationEV} completes the proof of Lemma~\ref{lem:EstimateLambdad}. 
\end{proof}

\subsection{Proofs of Proposition \ref{prop:Lambda_ab} and Lemma~\ref{lem:properties_Lambda}} \label{sec:pf_Lambda_ab}

The rest of this appendix is devoted to the proofs of Proposition \ref{prop:Lambda_ab} and Lemma~\ref{lem:properties_Lambda}. We start with some auxiliary results. The following result is a simple consequence of a transfer matrix analysis (see \cite[Appendix C]{ADK19}).

\begin{lemma}
If $\alpha \leq 2$ then $\norm{Z(\alpha, 1)} \leq 2$. If $\alpha > 2$ then $Z(\alpha,1)$ has two eigenvalues $\pm \frac{\alpha}{\sqrt{\alpha - 1}}$, and the remainder of the spectrum is contained in $[-2,2]$.
\end{lemma}

Since $Z(\beta, 1)$ and $Z(0,\beta)$ have the same spectrum in $\R \setminus \{0\}$ and since the matrix $\pB{\begin{smallmatrix}
0 & \sqrt{\alpha}
\\
\sqrt{\alpha} & 0
\end{smallmatrix}}$ has eigenvalues $\pm \sqrt{\alpha}$, Cauchy's interlacing inequalities yield the following result.

\begin{corollary} \label{cor:uniqueness_lambda}
If $\beta \leq 2$ then $Z(\alpha, \beta)$ has at most one eigenvalue in $(2,\infty)$.
\end{corollary}

\begin{lemma} \label{lem:Lambda_greater_2} 
If $\alpha > 2$ and \eqref{assumption_beta_pr} holds for some sufficiently small $c>0$ then 
$\Lambda(\alpha,\beta) >2$.  
\end{lemma} 

\begin{proof} 
We first show the well-definedness of the definition of $\Lambda$ in \eqref{eq:def_Lambda}. 
If $\alpha \geq 2$ and $\beta \geq 2 (\sqrt{2} - 1)$ then $(\alpha + \beta)^2 - 4 \alpha \geq 0$.  
We introduce $\mu \deq \alpha - \frac{\beta}{2} (\alpha + \beta) + \frac{\beta}{2} \sqrt{(\alpha + \beta)^2 - 4 \alpha}$. 
If $\beta \leq 6/5$ then 
$\alpha - \frac\beta 2 (\alpha + \beta) > 0$ and 
if $\beta \geq 6/5$ then $\frac{\beta}{2} \sqrt{(\alpha + \beta)^2 - 4 \alpha}  > \abs{\alpha - \frac{\beta}{2} 
( \alpha + \beta)}$.  
Hence, $\mu >0$ if $\alpha \geq 2$ and $\beta \geq 2(\sqrt{2}- 1)$ which implies the well-definedness of \eqref{eq:def_Lambda} 
as $\Lambda = {\alpha}/{\sqrt{\mu}}$.

We note that, since $\alpha^2 - 4 \alpha + 2 \beta(\alpha + \beta) = \alpha (\alpha - 2 + 2(\beta - 1))+ \beta^2 \geq 0$ if $\abs{\beta - 1} \leq  (\alpha- 2)$, our goal $\Lambda(\alpha,\beta) > 2$ is equivalent to 
$\alpha^4 - 8 \alpha^3 + 16 \alpha^2 + 4 \alpha^3 \beta + 4 \alpha^2 \beta ^2 - 16 \alpha^2 \beta>0$. 
The last expression coincides with $\alpha^2 ( (\alpha - 2)( \alpha - 2 + 4 (\beta - 1)) + 4 (\beta - 1)^2)$, 
which is positive if $\abs{\beta - 1} < (\alpha - 2)/4$. 
\end{proof}

We now deduce Proposition~\ref{prop:Lambda_ab} from these previous auxiliary results.

\begin{proof}[Proof of Proposition \ref{prop:Lambda_ab}]
We first note that $Z(\alpha, \beta)$ conjugated by $\diag(1,-1,1,-1, \dots)$ is equal to $-Z(\alpha, \beta)$.
Therefore, the spectrum of $Z(\alpha,\beta)$ is symmetric and it suffices to focus on the positive eigenvalue 
and the corresponding eigenvectors.
The eigenvalue-eigenvector equation for $Z(\alpha, \beta)$ reads
\begin{equation} \label{u_conditions1}
\lambda u_0 = \sqrt{\alpha} u_1 \,, \quad
\lambda u_1 = \sqrt{\alpha} u_0 + \sqrt{\beta} u_2\,, \quad
\lambda u_2 = \sqrt{\beta} u_1 + u_3\,, \quad
\lambda u_k = u_{k - 1} + u_{k+1} \quad (k \geq 3)\,.
\end{equation}
The final relation can be written as
\begin{equation*}
\begin{pmatrix}
u_{k+1} \\ u_{k}
\end{pmatrix}
=
T(\lambda)
\begin{pmatrix}
u_{k} \\ u_{k-1}
\end{pmatrix} \qquad (k \geq 3)\,, \qquad T(\lambda) \deq
\begin{pmatrix}
\lambda & -1
\\
1 & 0
\end{pmatrix}\,.
\end{equation*}
For $\lambda > 2$, the matrix $T(\lambda)$ has eigenvalues $\gamma(\lambda)$ and $1/\gamma(\lambda)$, where $\gamma(\lambda) \deq \frac{\lambda - \sqrt{\lambda^2 - 4}}{2} < 1$. Since the components of the eigenvector $u_k$ cannot grow exponentially, we require that the vector $\binom{u_3}{u_2}$ be collinear to the eigenvector associated with $\gamma(\lambda)$. Hence, we have $u_k = \gamma(\lambda) u_{k - 1}$ for $k \geq 3$. This already yields \ref{item:Z_evect} assuming $\lambda > 2$ from \ref{item:Z_eval}.

What remains is completing the proof of \ref{item:Z_eval}. 
Let $\alpha$ and $\beta$ satisfy the assumptions of Proposition \ref{prop:Lambda_ab}. 
As the uniqueness is taken care of by Corollary \ref{cor:uniqueness_lambda} and $\lambda = \Lambda(\alpha,\beta)> 2$ by Lemma~\ref{lem:Lambda_greater_2}, it suffices to 
show that $\lambda$ is an eigenvalue of $Z(\alpha,\beta)$. 
Together with $u_k = \gamma(\lambda) u_{k - 1}$ for $k \geq 3$, the conditions \eqref{u_conditions1} can be written as
\begin{equation} \label{u_conditions2}
\lambda u_0 = \sqrt{\alpha} u_1 \,, \qquad
\lambda u_1 = \sqrt{\alpha} u_0 + \sqrt{\beta} u_2\,, \qquad
\lambda u_2 = \sqrt{\beta} u_1 + u_3\,, \qquad u_3 = \gamma(\lambda) u_2\,.
\end{equation}
Thus, using the definition of $\Lambda(\alpha,\beta)$ in \eqref{eq:def_Lambda}, 
we see that $q(\alpha,\beta,\Lambda(\alpha,\beta)) = 0$, where we introduced 
 the quartic (biquadratic) polynomial
\begin{equation} \label{def_q}
q(\alpha, \beta, \lambda) \deq (1 - \beta) \lambda^4 + (\alpha \beta + \beta^2 - 2 \alpha) \lambda^2 + \alpha^2\,.
\end{equation}
Then a simple calculation shows the existence of $(u_i)_i \in \ell^2(\N)\setminus \{ 0 \}$ such that \eqref{u_conditions2} holds with $\lambda = \Lambda(\alpha,\beta)$, and hence $\Lambda(\alpha,\beta)$ is an eigenvalue of $Z(\alpha, \beta)$. 
\end{proof}

\begin{proof}[Proof of Lemma~\ref{lem:properties_Lambda}] 
By following the proof of Lemma~\ref{lem:Lambda_greater_2}, we obtain $\alpha - \frac{\beta}{2} ( \alpha + \beta) \gtrsim \alpha$ if $\beta \leq 6/5$.  
Thus, for $\mu$ defined in the proof of Lemma~\ref{lem:Lambda_greater_2}, we get $\mu \asymp \alpha$ if $\alpha \geq 2$ and $2 (\sqrt{2} - 1) \leq \beta \leq 6/5$. 
This proves \eqref{eq:approx_Lambda_alpha_x_beta_x_by_Lambda_alpha_x}. 

The identities in \eqref{Lambda_identities} follow by a simple computation from \eqref{eq:def_Lambda}. 
Owing to the well-definedness of $\Lambda$ shown in the proof of Lemma~\ref{lem:Lambda_greater_2}, 
$\Lambda$ is smooth in $\alpha$ and $\beta$ if $\alpha >2$ and $\beta \geq 2(\sqrt{2} - 1)$. 
Therefore, the positivity of the derivatives $\partial_\alpha \Lambda$ and $\partial_\beta \Lambda$ 
is a standard consequence of the Perron-Frobenius theorem as $\Lambda(\alpha,\beta)$ is 
the largest eigenvalue of $Z(\alpha,\beta)$ by Proposition~\ref{prop:Lambda_ab} \ref{item:Z_eval}. 

As $\mu \asymp \alpha$, we have $\mu \gtrsim 1$ and, thus, $\abs{\partial_\alpha \Lambda} + 
\abs{\partial_\beta \Lambda} + \abs{\partial_\alpha^2 \Lambda} +\abs{\partial_{\alpha\beta} \Lambda} +\abs{\partial_\beta^2 \Lambda} \lesssim 1$ for all $\alpha \in (2,C_*]$ and any constant $C_*$ satisfying $2 < C_* \lesssim 1$. 

Let $\f u = (u_i)_{i \in \N}$ be a normalized eigenvector of $Z(\alpha,\beta)$ corresponding to $\Lambda = \Lambda(\alpha,\beta)$. Note that 
it satisfies \eqref{Z_vect_components} and $\abs{u_i} \lesssim 1$ due to the normalization. 
Both properties will be used extensively in the following. 
By perturbation theory for eigenvalues, we obtain 
\begin{equation} \label{eq:derivatives_Lambda_in_terms_eigenvectors}
\partial_{\alpha}\Lambda=\frac{1}{\sqrt{\alpha}}u_{0}u_{1}=\frac{\Lambda}{\alpha}u_{0}^{2}\,, 
\qquad 
\partial_{\beta}\Lambda=\frac{1}{\sqrt{\beta}}u_{1}u_{2}=\frac{2}{\Lambda+\sqrt{\Lambda^{2}-4}}u_{1}^{2}\,. 
\end{equation}
As $\Lambda \asymp \sqrt\alpha$, this shows $\partial_\alpha \Lambda \lesssim \alpha^{-1/2}$ and $\partial_\beta 
\Lambda \lesssim \alpha^{-1/2}$. 
If $\Pi$ denotes the orthogonal projection onto $\f u^\perp$ then perturbation theory also yields 
\[
\partial_{\alpha}\f u=\frac{1}{2\sqrt{\alpha}}(\Pi (Z-\Lambda)\Pi)^{-1}(u_{1}\f e_{0}+u_{0}\f e_{1})\,, \quad \partial_{\beta}\f u=\frac{1}{2\sqrt{\beta}}(\Pi(Z-\Lambda)\Pi)^{-1}(u_{2}\f e_{1} + u_{1}\f e_{2})\,,
\]
where $\f e_i$ denotes the standard basis vector in $\R^\N$.

Hence, $\norm{\partial_\alpha \f u} \lesssim \alpha^{-1/2} (\Lambda - 2)^{-1}$ and $\norm{\partial_\beta \f u}
\lesssim (\Lambda - 2)^{-1}$ by Corollary~\ref{cor:uniqueness_lambda} and $\Lambda>2$ from Lemma~\ref{lem:Lambda_greater_2}. 
We now choose $C_*>2$ so large that $C_* \lesssim 1$ and $\Lambda -2 \asymp \sqrt{\alpha}$ for all $\alpha \geq C_*$. 
In particular, $\partial_\alpha \f u = O(\alpha^{-1})$ and $\partial_\beta \f u = O(\alpha^{-1/2})$ for all $\alpha \geq C_*$. 
Therefore, from \eqref{eq:derivatives_Lambda_in_terms_eigenvectors} and $\Lambda \asymp \sqrt\alpha$, we obtain 
\begin{align*} 
\partial_{\alpha}^{2}\Lambda& =-\frac{\Lambda u_{0}^{2}}{\alpha^{2}}+\frac{u_{0}^{2}}{ \alpha}\partial_{\alpha}\Lambda +\frac{2 \Lambda u_0}{\alpha}\partial_{\alpha}u_{0} = O(\alpha^{-3/2})\,, 
\\ 
\partial_{\alpha\beta}\Lambda& =\frac{4u_1 }{\Lambda+\sqrt{\Lambda^{2}-4}}\partial_{\alpha}u_{1}-\frac{2u_{1}^{2}(1+\Lambda (\Lambda^2-4)^{-1/2})}{(\Lambda+\sqrt{\Lambda^{2}-4})^2}\partial_{\alpha}\Lambda = O(\alpha^{-3/2})\,,
\\ 
\partial_{\beta}^{2}\Lambda & =-\frac{2u_{1}^{2}(1+\Lambda (\Lambda^2-4)^{-1/2})}{(\Lambda+\sqrt{\Lambda^{2}-4})^2}
\partial_{\beta}\Lambda+\frac{4u_1}{\Lambda+\sqrt{\Lambda^{2}-4}}\partial_{\beta}u_{1} = O(\alpha^{-1})\,.
\end{align*} 
This proves \eqref{Lambda_estimates}, which immediately implies \eqref{Lambda_expansion}. 
Since $\frac{\alpha}{\sqrt{\alpha - 1}} - 2 \gtrsim \frac{(\alpha-2)^2}{\alpha^{3/2}}$, 
the inequality in \eqref{Lambda_geq_2_est} follows from \eqref{Lambda_expansion} 
and the condition $\abs{\beta - 1} \leq c (\alpha - 2)$. 
\end{proof}

\section{Truncated inclusion-exclusion formula for point processes -- proof of Lemma~\ref{lem:incl_excl}} \label{sec:inclusion_exclusion}

By truncation of $\cal X$ and monotone convergence, it suffices to consider finite deterministic $\cal X$, which we order in some arbitrary fashion. 
In the following, the index $x$ ranges over $\cal X$, $i$ over $[n]$, and $A$ over (possibly empty) subsets of $\cal X$. We use the notation $\sum^*$ to denote a sum over disjoint subsets of $\cal X$, and $\sqcup$ to denote disjoint union. We find
\begin{multline} \label{P_exp_step1}
\P(\Phi(I_1) = k_1, \dots, \Phi(I_n) = k_n) = \E \qBB{\pBB{\prod_i \ind{\Phi(I_i) = k_i}} \prod_{x} \pBB{\sum_i \ind{Z_x \in I_i} + \pBB{1 - \sum_i \ind{Z_x \in I_i}}}}
\\
=\sum^*_{A_1, \dots, A_n} \pBB{\prod_i \ind{\abs{A_i} = k_i}} \E \qBB{\pBB{\prod_i \prod_{x \in A_i} \ind{Z_x \in I_i}} \pBB{\prod_{x \notin A_1 \sqcup \dots \sqcup A_n} \pBB{1 - \sum_i \ind{Z_x \in I_i}}}}\,.
\end{multline}
We expand the last product by expanding successively each factor, in increasing order of $x$, and stopping the expansion if $m +1$ terms $\ind{Z_x \in I_i}$ have been generated. Using the estimate $0 \leq 1 - \sum_i \ind{Z_x \in I_i} \leq 1$, we therefore find
\begin{multline*}
\prod_{x \notin A_1 \sqcup \dots \sqcup A_n} \pBB{1 - \sum_i \ind{Z_x \in I_i}} = \sum_{B_1, \dots, B_n \subset \cal X \setminus (A_1 \sqcup \cdots \sqcup A_n)}^* \ind{\sum_i \abs{B_i} \leq m} (-1)^{\sum_i \abs{B_i}} \pBB{\prod_i \prod_{x \in B_i} \ind{Z_x \in I_i}}
\\
+ O \pBB{\sum_{B_1, \dots, B_n \subset \cal X \setminus (A_1 \sqcup \cdots \sqcup A_n)}^* \ind{\sum_i \abs{B_i} = m+1} \pBB{\prod_i \prod_{x \in B_i} \ind{Z_x \in I_i}}}\,. 
\end{multline*}
Plugging this into \eqref{P_exp_step1}, we find
\begin{multline*}
\P(\Phi(I_1) = k_1, \dots, \Phi(I_n) = k_n)
\\
= \sum^*_{A_1, B_1,\dots, A_n, B_n} \pBB{\prod_i \ind{\abs{A_i} = k_i}} \ind{\sum_i \abs{B_i} \leq m} (-1)^{\sum_i \abs{B_i}} \E \pBB{\prod_i \prod_{x \in A_i \sqcup B_i} \ind{Z_x \in I_i}}
\\
+ O \pBB{\sum^*_{A_1, B_1,\dots, A_n, B_n} \pBB{\prod_i \ind{\abs{A_i} = k_i}} \ind{\sum_i  \abs{B_i} = m+1} \E \pBB{\prod_i \prod_{x \in A_i \sqcup B_i} \ind{Z_x \in I_i}}}\,.
\end{multline*}
Writing $\ell_i = \abs{B_i}$ and using the exchangeability of $(Z_x)_{x \in \cal X}$, we find that the first term is equal to
\begin{equation*}
\sum_{\ell_1, \dots, \ell_n \in \N } (-1)^{\sum_i \ell_i } \ind{\sum_i \ell_i \leq m}  \frac{1}{k_1 ! \ell_1! \cdots k_n ! \ell_n!} \, q_\Phi(I_1^{k_1 + \ell_1} \times \cdots \times I_n^{k_n + \ell_n})\,,
\end{equation*}
The error term is dealt with analogously. \qed

\section{Lipschitz continuity of the intensity measure $\rho$ -- proof of Lemma~\ref{lem:rho_Lipschitz}} \label{sec:lipschitz}

Recall the notations \eqref{def_Gg} as well as $\ang{x}$ from the proof of Lemma \ref{lem:kappa_condition}. Let $s \in \R$ be given. Define
\begin{equation*}
\ell_0 \deq \min\h{\ell \in \Z \col s + \theta(\fra u) (\ang{d \fra u} + \ell) \geq 0}\,.
\end{equation*}
Then for any $t$ satisfying
\begin{equation} \label{t_condition}
-\theta(\fra u) (\ang{d \fra u} + \ell_0) \leq t \leq s
\end{equation}
we have
\begin{equation} \label{small_l_est}
\sum_{\ell < \ell_0} \fra u^{\ang{d \fra u} + \ell} \, G \pb{t + \theta(\fra u) (\ang{d \fra u} + \ell)} \asymp
\sum_{\ell < \ell_0} \fra u^{\ang{d \fra u} + \ell} \asymp \fra u^{\ang{d \fra u} + \ell_0 - 1}\,,
\end{equation}
since $\frac{1}{2} \leq G(x) \leq 1$ for $x \leq 0$. Since the left-hand side of \eqref{small_l_est} for $t = s$  is bounded by $\rho(E_s)$, we conclude that
\begin{equation} \label{l_0_est}
\ell_0 \log \fra u \leq C \log \fra u + \log \rho(E_s)\,.
\end{equation}

To estimate the terms $\ell > \ell_0$, we use the elementary estimates
\begin{equation} \label{G_g_estimates}
G(x) \asymp \frac{1}{1+x} \, g(x)\,, \qquad g(x+y) \leq g(x) \, \ee^{-\frac{1}{2} y^2}
\end{equation}
for $x,y \geq 0$.
Thus we find, for any $t$ satisfying \eqref{t_condition},
\begin{equation} \label{large_l_est}
\sum_{\ell > \ell_0} \fra u^{\ang{d \fra u} + \ell} \, G \pb{t + \theta(\fra u) (\ang{d \fra u} + \ell)} \lesssim 
\fra u^{\ang{d \fra u} + \ell_0} \, G \pb{t + \theta(\fra u) (\ang{d \fra u} + \ell_0)}\,.
\end{equation}

From \eqref{small_l_est} and \eqref{large_l_est}, recalling $G(0) = \frac{1}{2}$, it is not hard to conclude that there exists $t_0$ satisfying \eqref{t_condition} such that
\begin{equation}\label{E_st_comp}
\fra u^{\ang{d \fra u} + \ell_0} \, G \pb{t_0 + \theta(\fra u) (\ang{d \fra u} + \ell_0)} \asymp \rho(E_s)\,.
\end{equation}
By using \eqref{G_g_estimates} and the definition of $g$, we conclude that
\begin{equation} \label{t_0_estimate}
t_0 + \theta(\fra u) (\ang{d \fra u} + \ell_0) \lesssim \sqrt{C + (\ang{d \fra u} + \ell_0) \log \fra u - \log \rho(E_s)} \lesssim \sqrt{\log \fra u}\,,
\end{equation}
where the last step follows from \eqref{l_0_est}.

We now have what we need to estimate the derivative of $\rho(E_s)$. We begin with
\begin{equation*}
-\frac{\dd}{\dd s} \sum_{\ell < \ell_0} \fra u^{\ang{d \fra u} + \ell} \, G \pb{s + \theta(\fra u) (\ang{d \fra u} + \ell)} 
= \sum_{\ell < \ell_0} \fra u^{\ang{d \fra u} + \ell} \, g \pb{s + \theta(\fra u) (\ang{d \fra u} + \ell)} \lesssim \sum_{\ell < \ell_0} \fra u^{\ang{d \fra u} + \ell} \lesssim \rho(E_s)\,,
\end{equation*}
since the left-hand side of \eqref{small_l_est} is bounded by $\rho(E_s)$ for $t = s$.
Moreover,
\begin{align*}
-\frac{\dd}{\dd s} \sum_{\ell > \ell_0} \fra u^{\ang{d \fra u} + \ell} \, G \pb{s + \theta(\fra u) (\ang{d \fra u} + \ell)} 
&= \sum_{\ell > \ell_0} \fra u^{\ang{d \fra u} + \ell} \, g \pb{s + \theta(\fra u) (\ang{d \fra u} + \ell)} 
\\
&\lesssim
\fra u^{\ang{d \fra u} + \ell_0} \, g \pb{s + \theta(\fra u) (\ang{d \fra u} + \ell_0)}\,,
\end{align*}
by the same argument as in \eqref{large_l_est}.

What remains, therefore, is to estimate
\begin{align*}
\fra u^{\ang{d \fra u} + \ell_0} \, g \pb{s + \theta(\fra u) (\ang{d \fra u} + \ell_0)} &\leq
\fra u^{\ang{d \fra u} + \ell_0} \, g \pb{t_0 + \theta(\fra u) (\ang{d \fra u} + \ell_0)}
\\
&\lesssim \fra u^{\ang{d \fra u} + \ell_0}\, \pb{1 + t_0 + \theta(\fra u) (\ang{d \fra u} + \ell_0)} G \pb{t_0 + \theta(\fra u) (\ang{d \fra u} + \ell_0)}
\\
&\lesssim \sqrt{\log \fra u} \, \rho(E_{s})\,,
\end{align*}
where in the first step we used that $g(x)$ is decreasing for $x \geq 0$, in the second step we used \eqref{G_g_estimates}, and in the last step we used \eqref{E_st_comp} and \eqref{t_0_estimate}. \qed

\section{Approximation of binomial random variables}

In this appendix we collect some standard quantitative approximation results for binomial and Poisson random variables.
We use the notation $\cal P_\mu$ to denote a Poisson random variable with parameter $\mu \geq 0$, and $\cal B_{n,p}$ to denote a binomial random variable with parameters $n \in \N^*$ and $p \in [0,1]$.

\begin{lemma}\label{lem:poisson_approx}
For $0 \leq v \leq \sqrt n$ and $0 \leq p \leq 1/\sqrt n$ we have 
\[ \P\big(\cal B_{n,p} = v \big) = \P(\cal P_{np} = v) \bigg( 1 + O \bigg( \frac{v^2}{n} + p^2 n \bigg) \bigg)\,. \] 
\end{lemma} 

\begin{proof} 
See \cite[Lemma~A.6]{ADK20}. 
\end{proof}

\begin{lemma}[Bennett's inequality] \label{lem:Bennett}
Recall the function $h$ from \eqref{eq:def_h}.
For $0 \leq \mu \leq n$ and  $a > 0$ we have
\begin{equation*}
\P(\cal B_{n,\mu/n} - \mu \geq a \mu) \leq \ee^{-\mu h(a)}\,, \qquad
\P(\cal B_{n,\mu/n} - \mu \leq - a \mu) \leq \ee^{-\mu a^2/2} \leq \ee^{-\mu h(a)}\,,
\end{equation*}
and $\frac{a^2}{2 (1 + a/3)}  \leq h(a) \leq \frac{a^2}{2}$. By taking $n \to \infty$, the same estimates hold with $\cal B_{n,\mu/n}$ replaced with $\cal P_\mu$.
\end{lemma}
\begin{proof}
See \cite[Section 2.7]{BLM13}.
\end{proof}

\begin{lemma}\label{lem:poisson_approx2}
For $v \geq 0$ and $0 \leq p \leq n^{-3/4}$ we have
\begin{equation} \label{Poisson_approx}
\P(\cal B_{n,p} \geq v) = \P(\cal P_{pn} \geq v) (1 + O(n^{-1/2})) + O( \ee^{-n^{1/4}})\,.
\end{equation}
\end{lemma}
\begin{proof}
The claim follows easily from Lemma \ref{lem:poisson_approx} combined with a tail estimate for $v \geq 10 \, n^{1/4}$ using Bennett's inequality (see Lemma~\ref{lem:Bennett}).
\end{proof}

\begin{lemma} \label{lem:Poisson_perturbation}
For $\abs{\epsilon} \leq 1/2$ and $v \geq 0$ we have
\begin{equation} \label{Poisson_perturbation}
\P(\cal P_{(1 + \epsilon)\mu} = v) = \P(\cal P_\mu = v) \ee^{O(\abs{\epsilon} (\mu + v))}\,.
\end{equation}
\end{lemma}
\begin{proof}
This is immediate from the definition of a Poisson random variable.
\end{proof}

The following result is a De Moivre-Laplace approximation of the Poisson distribution. We give a version with a quantitative error bound suitable for our needs. Recall the notations \eqref{def_Gg}.

\begin{lemma} \label{lem:ml_approx}
Let $0 \leq \xi \leq 1/6$. Then for $\mu \geq 1$ and $t \leq \mu^{\xi}$ we have
\begin{equation*}
\P(\cal P_\mu \geq \mu + \sqrt{\mu} t) = G(t) (1 + O(\mu^{3 \xi - 1/2}))\,.
\end{equation*}
\end{lemma}
\begin{proof}
We estimate
\begin{multline*}
\P(\cal P_\mu \geq \mu + \sqrt{\mu} t) = \sum_{s \in (\Z - \mu) / \sqrt{\mu}} \ind{s \geq t} \, \P(\cal P_\mu = \mu + \sqrt{\mu} s)
\\
=\sum_{s \in (\Z - \mu) / \sqrt{\mu}} \ind{s \geq t} \ind{\abs{s} \leq 2\mu^\xi} \, \P(\cal P_\mu = \mu + \sqrt{\mu} s) + O\pb{\P(\abs{\cal P_\mu - \mu} >  2 \mu^{1/2 + \xi})}\,.
\end{multline*}
By Lemma \ref{lem:Bennett}, the error term is bounded by $O(\ee^{-\mu^{2 \xi}})$.
Moreover, for $\abs{s} \leq 2 \mu^{\xi}$, Stirling's approximation followed by a Taylor expansion yields
\begin{equation*}
\P(\cal P_\mu = \mu + \sqrt{\mu} s) = \frac{1 + O(1/\mu)}{\sqrt{2 \pi (\mu + \sqrt{\mu} s)}} \ee^{-s^2/2 + O(\abs{s}^3 / \sqrt{\mu})} = \frac{1 + O(\mu^{3 \xi - 1/2})}{\sqrt{2 \pi \mu}} \, \ee^{-s^2 / 2}\,.
\end{equation*}
Thus,
\begin{equation*}
\P(\cal P_\mu \geq \mu + \sqrt{\mu} t) = \frac{1 + O(\mu^{3 \xi - 1/2})}{\sqrt{\mu}} \sum_{s \in (\Z - \mu) / \sqrt{\mu}} \ind{s \geq t} \ind{\abs{s} \leq 2\mu^{\xi}} \, g(s) + O(\ee^{-\mu^{2 \xi}})\,.
\end{equation*}
By the Euler-Maclaurin formula of order one, we find
\begin{align*}
\P(\cal P_\mu \geq \mu + \sqrt{\mu} t) &= \pb{1 + O(\mu^{3 \xi - 1/2})} \int_{t}^\infty \dd s \, \ind{\abs{s} \leq 2\mu^{\xi}} \, g(s)
\\
&\qquad + O \pbb{\mu^{-1/2} \ee^{-t^2/2} + \mu^{\xi - 1/2} \int_{t}^\infty \dd s \, \ind{\abs{s} \leq 2\mu^{\xi}} \, g(s)} + O(\ee^{-\mu^{2 \xi}})
\\
&=\pb{1 + O(\mu^{3 \xi - 1/2})} \int_{t}^\infty \dd s \, \ind{\abs{s} \leq 2\mu^{\xi}} \, g(s)\,,
\end{align*}
where in the last step we used that all error terms can be absorbed into the factor $O(\mu^{3 \xi - 1/2})$, because $t \leq \mu^{\xi}$ and hence $ \int_{t}^\infty \dd s \, \ind{\abs{s} \leq 2\mu^{\xi}} \, g(s) \asymp \frac{1}{t} \ee^{-t^2/2}$ for $t \geq 1$. 
Now the claim easily follows by estimating the contribution of the integral over $\abs{s} > 2 \mu^{\xi}$.
\end{proof}

\section{Auxiliary estimates and computations} \label{app:proof_auxiliary_estimates}

\begin{proof}[Proof of Lemma~\ref{lem:u_a}] 
We set $t = \log N/d$. Note that by assumption $t \geq 1/3$, and for the first claim of \eqref{eq:scaling_alpha_max} 
we have to show $\am \asymp \frac{t}{\log 4t}$. 
If $t \leq C$ for any $C \asymp 1$ then it is easy to see from \eqref{eq:def_alpha_max} 
that $\am \asymp 1$ and $\frac{t}{\log 4t} \asymp 1$. 
For $t \geq C$, we write $\am = \gamma \frac{t}{\log 4t}$ for some $\gamma >0$. From \eqref{eq:def_alpha_max}, we conclude that $\gamma \asymp 1$ for $t \geq C$ if $C \asymp 1$ is chosen large enough. 
This proves $\am \asymp \frac{t}{\log 4t}$. Since $t \gtrsim \log 4t$ for $t \geq 1/3$, we conclude the proof of the first claim of \eqref{eq:scaling_alpha_max}.
The second claim of \eqref{eq:scaling_alpha_max} follows easily from a Taylor expansion of the function $f$ from \eqref{def_f}.
\end{proof}

\begin{proof}[Proof of \eqref{eq:lower_bound_degree_cal_V}] 
From $d \leq 3 \log N$, the definition of $\am$ in \eqref{eq:def_alpha_max} and the monotonicity of $h^{-1}$, we first conclude that 
\begin{equation} \label{eq:am_minus_1_gtrsim_1} 
 \am - 1 = h^{-1}\bigg(\frac{\log N}{d} \bigg)\gtrsim 1\,. 
\end{equation}
This, in particular, implies that $ \am \log \am = \frac{\log N}{d} + \am - 1 \geq \frac{\log N}{d}$. 
Hence, we obtain 
\[ \am - \frac{c_* \delta}{\log \am} \geq \am \bigg( 1 - \frac{c_* \delta d}{\log N} \bigg) 
\geq \am \big( 1 - c_* C \big) \geq (1 + c) \vee \bigg( \frac{1}{2} \am \bigg)\,, 
\] 
where, in the last step, we chose $c_*$ sufficiently small and used \eqref{eq:am_minus_1_gtrsim_1}. 
This proves the first inequality in \eqref{eq:lower_bound_degree_cal_V}. 

The second inequality in \eqref{eq:lower_bound_degree_cal_V} follows directly from 
\eqref{eq:scaling_alpha_max}. 
\end{proof}

\begin{proof}[Proof of \eqref{eq:decomposition_H_minus_Lambda_v}] 
Observe that $H^{(x,r)} \f v = H \f v$ since $\supp \f v \subset B_r(x)$ while $H^{(x,r)} = \op{Adj}(\mathbb{G}|_{B_{r+1}(x)})/\sqrt{d}$. 

Since $A$ is a tree around $x$, we have $ A \f 1_x = \f 1_{S_1}$ and $A \f 1_{S_i} = \f 1_{S_{i+1}} +  \sum_{y \in S_{i-1}} \f 1_y N_y$ for $i \geq 1$. 
Consequently, for $i\geq 2$ and $ H = A/\sqrt{d}$, we obtain 
\begin{equation} \label{eq:identities_A_s_i} 
H  \f 1_x = \sqrt{\alpha_x} \f s_1, \qquad 
H \f s_1= \sqrt{\beta_x} \f s_{2} + \sqrt{\alpha_x} \f 1_x, 
\qquad  H \f s_i = \frac{\sqrt{\abs{S_{i+1}}}}{\sqrt{d\abs{S_i}}} \f s_{i+1} + \frac{1}{\sqrt{d\abs{S_i}}} \sum_{y \in S_{i-1}} \f 1_y N_y\,. 
\end{equation}
Hence, the definition of $\f v$ in \eqref{eq:def_v_intermediate_rigidity} yields 
\[ H \f v = \sqrt{\alpha_x} u_0 \f s_1 +  \sqrt{\beta_x} u_1 \f s_2 
+ \sqrt{\alpha_x} u_1 \f 1_x 
+ \sum_{i=2}^r \frac{\sqrt{\abs{S_{i+1}}}}{\sqrt{d \abs{S_i}}} u_i \f s_{i+1}  
+ \sum_{i=2}^r \frac{u_i}{\sqrt{d \abs{S_i}}} \sum_{y \in S_{i-1}} \f 1_y N_y\,. 
\] 
Owing to the definitions of $Z(\alpha_x,\beta_x)$ and $\Lambda = \Lambda(\alpha_x,\beta_x)$, we have 
\begin{equation} \label{eq:identities_lambda_u_i} 
\Lambda u_0 = \sqrt{\alpha_x} u_1, \qquad \Lambda u_1 = \sqrt{\beta_x} u_2 + \sqrt{\alpha_x} u_0, 
\qquad \Lambda u_2 = \sqrt{\beta_x} u_1 + u_3, \qquad \Lambda u_{i} = u_{i-1} + u_{i+1}  
\end{equation}
for $i \geq 3$.  Hence, 
\[ \Lambda \f v = \sum_{i=0}^r \Lambda u_i \f s_i = \sqrt{\alpha_x} u_1 \f 1_x  + 
\f s_1 ( \sqrt{\beta_x} u_2 + \sqrt{\alpha_x} u_0) + \f s_2 (\sqrt{\beta_x} u_1 + u_3) + 
\sum_{i=3}^r \f s_i ( u_{i-1} + u_{i+1} )\,. 
\] 
Thus, we get 
\[ \begin{aligned} 
\big( H - \Lambda\big) \f v  = \, & - \f s_1 \sqrt{\beta_x} u_2 - 
\f s_2  u_3 + 
 \sum_{i=2}^r \frac{\sqrt{\abs{S_{i+1}}}}{\sqrt{d \abs{S_i}}} u_i \f s_{i+1}  
+ \sum_{i=2}^r \frac{u_i}{\sqrt{d \abs{S_i}}} \sum_{y \in S_{i-1}} \f 1_y N_y 
-  \sum_{i=3}^r \f s_i  ( u_{i-1} + u_{i+1} )  \\ 
  = \, & \sum_{i=2}^{r-1} \f s_{i +1} u_i \bigg( \frac{\sqrt{\abs{S_{i+1}}}}{\sqrt{d \abs{S_i}}} - 1  \bigg)
+ \frac{\sqrt{\abs{S_{r+1}}}}{\sqrt{d \abs{S_r}}} u_r \f s_{r+1}
+ \frac{u_2}{\sqrt{d \abs{S_2}}} \sum_{y \in S_1} \f 1_y \bigg( N_y - \frac{\abs{S_2}}{\abs{S_1}} \bigg) 
+ \f s_r  u_{r+1}  \\ 
 & + \sum_{i=3}^{r} \frac{u_i}{\sqrt{d \abs{S_i}}} \sum_{y \in S_{i-1}} \f 1_y \big( N_y  - d \big) 
+ \sum_{i=3}^r \f s_{i-1} u_i \bigg( \frac{\sqrt{d\abs{{S_{i-1}}}}}{\sqrt{\abs{{S_{i}}}}} - 1 \bigg) \\ 
 = \, & \f w_2 + \f w_3 + \f w_4\,, 
\end{aligned} 
\] 
which completes the proof of \eqref{eq:decomposition_H_minus_Lambda_v}.
\end{proof}

\begin{proof}[Proof of \eqref{eq:scalar_v_H_minus_Lambda_v}] 
Since $\supp \f v \subset B_r(x)$ and $H$ is symmetric, we have 
\[ \scalar{\f v }{(H^{(x,r)}-\Lambda) \f v} = \scalar{\f v }{(H-\Lambda) \f v} = \sum_{i=0}^r u_i^2 \scalar{\f s_i}{(H-\Lambda)\f s_i} + 2 \sum_{i=0}^{r-1} \sum_{j=i+1}^{r} u_i u_j \scalar{\f s_j}{(H-\Lambda) \f s_i}\,. \] 
We conclude that $\scalar{\f s_i}{(H-\Lambda)\f s_i} = - \Lambda$ and
from \eqref{eq:identities_A_s_i}, for $j > i$, that 
\[ \scalar{\f s_{j}}{(H-\Lambda) \f s_i} = \scalar{\f s_{j}}{H\f s_i}= \begin{cases} 
0 & \text{ if }j \geq i + 2\\ 
\sqrt{\alpha_x} & \text{ if }i = 0,~j = 1\\ 
\sqrt{\beta_x} & \text{ if } i = 1, ~ j=2\\ 
\frac{\sqrt{\abs{S_{i+1}}}}{\sqrt{d\abs{S_i}}} & \text{ if } i \geq 2, ~j = i+1\,. 
\end{cases} 
\] 
Thus, \eqref{eq:scalar_v_H_minus_Lambda_v} follows since \eqref{eq:identities_lambda_u_i} yields 
\begin{align*} 
 - \sum_{i=0}^r \Lambda u_i^2 & = -2 ( u_0u_1\sqrt{\alpha_x} + u_1 u_2 \sqrt{\beta_x}) - \mathfrak{d} u_2 u_3 - \sum_{i=3}^r u_i  ( u_{i-1} + u_{i+1}) \\ 
& =  -2 ( u_0u_1\sqrt{\alpha_x} + u_1 u_2 \sqrt{\beta_x}) 
- 2 \sum_{i=2}^{r-1} u_i u_{i+1} -  u_r u_{r+1}\,. \qedhere
\end{align*}
\end{proof}

\begin{proof}[Proof of Lemma~\ref{lem:beta_x_assumption_beta_pr_checked}] 
If $d > (\log N)^{3/4}$ then $\abs{S_1(x)} = \alpha_x d \geq 2 d$ for all $x \in \cal U$ and \eqref{eq:beta_x_equal_1_error_term_on_cal_U} 
follows from \cite[Lemma~5.4 (i)]{ADK19}. 
If $d \leq (\log N)^{3/4}$ then $\abs{S_1(x)} = \alpha_x d \gtrsim \am d$ for all $x \in \cal U$. 
Thus,  \cite[Lemma~5.4 (i)]{ADK19} and \eqref{eq:scaling_alpha_max} imply \eqref{eq:beta_x_equal_1_error_term_on_cal_U}. Finally, 
 \eqref{assumption_beta_pr} with $\alpha = \alpha_x$ and $\beta = \beta_x$ follows directly from 
 \eqref{eq:beta_x_equal_1_error_term_on_cal_U} and \eqref{eq:UsetDef}. 
\end{proof}

\medskip

\paragraph{Acknowledgements}
The authors acknowledge funding from the European Research Council (ERC) under the European Union’s Horizon 2020 research and innovation programme, grant agreement No.\ 715539\_RandMat and the Marie Sklodowska-Curie grant agreement No.\ 895698.  Funding from the Swiss National Science Foundation through the NCCR SwissMAP grant and the grant 200020--200400 is also acknowledged.

\bibliography{bibliography} 
\bibliographystyle{amsplain-nodash}

\noindent
Johannes Alt (\href{mailto:johannes.alt@unige.ch}{johannes.alt@unige.ch}) -- University of Geneva and New York University.
\\
Rapha\"el Ducatez (\href{mailto:raphael.ducatez@ens-lyon.fr}{raphael.ducatez@ens-lyon.fr}) -- ENS Lyon, Unité de Mathématiques Pures et Appliqués (UMPA).
\\
Antti Knowles (\href{mailto:antti.knowles@unige.ch}{antti.knowles@unige.ch}) -- University of Geneva.

\end{document}